\documentclass[reqno,11pt, a4paper]{amsart}

\usepackage[T1]{fontenc} 

\usepackage{amssymb}
\usepackage{amsmath}
\usepackage{amsfonts}
\usepackage{amsthm}
\usepackage{mathtools}
\usepackage{mathabx}
\usepackage{fullpage}
\usepackage{comment}
\usepackage[dvipsnames]{xcolor}
\usepackage{graphicx}
\usepackage{float}

\usepackage{enumitem} 

\usepackage[export]{adjustbox}
\usepackage{bbold}
\usepackage[percent]{overpic}

\usepackage{booktabs,array}
\usepackage{url} 
\usepackage{hyperref}

\usepackage{pgfplots}
\pgfplotsset{compat=1.6}

\usetikzlibrary{quotes,angles}
\usetikzlibrary{decorations.pathreplacing}

\usepackage{mdframed}

\usepackage{caption}
\usepackage{subcaption}

\usepackage[nocompress]{cite}

\usepackage{todonotes}
\theoremstyle{plain}
\newtheorem{theorem}{Theorem}[section]
\newtheorem{lemma}[theorem]{Lemma}
\newtheorem{proposition}[theorem]{Proposition}

\newtheorem*{conjecture*}{Conjecture} 
 
\newtheorem{sublemma}[theorem]{\normalfont Claim}
 \numberwithin{equation}{section}

\theoremstyle{definition}
\newtheorem{definition}[theorem]{Definition}

\theoremstyle{remark}
\newtheorem{remark}[theorem]{Remark}

\newenvironment{subproof}[1][\proofname]{%
  \proof[Proof of claim]%
  \renewcommand\qedsymbol{\Large$\square$}%
}{\endproof}

\usepackage{scalerel,stackengine}

\usepackage{soul}

 \let \leq \leqslant
 \let \geq \geqslant

\DeclareMathOperator{\sgn}{sgn}

\DeclareMathOperator{\unif}{unif}

\DeclareMathOperator{\support}{supp}
 
\DeclareMathOperator{\dist}{dist}

\usepackage{fancyhdr}
\usepackage{graphics}
\usepackage{graphicx}

\definecolor{detailcolor00}{rgb}{0.4405, 0.204, 0.343}
\definecolor{detailcolor01}{rgb}{0.546, 0.215, 0.352}
\definecolor{detailcolor02}{rgb}{0.675, 0.247, 0.387} 
\definecolor{detailcolor03}{rgb}{0.775, 0.317, 0.455}
\definecolor{detailcolor04}{rgb}{0.830, 0.421, 0.553} 
\definecolor{detailcolor05}{rgb}{0.831, 0.533, 0.663}
\definecolor{detailcolor06}{rgb}{0.779, 0.619, 0.775}
\definecolor{detailcolor07}{rgb}{0.724, 0.694, 0.827}
\definecolor{detailcolor08}{rgb}{0.687, 0.770, 0.880}
\definecolor{detailcolor09}{rgb}{0.671, 0.839, 0.904}
\definecolor{detailcolor10}{rgb}{0.659, 0.872, 0.882}

\title{Wilson loops in the abelian lattice Higgs model}

\author{Malin P. Forsstr\"om}
\address[Malin P. Forsstr\"om]{Department of Mathematics, KTH Royal Institute of Technology, 100 44 Stockholm, Sweden.}
\email{malinpf@kth.se}

\author{Jonatan Lenells}
\address[Jonatan Lenells]{Department of Mathematics, KTH Royal Institute of Technology, 100 44 Stockholm, Sweden.}
\email{jlenells@math.kth.se}

\author{Fredrik Viklund}
\address[Fredrik Viklund]{Department of Mathematics, KTH Royal Institute of Technology, 100 44 Stockholm, Sweden.}
\email{fredrik.viklund@math.kth.se}

\begin{document}

\begin{abstract}
    We consider the lattice Higgs model on \(\mathbb{Z}^4\) in the fixed length limit, with structure group given by \( \mathbb{Z}_n \) for \( n \geq 2 \). We compute the expected value of the Wilson loop observable to leading order when the inverse temperature and hopping parameter are both sufficiently large compared to the length of the loop. The leading order term is expressed in terms of a quantity arising from the related but much simpler \( \mathbb{Z}_n \) model, which reduces to the usual Ising model when \(n=2\). As part of the proof, we construct a coupling between the lattice Higgs model and the \( \mathbb{Z}_n \) model. 
\end{abstract}

\maketitle

\section{Introduction}

\subsection{Background}
 
Lattice gauge theories were introduced by Wilson~\cite{w1974} (see also Wegner~\cite{w1971}) as lattice approximations of quantum field theories known as (Yang--Mills) gauge theories. Gauge theories appear for instance in the Standard Model, where they describe fundamental particle interactions. One advantage of the lattice models is that they are immediately mathematically well-defined, whereas it remains a substantial problem to give rigorous constructions of relevant continuum theories~\cite{jw}. In fact, one possible approach to the latter problem is to try passing to a scaling limit with a lattice theory. 

A discrete gauge field configuration assigns to each lattice edge an element of a given group $G$, known as the structure (or gauge-) group, and can be thought of as a discrete approximation of a connection one-form on a principal $G$-bundle on an underlying smooth manifold. 
The simplest lattice gauge theory is a pure gauge theory, realized as a probability measure on gauge field configurations. The choice of structure group is part of the model and is fixed at the outset. The action that Wilson introduced in~\cite{w1974} (and which determines the probability measure) involves a discrete analog of curvature obtained by considering microscopic holonomies around lattice plaquettes. The definition is set up in such a way that the resulting discrete action is exactly invariant with respect to gauge transformations of the discrete configuration space, while still approximating the continuum Yang--Mills action in an appropriate lattice size scaling limit. 

In the physics community, lattice gauge theories have been studied as statistical mechanics models in their own right, and have also been used as a basis for numerical simulations of the corresponding continuum models. Recently there has been a renewed interest in the rigorous analysis of four-dimensional gauge theories in the mathematical community using probabilistic techniques (see, e.g.,~\cite{c2019, sc2019, flv2020,c2021,f2021, gs2021}). For example, in the papers~\cite{c2019, sc2019, flv2020}, the leading order terms for the expectation of Wilson loop observables (see below) were computed for pure gauge theories with Wilson action and finite structure groups in the limit when the size of the loop and inverse temperature tend to infinity simultaneously in such a way that the Wilson loop expectation is non-trivial. (Structure groups that are of primary relevance to unresolved problems in physics, such as \(SU(N)\), appear to be out of reach at the moment.) This type of result was first obtained by Chatterjee in~\cite{c2019} and is different from the more classical results on area versus perimeter law (see, e.g., \cite{mms1979,fs1979,bf,w1974,os1978, fs1982}), which concern the decay rate of the Wilson loop expectation as the size of the loop tends to infinity while the inverse temperature is kept \emph{fixed}. Part of the motivation for the results in this paper instead stems from the question of how to renormalize in order to be able to take a scaling limit. We will comment further on this below.

Pure gauge theories model only the gauge field itself, and in order to advance towards physically relevant theories, it is necessary to consider models that include matter fields interacting with the gauge field, see, e.g.,~\cite{fs1979, os1978}. 
In this paper, we consider lattice gauge theories with Wilson's action for the gauge field, coupled to a scalar Bosonic field with a quartic Higgs potential. The resulting theory is called the \emph{lattice Higgs model}. We analyze this model in a particular limiting ``fixed length'' regime where the relative weight of the potential is infinite, forcing the Higgs field to have unit length. We only consider finite abelian structure groups. 

We will discuss the relation to previous work in more detail below, but let us now mention that the models we consider here have received significant attention in the physics literature: see for instance the works~\cite{b1974,bdi1975ii,bdi1975iii}, where calculations to obtain critical parameters of these models were first performed, and~\cite{jsj1980, ks1984}, in which phase diagrams were sketched. For further background, as well as more references, we refer the reader to~\cite{fs1979} and~\cite{s1988}.

Just as in \cite{c2019,flv2020,sc2019} our main goal is to analyze the asymptotic behavior of Wilson loop expectations, that is, expected values of ``holonomy variables'', computed along large loops when the inverse temperature is also large compared to the perimeter of the loop. Wilson introduced such observables as a means to detect whether quark confinement occurs, see~\cite{w1974}, and they have since been a primary focus in the analysis of lattice gauge theories. Here we compute the leading order term of the Wilson loop expectation (in the sense of \cite{c2019,flv2020,sc2019}) in the lattice Higgs model in the fixed length regime. We show that the leading order term can be expressed in terms of the expectation of a certain observable in a simpler \( \mathbb{Z}_n \) model of interacting spins on vertices. (In the case $n=2$ this is the Ising model.) Our theorems are of a similar type as the main results of~\cite{c2019,flv2020,sc2019}, but require substantial additional work due to the presence of an interacting field. 
Along the way, we develop tools that we believe may be useful to analyze other gauge theories coupled with matter fields, as well as other observables in the lattice Higgs model. In particular, we construct a coupling, in the sense of probability measures, of the lattice Higgs model and the \( \mathbb{Z}_n \) model, thereby making precise certain ideas appearing in the physics literature, see, e.g.,~\cite[Section~C]{fs1979}. We note however that we do not consider the question of lattice fields acquiring ``mass'' due to the presence of a Higgs field and the breaking of an exact gauge symmetry (see Section~\ref{sec:unitary_gauge}), cf., e.g., Section 4 of \cite{os1978}.

\subsection{Preliminary notation}
For \( m \geq 2 \), the graph $\mathbb{Z}^m$ has a vertex at each point \( x \in \mathbb{Z}^m \) with integer coordinates and a non-oriented edge between nearest neighbors. We will work with oriented edges throughout this paper, and for this reason we associate to each non-oriented edge \( \bar e \) in \( \mathbb{Z}^m \) two oriented edges \( e_1 \) and \( e_2 = -e_1 \) with the same endpoints as \( \bar e \) and opposite orientations. 

Let \( d\mathbf{e}_1 \coloneqq (1,0,0,\ldots,0)\), \( d\mathbf{e}_2 \coloneqq (0,1,0,\ldots, 0) \), \ldots, \( d\mathbf{e}_m \coloneqq (0,\ldots,0,1) \) be oriented edges corresponding to the unit vectors in \( \mathbb{Z}^m \). We say that an oriented edge \( e \) is \emph{positively oriented} if it is equal to a translation of one of these unit vectors, i.e.,\ if there is a \( v \in \mathbb{Z}^m \) and a \( j \in \{ 1,2, \ldots, m\} \) such that \( e = v + d{\mathbf{e}}_j \). 
When two vertices \( x \) and \( y \) in \( \mathbb{Z}^m \) are adjacent, we sometimes write \( (x,y) \) to denote the oriented edge with endpoints at \( x \) and \( y \) which is oriented from \( x \) to \( y \).
If \( v \in \mathbb{Z}^m \) and \( j_1 <   j_2 \), then \( p = (v +  d\mathbf{e}_{j_1}) \land  (v+ d\mathbf{e}_{j_2}) \) is a positively oriented 2-cell, also known as a  \emph{positively oriented plaquette}. 
We let \( B_N \) denote the set \(   [-N,N]^m \) of \( \mathbb{Z}^m \), and we let \( V_N \), \( E_N \), and \( P_N \) denote the sets of oriented vertices, edges, and plaquettes, respectively, whose end-points are all in \( B_N \).

Whenever we talk about lattice gauge theory we do so with respect to some (abelian) group \( (G,+)  \), referred to as the \emph{structure group}. We also fix a unitary and faithful representation \( \rho \) of \( (G,+) \). In this paper, we will always assume that \( G = \mathbb{Z}_n \) for some \( n \geq 2 \) with the group operation \( + \) given by standard addition modulo \( n \). Also, we will assume that \( \rho \) is a one-dimensional representation of \( G \). We note that a natural such representation is \( j\mapsto e^{j \cdot 2 \pi i/n} \).

Now assume that a structure group \( (G,+) \), a one-dimensional unitary representation \( \rho \) of \( (G,+) \), and an integer \( N\geq 1 \) are given.
We let \( \Sigma_{E_N} \) denote the set of all  \( G \)-valued  1-forms \( \sigma \) on \( E_N \), i.e., the set of all \( G \)-valued functions \(\sigma \colon  e \mapsto \sigma_e \) on \( E_N \) such that \( \sigma_e =  -\sigma_{-e} \) for all \( e \in E_N \).
We let \( \hat{\Phi}_{V_N} \) denote the set of all \( \mathbb{C} \)-valued functions \( \phi \colon x \mapsto \phi_x\) on \( V_N \) which  satisfy \( \phi_x/|\phi_x| \in \rho(G) \) for \( x \in V_N \) and \( \phi_x = 0 \) for \( x \notin V_N \). Similarly, we let $\Phi_{V_N}$ denote the set of all \( \phi \in \hat \Phi_{V_N} \) such that \( \phi_x  \in \rho(G) \) for each \( x \in V_N \).
When \( \sigma \in \Sigma_{E_N } \) and \( p \in P_N \), we let \( \partial p \) denote the four edges in the oriented boundary of \( p \) (see Section~\ref{sec: plaquettes}), and define
\begin{equation*}
    (d\sigma)_p \coloneqq \sum_{e \in \partial p} \sigma_e.
\end{equation*} 

Elements \( \sigma \in \Sigma_{E_N} \) will be referred to as \emph{gauge field configurations}, and elements \( \phi \in \hat{\Phi}_{V_N} \) will be referred to as \emph{Higgs field configurations}.

\subsection{Lattice Higgs model and Wilson loops}

The \emph{Wilson action functional} for pure gauge theory is defined by (see, e.g.,~\cite{w1974})
\begin{equation*}
    S^W(\sigma) \coloneqq - \sum_{p \in P_N}  \rho( (d\sigma)_p), \quad \sigma \in \Sigma_{E_N}.
\end{equation*}

For  \( \sigma \in \Sigma_{E_N}\) and \( e=(x,y) \in E_N\), we define a \emph{discrete covariant derivative} of \( \phi\in \hat \Phi_{V_N} \) along \( e \) by
\begin{equation*}
D_{e}\phi \coloneqq \rho(\sigma_{-e})\phi_y  - \phi_x.
\end{equation*} 
Then, for \( \sigma \in \Sigma_{E_N} \) and \( \phi \in \hat{\Phi}_{V_N} \), we define the corresponding kinetic action by the Dirichlet energy
\begin{align*}
    &S^D(\sigma, \phi) \coloneqq \frac{1}{2} \sum_{e \in E_N} |D_e \phi|^2   = \frac{1}{2}\sum_{e=(x,y)\in E_N}  
    \overline{(\rho(\sigma_{-e})\phi_y - \phi_x)} (\rho(\sigma_{-e})\phi_y - \phi_x) 
    \\ &\qquad = \frac{1}{2}\sum_{e=(x,y)\in E_N} \Bigl( |\phi_x|^2 + |\phi_y|^2 -  \overline{\phi_y}\overline{\rho( \sigma_{-e})} \phi_x-   \overline{\phi_x}\rho( \sigma_{-e}) \phi_y  \Bigr)\\
    & \qquad= \sum_{x \in V_N} \bigl|\{e \in E_N \colon x \in \partial e\} \bigr| \, |\phi_x|^2 + S^I(\sigma, \phi),
\end{align*}
where the interaction term $S^I$ is given by
\begin{equation*}
    S^I(\sigma, \phi) \coloneqq 
    -\sum_{e=(x,y)\in E_N}  \overline{\phi_y} \rho(\sigma_{e}) \phi_x.
\end{equation*} 
We finally consider a quartic ``sombrero'' potential
\begin{equation*}
    V(\phi) \coloneqq \sum_{x \in V_N}\bigl(|\phi_x|^2-1 \bigr)^2, \quad \phi \in \hat \Phi_{V_N}.
\end{equation*}
Given \( \beta, \kappa ,\zeta \geq 0 \), the action \( S_{\beta,\kappa, \zeta} \) for lattice gauge theory with Wilson action coupled to a Higgs field on \( B_N \) is, for \( \sigma \in \Sigma_{E_N} \) and \( \phi \in \hat{\Phi}_{V_N} \),  defined by
\begin{equation}\label{eq: general action}
    \begin{split}
        &S_{\beta,\kappa,\zeta}(\sigma, \phi)  \coloneqq  \beta S^W(\sigma) + S^D(\sigma, \phi) + (\kappa-1)S^I(\sigma, \phi) + \zeta V(\phi)  
        \\&\qquad= -\beta \sum_{p \in P_N}  \rho( (d\sigma)_p) - \kappa\sum_{e=(x,y) \in E_N}  \overline{\phi_y}    \rho(\sigma_{e})  {\phi_x} 
         \\&\qquad\qquad
         + \zeta \sum_{x \in V_N}  \bigl(|\phi_x|^2-1 \bigr)^2 + \sum_{x \in V_N} \bigl|\{e \in E_N \colon x \in \partial e\} \bigr|\,  |\phi_x|^2. 
    \end{split}
\end{equation}
The corresponding Gibbs measure can be written
\[
d\mu_{N,\beta, \kappa, \zeta}(\sigma, \phi) = Z^{-1}_{N,\beta,\kappa, \zeta} e^{-S_{\beta,\kappa,\zeta}(\sigma, \phi)}
\prod_{e \in E_N^+} d\mu_G(\sigma_e) 
\prod_{x \in V_N} d\mu_{\rho(G)}(\phi_x/|\phi_x|) d\mu_{\mathbb{R}_+}(|\phi_x|),
\]
where \( E_N^+ \) denotes the set of positively oriented edges in \( E_N \), \( d\mu_G \) is the uniform measure on \( G \), \( d\mu_{\rho(G)} \) is the uniform measure on \( \rho(G) \), and \( \mu_{\mathbb{R}_+} \) is the Lebesgue measure on \( \mathbb{R}_+ \).
%Here $d\mu^H_{\Sigma_{E_N}}$ is the uniform measure on $\Sigma_{E_N}$ and $d\mu^H_{\rho(G),V_N} $ is product of the uniform measure on $\rho(G)^{|V_N|}$ and Lebesgue measure on $\mathbb{R}_+^{|V_N|}$.
We refer to this lattice gauge theory as the \emph{lattice Higgs model}.
Since \( \sigma \in \Sigma_{E_N} \), we have \( \rho(\sigma_{-e}) = \overline{\rho(\sigma_{e})} \), and hence \( S_{\beta,\kappa,\zeta}(\sigma, \phi) \) is real for all \( \sigma \in \Sigma_{E_N} \) and \( \phi \in \hat \Phi_{V_N} \).
The quantity \( \beta \) is known as the \emph{gauge coupling constant} (or inverse temperature),  \( \kappa \) is known as the \emph{hopping parameter}, and \( \zeta \) is known as the \emph{quartic Higgs self coupling}. 
We will work with the model obtained from this action in the \emph{fixed length limit} \( \zeta \to \infty\), in which the Higgs field concentrates on the unit circle, i.e., \(|\phi_x| = 1\) for all \( x \in V_N \). (This is sometimes called the \emph{London limit} in the physics literature.) We do not discuss the limit of the Gibbs measure corresponding to $S_{\beta, \kappa, \zeta}$ as $\zeta \to \infty$, but simply from the outset adopt the action resulting from only considering \( \phi \in \Phi_{V_N} \), so that \( |\phi_x| = 1 \) for all \( x \in V_N \):
\begin{equation}\label{eq: general fixed length action}
    \begin{split}
        &S_{\beta,\kappa, \infty}(\sigma, \phi) \coloneqq  \beta S^W(\sigma) + \kappa S^I(\sigma, \phi)
        \\&\qquad = -\beta \sum_{p \in P_N}  \rho( (d\sigma)_p) - \kappa\sum_{e=(x,y) \in E_N}  \overline{\phi_y}  \rho(\sigma_e){\phi_x}, \quad \sigma \in \Sigma_{E_N},\, \phi \in \Phi_{V_N}.
    \end{split}
\end{equation} 
We then consider a corresponding probability measure \(\mu_{N,\beta, \kappa, \infty}\) on \(\Sigma_{E_N} \times \Phi_{V_N}\) given by
\begin{equation}\label{eq: general fixed length measure}
    \mu_{N,\beta, \kappa, \infty}(\sigma, \phi)  \coloneqq
    Z_{N,\beta,\kappa, \infty}^{-1} e^{-S_{\beta,\kappa, \infty}(\sigma, \phi)} , \quad \sigma \in \Sigma_{E_N},\, \phi \in \Phi_{V_N},
\end{equation}
where \( Z_{N,\beta,\kappa, \infty}\) is a normalizing constant. This is the \emph{fixed length lattice Higgs model}. We let \( \mathbb{E}_{N,\beta,\kappa,\infty} \) denote the corresponding expectation. In the physics literature, this model has been studied in, e.g.,~\cite{ejjsln1987,jsj1980,fmf1986,cis2002} (see also the review~\cite{fs1979}). In~\cite{c1980,ks1984}, phase diagrams for the theory described by~\eqref{eq: general action} are sketched. 
These diagrams suggest (see also, e.g.,~\cite{s1988}) that for large \( \zeta \), the model described by~\eqref{eq: general action} behaves very similarly to the fixed length model considered here. 

Observables attached to simple oriented loops in the underlying lattice are among the main objects of interest in lattice gauge theories. 
If  \( \gamma \subseteq E_N \) is such that
\begin{enumerate}[label=(\arabic*)]
    \item if \( e \in \gamma \) then \( -e \notin \gamma \), and
    \item the edges in \( \gamma \) can be ordered so that they form an oriented loop,
\end{enumerate}
then we say that \( \gamma \) is a \emph{simple loop} in \( E_N \). If \( \gamma \) is a simple loop in \( E_N \) for some \( N \geq 1 \), then we say that \( \gamma \) is a simple loop in \( \mathbb{Z}^m \). 
Given a gauge field configuration \( \sigma \in \Sigma_{E_N} \), the  \emph{Wilson loop observable} \( W_\gamma \) is defined by considering the discrete holonomy along \(\gamma\):
\begin{equation}\label{Wilsonloopdef}
    W_\gamma \coloneqq  \rho \Bigl( \sum_{e \in \gamma} \sigma_e \Bigr).
\end{equation}
Wilson argued that the decay rate of the expected value of \(W_\gamma\) carries information about whether \emph{quark confinement} occurs in the model, see~\cite{w1974}.
 
Figure~\ref{fig: simulation} shows the results of two simulations of the fixed length lattice Higgs model on \( \mathbb{Z}^2. \)
\begin{figure}[hbt]
    \centering
    \begin{subfigure}[t]{0.32\textwidth}\centering
        \centering
        \includegraphics[width=\textwidth]{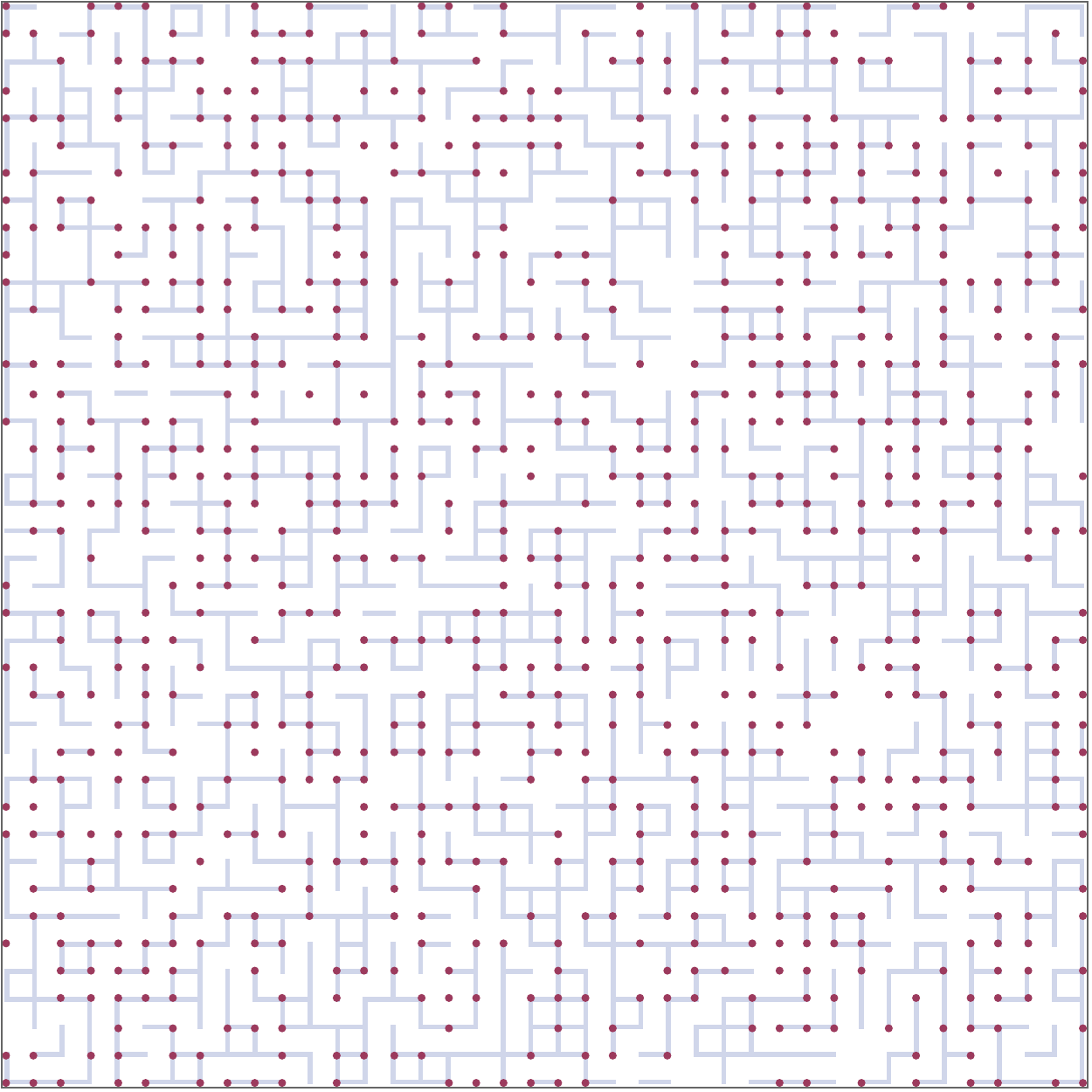}
        \caption{\(\kappa=0\)}
        \label{subfig: config1}
    \end{subfigure}
    \hfil
    \begin{subfigure}[t]{0.32\textwidth}\centering
        \centering
        \includegraphics[width=\textwidth]{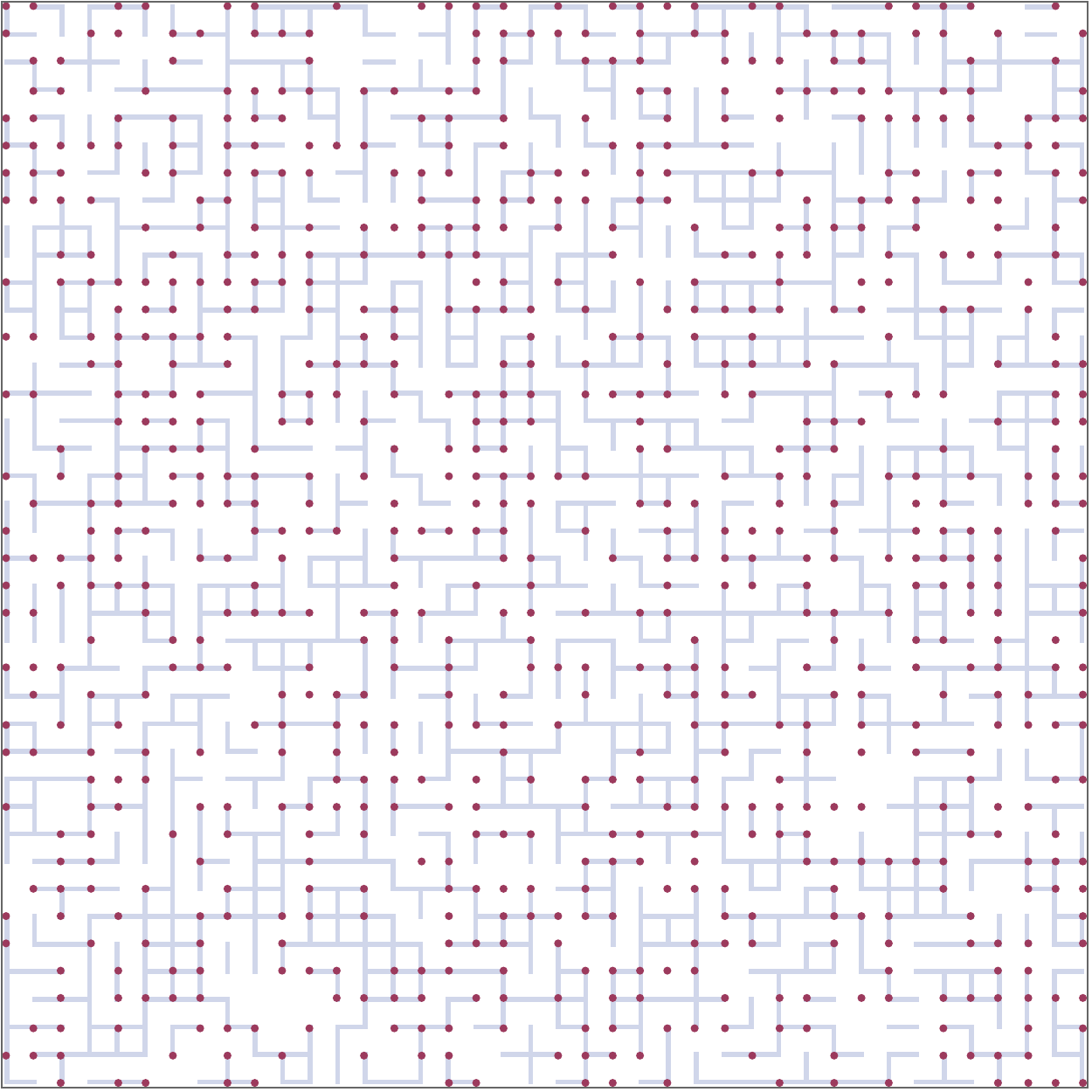}
        \caption{\(\kappa=0.3\)}
        \label{subfig: config2}
    \end{subfigure}
    \hfil
    \begin{subfigure}[t]{0.32\textwidth}\centering
        \centering
        \includegraphics[width=\textwidth]{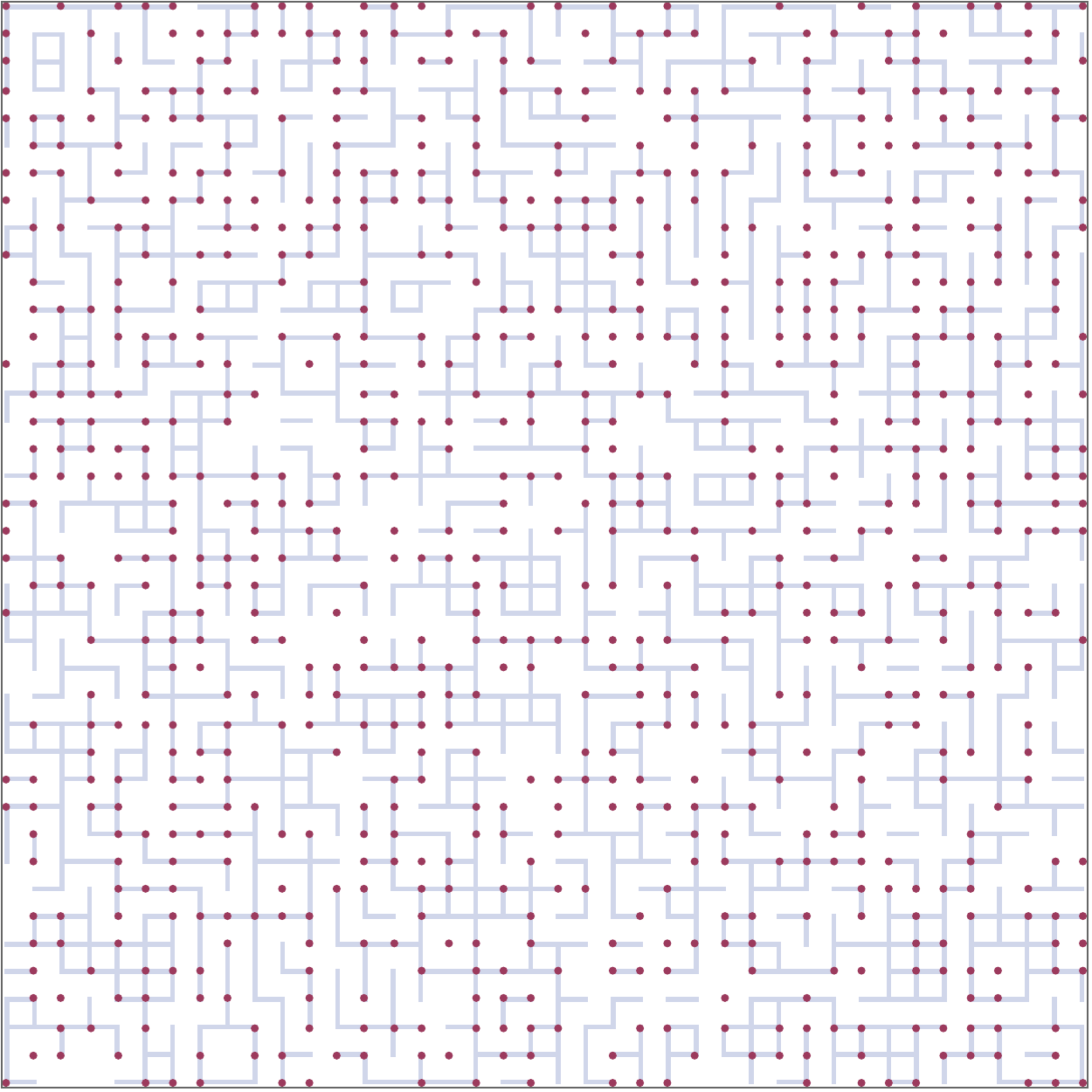}
        \caption{\(\kappa=0.5\)}
        \label{subfig: config3}
    \end{subfigure}
    
    \begin{subfigure}[t]{0.32\textwidth}\centering
        \centering
        \includegraphics[width=\textwidth]{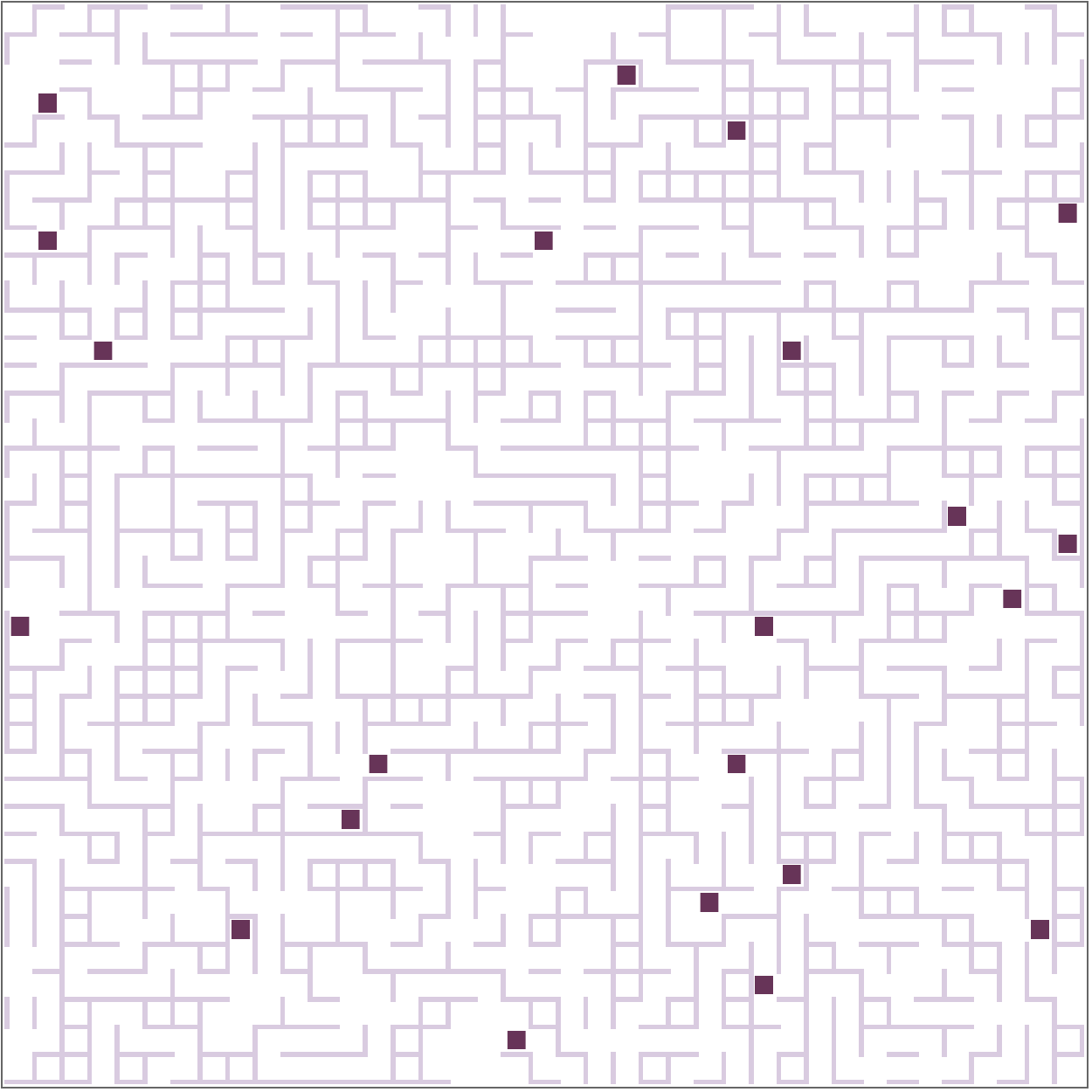}
        \caption{\(\kappa=0\)}
        \label{subfig: config4}
    \end{subfigure}
    \hfil
    \begin{subfigure}[t]{0.32\textwidth}\centering
        \centering
        \includegraphics[width=\textwidth]{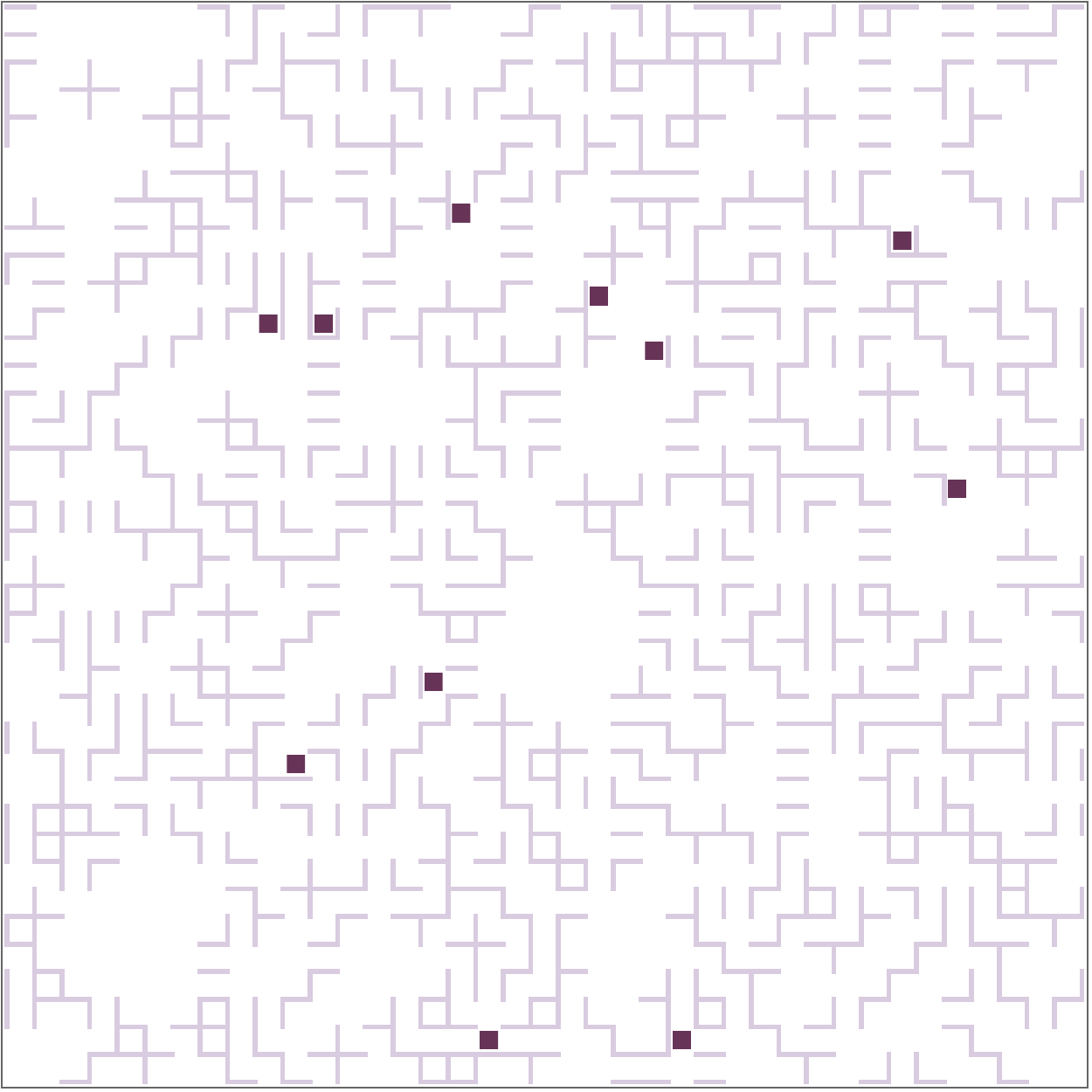}
        \caption{\(\kappa=0.3\)}
        \label{subfig: config5}
    \end{subfigure}
    \hfil
    \begin{subfigure}[t]{0.32\textwidth}\centering
        \centering
        \includegraphics[width=\textwidth]{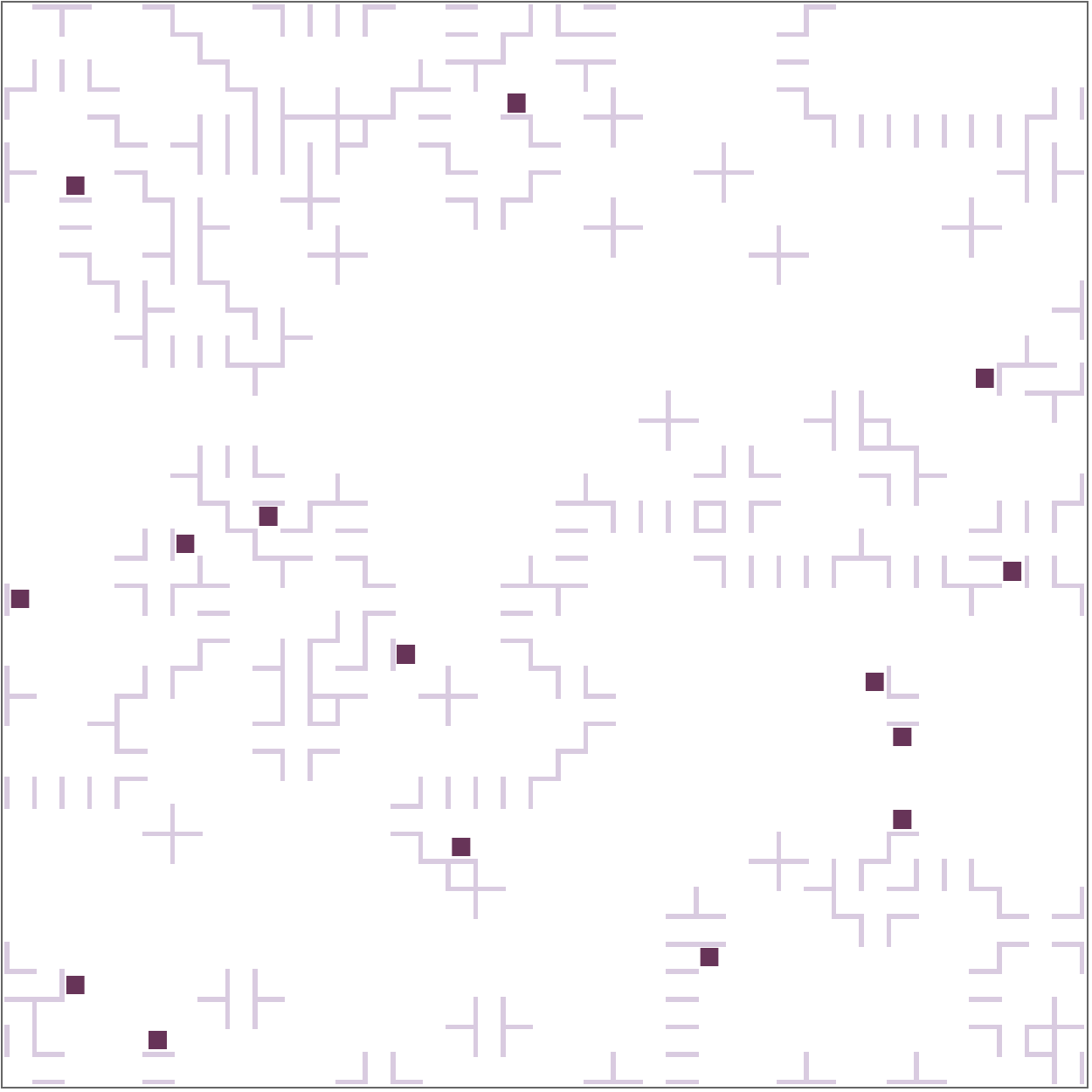}
        \caption{\(\kappa=0.5\)}
        \label{subfig: config6}
    \end{subfigure}
    \caption{In~\subref{subfig: config1},~\subref{subfig: config2},~and~\subref{subfig: config3}, we draw configurations \( (\sigma,\phi) \) obtained by simulation of the fixed length lattice Higgs model on a subset of \( \mathbb{Z}^2\), using \( G = \mathbb{Z}_2 \), \( \beta = 2.2 \), and  free boundary conditions. In~\subref{subfig: config1},~\subref{subfig: config2},~and~\subref{subfig: config3}, the colored edges correspond to edges \( e \) where the gauge field \( \sigma_e \neq 0 \), and the red dots to vertices \( x \) where the Higgs field \( \phi_x \neq 1 \).
    In~\subref{subfig: config4},~\subref{subfig: config5},~and~\subref{subfig: config6}, we draw the same configurations as in~\subref{subfig: config1},~\subref{subfig: config2},~and~\subref{subfig: config3}, but here the colored edges instead correspond to edges \( e =(x,y)\) with
    \( \overline{\phi_y}\rho(\sigma_e) \phi_x\neq 1 \), and the colored squares to plaquettes \( p \) with \( (d\sigma)_p \neq 0 \).}
    \label{fig: simulation}
\end{figure}

\subsection{Main results}

We define the measure \(\mu_{N,\infty, \kappa}\) on \(\Sigma_{E_N}^0 \coloneqq \{ \sigma \in \Sigma_{E_N} \colon d\sigma = 0 \}  \) by
\begin{equation}\label{eq: Zn model def}
    \mu_{N,\infty, \kappa} (\sigma) \coloneqq Z_{N,\infty,\kappa}^{-1} e^{\kappa \sum_{e \in E_N} \rho( \sigma_e)} ,\quad \sigma \in \Sigma_{E_N}^0,
\end{equation}
where \( Z_{N,\infty,\kappa} \) is a normalizing constant, and let  \(\mathbb{E}_{N,\infty, \kappa} \) denote the corresponding expectation. We note that if \( G = \mathbb{Z}_2 \), then \( \mu_{N,\infty,\kappa} \) represents the gradient field of the standard Ising model in the following sense. If \( \eta \in \mathbb{Z}_2^{V_N}\) is a spin configuration on \( V_N \), chosen according to the Ising model with coupling parameter \( \kappa \), and we for \( e = (x,y) \in E_N \) let \(  \sigma_e \coloneqq \eta_y-\eta_x \), then \( (\sigma_e)_{e \in E_N} \sim \mu_{N,\infty, \kappa} \).
In the same way, if \( G = \mathbb{Z}_n \) for some \( n \geq 3 \), then \( \mu_{N,\infty,\kappa} \) describes the distribution of the gradient field of the \( \mathbb{Z}_n \) model (also known as the clock model).

For \( \beta ,  \kappa \geq 0 \), we let \( \langle W_\gamma \rangle_{\beta,\kappa} \) and \( \langle W_\gamma \rangle_{\infty,\kappa} \) denote the infinite volume limits
\begin{equation}\label{infinitevolumelimits}
    \langle W_\gamma \rangle_{\beta,\kappa} \coloneqq \lim_{N \to \infty} \mathbb{E}_{N,\beta,\kappa,\infty}[W_\gamma] \quad \text{and} \quad \langle W_\gamma \rangle_{\infty,\kappa} \coloneqq \lim_{N \to \infty} \mathbb{E}_{N,\infty,\kappa}[W_\gamma],
\end{equation}
which are guaranteed to exist by the Ginibre inequalities (see Section~\ref{sect:ginibre}).

Finally, given a simple loop \( \gamma \), we let \( \gamma_c \) denote the set of \emph{corner edges} in \( \gamma \), that is, the set of edges   \( e \in \gamma \) which share a plaquette with at least one other edge in \( \pm \gamma \).

We are now ready to state our first main result.
\begin{theorem}\label{theorem: main result Z2}
    Consider the fixed length lattice Higgs model on \( \mathbb{Z}^4 \), with structure group \(G = \mathbb{Z}_2\), and representation \( \rho\colon G \to \mathbb{C} \) given by \( \rho(0) = 1 \) and \( \rho(1)=-1 \). 
    Suppose that \( \beta_0,\kappa_0 \geq 0\) satisfy
    \begin{enumerate}[label=\textnormal{[a\arabic*]}]
        \item \( e^{-4\beta_0}  + 4e^{-4\kappa_0} < 1\),
        \item \( 2^{1/3} 5  e^{-4\beta_0}  e^{-4\kappa_0/6} <1\), and
        \item \( 18^2 e^{-4\kappa_0}(2 + e^{-4\kappa_0})<1\).
    \end{enumerate}
    Then there is a constant \(C = C(\beta_0,\kappa_0)\), such that for any \( \beta \geq \beta_0 \) and \( \kappa \geq \kappa_0 \) with \( 3 \beta \geq 2\kappa \), and any  simple oriented loop \( \gamma \) in \( \mathbb{Z}^4 \), we have 
    \begin{equation}\label{eq: main result Z2}
      \begin{split}
        \Bigl|\langle W_\gamma\rangle_{\beta,\kappa}
        -
       \Upsilon_{\gamma,\beta,\kappa} \Bigr|
        \leq
        C \Bigl( e^{-4(\beta+\kappa/6)} + \sqrt{|\gamma_c|/|\gamma|}\Bigr)^{2/7},
    \end{split}
        \end{equation}
    where 
    \begin{equation}\label{Upsilondef}  \Upsilon_{\gamma, \beta,\kappa} \coloneqq  \bigl\langle e^{-2|\gamma|e^{-24\beta-4\kappa}  (1 + \frac{|\{ e \in \gamma \colon \sigma_e= 1\}|}{|\gamma|}(e^{ 8\kappa}-1))}\bigr\rangle_{\infty,\kappa}.
    \end{equation}
\end{theorem}

Notice that $\ell/\ell_c$ and $\beta$ must both tend to infinity simultaneously for the theorem statement to be non-trivial. In particular, Theorem~\ref{theorem: main result Z2} gives no information about a confinement/deconfinement phase transition nor about related quantitative information such as the exact perimeter law decay rate at fixed but sufficiently large $\beta, \kappa$. Conversely, Theorem~\ref{theorem: main result Z2} is not implied by any known such result. 

Roughly speaking, the conclusion of Theorem~\ref{theorem: main result Z2} may be interpreted as follows. We consider several cases.
\begin{enumerate}
    \item If \( |\gamma| e^{-24\beta-4\kappa}  \) is very large, then \(\Upsilon_{\gamma,\beta, \kappa} \approx 0 \).
    \item If \( |\gamma| e^{-24\beta-4\kappa} \) is very small, then \(  \Upsilon_{\gamma,\beta, \kappa} \approx 1 \). (This is immediate if \( \kappa \) is bounded from above, and follows in general because \( \mu_{N, \infty, \kappa}( \{ \sigma \in \Sigma_{E_N}^0 \colon \sigma_e = 1 \})\lesssim e^{-32\kappa} \) for any fixed edge \( e  \), see Proposition~\ref{proposition: minimal Ising}. 
    \item If \( |\gamma|\asymp e^{24\beta+4\kappa} \), then \(  \Upsilon_{\gamma,\beta, \kappa}  \) is bounded away from both $0$ and $1$. Hence the theorem identifies the correct relation between the size of the loop and $\beta, \kappa$ needed to obtain a non-trivial quantity asymptotically. 
\end{enumerate} 

\begin{remark}
    Part of the motivation for considering the type limit as in Theorem~\ref{theorem: main result Z2} is the following. In the statement, the lattice size is fixed but the size of the loop grows. Essentially equivalently, one may fix a loop and then let the lattice size tend to zero. In that setting, there is some belief in the physics community (see, e.g., Section 6 of \cite{os1978} and Chapter 2 of \cite{s}) that the type of limit we consider here (i.e., the inverse temperature tends to infinity at an appropriate rate with the inverse mesh-size) is what should be taken in order to define a non-trivial scaling limit of the model based on this particular class of observables, see also the discussion in \cite{c2018}. 
\end{remark}

\begin{remark}\label{remark: remark below}
    The assumptions of Theorem~\ref{theorem: main result Z2} are all satisfied if, e.g.,
    \begin{equation*}
        \kappa_0 > -\tfrac{1}{4}\log(\tfrac{5 \sqrt{13}}{18} - 1 ) \approx 1.61867 \quad \text{and} \quad
        \beta_0   > \tfrac{1}{4}\log(2^{1/3} 5) + \tfrac{1}{24}\log \bigl(\tfrac{5 \sqrt{13}}{18} - 1) \approx 0.190344. 
    \end{equation*}
\end{remark}

\begin{remark}
    In a forthcoming paper~\cite{f2021b}, a version of Theorem~\ref{theorem: main result Z2} that is valid also for so-called \emph{Wilson lines} is obtained, that is, for observables of the form 
    \begin{equation*}
        L_\gamma(\sigma,\phi)
        \coloneqq  \rho \bigl( \sigma(\gamma) - \phi(\partial \gamma) \bigr),\qquad \sigma\in \Sigma_{E_\infty},\, \phi\in \Phi_{V_N},
    \end{equation*}
    where \( \gamma \subseteq E_N \) is a simple oriented path along the boundary of some rectangle with side lengths \( \ell_1,\ell_2 \).
\end{remark}

\begin{remark}
    The constant \( C \) in Theorem~\ref{theorem: main result Z2} is given explicitly by
    \begin{equation*}
        C \coloneqq 2^{5/7}
            \cdot 
            \bigl\{ 2C_1C_0^{(7)} + C_3  + \sqrt{ C_{c,1} }  + \sqrt{ C_{c,2}  } +4\sqrt{3C_0^{(6)}}+2\sqrt{C_I}+2C_0^{(25)}+2
            \bigr\}^{1/3},
    \end{equation*}
    where \( C_0^{(6)} \), \( C_0^{(7)} \), and \( C_0^{(25)} \) are defined in~\eqref{eq: C0M}, \( C_3 \) is defined in~\eqref{eq: C3 def}, \( C_I \) is defined in~\eqref{eq: minimal Ising}, \( C_{c,1} \) is defined in~\eqref{eq: Cc1}, \( C_{c,2} \) is defined in~\eqref{eq: Cc2}, and \( C_1 \) is defined in~\eqref{eq: C1}. 
    These constants are all uniformly bounded from above for all \( \beta \geq \beta_0 \) and \( \kappa \geq \kappa_0 \), for any \( \beta_0 \) and \( \kappa_0 \) which satisfy the assumptions on \( \beta_0 \) and \( \kappa_0 \) in Theorem~\ref{theorem: main result Z2}.
\end{remark}

\begin{remark}
    When \( \kappa =0 \) we have \( \Upsilon_{\gamma,\beta,\kappa} = e^{-2|\gamma|e^{-24\beta}}  \). Consequently, in this case the conclusion of Theorem~\ref{theorem: main result Z2} (i.e.,~\eqref{eq: main result Z2}) agrees with the main result of~\cite{c2019} (in~\cite{c2019}, the constant \( \beta \) is the same as \( 2\beta \) in our paper). However, we stress that the case \( \kappa = 0 \) is not covered by Theorem~\ref{theorem: main result Z2}.
\end{remark}

\begin{remark}
    It would be interesting to find a more explicit formula for the Ising gradient field expression \(\Upsilon_{\gamma,\beta, \kappa}\). 
    It would also be interesting to establish the leading order behavior of \( \langle W_\gamma \rangle_{\beta,\kappa} \) as $\beta \to \infty$ when \( \kappa \) is small. However, our current argument uses a coupling between the abelian Higgs model and the \( \mathbb{Z}_n \) model which becomes less good when \( \kappa \to 0, \) and hence cannot be used directly in this regime.
\end{remark}

Our second theorem, which is a more general version of Theorem~\ref{theorem: main result Z2}, is valid for all finite cyclic groups \( \mathbb{Z}_n \) and without the assumption that \( 3\beta \geq 2\kappa \). In order to state this theorem, we need to introduce some additional notation.

For \( r \geq 0\) and \( g \in G \), we define
\begin{equation}\label{eq: varphi}
    \varphi_r(g) \coloneqq e^{r \Re (\rho(g)-\rho(0))}.
\end{equation}
We extend this notation to \( r = \infty \) by letting \( \varphi_\infty(g) \coloneqq \mathbb{1}_{g=0} \), i.e.,
\begin{equation*}
    \varphi_\infty(g) 
    \coloneqq 
    \begin{cases} 
        1 &\text{if } g = 0 , \\
        0 & \text{if }g \in G \smallsetminus \{0\}.
    \end{cases}
\end{equation*}
Next, for \( \hat g \in G \) and \( \beta,\kappa \geq 0 \), we define 
\begin{equation}\label{eq: def of thetag}
    \theta_{\beta,\kappa}(\hat g) \coloneqq \frac{\sum_{g \in G} \rho(g)  \varphi_\beta (g)^{12} \varphi_\kappa(g+\hat g)^2}{\sum_{g \in G} \varphi_\beta(g)^{12} \varphi_\kappa (g+\hat g)^2}.
\end{equation} 
For \( r \geq 0 \), let
\begin{equation}\label{eq: alpha01def}
    \alpha_0(r) \coloneqq \sum_{g \in G\smallsetminus \{ 0 \}} \varphi_r(g)^2 \quad \text{and} \quad  \alpha_1(r) \coloneqq \max_{g \in G\smallsetminus \{ 0 \}} \varphi_r(g)^2.
\end{equation} 
Next, for \( \beta,\kappa \geq 0 \), define
\begin{equation}\label{eq: alpha234def}
    \alpha_2(\beta,\kappa) \coloneqq \alpha_0(\beta) \alpha_0(\kappa)^{1/6},
    \quad
    \alpha_3(\beta,\kappa) \coloneqq  \bigl| 1-  \theta_{\beta,\kappa}(0)\bigr|,
    \quad
    \alpha_4(\beta,\kappa) \coloneqq  \max_{g \in G} \bigl| \theta_{\beta,\kappa}(g)-  \theta_{\beta,\kappa}(0)\bigr|,
\end{equation} 
and
\begin{equation}\label{eq: alpha5}
    \alpha_5(\beta,\kappa) \coloneqq  \min_{g_1,g_2, \ldots, g_6 \in G}\bigg( 1 -   \biggl| \frac{\sum_{g \in G} \rho(g) \bigl(\prod_{k =1}^6 \varphi_\beta(g+g_k)^2\bigr) \varphi_\kappa(g)^2}{\sum_{g \in G} \bigl(\prod_{k =1}^6\varphi_\beta(g+g_k)^2\bigr) \varphi_\kappa(g)^2} \biggr|\bigg).
\end{equation}  
Finally, we let 
\begin{align*}
    &\varepsilon_1(\beta,\kappa) \coloneqq \frac{\alpha_2(\beta,\kappa)^6}{\alpha_5(\beta,\kappa)}   ,\quad 
    \varepsilon_2(\beta,\kappa) \coloneqq \sqrt{\frac{\alpha_0(\kappa)^8 \alpha_4(\beta,\kappa)}{ \alpha_5(\beta,\kappa)}}
        +
        \sqrt{\frac{\alpha_3(\beta,\kappa) }{\alpha_5(\beta,\kappa)}},
        \\
  &   \varepsilon_3(\beta,\kappa) \coloneqq
    \sqrt{\frac{\alpha_4(\beta,\kappa) \alpha_0(\kappa)^2}{\alpha_2(\beta,\kappa)^6}}.
\end{align*}

In the general version of our main theorem below, we will work under the following three assumptions:
\begin{enumerate}[label=\textnormal{[A\arabic*]}]
        \item \( \alpha_1(\beta) + 4\alpha_0(\kappa) < 1\), \label{assumption: 1}
        \item \( 2^{1/3} 5  \alpha_0(\beta)  \alpha_0(\kappa)^{ 1/6} <1\), and \label{assumption: 2}
        \item \( 18^2 \alpha_0(\kappa)(2 + \alpha_0(\kappa))<1\). \label{assumption: 3}
\end{enumerate}
 
\begin{theorem}\label{theorem: the main result}\label{theorem: main result} 
    Let \(n \geq 2\) be an integer. Consider the lattice Higgs model on \( \mathbb{Z}^4 \) in the fixed-length limit, with structure group \(G = \mathbb{Z}_n\) and a faithful one-dimensional representation \( \rho \) of \( G \). 
    Suppose that \( \beta_0,\kappa_0 \geq 0 \) satisfy~\ref{assumption: 1},~\ref{assumption: 2}, and~\ref{assumption: 3} with \( \beta = \beta_0 \) and \( \kappa=\kappa_0 \). Then, for any \( \beta \geq \beta_0 \), any \( \kappa \geq \kappa_0 \), and any simple oriented loop \( \gamma \) in \( \mathbb{Z}^4 \), we have 
    \begin{equation}\small\label{eq: main theorem}
        \begin{split}
            &\Bigl|\langle W_\gamma\rangle_{\beta,\kappa} -  \Upsilon_{\gamma,\beta,\kappa} \Bigr|
            \leq 
            K \, \Bigl( \alpha_2(\beta,\kappa) + \sqrt{|\gamma_c|/|\gamma|}\Bigr)^{\mathrlap{1 - 2|\gamma| /(3|\gamma|-|\gamma_c| )}}
            ,
        \end{split}
    \end{equation} 
    where
    \begin{equation*}
        \Upsilon_{\gamma,\beta,\kappa} \coloneqq  \Bigl\langle \, \prod_{e \in \gamma } \theta_{\beta,\kappa}(\sigma_e ) \Bigr\rangle_{\infty,\kappa}
    \end{equation*}
    and
    \begin{align*}
        &K \coloneqq 2 \biggl( \bigl( 2C_1C_0^{(7)} +C_3 \bigr) \varepsilon_1(\beta,\kappa) 
        +
       \Bigl( \sqrt{C_{c,1}}+\sqrt{C_{c,2}} \max\bigl( 1,\alpha_2(\beta,\kappa)^6\bigr) \Bigr)\, \varepsilon_1(\beta,\kappa) \varepsilon_3(\beta,\kappa)   
        \\&\qquad\qquad+ 
        \sqrt{12C_0^{(6)}} \sqrt{\varepsilon_1(\beta,\kappa)\varepsilon_2(\beta,\kappa)^2}
        +
        \sqrt{C_I}\varepsilon_2(\beta,\kappa)  
        +
        2C_0^{(25)}   \varepsilon_1(\beta,\kappa)^4   \biggr)^{\mathrlap{1 - 2|\gamma| /(3|\gamma|-|\gamma_c|)}}.
    \end{align*}
\end{theorem}

\begin{remark}
    The constants in Theorem~\ref{theorem: the main result} are all defined later in the paper. In detail, \( C_0^{(6)} \), \( C_0^{(7)} \), and \( C_0^{(25)} \) are defined in~\eqref{eq: C0M}, \( C_3 \) is defined in~\eqref{eq: C3 def}, \( C_I \) is defined in~\eqref{eq: minimal Ising}, \( C_{c,1} \) is defined in~\eqref{eq: Cc1}, \( C_{c,2} \) is defined in~\eqref{eq: Cc1}, and \( C_1 \) is defined in~\eqref{eq: C1}. 
    These constants are all uniformly bounded from above for all \( \beta \geq \beta_0 \) and \( \kappa \geq \kappa_0 \) for any \( \beta_0 \) and \( \kappa_0 \) which satisfy Assumptions~\ref{assumption: 1}--\ref{assumption: 3}.
\end{remark}

\begin{remark}
    When \( G = \mathbb{Z_2} \), we have 
    \( \varepsilon_1(\beta,\kappa) \leq 1 \), \( \varepsilon_2(\beta,\kappa) \leq \sqrt{2} \), and \( \varepsilon_3(\beta,\kappa) \leq \sqrt{2} \) (see the proof of Theorem~\ref{theorem: main result Z2}).
    Heuristic arguments suggest that similar inequalities are valid also for \( n \geq 3 \). In particular, this would imply that \( K \) on the right-hand side of~\eqref{eq: main theorem} can be uniformly bounded from above for all \( \beta \geq \beta_0 \) and \( \kappa \geq \kappa_0 \), but we have not shown this.
\end{remark}

\begin{remark} 
   We  assume throughout the paper that the underlying lattice is \( \mathbb{Z}^4 \). However,  most of our results can be extended to other dimensions by simply changing the constants and the assumptions on \( \beta \) and \( \kappa \). Whenever we draw figures or simulations, we will for simplicity do this for \( \mathbb{Z}^2 \). One should be aware that although this gives a general idea of how the model behaves, there are qualitative differences between \( \mathbb{Z}^2 \) and \( \mathbb{Z}^d \) for \( d \geq 3 \). 
\end{remark}

\subsection{Relation to other work}

    In this section we summarize known results about pure lattice gauge theory and the lattice Higgs model on \( \mathbb{Z}^4 \), we mainly concentrate on the case when the structure group is finite. We then compare these works with the questions considered in this paper.

    Pure lattice gauge theory with structure group \( \mathbb{Z}_2 \) (also known as Ising gauge theory) was first introduced by Wegner in~\cite{w1971} as an example of a model exhibiting a phase transition without a local order parameter. It was argued that a phase transition could be detected by the decay rate of the Wilson loop expectation as the size of the loop tends to infinity while the parameters are kept fixed. 
    The same model, but in a more general setting, was then independently introduced by Wilson in~\cite{w1974} as a discretization of Yang-Mills theories and the decay rate of Wilson loop expectations was proposed as a condition for the occurrence of quark confinement.
    
    In terms of the Wilson loop expectation, the belief is that there are in general (at least) two phases:  a confining phase, with  ``area law'' decay, in which the Wilson loop expectation decays like \( e^{-c\, \mathrm{area}(\gamma)}, \) where \( \mathrm{area}(\gamma) \) is the surface area of the minimal surface that has \( \gamma \) as its boundary, as well as a non-confining phase, with ``perimeter law'' decay, in which the expectation decays like \( e^{-c'|\gamma|}. \) Here \( c,c' > 0 \) are constants which may depend on the parameters of the considered model. Significant effort has been put into making this precise and rigorously understanding the phase diagrams of the various lattice gauge theories.

    In~\cite{bdi1975ii}, it was shown that in pure lattice gauge theory, if \( \beta>0 \) is sufficiently small then the Wilson loop observable follows an area law (see also~\cite{os1978}). Also in~\cite{bdi1975ii}, the lattice Higgs model~\eqref{eq: general fixed length measure} studied in this paper was first considered, with the goal of understanding the phase structure of the model.
    Extending this line of work, in~\cite{fs1982} (see also~\cite{g1980,g1980b}) it was shown that pure lattice gauge theory with Villain action and structure group \( G = \mathbb{Z}_n \) 
    has at least two phases when \( n \) is sufficiently large, corresponding to the Wilson loop observable having area and perimeter law respectively. 
    In contrast, in~\cite{ejkfk1987b} it was shown that the Wilson loop expectation follows a perimeter law whenever \( \kappa > 0 \) for any finite group.
    We remark that it is also known~\cite{sj1982} that perimeter law is the slowest possible decay for Wilson loop expectations in a pure lattice gauge theory.

    The type of results obtained in this paper (see also~\cite{c2019,flv2020,sc2019,f2021}) are different from the more classical results on area versus perimeter law, which concern the decay rate of the Wilson loop expectation as the size of the loop tends to infinity while the inverse temperature is kept \emph{fixed}. 
    Instead, the main goal of these papers is to compute the leading order term for the expectation of Wilson loop observables in the limit when the size of the loop and inverse temperature tend to infinity simultaneously in such a way that the Wilson loop expectation is non-trivial. The main motivation for considering this limit is that it is believed that to obtain an interesting scaling limit, the parameter \( \beta \) would have to tend to infinity with the perimeter of the loop \( \gamma\), see, e.g.,~\cite{w1974, os1978}.
    This was first done in~\cite{c2019} for Wilson loop observables in a pure lattice gauge theory with Wilson action and structure group \( \mathbb{Z}_2.\)
    The results of~\cite{c2019} were extended to general finite groups in~\cite{flv2020,sc2019}. In this paper we prove an analogous result in the presence of an external field. 
    A central idea in all these papers is the decomposition of spin configurations into vortices, an idea sometimes referred to as polymer (or defect) expansion. This idea was used already in~\cite{mms1979} and only works when the structure group is discrete. One of the contributions of the current paper is to extend this idea to the case \( \kappa > \kappa_c. \) In particular, we construct and use a coupling between lattice gauge theory and its \( \beta \to \infty \) limit, which for \( G = \mathbb{Z}_2 \) is given by the discrete differential of the Ising model. By using this coupling, we can handle also the case when \( \kappa \) does not grow with the perimeter of the loop.
    
    When \( \kappa > 0, \) it is known from~\cite{ejkfk1987b} that \( \langle W_\gamma \rangle_{\beta,\kappa,\infty} \) follows a perimeter law. In other words, it is known that for any fixed \( \beta,\kappa > 0, \) one has 
    \begin{equation}\label{eq: log limits}
        0< \liminf_{|\gamma| \to \infty} -|\gamma|^{-1}\log \langle W_\gamma \rangle_{\beta,\kappa,\infty}  \leq  \limsup_{|\gamma| \to \infty} -|\gamma|^{-1}\log \langle W_\gamma \rangle_{\beta,\kappa,\infty}  < \infty.
    \end{equation}
    In~\cite{fs1979}, using a first cumulant approximation/dilute gas expansion, the leading behavior of the limits in~\eqref{eq: log limits} was predicted to be given by \( e^{-2|\gamma|e^{-24\beta-4\kappa}}. \)
    In the upcoming paper~\cite{f2021}, we show that the expression for \( \Upsilon_{\gamma,\beta,\kappa} \) in Theorem~\ref{theorem: main result Z2} can be replaced by \begin{equation*}
        e^{-2| \gamma|e^{-24\beta-4\kappa}\bigl(1+\frac{1}{2}(e^{8\kappa}-1) (1-\langle \rho(\sigma_e-(d\phi)_e) \rangle_{\infty,\kappa,\infty})\bigr)} 
    \end{equation*}
    for an arbitrary edge \( e \in E_N\),  without increasing the size of the error term. Thus, when \( \kappa \) is sufficiently large, the leading order term of our main theorem matches the predictions from~\cite{fs1979}.
    However, we remark that since our main theorem requires \( \beta \) and \( |\gamma| \) to tend to infinity simultaneously, one cannot deduce anything about the limit of \( |\gamma|^{-1}\log \langle W_\gamma \rangle_{\beta,\kappa,\infty} \) as \( |\gamma| \to \infty \) for fixed $\beta$ and $\kappa$ from this result.

\subsection{Structure of the paper}

The rest of this paper will be structured as follows.
In Section~\ref{sec: preliminaries}, we give a brief background on the language of, and tools from, discrete exterior calculus, which will be used throughout the rest of this paper. Much of this is standard material, but in  Section~\ref{sec: the partial ordering}, we introduce a partial ordering of differential forms, and in Section~\ref{sec: irreducibility}, we use this definition to define a concept of irreducibility for spin configurations. We also describe the unitary gauge, and use this to give an alternative expression for the probability measure $\mu_{N,\beta, \kappa, \infty}$.  In Section~\ref{sec: vortices}, we use the definition of irreducibility to define vortices, and establish some of their properties.
In Section~\ref{sec: activity}, we define the activity of a gauge field configuration. This notation is then used in Section~\ref{sec: distribution of configurations} and Section~\ref{section: distributions of frustrated plaquettes}  to give upper bounds on the probability that a gauge field configuration locally has certain properties. 
In Section~\ref{sec: coupling}, we present a coupling between the measures \( \mu_{N,\beta,\kappa} \) and \( \mu_{N,\beta,\infty} \).
In Section~\ref{section: technical inequalities}, we derive a few useful properties of the function \( \theta_{\beta,\kappa} \).
In Section~\ref{sec:WtoWprime}, we use ideas from~\cite{c2019} in order to replace the expected value of a Wilson loop observable with a simpler expectation.
In Section~\ref{sec: resampling}, we use a resampling trick to rewrite the expectation obtained in Section~\ref{sec:WtoWprime}.
In Section~\ref{sec:applyingcoupling}, we then apply the coupling introduced in Section~\ref{sec: coupling} to this expectation. 
Finally, in Section~\ref{sec: main result}, we use the results from the previous sections to prove Theorem~\ref{theorem: main result Z2} and Theorem~\ref{theorem: main result}.

\section{Preliminaries}\label{sec: preliminaries}
In this section, we introduce and discuss needed auxiliary results and  notation. Many statements given are certainly well-known, but in order to have the paper self-contained, and since clean references do not always seem available, we choose to include them here. Sections~\ref{sec: the partial ordering},~\ref{sec: irreducibility}, ~\ref{sec: minimal configurations}, and ~\ref{sec:unitary_gauge} however contain new results. 
Throughout this section, we assume that \( N \geq 1 \) is given.

\subsection{Discrete exterior calculus}
In what follows, we give a brief overview of discrete exterior calculus on the cell complex of \(  [-N,N]^m \cap  \mathbb{Z}^m \), for \( m \geq 1 \).
For a more thorough background on discrete exterior calculus, we refer the reader to~\cite{c2019}. 

All of the results in this subsection are obtained under the assumption that an abelian group \( G \), which is not necessarily finite, has been given. In particular, they all hold for both \( G=\mathbb{Z}_n \) and \( G=\mathbb{Z} \).

\subsubsection{Boxes}

Any set of the form \( \bigl( [a_1,b_1] \times \cdots \times [a_m,b_m] \bigr) \cap \mathbb{Z}^m \) where, for each \( j \in \{ 1,2, \ldots, m \} \), \( a_j, b_j  \in \mathbb{Z} \) satisfy $a_j < b_j$, will be referred to as a  \emph{box}. If all the intervals $[a_j,b_j]$, $1 \leq j \leq m$, have the same length, then the set \( \bigl( [a_1,b_1] \times \cdots \times [a_m,b_m] \bigr) \cap \mathbb{Z}^m \) will be referred to as a {\it cube}.

Recall that \( B_N \) is the cube \( [-N,N]^m  \cap \mathbb{Z}^m \).

\subsubsection{Oriented edges (1-cells)}

The graph \(\mathbb{Z}^m\) has a vertex at each point \( x \in \mathbb{Z}^m \) with integer coordinates and an (undirected) edge between each pair of nearest neighbors. 
We associate to each undirected edge \( \bar e \) in \( \mathbb{Z}^m \) two directed or \emph{oriented} edges \( e \) and \( -e \) with the same endpoints as \( \bar e \); oriented in opposite directions.

Let \( \mathbf{e}_1 \coloneqq (1,0,0,\ldots,0)\), \( \mathbf{e}_2 \coloneqq (0,1,0,\ldots, 0) \), \ldots, \( \mathbf{e}_m \coloneqq (0,\ldots,0,1) \) and let \( d{\mathbf{e}}_1 \), \ldots, \( d{\mathbf{e}}_m \)  denote the \( m \) oriented edges with one endpoint at the origin which naturally correspond to these unit vectors. We say that an oriented edge \( e \) is \emph{positively oriented} if it is equal to a translation of one of these unit vectors, i.e., if there exists a point \( x \in \mathbb{Z}^m \) and an index \( j \in \{ 1,2, \ldots, m\} \) such that \( e = x + d{\mathbf{e}}_j \). If \( x \in \mathbb{Z}^m \) and  \( j \in \{ 1,2, \ldots, m\} \), then we write \( dx_j \coloneqq x + d{\mathbf{e}}_j  \).

We let \( E_N \) denote the set of oriented edges whose endpoints are both in the box \( B_N \).

\subsubsection{Oriented \( k \)-cells}
For any two oriented edges \( e_1  \) and \( e_2\), we consider the wedge product \( e_1 \wedge e_2  \) satisfying
\begin{equation}\label{eq: reversing wedge order}
     e_1 \wedge e_2 = -(e_2 \wedge e_1) = (-e_2) \wedge e_1 = e_2 \wedge (-e_1),
\end{equation}
and \( e_1 \wedge e_1 = 0 \). If \( e_1 \), \( e_2 \), \ldots, \( e_k \) are oriented edges which  do not share a common endpoint, we  set \( e_1 \wedge e_2  \wedge \cdots \wedge e_k = 0 \).
If \( e_1 \), \( e_2 \), \ldots, \( e_k \) are oriented edges and \( e_1 \wedge \cdots \wedge e_k \neq 0 \), we say that \( e_1 \wedge \cdots  \wedge e_k \) is an oriented \emph{\( k \)-cell}. If there exists an \( x \in \mathbb{Z}^m \) and \( j_1 <j_2 <\cdots< j_k \) such that \( e_i = d{x}_{j_i} \), then we say that \( e_1 \wedge \cdots \wedge e_k \) is \emph{positively oriented} and that \( -( e_1 \wedge \cdots \wedge e_k) \) is \emph{negatively oriented}. Using~\eqref{eq: reversing wedge order}, this defines an orientation for all \( k \)-cells. 

Whenever \( S \) is a set of \( k \)-cells, we let \( S^+ \) denote the set of positively oriented \( k \)-cells in \( S \). If a set \( S \) of \( k \)-cells satisfies \( S = -S \coloneqq \{ c \colon -c \in S \} \), then we say that \( S \) is \emph{symmetric}.

\subsubsection{Oriented plaquettes} \label{sec: plaquettes}
We will usually say \emph{oriented plaquette} instead of oriented \(2\)-cell.
If \( x \in \mathbb{Z}^m \) and \( j_1 <   j_2 \), then  \( p \coloneqq dx_{j_1} \wedge d{{x}}_{j_2} \) is a positively oriented plaquette, and we define 
\begin{equation*}
    \partial p \coloneqq \{ 
    dx_{j_1},  
    (d(x +  \mathbf{e}_{j_1}))_{j_2}, 
    -(d(x + \mathbf{e}_{j_2}))_{j_1}, 
    - dx_{j_2}
    \}.
\end{equation*} 
If \( e \) is an oriented edge, we let \( \hat \partial e \) denote the set of oriented plaquettes \( p \) such that \( e \in \partial p \).
We let \( P_N \) denote the set of oriented plaquettes whose edges are all contained in \( E_N \).

\subsubsection{Discrete differential forms} 

A \( G \)-valued function \( f \) defined on the set of \( k \)-cells in \( V_N \) with the property that \( f(c) = -f(-c) \) is called a \emph{\( k \)-form}. 
If \( f \) is a \( k \)-form in \( \Sigma_k \) which takes the value \( f_{j_1,\ldots, j_k}(x) \) on \( dx_{j_1} \wedge \cdots \wedge dx_{j_k} \), it is useful to represent its values on the $k$-cells at $x \in V_N$ by the formal expression
\begin{equation*}
    f(x) = \sum_{1 \leq j_1 < \cdots < j_k\leq m} f_{ j_1,\ldots, j_k}(x) \,  dx_{j_1} \wedge \cdots \wedge dx_{j_k}.
\end{equation*}
To simplify notation, if \( c \coloneqq  dx_{j_1} \wedge \cdots \wedge dx_{j_k} \) is a \( k \)-cell and \( f \) is a \( k \)-form we often write \( f_c \) instead of \( f(c) = f_{j_1,\ldots, j_k}(x) \). We say that a \( k \)-form \( f \) is \emph{non-trivial} if there is at least one \( k \)-cell \( c \) such that \( f_c \neq 0 \).

Given a \( k \)-form \( f \), we let \( \support f \) denote the support of \( f \), i.e., the set of all oriented \( k \)-cells \( c \) such that \( f(c) \neq 0 \).
For a symmetric set \( E \subseteq E_N \), we let \( \Sigma_E \) denote the set of all \( G \)-valued 1-forms with support contained in \( E \). 
More generally, For \( k = 0,1,\dots, m \), we let \( \Sigma_k \) denote the set of \( k \)-forms with support contained in \( B_N \). We note that with this notation, \( \Sigma_1 = \Sigma_{E_N} \).

\subsubsection{The exterior derivative}
Given \( h \colon \mathbb{Z}^m \to G \), \( x \in \mathbb{Z}^m \), and \( i \in \{1,2, \ldots, m \} \), we let 
\begin{equation*}
    \partial_i h(x) \coloneqq h(x+\mathbf{e}_i) - h(x) .
\end{equation*}
If \( k \in \{ 0,1,2, \ldots, m-1 \} \) and \( f \) is a \( G \)-valued \( k \)-form in \( \Sigma_k \), we define the \( (k+1) \)-form \( df \) in \( \Sigma_{k+1}\) via the formal expression
\begin{equation*}
    df(x) = \sum_{1 \leq j_1 < \cdots < j_k\leq m} \sum_{i=1}^m \partial_i f_{j_1,\ldots, j_k} (x) \, dx_i \wedge (dx_{j_1} \wedge \cdots \wedge dx_{j_k}). %, \quad x \in V_N.
\end{equation*}
The operator \( d \) is called the \emph{exterior derivative.}

\subsubsection{Boundary operators}\label{sec: boundary}
If \( \hat x  \in V_N \) and \( j \in \{ 1,2, \ldots, m \} \), then for any \( x \in V_N\), we have
\begin{equation*}\label{equation: simplest differential form}
    \begin{split}
        &d(\mathbb{1}_{x = \hat x} \, dx_j) 
        = \sum_{i=1}^m (\partial_i \mathbb{1}_{x = \hat x} )\,  dx_i \wedge dx_j
        = \sum_{i=1}^m (\mathbb{1}_{x + \mathbf{e}_i= \hat x} -  \mathbb{1}_{x = \hat x} )\,  dx_i \wedge dx_j
        \\&\qquad = \sum_{i=1}^{j-1} (\mathbb{1}_{x + \mathbf{e}_i = \hat x} -  \mathbb{1}_{x = \hat{x}} )\,  dx_i \wedge dx_j
        - \sum_{i=j+1}^{m} (\mathbb{1}_{x + \mathbf{e}_i = \hat x} -  \mathbb{1}_{x = \hat x} )\,    dx_j \wedge dx_i.
    \end{split}
\end{equation*}
Here we are writing \(\mathbb{1}_{x = \hat x}\) for the Dirac delta function of \(x\) with mass at \(\hat x\).
From this equation, it follows that whenever \(  e = \hat{x} + d\mathbf{e}_j = d\hat x_j \) is an oriented edge and \( p \) is an oriented plaquette, we have
\begin{equation}\label{d1dxjp}
     \bigl(d(\mathbb{1}_{x = \hat x} \, dx_j )\bigr)_p =
     \begin{cases}1 &\text{if }  e \in \partial p, \cr 
     -1 &\text{if } - e \in \partial p,
     \cr 0 &\text{else.} \end{cases}
\end{equation}
This implies in particular that if \( 1 \leq j_1 < j_2 \leq m \), \( p =   dx_{j_1} \wedge dx_{j_2}  \) is a plaquette, and \( f \) is a \( 1 \)-form, then
\begin{equation*}
    (df)_p = (df)_{j_1,j_2}(x)
    = \bigg(d\bigg(\sum_{j=1}^m f_j(x) dx_j\bigg)\bigg)_p
=  \sum_{\hat{x} \in \mathbb{Z}^m} \sum_{j=1}^m f_j(\hat{x}) \big(d\big( \mathbb{1}_{x = \hat{x}} dx_j\big)\big)_p
    = \sum_{e \in \partial p}  f_e.
\end{equation*}
Analogously, if \( k \in \{ 1,2, \ldots, m \} \) and \(  c  \) is a \( k \)-cell, we define \( \partial  c \) as the set of all \( (k-1) \)-cells \( \hat c =   d\hat x_{j_1} \wedge \cdots \wedge d\hat x_{j_{k-1}}\) such that
\begin{equation}\label{eq: gen boundary def}
    \bigl(d(\mathbb{1}_{x = \hat x} \, d\hat x_{j_1} \wedge \dots \wedge d\hat x_{j_{k-1}}) \bigr)_{ c} = 1.
\end{equation}
Using this notation, one can show that if \( f \) is a \( k \)-form and \( c_0 \) is a \( (k+1)\)-cell, then
\begin{equation}\label{dfc0sumdfc}
    (d f)_{c_0} = \sum_{c \in  \partial c_0} f_c.
\end{equation}

If \( k \in \{ 1,2, \ldots, m \} \) and \( \hat c \) is a \( k \)-cell, the set \( \partial \hat c \) will be referred to as the \emph{boundary} of \( \hat c \). When \( k \in \{ 0,1,2, \ldots, m-1 \} \) and \( c \) is a \( k \)-cell, we also define the \emph{co-boundary} \( \hat \partial c \) of \( c \) as the set of all \( (k+1) \)-cells \( \hat c \) such that \( c \in \partial \hat c \). 
If \( B \) is a set of \( k \)-cells, then we write \( \partial B = \bigcup_{c \in B} \partial c \), and \( \hat \partial B = \bigcup_{c \in B} \hat \partial c  \).
 
Using~\eqref{dfc0sumdfc}, we immediately obtain the following lemma which is a discrete version of the Bianchi identity.  
\begin{lemma}\label{lemma: Bianchi}
    If \( \sigma \in \Sigma_{E_N}\) and \( m \geq 3 \), then for any oriented  3-cell \( c \) in \( B_N \) we have  \begin{equation}\label{eq: Bianchi}
        \sum_{p \in \partial c} (d\sigma)_p = 0.
    \end{equation}
\end{lemma}

\subsubsection{Boundary cells}

Recall that an edge is said to be in a box \( B \) if both its endpoint are in \( B \). More generally, for \( k \in \{ 1,2,\ldots, m \} \), a \( k \)-cell \( c \) is said to be in \( B \) if all its corners are in \( B \).

An edge \( e \in E_N \)   is said to be a \emph{boundary edge} of \( B \) if there exists a plaquette \( p \in \hat \partial e \) which is not in \( B \).
Analogously, a plaquette \( p \in P_N \) is said to be a \emph{boundary plaquette} of \( B \) if there is a 3-cell  \( c \in \hat   \partial p\) which is not in \( B \).
More generally, for \( k \in \{ 0,1, \ldots, m-1 \} \),  a \( k \)-cell \( c \) in \( B \) is said to be a \emph{boundary cell} of \( B \), or equivalently to be in \emph{the boundary} of \( B \), if there is a \( (k+1) \)-cell \( \hat c \in \hat \partial c \) which is not in \( B \).

\subsubsection{The Poincar\'e lemma}
For \( k \in \{ 1,\ldots, m \} \), 
we say that a \( k \)-form \( f \) is \emph{closed} if \( df = 0 \).

\begin{lemma}[The Poincar\'e lemma, Lemma 2.2 in~\cite{c2019}]\label{lemma: poincare}
  Let \( k \in \{ 0,1, \ldots, m-1\} \) and let \( B \) be a box in \( \mathbb{Z}^m \). Then the exterior derivative \( d \) is a surjective map from the set of \(G\)-valued \( k \)-forms with support contained in \(B\) onto the set of G-valued closed \((k+1)\)-forms with support contained in \(B\). Moreover, if \(G\) is finite and \( j \) is the number of closed \(G\)-valued \( k \)-forms with support contained in \(B\), then this map is a \(j\)-to-\(1\) correspondence. Lastly, if \( k \in \{ 0,1,2, \ldots, m-2 \} \) and \(f\) is a closed \( (k+1) \)-form that vanishes on the boundary of \(B\), then there is a \(k\)-form \( h\) that also vanishes on the boundary of \(B\) and satisfies \(dh = f\).
\end{lemma}
 
The set of closed \( G \)-valued 1-forms \( \sigma \in \Sigma_{E_N} = \Sigma_1 \) will be denoted by \( \Sigma_{E_N}^0 \), and the set of closed \( G \)-valued 2-forms \( \omega \in \Sigma_2 \)  will be denoted by \( \Sigma_{P_N} \). By Lemma~\ref{lemma: poincare}, \( \omega \in \Sigma_{P_N} \) if and only if \( \omega = d\sigma \) for some \( \sigma \in \Sigma_{E_N} \). The 2-forms in \( \Sigma_{P_N} \) will be referred to as \emph{plaquette configurations}.

\subsubsection{The co-derivative}\label{sec: coderivative}

Given \( h \colon \mathbb{Z}^m \to G \), \( x \in \mathbb{Z}^m \), and \( i \in \{ 1,2, \ldots, m \} \), we let 
\begin{equation*}
    \bar \partial_i h(x) \coloneqq h(x) - h(x-\mathbf{e}_i).
\end{equation*}
When \( k \in \{ 1,2, \ldots, m \} \) and \( f \) is a \( G \)-valued \( k \)-form in \( \Sigma_k \), we define the \( (k-1) \)-form \( \delta f \) in \( \Sigma_{k-1}\) by
\begin{equation*}
    \delta f(x) \coloneqq \sum_{1 \leq j_1 < \cdots < j_k \leq m} \sum_{i = 1}^k   (-1)^{i} \, \bar \partial_{j_i} f_{j_1,\ldots, j_k}(x) \, dx_{j_1} \wedge \cdots \wedge dx_{j_{i-1}}\wedge dx_{j_{i+1}} \wedge \cdots \wedge  dx_{j_{k}}.% ,\quad x \in \mathbb{Z}^m.
\end{equation*}  
The operator \( \delta \) is called the \emph{co-derivative}. 
Note that if \( 1 \leq i_1 < i_2 \leq m \), \(\hat x  \in \mathbb{Z}^m\), \( p_0 = d\hat x_{i_1} \wedge d\hat x_{i_2} \in P_N \), and the \( 2 \)-form \( f \) is defined by
\begin{equation*}
    f(x) \coloneqq \mathbb{1}_{x = \hat x} \, d x_{i_1}\wedge d x_{i_2},\quad x\in V_N
\end{equation*}  
then for any \( x\in V_N \) and \( j,j_1,j_2 \in \{ 1,2, \ldots, m \} \) with \(j_1 < j_2\), we have
\begin{equation*}
    \begin{split} 
        &\bar \partial_{j} f_{j_1,j_2}(x) 
        = \bar \partial_j \mathbb{1}_{x = \hat x,\, j_1 = i_1, j_2 = i_2}
        = \mathbb{1}_{x = \hat x,\, j_1 = i_1, j_2 = i_2}
        - \mathbb{1}_{x = \hat x-\mathbf{e}_j,\, j_1 = i_1, j_2 = i_2} ,
    \end{split}
\end{equation*}
and hence
\begin{align*}
    &\delta f(x) 
    = 
    \sum_{1 \leq j_1 < j_2 \leq m} \big[(-1)^1 \bar{\partial}_{j_1}f_{j_1, j_2}(x) dx_{j_2} + (-1)^2 \bar{\partial}_{j_2} f_{j_1, j_2}(x) dx_{j_1}\big]  
     \\&\qquad
   = 
    -  \bar \partial_{i_1} \mathbb{1}_{x=\hat x}  dx_{i_2} 
    + 
    \bar \partial_{i_2} \mathbb{1}_{x=\hat x} dx_{i_1}   
    = 
    - \mathbb{1}_{x = \hat x } dx_{i_2}
    +\mathbb{1}_{x-\mathbf{e}_{i_1} = \hat x}\, dx_{i_2}
    + \mathbb{1}_{x = \hat x}\, dx_{i_1}
    - \mathbb{1}_{x-\mathbf{e}_{i_2} = \hat x}\, dx_{i_1}.
\end{align*}
In other words, for any \( e \in E_N \) we have
\begin{equation*}
    (\delta f)_e = 
    \begin{cases}
        1 &\text{if } e \in \partial p_0, \cr 
        -1 &\text{if } -e \in \partial p_0, \cr 
        0 &\text{else.}
    \end{cases}
\end{equation*}
This implies that if \( 1 \leq j \leq m \), \( e = dx_{j} \) is an edge, and \( f \) is any \( 2 \)-form, then
\begin{equation*}
    (\delta f)_e = (\delta f)_{j}(x) 
    = \sum_{\hat{x} \in \mathbb{Z}^m} \sum_{1 \leq i_1 < i_2 \leq m} f_{i_1, i_2}(\hat{x}) \bigl(\delta(\mathbb{1}_{x = \hat{x}} dx_{i_1} \wedge dx_{i_2})\bigr)_e
    = \sum_{p \in \hat{\partial} e}  f_p.
\end{equation*} 
More generally, one can show that whenever \( f \) is a \( k \)-form and \( c_0 \) is a \( (k-1)\)-cell, then
\begin{equation}\label{eq: deltafc0}
    (\delta f)_{c_0} = \sum_{c \in \hat \partial c_0} f_c.
\end{equation}

\begin{lemma}[The Poincar\'e lemma for the co-derivative, Lemma 2.7 in~\cite{c2019}]\label{lemma: lemma 2.7}
    Let \( k \in \{ 1,2, \ldots, m-1 \} \). Let \( f \) be a \( G \)-valued \(k\)-form which satisfies \( \delta f  = 0 \). Then there is a \( (k+1) \)-form \( h \) such that \( f = \delta h \). Moreover, if \( f \) is equal to zero outside a box \( B \), then there is a choice of \( h \) that is equal to zero outside \( B \).
\end{lemma}

\subsubsection{The Hodge dual}

The lattice \( \mathbb{Z}^m \) has a natural dual, called the \( \emph{dual lattice} \) and denoted by \( *\mathbb{Z}^m \). In this context, the lattice \( \mathbb{Z}^m \) is called the \emph{primal lattice.} The vertices of the dual lattice \( *\mathbb{Z}^m \) are placed at the centers of the \( m \)-cells of the primal lattice. More generally, for \( k \in \{ 0,1,\ldots, m \} \), there is a bijection between the set of \( k \)-cells of \( \mathbb{Z}^m \) and the set of \( (m-k) \)-cells of \(*\mathbb{Z}^m \) defined as follows. For each \( x \in \mathbb{Z}^m \), let \( y \coloneqq *(dx_1 \wedge \cdots \wedge dx_m) \in *\mathbb{Z}^m\) be the point at the center of the primal lattice \( m \)-cell \( dx_1 \wedge \cdots \wedge dx_m \).
Let \( dy_1 = y-d\mathbf{e}_1, \ldots, dy_m = y-d\mathbf{e}_m \) be the edges coming out of \( y \) in the \emph{negative} direction.
Next,  let \( k \in \{ 0,1, \ldots, m \} \) and assume that \( 1 \leq i_1 < \cdots < i_k \leq m \) are given. If \(x \in \mathbb{Z}^m\), then \( c = dx_{i_1} \wedge \cdots \wedge dx_{i_k} \) is a \( k \)-cell in \( \mathbb{Z}^m \).  Let \( {j_1}, \ldots, j_{m-k} \) be any enumeration of \( \{ 1,2, \ldots, m \} \smallsetminus \{ i_1, \ldots, i_k \} \), and let \(  \sgn (i_1,\ldots, i_k, j_{1}, \ldots, j_{m-k} )\) denote the sign of the permutation that maps \((1,2,\ldots, m)\) to \( (i_1,\ldots, i_k, j_{1}, \ldots, j_{m-k} ) \). Define 
\begin{equation*}
    *(dx_{i_1} \wedge \cdots \wedge dx_{i_k}) =  \sgn (i_1,\ldots, i_k, j_{1}, \ldots, j_{m-k} )\, dy_{j_1} \wedge \cdots \wedge dy_{j_{m-k}}
\end{equation*}
and, analogously, define
\begin{equation*}
    \begin{split}
         *&(dy_{j_1} \wedge \cdots \wedge dy_{j_{m-k}}) =
        \sgn (j_{1}, \ldots, j_{m-k},i_1,\ldots, i_k)\, dx_{i_1} \wedge \cdots \wedge dx_{i_k}
        \\&\qquad =
        (-1)^{k(m-k)}  \sgn (i_1,\ldots, i_k, j_{1}, \ldots, j_{m-k} )\, dx_{i_1} \wedge \cdots \wedge dx_{i_k}.
    \end{split}
\end{equation*}

If \( B \) is a box in \( \mathbb{Z}^m \), then
\begin{equation*}
    *B \coloneqq \bigl\{ y \in *\mathbb{Z}^m \colon \exists x \in B \text{ such that } y \text{ is a corner in } *x \bigr\}.
\end{equation*}
Note that with this definition, \( B \subsetneq *{*B} \).

Given a \( G \)-valued \( k \)-form \( f \) on \( \mathbb{Z}^m \), we define the \emph{Hodge dual} \( *f \) of \( f \)  as
\begin{equation*}
    *f(y) \coloneqq \sum_{1 \leq i_1 < \cdots < i_k \leq m} f_{i_1,\ldots, i_k}(x) \sgn (i_1,\ldots, i_k,j_{1}, \ldots, j_{m-k}) \, dy_{j_1} \wedge \cdots \wedge dy_{j_{m-k}},
\end{equation*}
where \( y = *(dx_1 \wedge \cdots \wedge dx_m)\), and in each term, the sequence \( j_1, \ldots, j_{m-k} \) depends on the sequence \( i_1,\ldots, i_k \).
One verifies that with these definitions, we have \( *f(*c) = f(c) \) and 
\begin{equation}\label{hodgehodge}
    *(*f) = (-1)^{k(m-k)}f.
\end{equation}
The exterior derivative on the dual cell lattice is defined by
\begin{equation*}
    df(y) \coloneqq \sum_{1 \leq j_1 < \cdots < j_k\leq m} \sum_{i=1}^m \bar \partial_i f_{j_1,\ldots, j_k} (y) \, dy_i \wedge (dy_{j_1} \wedge \cdots \wedge dy_{j_k}). %, \quad y \in *\mathbb{Z}^m.
\end{equation*}

The next two lemmas describe how the Hodge dual relates to previously given definitions.
 
\begin{lemma}[Lemma 2.3 in~\cite{c2019}]\label{lemma: lemma 2.3}
    For any \( G \)-valued \( k \)-form \( f \), % on \( \mathbb{Z}^m \)
    and any \( x \in V_N  \), 
    \begin{equation*}
        \delta f(x)= (-1)^{m(k+1)+1} *(d (*f(y))),
    \end{equation*}
    where \( y =*(dx_1 \wedge \cdots \wedge dx_n )\) is the center of the \( m \)-cell \( dx_1 \wedge \cdots \wedge dx_m \).
\end{lemma}
 
The following lemma was proved in the special case when \(B\) is a cube in \cite[Lemma 2.4]{c2019}; the proof when \(B\) is a box is analogous.

\begin{lemma}\label{lemma: lemma 2.4}
    Let \( B \)  be any box in \( \mathbb{Z}^m \). Then a \( k \)-cell \(c\) is outside \(B\) if and only if \(*c\) is either outside \(*B\) or in the boundary of \(*B \). Moreover, if \( c\) is a \(k\)-cell outside \(B\) that contains a \( (k-1)\)-cell of \(B\), then \( *c \) belongs to the boundary of \(*B \).
\end{lemma}

\subsubsection{Restrictions of forms}

If \( \sigma \in \Sigma_{E_N} \) and \( E \subseteq E_N \) is symmetric,  we define \( \sigma|_E \in \Sigma_{E_N} \) by, for \( e \in E_N \), letting
\begin{equation*}
    (\sigma|_E)_e \coloneqq \begin{cases}
    \sigma_e &\text{if }  e \in E, \cr 
    0 &\text{else.}
    \end{cases}
\end{equation*}
Similarly, if \( \omega \in \Sigma_2 \) and \( P \subseteq P_N \) is symmetric, we define \( \omega|_P \in \Sigma_2 \) by, for \( p \in P_N \), letting
\begin{equation*}
    (\omega|_P)_p \coloneqq \begin{cases}
    \omega_p &\text{if } p \in P, \cr 
    0 &\text{else.}
    \end{cases}
\end{equation*}

\subsubsection{A partial ordering of \texorpdfstring{\( k \)}{k}-forms}\label{sec: the partial ordering} 

We now introduce a partial ordering on \( k \)-forms which will be useful both when talking about gauge field configurations and when introducing vortices in later sections. 

\begin{definition}\label{def: partial order}
    When \( k \in \{ 0,1,\dots, m-1 \} \) and \( \omega,\omega' \in \Sigma_k \), we write \( \omega' \leq  \omega \) if
    \begin{enumerate}[label=(\roman*)]
        \item \( \omega' = \omega|_{\support \omega'} \), and
        \item \( d\omega' = (d\omega)|_{\support d\omega'} \).
    \end{enumerate}
    If \( \omega' \neq \omega \) and \( \omega' \leq \omega \), we write \( \omega'< \omega \).
\end{definition}

The following lemma collects some basic facts about the relation $\leq$ on $\Sigma_k$. In particular, it shows that $\leq$ is a partial order on $\Sigma_k$.
\begin{lemma}\label{lemma: the blue lemma}
    Let \( k \in \{ 0,1,\dots, m-1 \} \) and \(\omega, \omega', \omega'' \in \Sigma_k \). The relation $\leq$ on $\Sigma_k$ has the following properties.
    \begin{enumerate}[label=\textnormal{(\roman*)}]
        \item Reflexivity: \( \omega \leq \omega \). \label{property 1}
        \item Antisymmetry: If \(\omega' \leq \omega \) and \(\omega \leq \omega'\), then \( \omega = \omega'\).\label{property 2}
        \item Transitivity: If \( \omega'' \leq \omega' \) and \( \omega' \leq \omega \), then \( \omega'' \leq \omega \).\label{property 3}
        \item If \( \omega' \leq \omega \), then \( \omega-\omega' = \omega|_{E_N \smallsetminus (\support \omega')} \leq \omega \). \label{property 4}
        \item If \( \omega' \leq \omega \), then \( \support d\omega'\) and \( \support d(\omega-\omega') \) are disjoint. \label{property 5} 
    \end{enumerate}
\end{lemma} 

\begin{proof}
    Properties~\ref{property 1}--\ref{property 3} follow easily from the definitions. For example, to show~\ref{property 3}, we note that if \( \omega'' \leq \omega'\) and \( \omega' \leq \omega\), then 
    \[ 
        d\omega'' = (d\omega')|_{\support d\omega''} = \big((d\omega)|_{\support d\omega'}\big)\big|_{\support d\omega''}
        = (d\omega)|_{\support d\omega''},
    \]
    where we have used that $\support d\omega'' \subseteq \support d\omega'$ in the last step. 

    To prove~\ref{property 4}, suppose that \( \omega' \leq \omega\), and let \( S \coloneqq \support \omega \) and \( S' \coloneqq \support\omega'\). Since \( \support (\omega|_{S \smallsetminus S'}) = S \smallsetminus S'\), we have \( \omega|_{S \smallsetminus S'} = \omega|_{\support (\omega|_{S \smallsetminus S'}) }.\)
    Moreover, since \( \omega = \omega|_{S \smallsetminus S'} + \omega',\)
    \begin{align}\label{dsigmaEEprime}
        d(\omega|_{S \smallsetminus S'}) = d\omega - d\omega' = d\omega - (d\omega)|_{\support d\omega'} = (d\omega)|_{(\support d\omega) \smallsetminus (\support d\omega')}.
    \end{align}
    Taking the support of both sides, we find that \(\support d(\omega|_{S \smallsetminus S'}) = (\support d\omega) \smallsetminus (\support d\omega')\), and hence (\ref{dsigmaEEprime}) implies that \( d(\omega|_{S \smallsetminus S'}) = (d\omega)|_{\support d(\omega|_{S \smallsetminus S'})}.\) This shows that \( \omega|_{E_N \smallsetminus S'} = \omega|_{S \smallsetminus S'} \leq \omega\) and completes the proof of~\ref{property 4}.
    
    To prove~\ref{property 5}, suppose that \( \omega' \leq \omega\) and that \( \omega'' = \omega - \omega' \). Let \( S' \coloneqq \support d\omega' \) and  \( S'' \coloneqq \support d\omega'' \).  
    Assume for contradiction that there exists \( c \in S' \cap S'' \). Then \( c \in \support d \omega' \), and hence \(  (d\omega')_c \neq 0 \). Since \( \omega' \leq  \omega \) and \( c \in S' \), we have \( (d\omega')_c = ((d \omega)|_{S'})_c = (d \omega)_c \) and hence \( (d \omega)_c \neq 0 \). Analogously, since \( \omega'' \leq  \omega \), we have \( (d\omega'')_c  = (d \omega)_c \). Combining these observations, it follows that 
    \begin{equation*}
        0 \neq (d \omega)_c = (d(\omega' + \omega''))_c = (d \omega' )_c + (d  \omega'')_c
        = (d \omega)_c  + (d \omega)_c 
        = 2 (d \omega)_c, 
    \end{equation*}
    which is a contradiction. This proves~\ref{property 5}.
\end{proof}

The next lemma guarantees the existence of minimal elements satisfying certain constraints.
\begin{lemma}\label{lemma: reduction V}
    Let \( k \in \{ 0,1,\dots, m-1 \}\), let \( \Omega \subseteq \Sigma_k \), and let \( \omega \in \Omega \). 
    Then there is \(  \omega' \leq \omega \) such that
    \begin{enumerate}[label=\textnormal{(\roman*)}]
        \item \(   \omega' \in \Omega \), and \label{item: new fix lemma i}
        \item there is no \( \omega'' <  \omega' \) such that \( \omega'' \in \Omega \).\label{item: new fix lemma ii}
    \end{enumerate}
\end{lemma}

\begin{proof}
    Since \( N \) is finite, the set \( \Sigma_k \) is finite, and hence \( \Omega \) is finite.
    Assume for contradiction that there is no \( \omega' \leq \omega \) such that~\ref{item: new fix lemma i} and~\ref{item: new fix lemma ii} both hold. Then there exists an infinite sequence \( \omega^{(1)},\omega^{(2)},\omega^{(3)},... \in \Omega \) such that \( \omega > \omega^{(1)} > \omega^{(2)} > \ldots \).
    Now let \( 1 \leq i < j \). By Lemma~\ref{lemma: the blue lemma}, we have \( \omega^{(j)} < \omega^{(i)} \), and hence \( \omega^{(j)} \neq \omega^{(i)} \). In particular, all the \( \omega^{(i)}\) must be distinct. Since this contradicts the assumption that \( \Omega \) is finite, the desired conclusion follows.
\end{proof}

\subsubsection{Irreducible configurations}\label{sec: irreducibility}

The partial ordering given in Definition~\ref{def: partial order} allows us to introduce a notion of irreducibility.

\begin{definition}\label{def: irreducible kform}
    When \( k \in \{ 0,1,\dots, m-1 \} \), a configuration \( \omega \in \Sigma_k \) is said to be \emph{irreducible} if there is no non-trivial configuration \( \omega' \in \Sigma_k \) such that \( \omega' < \omega \).
\end{definition}

Equivalently, \( \omega \in \Sigma_k \) is irreducible if there is no  non-empty set \( S \subsetneq \support \omega \) such that \( \support d(\omega|_S) \) and \( \support d (\omega|_{S^c}) \) are disjoint.
Note that if \( \omega \in \Sigma_k \)  satisfies \( d\omega = 0 \), then \( \omega \) is irreducible if and only if there is no non-empty set \( S \subsetneq \support \omega \) such that \(   d(\omega|_S) = d(\omega|_{S^c}) = 0\). 

The main reason the notion of irreducibility is important to us is the following lemma.

\begin{lemma}\label{lemma: lemma sum of irreducible configurations} 
    Let \( k \in \{ 0,1,\dots, m-1 \} \), and let $\omega\in \Sigma_k $ be non-trivial.
    Then there is an integer \( j \geq 1 \) and $k$-forms \( \omega^{(1)}, \ldots, \omega^{(j)} \in \Sigma_k \) such that
    \begin{enumerate}[label=\textnormal{(\roman*)}]
        \item\label{lemma210property1} for each \( i \in \{ 1,2, \ldots, j \} \), \( \omega^{(i)}  \) is non-trivial and irreducible,
        \item\label{lemma210property2} for each \( i \in \{ 1,2, \ldots, j \} \), \( \omega^{(i)} \leq \omega \),
        \item\label{lemma210property3} \( \omega^{(1)}, \dots, \omega^{(j)} \) have disjoint supports, and
        
        \item\label{lemma210property4} \( \omega = \omega^{(1)} + \dots + \omega^{(j)} \).

         %\item \label{lemma210property5} \( d\omega^{(1)}, \dots, d\omega^{(j)} \) have disjoint supports. 
         
    \end{enumerate} 
\end{lemma}

\begin{proof}
    If \( \omega \) is irreducible, then the conclusion of the lemma trivially holds, and hence we can assume that this is not the case.  
    Since \( N \) is finite and \( \omega \in \Sigma_k \),
    the set \( \support \omega \) is finite, and hence the set \( \Omega \coloneqq \{\omega' \in \Sigma_k \colon 0 < \omega' \leq \omega\} \subseteq \Sigma_k \) must also be finite. Consequently, by Lemma~\ref{lemma: reduction V}, there is a non-trivial and irreducible configuration \(   \omega^{(1)} \in \Sigma_k \) such that \(  \omega^{(1)} \leq \omega \). By definition, we have
    \begin{equation*} 
        \omega = \omega^{(1)} + (\omega-\omega^{(1)}),
    \end{equation*}
    and since \(  \omega^{(1)} \leq \omega \), we have \( \support (\omega-\omega^{(1)})  \subsetneq \support \omega \). Moreover, since \( \omega^{(1)} \leq \omega \), the configurations \( \omega^{(1)} \) and \( \omega-\omega^{(1)} \)  have disjoint supports.
    
    Since $\omega-\omega^{(1)} \in \Sigma_k $ and \( |\support (\omega-\omega^{(1)})| < | \support \omega| < \infty \), either \( \omega-\omega^{(1)}  \) is irreducible, or we can repeat the above argument with \( \omega \)  replaced with \( \omega-\omega^{(1)} \) to find a non-trivial and irreducible  configuration \( \omega^{(2)} \leq \omega-\omega^{(1)} \). Then
    \begin{equation*} 
        \omega-\omega^{(1)} = \omega^{(2)} + (\omega-\omega^{(1)}-\omega^{(2)}).
    \end{equation*}
    and hence 
    \begin{equation*}
        \omega = \omega^{(1)} + \omega^{(2)} + (\omega-\omega^{(1)}-\omega^{(2)}).
    \end{equation*}
    We now make the following observations.
    \begin{enumerate}[label=(\arabic*)]
    
        \item Since \( \omega^{(2)} \leq \omega-\omega^{(1)} \) and \( \omega-\omega^{(1)} \leq \omega \), it follows from Lemma \ref{lemma: the blue lemma}\ref{property 3} that \( \omega^{(2)} \leq \omega \).
        
        \item Since \( \omega^{(2)} \leq \omega-\omega^{(1)} \), the configurations \( \omega^{(2)} \) and \(  \omega-\omega^{(1)} -\omega^{(2)}\) have disjoint supports and \( \support \omega^{(2)} \subsetneq \support (\omega-\omega^{(1)}) \). Since \( \omega^{(1)} \leq \omega \), the configurations \( \omega^{(1)} \) and \( \omega - \omega^{(1)} \) have disjoint supports, and hence \( \omega^{(1)} \) and \( \omega^{(2)} \) have disjoint supports. In addition, we have 
        \begin{equation*}
            \support (\omega - \omega^{(1)}-\omega^{(2)}) \subsetneq \support (\omega -\omega^{(1)}) \subsetneq \support \omega .
        \end{equation*}
    \end{enumerate} 
    
    Repeating the above argument, we obtain a sequence \( \omega^{(1)}, \omega^{(2)}, \omega^{(3)},\ldots \leq \omega \) of non-trivial and irreducible configurations in \( \Sigma_k \) with disjoint supports, which satisfies
    \begin{equation*}
        \support \omega \supsetneq \support (\omega-\omega^{(1)}) \supsetneq \support (\omega-\omega^{(1)}-\omega^{(2)}) \supsetneq \dots
    \end{equation*}
    Since \( \support \omega \) is a finite set, this process must eventually stop. Equivalently, there must exist some \( j \geq 1 \) such that \( \omega - \omega^{(1)} -\dots -\omega^{(j)} \) is trivial. For this \( j \), we have
    \begin{equation*}
        \omega = \omega^{(1)} + \dots+ \omega^{(j)} + (\omega - \omega^{(1)} -\dots -\omega^{(j)} ),
    \end{equation*}
    and hence the existence of $k$-forms satisfying  \ref{lemma210property1}--\ref{lemma210property4} follows. %Property \ref{lemma210property5} is a direct consequence of \ref{lemma210property2} and  \ref{lemma210property4}.
\end{proof}
 
\begin{remark}\label{remark: decomposition is not unique}
    The decomposition in Lemma~\ref{lemma: lemma sum of irreducible configurations} is not unique. To see this, consider first the \( \mathbb{Z}^2 \)-lattice. Fix some plaquette  \( p   \) and define a gauge field configuration \( \sigma \) by  \( e \mapsto \sigma_e \coloneqq \mathbb{1}(e \in   \partial p)-\mathbb{1}(-e \in   \partial p)   \). Then it is easy to see that there are three distinct ways to write \(  \sigma \) as a sum of irreducible gauge field configurations \( \sigma_1 \) and \( \sigma_2   \) with disjoint supports (see Figure~\ref{fig: non-unique decomposition example}). With slightly more work, an analogous argument gives a counter-example also for the \( \mathbb{Z}^4 \)-lattice. (Let \( c \) be a 4-cell and for \( e \in E_N \) define \( \sigma_e \coloneqq \mathbb{1}(e \in \partial\partial\partial c)-\mathbb{1}(-e \in \partial\partial\partial c) \).)
\end{remark}

\begin{figure}
    \centering 
    \pgfmathsetmacro{\bd}{0.2} 
    \begin{subfigure}[t]{0.22\textwidth}\centering
        \centering
        \begin{tikzpicture}[scale=0.6] 
    
            \draw[help lines] (3,1) grid (8,6);
         
            \draw[very thick, dotted, detailcolor00] (5,4) -- (6,4);
            \draw[very thick, detailcolor01] (6,4) -- (6,3);   
            \draw[very thick, dotted, detailcolor00] (6,3) -- (5,3);   
            \draw[very thick, detailcolor01] (5,3) -- (5,4);  
            
            \fill[fill=detailcolor09, fill opacity=0.6] (4+\bd,3+\bd) -- (4+\bd,4-\bd) -- (5-\bd,4-\bd) -- (5-\bd,3+\bd) -- (4+\bd,3+\bd); 
            \fill[fill=detailcolor09, fill opacity=0.6] (6+\bd,3+\bd) -- (6+\bd,4-\bd) -- (7-\bd,4-\bd) -- (7-\bd,3+\bd) -- (6+\bd,3+\bd);  
            \fill[fill=detailcolor09, fill opacity=0.6, fill opacity=0.5] (5+\bd,4+\bd) -- (5+\bd,5-\bd) -- (6-\bd,5-\bd) -- (6-\bd,4+\bd) -- (5+\bd,4+\bd); 
            \fill[fill=detailcolor09, fill opacity=0.6, fill opacity=0.5] (5+\bd,2+\bd) -- (5+\bd,3-\bd) -- (6-\bd,3-\bd) -- (6-\bd,2+\bd) -- (5+\bd,2+\bd); 
           
        \end{tikzpicture} 
        \caption{}
    \end{subfigure} 
    \begin{subfigure}[t]{0.22\textwidth}\centering
        \begin{tikzpicture}[scale=0.6] 
    
            \draw[help lines] (3,1) grid (8,6);
         
            \draw[very thick, dotted, detailcolor00] (5,4) -- (6,4);
            \draw[very thick, dotted, detailcolor00] (6,4) -- (6,3);   
            \draw[very thick, detailcolor01] (6,3) -- (5,3);   
            \draw[very thick, detailcolor01] (5,3) -- (5,4);   
          
            \fill[fill=detailcolor09, fill opacity=0.6] (4+\bd,3+\bd) -- (4+\bd,4-\bd) -- (5-\bd,4-\bd) -- (5-\bd,3+\bd) -- (4+\bd,3+\bd); 
            \fill[fill=detailcolor09, fill opacity=0.6] (5+\bd,2+\bd) -- (5+\bd,3-\bd) -- (6-\bd,3-\bd) -- (6-\bd,2+\bd) -- (5+\bd,2+\bd); 
            \fill[fill=detailcolor09, fill opacity=0.6] (5+\bd,4+\bd) -- (5+\bd,5-\bd) -- (6-\bd,5-\bd) -- (6-\bd,4+\bd) -- (5+\bd,4+\bd); 
            \fill[fill=detailcolor09, fill opacity=0.6] (6+\bd,3+\bd) -- (6+\bd,4-\bd) -- (7-\bd,4-\bd) -- (7-\bd,3+\bd) -- (6+\bd,3+\bd);c
        \end{tikzpicture} 
        \caption{}
    \end{subfigure}
    \begin{subfigure}[t]{0.22\textwidth}\centering
        \begin{tikzpicture}[scale=0.6] 
    
            \draw[help lines] (3,1) grid (8,6);
         
            \draw[very thick, detailcolor01] (5,4) -- (6,4);
            \draw[very thick, dotted, detailcolor00] (6,4) -- (6,3);   
            \draw[very thick, dotted, detailcolor00] (6,3) -- (5,3);   
            \draw[very thick, detailcolor01] (5,3) -- (5,4);   
          
            \fill[fill=detailcolor09, fill opacity=0.6] (4+\bd,3+\bd) -- (4+\bd,4-\bd) -- (5-\bd,4-\bd) -- (5-\bd,3+\bd) -- (4+\bd,3+\bd); 
            \fill[fill=detailcolor09, fill opacity=0.6] (5+\bd,4+\bd) -- (5+\bd,5-\bd) -- (6-\bd,5-\bd) -- (6-\bd,4+\bd) -- (5+\bd,4+\bd); 
            \fill[fill=detailcolor09, fill opacity=0.6] (5+\bd,2+\bd) -- (5+\bd,3-\bd) -- (6-\bd,3-\bd) -- (6-\bd,2+\bd) -- (5+\bd,2+\bd); 
            \fill[fill=detailcolor09, fill opacity=0.6] (6+\bd,3+\bd) -- (6+\bd,4-\bd) -- (7-\bd,4-\bd) -- (7-\bd,3+\bd) -- (6+\bd,3+\bd); 
        \end{tikzpicture} 
        \caption{}
    \end{subfigure}
    
    \caption{Given a plaquette \( p  \) in \( \mathbb{Z}^2 \), define a gauge field configuration \( \sigma  \) on \( \mathbb{Z}^2 \) by  \( e \mapsto \sigma_e \coloneqq \mathbb{1}(e \in   \partial p)-\mathbb{1}(-e \in   \partial p)   \). The three figures above represent three different ways to write
    \(  \sigma \) as a sum \( \sigma_1+\sigma_2\) of non-trivial and irreducible gauge field configurations \( \sigma_1 \) (with support at solid lines) and \( \sigma_2  \) (with support at dotted lines) with disjoint supports. The plaquettes at which \( d\sigma \) has support are drawn blue. In particular, the above figures show that the decompositions whose existence are guaranteed by Lemma~\ref{lemma: lemma sum of irreducible configurations}, are in general not unique. (See also Remark~\ref{remark: decomposition is not unique}.)}
    \label{fig: non-unique decomposition example}
\end{figure}
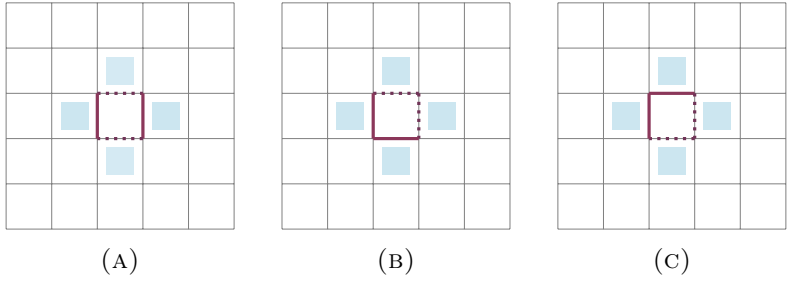

In the special case that \( k = 2 \) and \( \omega \in \Sigma_2 \) satisfies \( d\omega = 0 \), Lemma~\ref{lemma: lemma sum of irreducible configurations} gives the following lemma.
 
\begin{lemma}\label{lemma: lemma sum of irreducible vortices}
    Let $\omega\in \Sigma_{P_N} $ be non-trivial. Then there is a \( k \geq 1 \) and \( \omega_1, \ldots, \omega_k \in \Sigma_{P_N} \) such that
    \begin{enumerate}[label=\textnormal{(\roman*)}]
        \item for each \( j \in \{ 1,2, \ldots, k \} \), \( \omega_j  \) is non-trivial and irreducible,
        \item for each \( j \in \{ 1,2, \ldots, k \} \), \( \omega_j \leq \omega \),
        \item \( \omega_1 ,\dots, \omega_k \) have disjoint supports, and
        \item \( \omega = \omega_1 + \dots + \omega_k \).
    \end{enumerate} 
\end{lemma}

\begin{proof}
    The desired conclusion will follow immediately from Lemma~\ref{lemma: lemma sum of irreducible configurations} if we can show that whenever \( \omega' \in \Sigma_2 \) is such that \( \omega' \leq \omega \), then \( d\omega' = 0 \), and hence \( \omega' \in \Sigma_{P_N} \). However, this is an immediate consequence of the definition of the relation \( \leq \).
\end{proof}

\subsubsection{Oriented surfaces}\label{sec: oriented surfaces}

In this section, we introduce oriented surfaces, and outline their connection to simple loops. 

\begin{definition}
    A \( \mathbb{Z} \)-valued \( 2 \)-form \( q \)  is said to be an \emph{oriented surface} if 
    \begin{equation*}
         \sum_{p \in \hat \partial e} q_p \in \{ -1,0,1 \}
    \end{equation*}
    for all \( e \in E_N \).
    If \( q \) is an oriented surface, then the \emph{boundary} of \( q \) is the set 
    \begin{equation*}
        B_q \coloneqq \biggl\{ e \in E_N \colon \sum_{p \in \hat \partial e} q_p = 1 \biggr\}.
    \end{equation*}
    An edge \( e \in E_N \) is an \emph{internal edge} of \( q \) if there is a \( p \in \support q \) such that \( e \in \partial p \), but neither \( e \) nor \( -e \) belongs to $B_q$.
\end{definition}

The following lemma is a discrete analogue of Stokes' theorem. If \( q \) is an oriented surface and \( p \in P_N \), we write $q^+_p \coloneqq \max(q_p,0)$.
If \( a \in \mathbb{Z} \) and \( g \in G \), we let \( ag \) denote the sum $g + \dots + g$ with $a$ terms.

\begin{lemma}[Lemma~5.2 in~\cite{flv2020}]\label{lemma: stokes}
    Let \( q \) be an oriented surface with boundary \( B_q \). For any \( \sigma \in \Sigma_{E_N} \), we have
    \begin{equation*}
        \sum_{p\in P_N  } q_p^+  (d\sigma)_p = \sum_{e \in B_q} \sigma_e.
    \end{equation*}
\end{lemma}

\begin{lemma}[Lemma~5.3 in~\cite{flv2020}]\label{lemma: oriented loops}
    Let \( \gamma \) be a simple loop in \( E_N   \). 
    %, which is contained in a box \( B \subseteq B_N \). 
    Then there exists an oriented surface \( q \) 
    %with support contained in \( B \) and 
    such that \( \gamma \) is the boundary of \( q \).  
\end{lemma}
 
If \( q \) is an oriented surface and \( \hat p \in \support q \), then \( \hat p \) is an \emph{internal plaquette} of \( q \) if 
\begin{equation*}
    \sum_{p \in \hat \partial e} q_p = 0 \quad \text{for every $e \in \partial \hat p$}.
\end{equation*}
Equivalently, a plaquette \( p \in \support q \) is an internal plaquette if   \( \pm \partial p \cap B_q = \emptyset\).

\subsubsection{Minimal configurations}\label{sec: minimal configurations}

In this section, we assume that \( m = 4 \). In other words, we assume that we are working on the \(\mathbb{Z}^4 \)-lattice.

The first lemma of this section gives a relationship between the size of the support of a plaquette configuration and the size of the support of the corresponding gauge field configuration.
\begin{lemma}\label{lemma: 6 plaquettes per edge}
    Let \( \sigma \in \Sigma_{E_N} \). Then
    \begin{equation*}
        |\support  \sigma| \geq |\support d \sigma|/6.
    \end{equation*}
\end{lemma}

\begin{proof}
    Assume that \( p \in \support d \sigma \). Since \( (d \sigma)_p \neq 0 \), there must exist at least one edge \( e \in \partial p \subseteq E_N \) such that \(  \sigma_e \neq 0 \), and hence \( e \in \support  \sigma \). Since each edge \( e \in E_N \) satisfies \( \bigl| \hat \partial e \cap P_N \bigr| \leq 6 \), the desired conclusion follows.
\end{proof}

\begin{lemma}\label{lemma: minimal vortex I}
    Let \( \omega \in \Sigma_{P_N} \) be non-trivial, and assume that there is a plaquette \( p \in \support \omega \) such that \( \hat \partial \partial p \subseteq P_N \). Then
    \begin{enumerate}[label=\textnormal{(\roman*)}]
        \item \( | (\support \omega)^+ | \geq 6\), and \label{item: minimal vortex I i}
        \item if \(| (\support \omega)^+| = 6\), then there is an edge \( e_0 \in E_N \) such that \( \support \nu = \hat \partial e_0 \cup \hat \partial (-e_0)\).\label{item: minimal vortex}
    \end{enumerate}
\end{lemma}
For a proof of Lemma~\ref{lemma: minimal vortex I}, see, e.g., Lemma~3.4.6~in~\cite{sc2019}.

\begin{remark}
    Without the assumption that \( \hat \partial \partial p \subseteq P_N \), the conclusion of Lemma~\ref{lemma: minimal vortex I} does not hold. To see this, let \( e \in E_N \) be in the boundary of \( B_N \), assume that \( \sigma \in \Sigma_{E_N} \) has support exactly at \( \pm e \). Then \( d\sigma \in \Sigma_{P_N} \), but we have \( |(\support d\sigma)^+| \in \{ 3,4,5\}, \) depending on the position of \( e \) at the boundary.
\end{remark}

\begin{lemma}\label{lemma: small 1forms}
    Let \( \sigma \in \Sigma_{E_N}^0 \) be non-trivial, and assume that there is an edge \( e\in \support \sigma \) such that \( \partial \hat \partial \partial \hat \partial e \subseteq E_N\). Then \(|(\support \sigma)^+ | \geq 8 \).
\end{lemma}

\begin{proof}
    Let \( e\in \support \sigma \) be such that \( \partial \hat \partial \partial \hat \partial e \subseteq E_N\) (see Figure~\ref{fig: edge}). Since $\sigma$ is closed, \( (d\sigma)_p = 0 \) for each \( p \in \hat \partial e \). Hence, since \( \sigma_e \neq 0 \), for each \( p \in \hat \partial e \) there exists an edge \( e' \in \partial p \smallsetminus \{ e \} \) such that \( \sigma_{e'} \neq 0 \) (see Figure~\ref{fig: edge and selection}). Since  \( |\hat \partial e| = 6 \) (see Figure~\ref{fig: edge and adjacent plaquettes}), \( \partial \hat \partial e \subseteq E_N \), and for distinct \( p,p' \in \hat \partial e \) the sets  \( \partial p\smallsetminus \{ e \} \) and  \( \partial p'\smallsetminus \{ e \} \) are disjoint, it follows that
    \begin{equation*}
        \bigl|\bigl((\partial \hat \partial e)\smallsetminus \{ e \}\bigr) \cap \support \sigma \bigr| \geq 6.
    \end{equation*}
    
    Next, fix some \( e' \in (\partial \hat \partial e) \smallsetminus \{ e \} \) with \( \sigma_{e'} \neq 0 \). Then \( \hat \partial e' \) contains at least one plaquette \( p' \) with 
    \begin{equation*}
        \bigl( \partial p' \smallsetminus \{ e' \} \bigr) \cap \bigl( \partial \hat \partial e\bigr) = \emptyset.
    \end{equation*}
    Fix such a plaquette \( p' \) (see Figure~\ref{fig: edge and selection and plaquette}).
    Since \( d\sigma = 0 \), we must have \( (d\sigma)_{p'} = 0 \). Consequently, since \( e' \in \partial p' \) and \( \sigma_{e'} \neq 0 \), there must exist \( e'' \in \partial p' \smallsetminus \{ e' \} \) such that \( \sigma_{e''} \neq 0 \). Note that by the choice of \( p' \) we have \( e'' \in \partial \hat \partial \partial \hat \partial e \smallsetminus  \partial \hat \partial e \subseteq E_N\). Consequently, we must have
    \begin{equation*}
        \bigl| \bigl( (\partial \hat \partial \partial \hat \partial e ) \smallsetminus (\partial \hat \partial e )\bigr) \cap \support \sigma \bigr| \geq 1.
    \end{equation*}
    Combining the above observations, and recalling that by assumption, \( \partial \hat \partial \partial \hat \partial e \in E_N \), we obtain 
    \begin{equation*}
        |(\support \sigma)^+| \geq |\{ e \}| + \bigl|\bigl((\partial \hat \partial e)\smallsetminus \{ e \}\bigr) \cap \support \sigma\bigr| + \bigl| \bigl( (\partial \hat \partial \partial \hat \partial e ) \smallsetminus (\partial \hat \partial e )\bigr) \cap \support \sigma \bigr| 
        \geq 1+6+1 = 8,
    \end{equation*}
    which is the desired conclusion.
\end{proof}

\begin{figure}[ht]
     \centering
     \begin{subfigure}[t]{0.3\textwidth}
         \centering
         \begin{tikzpicture}
            \draw[white, line width=0.01mm] (0,0) -- (1,0) -- (1,1) -- (0,1) -- (0,0);
            \draw[white, line width=0.01mm] (0,0) -- (0,1) -- (-1,1) -- (-1,0) -- (0,0);
            \draw[white, line width=0.01mm] (0,0) -- (0.55,-0.3) -- (0.55,0.7) -- (0,1) -- (0,0);
            \draw[white, line width=0.01mm] (0,0) -- (0,1) -- (-0.55,1.3) -- (-0.55,0.3) -- (0,0);
            \draw[white, line width=0.01mm] (0,0) -- (0.7,0.25) -- (0.7,1.25) -- (0,1) -- (0,0);
            \draw[white, line width=0.01mm] (0,0) -- (0,1) -- (-0.7,0.75) -- (-0.7,-0.25) -- (0,0);
            
            \draw[thick, detailcolor03] (0,0) -- (0,1);
        \end{tikzpicture}
         \caption{An edge \( e \in \support \sigma\) (red).}
         \label{fig: edge}
     \end{subfigure}
     \hfill
     \begin{subfigure}[t]{0.3\textwidth}
         \centering
         \begin{tikzpicture}
            \filldraw[fill=detailcolor07, fill opacity=0.14, draw opacity=0.4, line width=0.01mm] (0,0) -- (1,0) -- (1,1) -- (0,1);
            \filldraw[fill=detailcolor07, fill opacity=0.14, draw opacity=0.4, line width=0.01mm] (0,1) -- (-1,1) -- (-1,0) -- (0,0);
            \filldraw[fill=detailcolor07, fill opacity=0.14, draw opacity=0.4, line width=0.01mm] (0,0) -- (0.55,-0.3) -- (0.55,0.7) -- (0,1);
            \filldraw[fill=detailcolor07, fill opacity=0.14, draw opacity=0.4, line width=0.01mm] (0,1) -- (-0.55,1.3) -- (-0.55,0.3) -- (0,0);
            \filldraw[fill=detailcolor07, fill opacity=0.14, draw opacity=0.4, line width=0.01mm] (0,0) -- (0.7,0.25) -- (0.7,1.25) -- (0,1);
            \filldraw[fill=detailcolor07, fill opacity=0.14, draw opacity=0.4, line width=0.01mm] (0,1) -- (-0.7,0.75) -- (-0.7,-0.25) -- (0,0);
            
            \draw[thick, detailcolor03] (0,0) -- (0,1);
        \end{tikzpicture}
         \caption{An edge \( e \in \support \sigma \) (red) and the plaquettes in the set \( \hat \partial e \) (purple).}
         \label{fig: edge and adjacent plaquettes}
     \end{subfigure}
     \hfill
     \begin{subfigure}[t]{0.3\textwidth}
         \centering
         \begin{tikzpicture}
            \filldraw[fill=detailcolor07, fill opacity=0.14, draw opacity=0.4, line width=0.01mm] (0,0) -- (1,0) -- (1,1) -- (0,1) -- (0,0);
            \filldraw[fill=detailcolor07, fill opacity=0.14, draw opacity=0.4, line width=0.01mm] (0,0) -- (0,1) -- (-1,1) -- (-1,0) -- (0,0);
            \filldraw[fill=detailcolor07, fill opacity=0.14, draw opacity=0.4, line width=0.01mm] (0,0) -- (0.55,-0.3) -- (0.55,0.7) -- (0,1) -- (0,0);
            \filldraw[fill=detailcolor07, fill opacity=0.14, draw opacity=0.4, line width=0.01mm] (0,0) -- (0,1) -- (-0.55,1.3) -- (-0.55,0.3) -- (0,0);
            \filldraw[fill=detailcolor07, fill opacity=0.14, draw opacity=0.4, line width=0.01mm] (0,0) -- (0.7,0.25) -- (0.7,1.25) -- (0,1) -- (0,0);
            \filldraw[fill=detailcolor07, fill opacity=0.14, draw opacity=0.4, line width=0.01mm] (0,0) -- (0,1) -- (-0.7,0.75) -- (-0.7,-0.25) -- (0,0);
            
            \draw[thick, detailcolor03] (0,0) -- (0,1);
            
            \draw[thick, Black] (0,1) -- (1,1);
            \draw[thick, Black] (0,1) -- (-1,1);
            \draw[thick, Black] (0,1) -- (0.55,0.7);
            \draw[thick, Black] (0,1) -- (-0.55,1.3);
            \draw[thick, Black] (0,1) -- (0.7,1.25);
            \draw[thick, Black] (0,1) -- (-0.7,0.75); 
            
        \end{tikzpicture}
         \caption{An edge \( e \in \support \sigma \) (red) and an edge \( e_p \in (\partial p \smallsetminus \{ e \}) \cap \support \sigma \) for each plaquette \( p \in \hat \partial e \) (black) (depending on \( \sigma \)).}
         \label{fig: edge and selection}
     \end{subfigure}
     
     \vspace{4ex}
     
     \begin{subfigure}[t]{0.3\textwidth}
         \centering
         \begin{tikzpicture}
         
            \filldraw[fill=detailcolor07, fill opacity=0.14, draw opacity=0.4, line width=0.01mm] (0,0) -- (1,0) -- (1,1) -- (0,1) -- (0,0);
            \filldraw[fill=detailcolor07, fill opacity=0.14, draw opacity=0.4, line width=0.01mm] (0,0) -- (0,1) -- (-1,1) -- (-1,0) -- (0,0);
            \filldraw[fill=detailcolor07, fill opacity=0.14, draw opacity=0.4, line width=0.01mm] (0,0) -- (0.55,-0.3) -- (0.55,0.7) -- (0,1) -- (0,0);
            \filldraw[fill=detailcolor07, fill opacity=0.14, draw opacity=0.4, line width=0.01mm] (0,0) -- (0,1) -- (-0.55,1.3) -- (-0.55,0.3) -- (0,0);
            \filldraw[fill=detailcolor07, fill opacity=0.14, draw opacity=0.4, line width=0.01mm] (0,0) -- (0.7,0.25) -- (0.7,1.25) -- (0,1) -- (0,0);
            \filldraw[fill=detailcolor07, fill opacity=0.14, draw opacity=0.4, line width=0.01mm] (0,0) -- (0,1) -- (-0.7,0.75) -- (-0.7,-0.25) -- (0,0);
            
            \draw[thick, detailcolor03] (0,0) -- (0,1); 
            
            \draw[thick, Black, opacity=0] (0,1) -- (0,2);
            
            \draw[thick, Black] (0,1) -- (1,1);
            \draw[ultra thick,dashed, Black] (0,1) -- (-1,1);
            \draw[thick, Black] (0,1) -- (0.55,0.7);
            \draw[thick, Black] (0,1) -- (-0.55,1.3);
            \draw[thick, Black] (0,1) -- (0.7,1.25);
            \draw[thick, Black] (0,1) -- (-0.7,0.75);  
            
        \end{tikzpicture}
         \caption{An edge \( e \in \support \sigma \) (red), a choice of one edge \( e_p \in (\partial p \smallsetminus \{ e \}) \cap \support \sigma \) for each plaquette \( p \in \hat \partial e \) (black), and an edge \( e' \in \cup_{p \in \hat \partial e} \{ e_p \}\) (dashed).}
         \label{fig: edge and selection and edge}
     \end{subfigure}
     \hfill
     \begin{subfigure}[t]{0.3\textwidth}
         \centering
         \begin{tikzpicture}
            \filldraw[fill=detailcolor07, fill opacity=0.14, draw opacity=0.4, line width=0.01mm] (0,0) -- (1,0) -- (1,1) -- (0,1) -- (0,0);
            \filldraw[fill=detailcolor07, fill opacity=0.14, draw opacity=0.4, line width=0.01mm] (0,0) -- (0,1) -- (-1,1) -- (-1,0) -- (0,0);
            \filldraw[fill=detailcolor07, fill opacity=0.14, draw opacity=0.4, line width=0.01mm] (0,0) -- (0.55,-0.3) -- (0.55,0.7) -- (0,1) -- (0,0);
            \filldraw[fill=detailcolor07, fill opacity=0.14, draw opacity=0.4, line width=0.01mm] (0,0) -- (0,1) -- (-0.55,1.3) -- (-0.55,0.3) -- (0,0);
            \filldraw[fill=detailcolor07, fill opacity=0.14, draw opacity=0.4, line width=0.01mm] (0,0) -- (0.7,0.25) -- (0.7,1.25) -- (0,1) -- (0,0);
            \filldraw[fill=detailcolor07, fill opacity=0.14, draw opacity=0.4, line width=0.01mm] (0,0) -- (0,1) -- (-0.7,0.75) -- (-0.7,-0.25) -- (0,0);
            
            \draw[thick, detailcolor03] (0,0) -- (0,1);
            
            \draw[thick, Black] (0,1) -- (1,1);
            \draw[ultra thick,dashed, Black] (0,1) -- (-1,1);
            \draw[thick, Black] (0,1) -- (0.55,0.7);
            \draw[thick, Black] (0,1) -- (-0.55,1.3);
            \draw[thick, Black] (0,1) -- (0.7,1.25);
            \draw[thick, Black] (0,1) -- (-0.7,0.75); 
            
            \filldraw[fill=detailcolor08, fill opacity=0.6, line width=0.01mm](0,1) -- (-1,1) -- (-1,2) -- (0,2) -- (0,1);
        \end{tikzpicture}
         \caption{An edge \( e \in \support \sigma \) (red), a choice of one edge \( e_p \in (\partial p \smallsetminus \{ e \}) \cap \support \sigma \) for each plaquette \( p \in \hat \partial e \) (black), and a plaquette \( p' \in \hat \partial e'  \) such that \( (\partial p' \smallsetminus \{ e' \}) \cap (\partial \hat \partial e) = \emptyset\) (blue).}
         \label{fig: edge and selection and plaquette}
     \end{subfigure}
     \hfill
     \begin{subfigure}[t]{0.3\textwidth}
         \centering
         \begin{tikzpicture}
            \filldraw[fill=detailcolor07, fill opacity=0.14, draw opacity=0.4, line width=0.01mm] (0,0) -- (1,0) -- (1,1) -- (0,1) -- (0,0);
            \filldraw[fill=detailcolor07, fill opacity=0.14, draw opacity=0.4, line width=0.01mm] (0,0) -- (0,1) -- (-1,1) -- (-1,0) -- (0,0);
            \filldraw[fill=detailcolor07, fill opacity=0.14, draw opacity=0.4, line width=0.01mm] (0,0) -- (0.55,-0.3) -- (0.55,0.7) -- (0,1) -- (0,0);
            \filldraw[fill=detailcolor07, fill opacity=0.14, draw opacity=0.4, line width=0.01mm] (0,0) -- (0,1) -- (-0.55,1.3) -- (-0.55,0.3) -- (0,0);
            \filldraw[fill=detailcolor07, fill opacity=0.14, draw opacity=0.4, line width=0.01mm] (0,0) -- (0.7,0.25) -- (0.7,1.25) -- (0,1) -- (0,0);
            \filldraw[fill=detailcolor07, fill opacity=0.14, draw opacity=0.4, line width=0.01mm] (0,0) -- (0,1) -- (-0.7,0.75) -- (-0.7,-0.25) -- (0,0);
            
            \draw[thick, detailcolor03] (0,0) -- (0,1);
            
            \draw[thick, Black] (0,1) -- (0,2);
            
            \draw[thick, Black] (0,1) -- (1,1);
            \draw[thick, Black] (0,1) -- (-1,1);
            \draw[thick, Black] (0,1) -- (0.55,0.7);
            \draw[thick, Black] (0,1) -- (-0.55,1.3);
            \draw[thick, Black] (0,1) -- (0.7,1.25);
            \draw[thick, Black] (0,1) -- (-0.7,0.75); 
        \end{tikzpicture}
         \caption{An edge \( e \) (red) and a set of edges (black) which correspond to a gauge field configuration \( \sigma \) with \( e \in \support \sigma \) and \( |(\support \sigma)^+| =8\).}
         \label{fig:five over x}
     \end{subfigure}
        \caption{The figures above illustrate the setting of Lemma~\ref{lemma: small 1forms}.}
        \label{fig: small 1forms}
\end{figure}
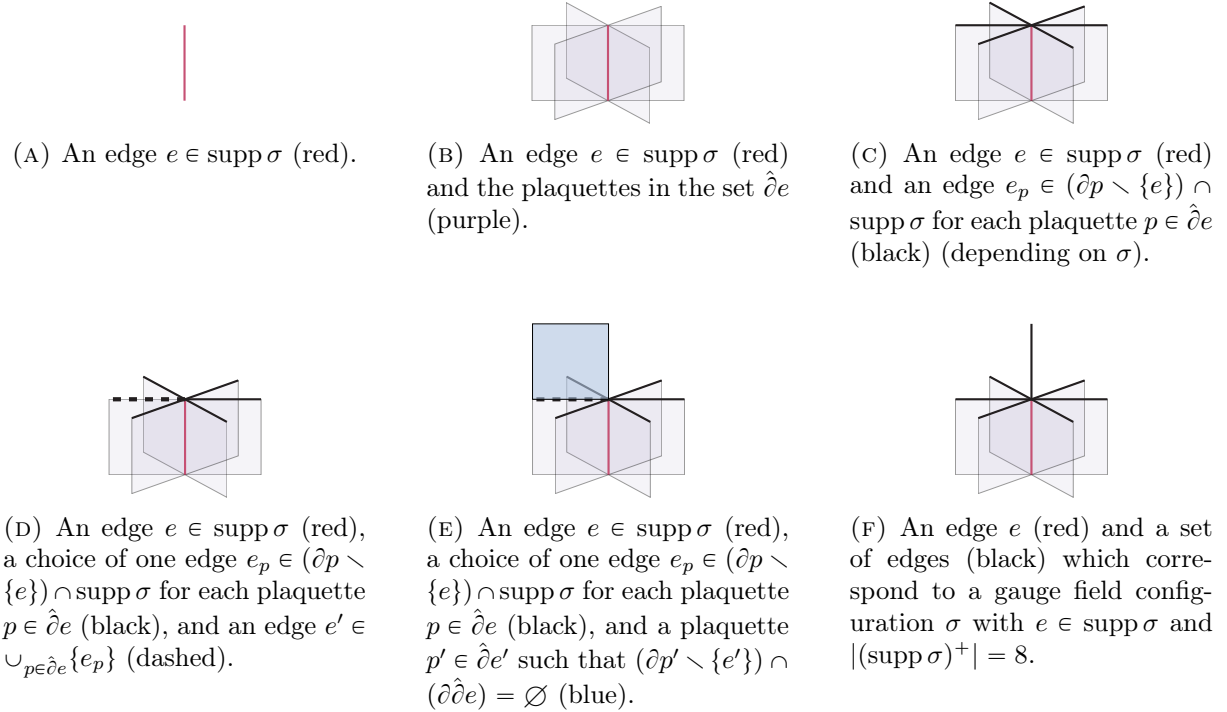

\begin{lemma}\label{lemma: other minimal configuration}
    Let \( g \in G\smallsetminus \{ 0 \}\), let \( e \in E_N \), and let \( \sigma \in \Sigma_{E_N} \) be such that \( (d\sigma)_p - \sigma_e = g \) for all \( p \in \hat \partial e \). Assume further that \( \partial \hat \partial \partial \hat \partial e \subseteq E_N\). Then, either
    \begin{enumerate}[label=\textnormal{(\roman*)}]
        \item \( \bigl|(\support \sigma)^+ \bigr| \geq 7\), or
        \item \( \bigl|(\support \sigma)^+\bigr| = 6\) and \( \bigl|(\support d\sigma)^+\bigr| \geq 12\).
    \end{enumerate}
\end{lemma}
 
\begin{proof}
    For each \( p \in \hat \partial e \), we have \( (d\sigma)_p - \sigma_e = g \neq 0 \) by assumption, and hence  \(| (\partial p\smallsetminus \{ e \}) \cap \support \sigma| \geq 1 \). Since \( |\hat \partial e|=6 \), \( \partial \hat \partial e \subseteq E_N\), and since for distinct \( p,p' \in \hat \partial e \), the sets \( \partial p\smallsetminus \{ e \} \) and \( \partial p'\smallsetminus \{ e \} \) are disjoint, it follows that \( |(\support \sigma)^+ | \geq  6 \). 
    
    Now assume that \( |(\support \sigma)^+ | = 6 \), and let \( \{ p_1,p_2,\dots,p_6 \} \coloneqq \hat \partial e \).
    For any \( j \in [6] \coloneqq \{1,2,3,4,5,6 \} \), we have \( (d\sigma)_{p_j}- \sigma_e = g \). 
    Since for any distinct \( j,j' = 1,2, \dots, 6 \), the sets \( \partial p_j \smallsetminus \{ e \} \) and \( \partial p_{j'} \smallsetminus \{ e \} \) are disjoint, and we have \( |\support \sigma| = 2 \cdot 6 \), it follows that there  exist \( e_1 \in \partial p_1 \smallsetminus \{ e \} \), \dots, \( e_6 \in \partial p_6\smallsetminus \{ e \} \), such that \( \support \sigma = \{ \pm e_1, \dots,  \pm e_6 \} \). 
    Now note that for any choice of \( e_j \in \partial p_j \smallsetminus \{ e \} \), \(j \in \{1,2, \dots, 6 \} \), we have \( |(\hat \partial e_j) \smallsetminus \bigcup_{k \in [6] \smallsetminus \{ j \}} \hat \partial e_k | \geq 2\) (see Figure~\ref{figure: more minimal configurations c}). 
    Also, we have \( \partial \hat \partial e_j \subseteq \partial \hat \partial \partial \hat \partial e \subseteq E_N,\) and hence \(  \hat \partial e_j \subseteq P_N.\)
    Consequently, since \( \support \sigma = \{ \pm e_1, \dots,  \pm e_6 \} \),
    \begin{equation*}
        \bigl|(\support d\sigma)^+ \bigr| \geq \sum_{j=1}^6 \biggl|(\hat \partial e_j) \, \setminus \bigcup_{k \in [6] \smallsetminus \{ j \}} \hat \partial e_k \biggr| \geq 12.
    \end{equation*}
    This concludes the proof.
\end{proof}

\begin{figure}[ht]
    \centering
    \begin{subfigure}[t]{0.3\textwidth}
         \centering
         \begin{tikzpicture}
            \filldraw[fill=detailcolor07, fill opacity=0.0, draw opacity=0.0, line width=0.01mm] (0,0) -- (1,0) -- (1,1) -- (0,1); 
            
            \filldraw[fill=detailcolor07, fill opacity=0.0, draw opacity=0.0, line width=0.01mm] (0,1) -- (-1,1) -- (-1,0) -- (0,0); 
            
            \filldraw[fill=detailcolor07, fill opacity=0.0, draw opacity=0.0, line width=0.01mm] (0,0) -- (0.55,-0.3) -- (0.55,0.7) -- (0,1); 
            
            \filldraw[fill=detailcolor07, fill opacity=0.0, draw opacity=0.0, line width=0.01mm] (0,1) -- (-0.55,1.3) -- (-0.55,0.3) -- (0,0);
            
            \filldraw[fill=detailcolor07, fill opacity=0.0, draw opacity=0.0, line width=0.01mm] (0,0) -- (0.7,0.25) -- (0.7,1.25) -- (0,1);
            
            \filldraw[fill=detailcolor07, fill opacity=0.0, draw opacity=0.0, line width=0.01mm] (0,1) -- (-0.7,0.75) -- (-0.7,-0.25) -- (0,0); 
            
            \draw[thick, detailcolor03] (0,0) -- (0,1); 
            \draw[draw opacity=0.0, line width=0.01mm] (0,1) -- (0,2); 
            
        \end{tikzpicture}
        \caption{The edge \( e \) (red).}
        \label{figure: more minimal configurations a}
        \end{subfigure}
        \hfil
        \begin{subfigure}[t]{0.3\textwidth}
         \centering
         \begin{tikzpicture}
            \filldraw[fill=detailcolor07, fill opacity=0.14, draw opacity=0.4, line width=0.01mm] (0,0) -- (1,0) -- (1,1) -- (0,1); 
            
            \filldraw[fill=detailcolor07, fill opacity=0.14, draw opacity=0.4, line width=0.01mm] (0,1) -- (-1,1) -- (-1,0) -- (0,0); 
            
            \filldraw[fill=detailcolor07, fill opacity=0.14, draw opacity=0.4, line width=0.01mm] (0,0) -- (0.55,-0.3) -- (0.55,0.7) -- (0,1); 
            
            \filldraw[fill=detailcolor07, fill opacity=0.14, draw opacity=0.4, line width=0.01mm] (0,1) -- (-0.55,1.3) -- (-0.55,0.3) -- (0,0);
            
            \filldraw[fill=detailcolor07, fill opacity=0.14, draw opacity=0.4, line width=0.01mm] (0,0) -- (0.7,0.25) -- (0.7,1.25) -- (0,1);
            
            \filldraw[fill=detailcolor07, fill opacity=0.14, draw opacity=0.4, line width=0.01mm] (0,1) -- (-0.7,0.75) -- (-0.7,-0.25) -- (0,0); 
             
            \draw[thick, detailcolor03] (0,0) -- (0,1); 
            \draw[draw opacity=0.0, line width=0.01mm] (0,1) -- (0,2); 
            
        \end{tikzpicture}
        \caption{The edge \( e \) (red), and the set \( \hat \partial e \) (purple).}
        \label{figure: more minimal configurations b}
        \end{subfigure}
        \hfil
        \begin{subfigure}[t]{0.3\textwidth}
         \centering
         \begin{tikzpicture}
            \filldraw[fill=detailcolor07, fill opacity=0.14, draw opacity=0.4, line width=0.01mm] (0,0) -- (1,0) -- (1,1) -- (0,1);
            \filldraw[fill=detailcolor07, fill opacity=0.14, draw opacity=0.4, line width=0.01mm] (0,1) -- (1,1) -- (1,2) -- (0,2);
            
            \filldraw[fill=detailcolor07, fill opacity=0.14, draw opacity=0.4, line width=0.01mm] (0,1) -- (-1,1) -- (-1,0) -- (0,0);
            \filldraw[fill=detailcolor07, fill opacity=0.14, draw opacity=0.4, line width=0.01mm] (0,2) -- (-1,2) -- (-1,1) -- (0,1); 
            
            \filldraw[fill=detailcolor07, fill opacity=0.14, draw opacity=0.4, line width=0.01mm] (0,0) -- (0.55,-0.3) -- (0.55,0.7) -- (0,1);
            \filldraw[fill=detailcolor07, fill opacity=0.14, draw opacity=0.4, line width=0.01mm] (0,1) -- (0.55,0.7) -- (0.55,1.7) -- (0,2); 
            
            \filldraw[fill=detailcolor07, fill opacity=0.14, draw opacity=0.4, line width=0.01mm] (0,1) -- (-0.55,1.3) -- (-0.55,0.3) -- (0,0);
            \filldraw[fill=detailcolor07, fill opacity=0.14, draw opacity=0.4, line width=0.01mm] (0,2) -- (-0.55,2.3) -- (-0.55,1.3) -- (0,1);
            \filldraw[fill=detailcolor07, fill opacity=0.14, draw opacity=0.4, line width=0.01mm] (0,0) -- (0.7,0.25) -- (0.7,1.25) -- (0,1);
            \filldraw[fill=detailcolor07, fill opacity=0.14, draw opacity=0.4, line width=0.01mm] (0,1) -- (0.7,1.25) -- (0.7,2.25) -- (0,2);
            
            \filldraw[fill=detailcolor07, fill opacity=0.14, draw opacity=0.4, line width=0.01mm] (0,1) -- (-0.7,0.75) -- (-0.7,-0.25) -- (0,0);
            \filldraw[fill=detailcolor07, fill opacity=0.14, draw opacity=0.4, line width=0.01mm] (0,2) -- (-0.7,1.75) -- (-0.7,0.75) -- (0,1);
            
            \draw[thick, detailcolor03] (0,0) -- (0,1);
            \draw[draw opacity=0.4, line width=0.01mm] (0,1) -- (0,2); 
            
            \draw[thick, Black] (0,1) -- (1,1);
            \draw[thick, Black] (0,1) -- (-1,1);
            \draw[thick, Black] (0,1) -- (0.55,0.7);
            \draw[thick, Black] (0,1) -- (-0.55,1.3);
            \draw[thick, Black] (0,1) -- (0.7,1.25);
            \draw[thick, Black] (0,1) -- (-0.7,0.75); 
        \end{tikzpicture}
        \caption{The edge \( e \) (red), the edges \( e_1,e_2, \dots, e_6 \) (black), and the set \( \bigcup_{j \in \{ 1,2,3,4,5,6\}} \bigl( (\hat \partial e_j) \, \setminus \bigcup_{k \in [6]\smallsetminus \{ j \}} \hat \partial e_k\bigr)\) (purple).}
        \label{figure: more minimal configurations c}
        \end{subfigure} 
        
        \caption{Given an edge \( e \) (see~\subref{figure: more minimal configurations a}) and a labelling \( \{ p_1,\dots, p_6\} \coloneqq \hat \partial e \) (see~\subref{figure: more minimal configurations b}),
        \subref{figure: more minimal configurations c} shows the set \( \bigcup_{j \in [6]} \bigl( (\hat \partial e_j) \, \setminus \bigcup_{k \in [6]\smallsetminus \{j \}} \hat \partial e_k \bigr)\) (purple) for one choice of \( e_1 \in \partial p_1 \smallsetminus \{ e \} \), \dots, \( e_6 \in \partial p_6 \smallsetminus \{ e \}\) (black). }
    \label{figure: more minimal configurations}
\end{figure}
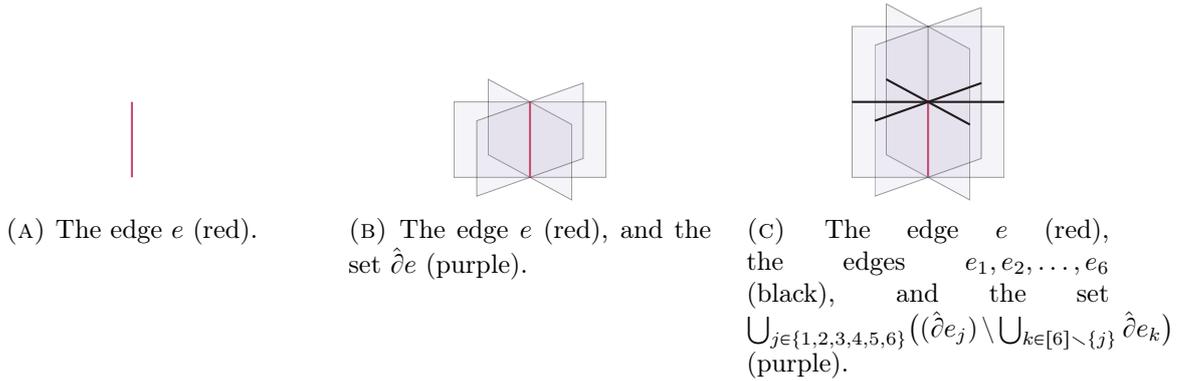

\subsection{Unitary gauge}\label{sec:unitary_gauge}

The purpose of this section is to present Proposition~\ref{proposition: unitary gauge one dim}, which shows that when considering the expected value of a Wilson loop, by applying a gauge transformation, the action defined in~\eqref{eq: general fixed length action} can be replaced by an alternative simpler action which does not explicitly involve the Higgs field.

We now introduce some additional notation.
For \( \eta \in G^{V_N}\), consider the bijection \( \tau \coloneqq \tau_\eta \coloneqq \tau_\eta^{(1)} \times \tau_\eta^{(2)} \colon \Sigma_{E_N} \times \Phi_{V_N} \to  \Sigma_{E_N} \times \Phi_{V_N}\), defined by
\begin{equation}\label{eq: gauge transform}
    \begin{cases}
     \sigma_{e} \mapsto -\eta_x +\sigma_e + \eta_y, & e=(x,y)\in E_N, \cr
     \phi_x \mapsto  \rho(\eta_x)\phi_x, & x \in V_N.
    \end{cases}
\end{equation}
Any mapping \( \tau \) of this form is called a \emph{gauge transformation}, and functions \( f: \Sigma_{E_N} \times \Phi_{V_N}  \to \mathbb{C} \) which are invariant under such mappings in the sense that $f= f \circ \tau$ are said to be \emph{gauge invariant}. One easily verifies that the Wilson loop observable \( W_\gamma \) is gauge invariant.
Next, for \( \beta, \kappa \geq   0 \) and  \( \sigma \in E_N \), let
\begin{equation}\label{eq: fixed length action}
    S_{\beta,\kappa}(\sigma) 
    \coloneqq
    -\beta \sum_{p \in P_N}      \rho((d  \sigma)_p) - \kappa \sum_{e \in E_N}   \rho( \sigma_e).
\end{equation} 
Define
\begin{equation}\label{eq: fixed length unitary measure}
    \mu_{N,\beta,\kappa}(\sigma) 
    \coloneqq
    Z_{N,\beta,\kappa}^{-1} e^{-S_{\beta,\kappa}(\sigma)},
\end{equation} 
where \( Z_{N,\beta,\kappa} \) is a normalizing constant, 
and let \( \mathbb{E}_{N,\beta,\kappa} \) denote the corresponding expectation.

\begin{proposition}\label{proposition: unitary gauge one dim}
Let \( \beta,\kappa \geq 0 \), let \( \gamma \) be a simple loop in \( E_N \), and let \( f \colon \Sigma_{E_N} \times \Phi_{V_N} \to \mathbb{C}\) be gauge invariant.  Then 
\begin{equation*} 
    \mathbb{E}_{N,\beta,\kappa,\infty}[f(\sigma,\phi)] = 
    \mathbb{E}_{N,\beta,\kappa}[f(\sigma,1)], 
\end{equation*}
where \(\mathbb{E}_{N,\beta,\kappa,\infty}\) denotes expectation with respect to the measure defined in \eqref{eq: general fixed length measure}. In particular, 
\begin{equation*}
    \mathbb{E}_{N,\beta,\kappa,\infty}[W_\gamma] = 
    \mathbb{E}_{N,\beta,\kappa}[W_\gamma], 
\end{equation*}
where \(W_\gamma\) is the Wilson loop observable defined in \eqref{Wilsonloopdef}.
\end{proposition}

Proposition~\ref{proposition: unitary gauge one dim} is considered well-known in the physics literature.
Using this proposition, we will work with \( \sigma \sim \mu_{N,\beta, \kappa} \) rather than \( (\sigma,\phi) \sim \mu_{N,\beta,\kappa,\infty}\) throughout the rest of this paper.

\begin{lemma}\label{lemma: gauge transform}
    Let \( \beta,\kappa \geq 0 \), \( \sigma \in \Sigma_{E_N} \), \( \phi \in \Phi_{V_N} \), and \( \eta \in G^{V_N} \). Then 
    \begin{equation*}
    S_{\beta,\kappa,\infty}(\sigma,\phi ) =  S_{\beta,\kappa,\infty}(\tau_\eta(\sigma,\phi)).
\end{equation*}
In other words, the action functional for the fixed length lattice Higgs model is gauge invariant.
\end{lemma}

\begin{proof}
    For each \( p \in P_N \), since each corner in \( p \) is an endpoint of exactly two edges in \( \partial p \), we have
    \begin{equation*}
        (d\sigma)_p = \sum_{e \in \partial p} \sigma_e
        = \sum_{e = (x,y)\in \partial p} (-\eta_x+\sigma_e + \eta_y).
    \end{equation*}
    Moreover, for any edge \( e = (x,y) \in E_N \), we have
    \begin{equation*}
        \overline{\phi_y} \rho(\sigma_e) \phi_x = 
        \overline{ \phi_y} \overline{\rho(\eta_y)}\rho(\eta_y) \rho(\sigma_e)\overline{\rho(\eta_x)}\rho(\eta_x) \phi_x = 
        \overline{\rho(\eta_y) \phi_y}   \rho(-\eta_x + \sigma_e+\eta_y) \rho(\eta_x) \phi_x.
    \end{equation*}
    From this the desired conclusion immediately follows.
\end{proof}

The main tool in the proof of Proposition~\ref{proposition: unitary gauge one dim} is the following lemma.

\begin{lemma}\label{lemma: gauge}
    Let \( \beta,\kappa \geq 0 \), and assume that  \( f \colon \Sigma_{E_N} \times \Phi_{V_N} \to \mathbb{C} \) is a gauge invariant function. Then, for any \( \phi \in \Phi_{V_N} \), we have
    \begin{equation}\label{eq: gauge}
        \sum_{\sigma \in \Sigma_{E_N}}  f(\sigma,\phi) e^{-S_{\beta,\kappa, \infty}(\sigma, \phi)} 
        =
        \sum_{\sigma \in \Sigma_{E_N}}  f(\sigma,1  ) e^{-S_{\beta,\kappa}(\sigma)} .
    \end{equation}
\end{lemma}

\begin{proof}
    Fix some \( \phi\in \Phi_{V_N} \).
    Let \( \eta \in G^{V_N} \) be such that  \( \rho(\eta_x) \phi_x = 1 \) for all \( x \in V_N.\) Then, by Lemma~\ref{lemma: gauge}, we have
    \begin{align*}
      &S_{\beta,\kappa, \infty}(\sigma, \phi) =
        S_{\beta,\kappa, \infty}\bigl(\tau_\eta(\sigma, \phi)\bigr)  ,\quad \sigma \in \Sigma_{E_N}.
    \end{align*}
    For our choice of \( \eta \), we have  \( \tau_\eta^{(2)}(\phi) = \rho(\eta_x) \phi_x = 1 \) for all \( x \in V_N \), and hence
    \begin{align*}
      &
        S_{\beta,\kappa, \infty}\bigl(\tau_\eta(\sigma, \phi)\bigr) = S_{\beta,\kappa}\bigl(\tau_\eta^{(1)}(\sigma)\bigr),\quad \sigma \in \Sigma_{E_N},
    \end{align*}
    and since \( \tau_\eta \) is a gauge transformation and \( f \) is gauge invariant, we also have
    \begin{equation*}
        f(\sigma,\phi) = f\bigl(\tau_\eta(\sigma,\phi)\bigr) =  
        f\bigl(\tau_\eta^{(1)}(\sigma),1 \bigr), \quad  \sigma \in \Sigma_{E_N}.
    \end{equation*}
    Combining the previous equations, it follows that for any \( \sigma \in \Sigma_{E_N} \), we have
    \begin{align*}
        &
        f(\sigma,\phi) e^{-S_{\beta,\kappa, \infty}(\sigma, \phi)} 
        =
        f(\tau_\eta^{(1)}(\sigma),1) e^{-S_{\beta,\kappa}(\tau_\eta^{(1)}(\sigma))} .
    \end{align*}
    Finally, since \( \tau_\eta^{(1)} \colon \Sigma_{E_N} \to \Sigma_{E_N}\) is a bijection, we have 
    \begin{align*}
        & 
        \sum_{\sigma \in \Sigma_{E_N}}   f\bigl(\tau_\eta^{(1)}(\sigma),1\bigr) e^{-S_{\beta,\kappa}(\tau_\eta^{(1)}(\sigma))}
        =
        \sum_{\sigma \in \Sigma_{E_N}}   f(\sigma,1) e^{-S_{\beta,\kappa}(\sigma)}.
    \end{align*}
    Combining the two previous equations, we obtain~\eqref{eq: gauge} as desired.
\end{proof}

\begin{proof}[Proof of Proposition~\ref{proposition: unitary gauge one dim}]
    Since the functions \( (\sigma,\phi) \mapsto f(\sigma,\phi) \) and  \( (\sigma,\phi) \mapsto 1 \) are both gauge invariant, it follows from Lemma~\ref{lemma: gauge} that
    \begin{align*}     
        &\mathbb{E}_{N,\beta,\kappa,\infty}[f(\sigma,\phi)] =
        \frac{
            \sum_{\phi \in \Phi_{V_N}} \sum_{\sigma \in \Sigma_{E_N}}
            f(\sigma,\phi)
            e^{-S_{\beta,\kappa,\infty}(\sigma,\phi)}
        }{
            \sum_{\phi \in \Phi_{V_N}} \sum_{\sigma \in \Sigma_{E_N}}
            e^{-S_{\beta,\kappa,\infty}(\sigma,\phi)}
        }
        = 
        \frac{
            \sum_{\phi \in \Phi_{V_N}} 
            \sum_{\sigma \in \Sigma_{E_N}}
            f(\sigma,1)
            e^{-S_{\beta,\kappa}(\sigma)}
        }{
            \sum_{\phi \in \Phi_{V_N}}\sum_{\sigma \in \Sigma_{E_N}} e^{-S_{\beta,\kappa}(\sigma)}
        }
        \\&\qquad = 
        \frac{
            \sum_{\sigma \in \Sigma_{E_N}} 
            f(\sigma,1)
            e^{-S_{\beta,\kappa}(\sigma)}
        }{
            \sum_{\sigma \in \Sigma_{E_N}}  e^{-S_{\beta,\kappa}(\sigma)}
        }
        =
        \mathbb{E}_{N,\beta,\kappa}[f(\sigma,1)],
    \end{align*}
    which is the desired conclusion. 
\end{proof}

The following result is not used in this paper, but since the proof is quite short given the work done in this section we choose to include it. 
\begin{proposition}\label{proposition: gauge relationships i}
    Let \( \beta,\kappa \geq 0 \). Then the following hold.
    \begin{enumerate}[label=\textnormal{(\roman*)}]
        \item If \( (\sigma,\phi) \sim \mu_{\beta,\kappa,\infty} \), then the marginal distribution of \( \phi \) is uniform on \( \rho(G)^{V_N} \). \label{item: gauge relationships i}
        
        \item If \( (\sigma,\phi) \sim \mu_{\beta,\kappa,\infty} \), and \( \eta_\phi \in G^{V_N} \) is such that \( \rho( \eta_\phi) = \phi \), then \( \tau_{-\eta_\phi}^{(1)} (\sigma)\sim \mu_{\beta,\kappa}.\)\label{item: gauge relationships ii}
        
        \item  If \( \sigma \sim \mu_{\beta,\kappa} \) and \( \eta \sim \unif G^{V_N}  \), then \( \tau_\eta(\sigma,1) \sim \mu_{\beta,\kappa,\infty}.\)\label{item: gauge relationships iii}
    \end{enumerate}
\end{proposition}

\begin{proof}
    Fix some \( \phi \in \Phi_{V_N} \). By Lemma~\ref{lemma: gauge}, applied with \( f \equiv 1 \)  (which is clearly gauge invariant), we have  
    \begin{equation*}%\label{eq: gauge}
        \sum_{\sigma \in \Sigma_{E_N}}  e^{-S_{\beta,\kappa, \infty}(\sigma, \phi)} 
        =
        \sum_{\sigma \in \Sigma_{E_N}}  e^{-S_{\beta,\kappa}(\sigma)}.
    \end{equation*}
    Since the right-hand side of the previous equation is independent of the choice of \( \phi \), we obtain~\ref{item: gauge relationships i}.

    To see that~\ref{item: gauge relationships ii} holds, fix some \(  \sigma' \in \Sigma_{E_N} \). For each \( \phi \in \Phi_{V_N} \) and \( \sigma' \in \Sigma_{E_N} \), by Lemma~\ref{lemma: gauge transform},  we have 
    \begin{equation*} 
        S_{\beta,\kappa}\bigl(\sigma' \bigr)
        = 
        S_{\beta,\kappa,\infty}\bigl(\sigma',1 \bigr)
        = 
        S_{\beta,\kappa,\infty}\bigl(\tau_{\eta_\phi}(\sigma',1) \bigr).
    \end{equation*} 
    Consequently, 
    \begin{align*}
        &\sum_{\substack{\sigma \in \Sigma_{E_N},\phi \in \Phi_{V_N} \mathrlap{\colon}\\ \tau_{\eta_\phi}(\sigma',1) = (\sigma,\phi)}} e^{-S_{\beta,\kappa,\infty}( \sigma,\phi)}
        =
        \sum_{\substack{\sigma \in \Sigma_{E_N},\phi \in \Phi_{V_N} \mathrlap{\colon}\\ \tau_{\eta_\phi}(\sigma',1) = (\sigma,\phi)}} e^{-S_{\beta,\kappa,\infty}( \tau_{\eta_\phi}(\sigma',1))}
        =
        e^{-S_{\beta,\kappa}( \sigma')}\sum_{\substack{\sigma \in \Sigma_{E_N},\phi \in \Phi_{V_N} \mathrlap{\colon}\\ \tau_{\eta_\phi}(\sigma',1) = (\sigma,\phi)}} 1.
    \end{align*}
    Since for any \( \phi \in \Phi_{V_N} \), the function \( \tau_{\eta_\phi}^{(1)} \colon \Sigma_{E_N} \to \Sigma_{E_N} \) is a bijection, and \( \tau_{\eta_\phi}^{(2)}(1) = \phi \), the last sum in the previous equation is equal to \( |\Phi_{V_N}| \). Noting that \( \tau_{\eta_\phi}^{-1} = \tau_{-\eta_\phi} \), it follows that~\ref{item: gauge relationships ii} holds.
    
    Finally, to see that~\ref{item: gauge relationships iii} holds, fix some \( \sigma \in \Sigma_{E_N} \) and \( \phi \in \Phi_{V_N} \).
    Then, for any \( \eta \in G^{V_N} \) and \( \sigma' \in \Sigma_{E_N} \), by Lemma~\ref{lemma: gauge transform}, we have
    \begin{equation*}
        S_{\beta,\kappa}(\sigma') = S_{\beta,\kappa,\infty} (\sigma',1) = 
        S_{\beta,\kappa,\infty} (\tau_\eta(\sigma',1))
    \end{equation*} 
    It follows that 
    \begin{equation*}
        \sum_{\substack{\sigma' \in \Sigma_{E_N}, \eta \in G^{V_N} \mathrlap{\colon} \\ \tau_\eta(\sigma',1) = (\sigma,\phi)}}e^{-S_{\beta,\kappa}(\sigma')}
        =
        \sum_{\substack{\sigma' \in \Sigma_{E_N}, \eta \in G^{V_N} \mathrlap{\colon} \\ \tau_\eta(\sigma',1) = (\sigma,\phi)}}e^{-S_{\beta,\kappa,\infty}(\tau_\eta(\sigma',1))}
        =
        e^{-S_{\beta,\kappa,\infty}(\sigma,\phi)} \sum_{\substack{\sigma' \in \Sigma_{E_N}, \eta \in G^{V_N} \mathrlap{\colon} \\ \tau_\eta(\sigma',1) = (\sigma,\phi)}} 1.
    \end{equation*}
    Since for any \( \eta \in G^{V_N} \), \( \tau_{\eta}^{(1)} \colon \Sigma_{E_N} \to \Sigma_{E_N} \) is a bijection,  \( \tau_{\eta}^{(2)}(1) = \rho(\eta) \), and \( \rho \) is faithful, the last sum in the previous equation is equal to \( 1 \). This shows that~\ref{item: gauge relationships iii} holds. 
\end{proof}

\begin{remark} 
    Using Proposition~\ref{proposition: gauge relationships i}, it follows that the simulations in Figure~\ref{subfig: config4}, Figure~\ref{subfig: config5}, and~Figure~\ref{subfig: config6} can be thought of as simulations of \( \sigma \sim \mu_{\beta,\kappa} \), with purple edges corresponding to edges \( e \in E_N \) with \( \sigma_e \neq 0 \).
\end{remark}

The main reason to work in unitary gauge is that it ensures, with high probability, that \( \sigma_e = 0 \) for sufficiently many edges \( e \) for the gauge field configuration \( \sigma \) to split into relatively small irreducible components (especially when \( \kappa \) is large and \( \sigma\sim \mu_{N,\beta,\kappa} \)). This observation is used on several occasions in the rest of the paper. We illustrate this idea in Figure~\ref{fig: simulation color}.
\begin{figure}[ht]
    \centering
    \begin{subfigure}[t]{0.48\textwidth}\centering
        \includegraphics[width=\textwidth]{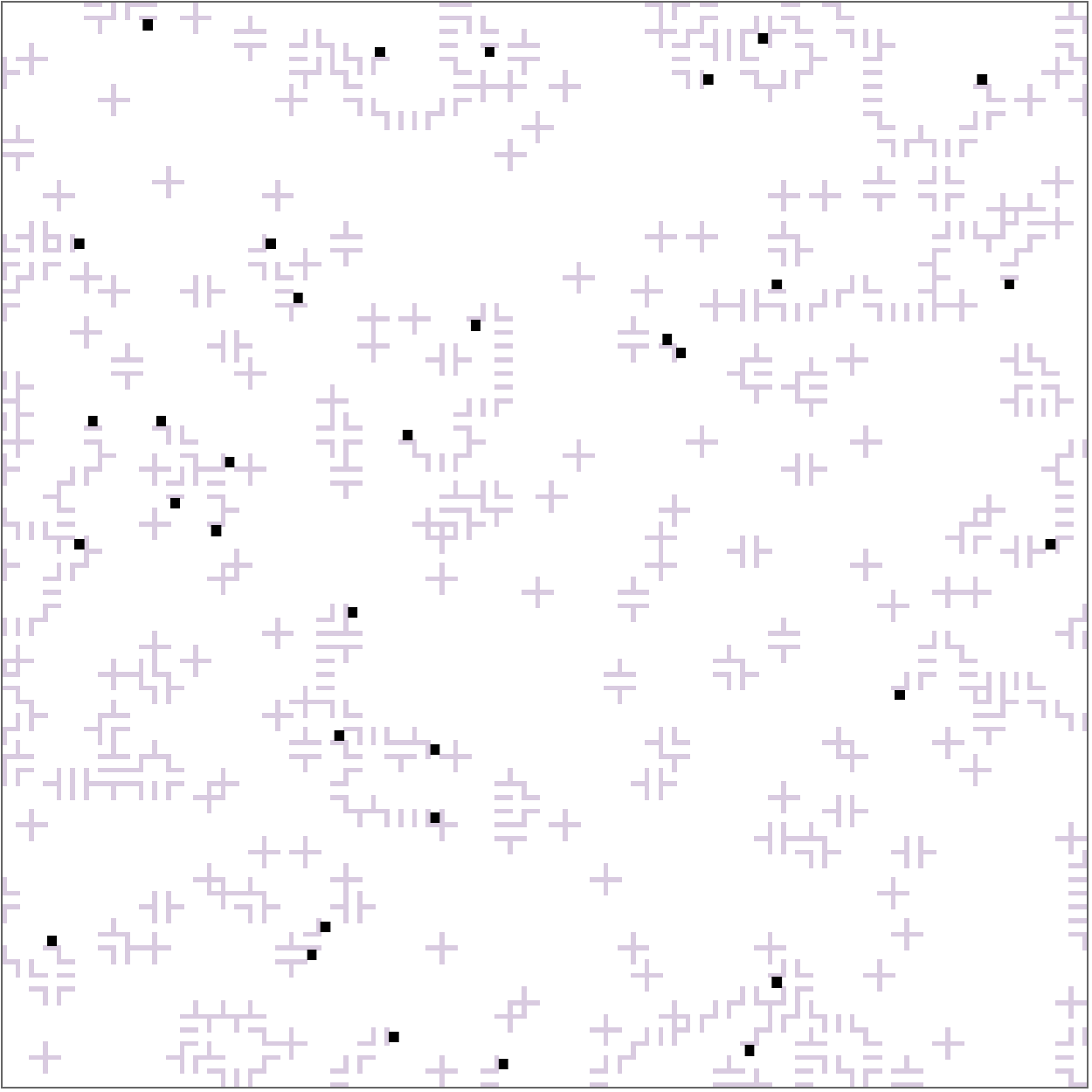}
        \caption{}\label{subfig: uncolored spins}
    \end{subfigure}
    \hfil
    \begin{subfigure}[t]{0.48\textwidth}\centering
        \includegraphics[width=\textwidth]{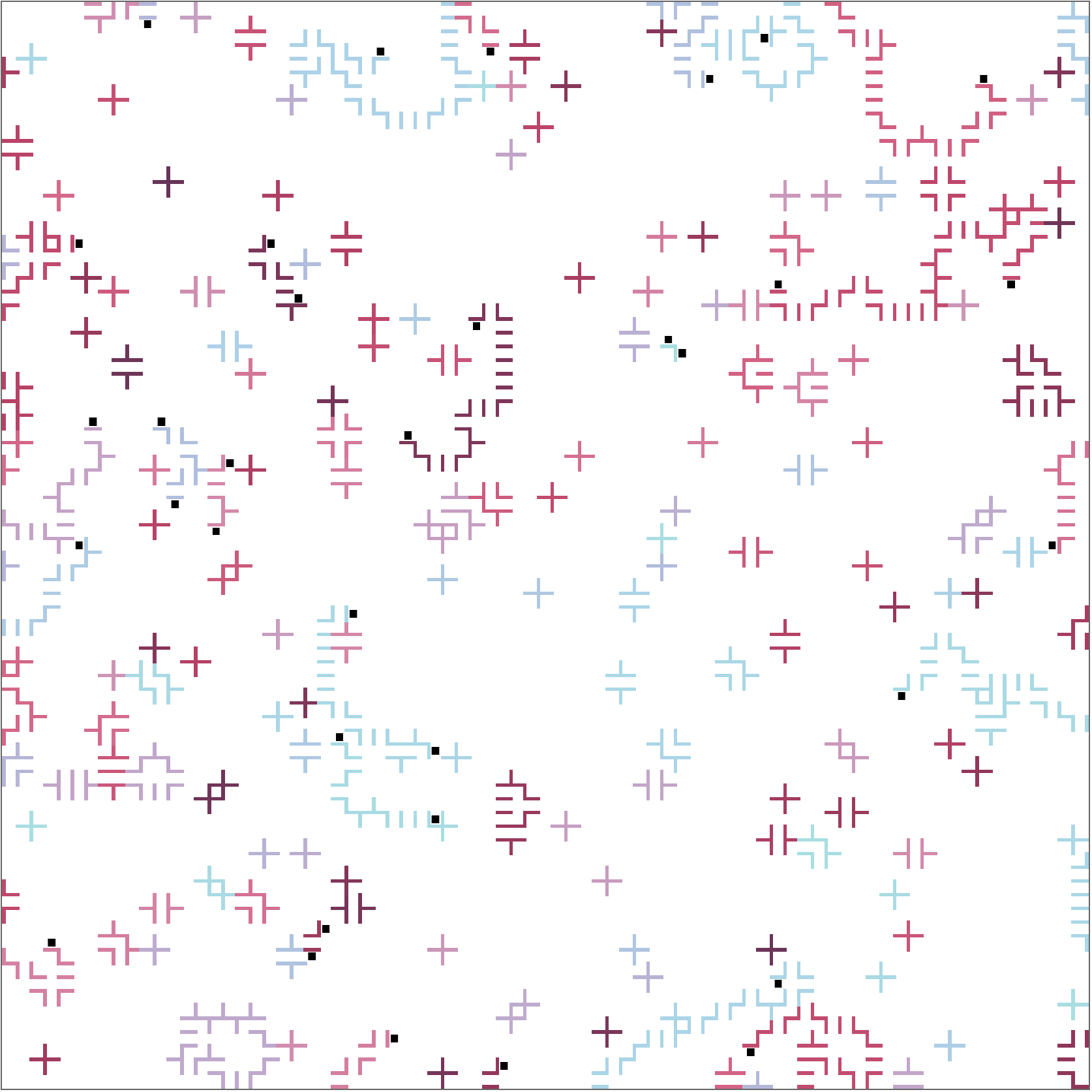}
        \caption{}\label{subfig: colored spins}
    \end{subfigure}
    \caption{This figure shows a gauge field configuration \( \sigma \) obtained by simulation of a fixed length lattice Higgs model on a subset of \( \mathbb{Z}^2\), using \( G = \mathbb{Z}_2 \), free boundary conditions, \( \beta=2.2 \) and \( \kappa = 0.48 \). In the figures, the colored edges correspond to edges \( e \) with  \( \sigma_e = 1 \) and the black squares to plaquettes \(p\) with \((d\sigma)_p = 1.\)
    In~\subref{subfig: colored spins}, we have colored the supports of distinct irreducible gauge field configurations \(  \sigma' \leq \sigma\) using different colors.} \label{fig: simulation color}
\end{figure}

\subsection{Existence of the infinite volume limit}\label{sect:ginibre}

In this section, we recall a result which shows existence and translation invariance of the infinite volume limit \( \langle W_\gamma \rangle_{\beta,\kappa} \) defined in \eqref{infinitevolumelimits} of the expectation of Wilson loop observables. This result is well-known, and is often mentioned in the literature as a direct consequence of the Ginibre inequalities~\cite{g1980}. A full proof of this result in the special case \( \kappa = 0 \) was included in~\cite{flv2020}, and the general case can be proven completely analogously, hence we omit the proof here. 

\begin{proposition}\label{proposition: limit exists}
    Let \( G = \mathbb{Z}_n \), \( M \geq 1 \), and let \( f \colon \Sigma_{E_M} \to \mathbb{R}\).
    For \( M' \geq M \), we abuse notation and let \( f \) denote the natural extension of \( f \) to \( E_{M'} \) (i.e., the unique function such that \( f(\sigma) = f(\sigma|_{E_M}) \) for all \( \sigma \in \Sigma_{E_{M'}} \)).
    Further, let \( \beta \in [0, \infty] \) and \( \kappa \geq 0 \). Then the following hold.
    \begin{enumerate}[label=\textnormal{(\roman*)}]
        \item The limit \( \lim_{N \to \infty} \mathbb{E}_{N,\beta,\kappa} \bigl[ f(\sigma) \bigr] \) exists.
        \item For any translation \( \tau \) of \( \mathbb{Z}^m \), we have \( \lim_{N \to \infty} \mathbb{E}_{N,\beta,\kappa}  \bigl[f \circ \tau(\sigma)\bigr] =  \lim_{N \to \infty} \mathbb{E}_{N,\beta,\kappa}\bigl[ f(\sigma) \bigr] \).
    \end{enumerate} 
\end{proposition}

\subsection{Notation and standing assumptions}\label{sec:standing assumptions}

In the remainder of the paper, we will assume that \( {G = \mathbb{Z}_n} \) for some integer \( n \geq 2 \). We will also assume that \( \rho\colon G \to \mathbb{C} \) is a unitary, faithful, and one-dimensional representation of \( G \). Any such representation has the form \( j \mapsto e^{i \cdot j \cdot p \cdot 2 \pi /n} \) for some integer \( p \in \{ 0,1, \ldots, n-1 \} \) which is relatively prime to \( n \).
Moreover, we will assume that the Higgs field has fixed length, i.e., we will work with the measure as defined in~\eqref{eq: general fixed length measure}.
Except in Section \ref{sec: main result}, we will assume that some integer \( N \geq 1 \) has been fixed.

\section{Vortices}\label{sec: vortices}

In this section, we use the notion of irreducibility introduced in Section~\ref{sec: irreducibility} to define what we refer to as vortices.  Vortices will play a central role in this paper. We mention that the definition of a vortex given in Definition~\ref{def: vortex} below is identical to the definition used in~\cite{flv2020}, but is different from the corresponding definitions in~\cite{sc2019}~and~\cite{c2019}.
    
\begin{definition}[Vortex]\label{def: vortex}
    Let \( \sigma \in \Sigma_{E_N} \). A non-trivial and irreducible plaquette configuration \( \nu \in \Sigma_{P_N} \) is said to be a \emph{vortex} in \( \sigma \) if \( \nu \leq d\sigma \), i.e., if \( (d\sigma)_p = \nu_p\) for all \( p \in \support \nu \). 
\end{definition}

We say that \( \sigma \in \Sigma_{E_N} \) has a vortex at \( V \subseteq P_N \) if \( (d\sigma)|_V \) is a vortex in \( \sigma \).  
With Lemma~\ref{lemma: minimal vortex I} in mind, we say that a vortex \( \nu \) is a \emph{minimal vortex} if there is \( p \in \support \nu \) with \( \hat\partial  \partial \nu \subseteq P_N, \) and  \( |\support \nu| = 12\). 

\begin{lemma}\label{lemma: minimal vortex II}
    Let \( \sigma \in \Sigma_{E_N} \), and let \( \nu \in \Sigma_{P_N}\) be a minimal vortex in \( \sigma \) such that \( \hat \partial \partial \support \nu \subseteq P_N \). Then there is an edge \( e_0 \in E_N \) and a \( g \in G\smallsetminus \{ 0 \} \) such that 
    \begin{equation}\label{eq: minimal vortex as df}
        \nu = d\bigl(g \mathbb{1}_{e = e_0} - g \mathbb{1}_{e = -e_0}\bigr).
    \end{equation}
    In particular, \( (d\sigma)_p =  \nu_p = g \) for all \( p \in \hat \partial e_0 \).
\end{lemma}

\begin{proof}
    Let \( e_0 \in E_N \) be such that  \( \support \nu = \hat \partial e_0 \cup \hat \partial (-e_0) \) (such an \( e_0 \) is guaranteed to exist by~Lemma~\ref{lemma: minimal vortex I}\ref{item: minimal vortex}).  
    Assume for contradiction that there is no \( g \in G\smallsetminus \{ 0 \} \) such that~\eqref{eq: minimal vortex as df} holds.
    Then, since \( {\support \nu = \hat \partial e_0 \cup \hat \partial (-e_0)} \), there exists a \( 3\)-cell~\(c\) and two plaquettes \( p_1,p_2 \in \support \nu  \cap \partial c \) with \( e_0 \in p_1 \) and \( -e_0 \in p_2 \) such that \( \nu_{p_1} \neq -\nu_{p_2}. \) Since any \( 3 \)-cell contains at most two plaquettes in \( \hat \partial e_0 \cup \hat \partial (-e_0) \), it follows that
    \begin{equation*}
        \sum_{p \in \partial c} \nu_p = \nu_{p_1} + \nu_{p_2} \neq 0.
    \end{equation*}
    Using that \(\nu = d\sigma'\) for some \(\sigma' \in \Sigma_{E_N}\) by Lemma~\ref{lemma: poincare}, this contradicts Lemma~\ref{lemma: Bianchi}, and hence~\eqref{eq: minimal vortex as df} must hold. As a direct consequence, it follows that \( (d\sigma)_p =  \nu_p = g \) for \( p \in \hat \partial e_0 \).
\end{proof}

If \( \sigma \in \Sigma_{E_N} \) and \( \nu \in \Sigma_{P_N} \) is a minimal vortex in \( \sigma \) which can be written as in~\eqref{eq: minimal vortex as df} for some \( e_0 \in E_N \) and \( g \in G \backslash \{ 0 \} \), then we say that \( \nu \) is a \emph{minimal vortex centered at \( e_0 \).}

\begin{lemma}\label{lemma: removing internal min vort}
    Let \( q \) be an oriented surface with boundary \( B_q \), and let \( e_0 \) be an internal edge of \( q \). Further, let \( \sigma \in \Sigma_{E_N} \), and let \( \nu \in \Sigma_{P_N}\) be a minimal vortex in \( \sigma \),  centered at \( e_0 \). Then
    \begin{equation*} 
        \sum_{p\in \support \nu } q_p^+  \nu_p   = 0.
    \end{equation*}
\end{lemma}

\begin{proof}
    Since \( \nu \) is a vortex centered at \( e_0 \),  there is a \( g \in G\smallsetminus \{ 0 \} \) such that if we for \( e \in E_N \) define  \( \sigma_e \coloneqq g \mathbb{1}_{e = e_0} - g \mathbb{1}_{e = -e_0}, \) then \( \nu = d\sigma \). By Lemma~\ref{lemma: stokes}, it follows that
    \begin{equation}\label{eq: rem int min vort 1}
        \sum_{p\in \support \nu } q_p^+  \nu_p = \sum_{p\in P_N } q_p^+  \nu_p = \sum_ {p\in P_N } q_p^+  (d\sigma)_p = \sum_{e \in B_q } \sigma_e .
    \end{equation}
    Since \( e_0 \) is an internal edge of \( q \), we have \( \{ e_0,-e_0 \} \cap  B_q = \emptyset\). Since \( \support \sigma = \{ e_0,-e_0 \} \), it follows that the right-hand side of~\eqref{eq: rem int min vort 1} vanishes.
\end{proof}

The next lemma is a version of Lemma~\ref{lemma: removing internal min vort}, valid for general vortices.
\begin{lemma}[Lemma~5.4 in~\cite{flv2020}]\label{lemma: 3.2}
    Let \( \sigma \in \Sigma_{E_N} \), and let \( \nu \in \Sigma_{P_N} \) be a vortex in \( \sigma \).
    Let \( q \) be an oriented surface. 
    If there is a box \( B \) containing \( \support \nu \) such that \( (*\! * \! B) \cap \support q \) consists of only internal plaquettes of \( q \), then
	\begin{equation}\label{eq: vortex surface equation}
	    \sum_{p \in P_N}  q^+_p \nu_p = 0.
    \end{equation} 
\end{lemma}

\begin{lemma}\label{lemma: vortex flip}
Let \( \sigma \in \Sigma_{E_N}\), and let \( \nu \) be a vortex in \( \sigma \). Then \( d\sigma - \nu  \in \Sigma_{P_N} \).
\end{lemma}

\begin{proof}
    Since \( \nu \) is a vortex in \( \sigma \), we have \( \nu \in \Sigma_{P_N} \), and hence \( d\nu = 0 \). Consequently, we have \( d(d\sigma-\nu) = dd\sigma - d\nu = 0-0 = 0 \), and hence \( d\sigma - \nu  \in \Sigma_{P_N} \).
\end{proof}

Before ending this section, we state and prove two lemmas which give connections between vortices in a given gauge field configuration \( \sigma   \) and in a gauge field configuration \( \sigma' \leq \sigma \).

\begin{lemma}\label{lemma: vortex transfer}
    Let \( \sigma' ,\sigma \in \Sigma_{E_N} \) be such that \( \sigma' \leq \sigma \), and let \( \nu \in \Sigma_{P_N}\) be a vortex in \( \sigma' \). Then  \( \nu \) is a vortex in \( \sigma \).
\end{lemma}
 
\begin{proof}
    Assume that \( p \in \support \nu \). Since \( \nu \) is a vortex in \( \sigma' \), we have \( (d\sigma')_p = \nu_p \neq 0 \), and hence \( p \in \support d\sigma' \).
    Since \( \sigma' \leq \sigma \) and \( p \in \support d\sigma', \) it follows from the definition of \( \leq \) that 
    \begin{equation*}
        (d\sigma)_p = ((d\sigma)|_{\support d\sigma'})_p = (d\sigma')_p = \nu_p
    \end{equation*}
    Since this holds for all \( p \in \support \nu \), \( \nu \) is a vortex is \( \sigma \).
\end{proof}

\begin{lemma}\label{lemma: reduction IV}
    Let \(  \sigma \in \Sigma_{E_N} \) and let 
    \( \nu \in \Sigma_{P_N} \) be a vortex in \( \sigma \). 
    Then there is \(  \sigma' \leq \sigma \) such that
    \begin{enumerate}[label=\textnormal{(\roman*)}]
        \item \( \nu \leq d \sigma' \), and 
        \item there is no \( \sigma''  <  \sigma' \) such that \( \nu \leq d \sigma'' \).
    \end{enumerate}
\end{lemma}

\begin{proof}
    Since \( \nu \) is a vortex in \( \sigma \), we have \( \nu \leq d\sigma \). Let \( \Omega \coloneqq \{ \omega \in \Sigma_{E_N} \colon \nu \leq d\omega \} \). Since \( \Sigma_{E_N} \) is a finite set and \( \Omega \subseteq \Sigma_{E_N}  \), \( \bigl| \Omega \bigr| \) is finite. The desired conclusion thus follows immediately from applying Lemma~\ref{lemma: reduction V} with \( k = 1 \).
\end{proof}

\section{Activity of gauge field configurations}\label{sec: activity}

In this section, we use the function \( \varphi_r \), introduced in~\eqref{eq: varphi}, to define the so-called \emph{activity} of a spin configuration. To this end, recall that for \( r \in [0,\infty]\) and \( g \in G \), we have defined
\begin{equation*}
    \varphi_r(g) = 
    \begin{cases} 
        e^{r \Re (\rho(g)-\rho(0))} &\text{if } r<\infty, \\ \mathbb{1}_{g=0} &\text{if } r=\infty.
    \end{cases}
\end{equation*} 
Since \( \rho \) is a unitary representation of \( G \), for any \( g \in G \) we have \( \rho(g) = \overline{\rho(-g)} \), and hence \( \Re \rho(g) = \Re \rho(-g) \). In particular, this implies that for any \( g \in G \) and any \( r \geq 0\), we have
\begin{equation} \label{eq: phi is symmetric}
    \varphi_r(g)
    =
    e^{ r (\Re \rho(g)-\rho(0))  }
    =
    e^{r \beta (\Re \rho(-g)-\rho(0)) }
    =
    \varphi_r(-g).
\end{equation}
Clearly, we also have \( \varphi_\infty(g) = \varphi_\infty(-g) \) for all \( g \in G \).
Moreover, if \( a \geq 0 \) and \( r \geq 0 \), then
\begin{equation*}
    \varphi_r(g)^a = \varphi_{ar}(g).
\end{equation*}
Abusing notation, for \( \sigma \in \Sigma_{E_N} \) and \( r \in [0,\infty] \), we let
\begin{equation*}
    \varphi_r(\sigma) \coloneqq \prod_{e \in E_N} \varphi_r(\sigma_e),
\end{equation*}
and for \( \omega \in \Sigma_{P_N} \) we let
\begin{equation*}
    \varphi_r(\omega) \coloneqq \prod_{p \in P_N}\varphi_r(\omega_p).
\end{equation*}
In analogy with~\cite{b1984}, for \( \beta \in [0,\infty ] \) and \( \kappa \geq 0 \), we define the \emph{activity} of \( \sigma \in \Sigma_{E_N} \) by
\begin{equation*}
    \varphi_{\beta,\kappa}(\sigma) \coloneqq 
    \varphi_\kappa(\sigma) \varphi_\beta(d\sigma)
    =
    \prod_{e \in E_N} \varphi_\kappa(\sigma_e) \prod_{p \in P_N}\varphi_\beta \bigl((d\sigma)_p \bigr).
\end{equation*}
Note that with this notation, for \( \sigma \in \Sigma_{E_N} \), \( \beta \in [0,\infty] \), and \( \kappa \geq 0 \), we have
\begin{equation}\label{eq: mubetakappaphi}
    \mu_{N,\beta,\kappa}(\sigma) = \frac{\varphi_{\beta,\kappa} (\sigma)}{\sum_{\sigma' \in \Sigma_{E_N}} \varphi_{\beta,\kappa}(\sigma')}.
\end{equation}
When \( \beta=\infty, \) we have \( \varphi_{\infty,\kappa}(\sigma) = \varphi_\kappa(\sigma) \) if \( \sigma \in \Sigma_{E_N}^0 \) and \( \varphi_{\infty,\kappa}(\sigma)  = 0 \) if \( \sigma \in \Sigma_{E_N}\smallsetminus \Sigma_{E_N}^0. \) Hence, if \( \sigma \in \Sigma_{E_N} ,\) then
\begin{equation}\label{eq: mubetakappaphiinfty}
    \mu_{N,\infty,\kappa}(\sigma) = \frac{\varphi_{\kappa} (\sigma) \cdot \mathbf{1}_{\sigma \in \Sigma_{E_N}^0}}{\sum_{\sigma' \in \Sigma_{E_N}^0} \varphi_{\kappa}(\sigma')}.
\end{equation}

The next lemma, which extends Lemma~2.6 in~\cite{f2021} from the case \( \kappa = 0 \) (see also Lemma~3.2.3 in~\cite{sc2019}), gives a useful factorization property of \( \varphi_{\beta,\kappa} \). 

\begin{lemma}\label{lemma: factorization property}
Let \( \sigma,\sigma' \in \Sigma_{E_N} \) be such that \( \sigma' \leq \sigma \), let \( \beta \in [0,\infty] \), and let \( \kappa \geq 0 \). Then 
\begin{equation}\label{eq: factorization property}
    \varphi_{\beta,\kappa}(\sigma) = \varphi_{\beta,\kappa}(\sigma')\varphi_{\beta,\kappa}(\sigma-\sigma').
\end{equation}
\end{lemma}

\begin{proof}
    To simplify notation, let \( E \coloneqq  \support \sigma' \) and \( P \coloneqq \support d\sigma' \). Note that 
    \begin{equation}\label{eq: eq with terms}
        \begin{split}
        & \varphi_{\beta,\kappa}(\sigma) = \prod_{p \in P_N} \varphi_{\beta}  \bigl( (d\sigma)_p\bigr) \prod_{e \in E_N}\varphi_\kappa(\sigma_e)
        \\&\qquad=
        \biggl[ \, \prod_{p \in P} \varphi_{\beta}  \bigl( (d\sigma)_p\bigr) \prod_{e \in E}\varphi_\kappa(\sigma_e) \biggr] \cdot \biggl[ \,  \prod_{p \in P_N\smallsetminus P} \varphi_{\beta}  \bigl( (d\sigma)_p\bigr) \prod_{e \in E_N\smallsetminus E}\varphi_\kappa(\sigma_e) \biggr].
        \end{split}
    \end{equation} 
    
    We now make the following observations.
    \begin{enumerate}[label=(\arabic*)]
        \item Since \( \sigma' \leq \sigma \), for \( e \in E \)  we have \( \sigma_e = \sigma'_e \). On the other hand, since \( E = \support \sigma' \), if \( e \in E_N \smallsetminus E \), then \( \sigma'_e = 0 \). Since \( \varphi_\kappa (0) = 1 \), it follows that
        \begin{equation}\label{eq: term 1}
            \prod_{e \in E} \varphi_\kappa(\sigma_e)    
            =
            \prod_{e \in E} \varphi_\kappa(\sigma'_e) 
            = 
            \prod_{e \in E_N} \varphi_\kappa(\sigma'_e).
        \end{equation}  
        \label{item: i}

        \item Since \( \sigma' \leq \sigma \), for \( e \in E \)  we have \( \sigma_e = \sigma'_e \), and hence
        \begin{equation*}
            (\sigma-\sigma')_e = \sigma_e - \sigma'_e = \sigma'_e - \sigma'_e = 0.
        \end{equation*}
        On the other hand, if \( e \in E_N \smallsetminus E \), then \( \sigma_e' = 0 \), and hence
        \begin{equation*}
            \sigma_e = \sigma_e - 0 = \sigma_e - \sigma'_e = (\sigma- \sigma')_e.
        \end{equation*}
        Since \( \varphi_\kappa(0)=1 \), it follows that
        \begin{equation}\label{eq: term 2}
            \prod_{e \in E_N\smallsetminus E}  \varphi_\kappa ( \sigma_e )
            =
            \prod_{e \in E_N\smallsetminus E}  \varphi_\kappa \bigl( (\sigma-\sigma')_e \bigr)
            = 
            \prod_{e \in E_N}  \varphi_\kappa \bigl( (\sigma-\sigma')_e \bigr).
        \end{equation}

        \item Since \( \sigma' \leq \sigma \) and \( P = \support d\sigma' \), for \( p \in P \) we have \( (d\sigma)_p = (d\sigma')_p \), and for \( p \in P_N \smallsetminus P \) we have \( (d\sigma')_p = 0 \). Consequently, since \( \varphi_\beta(0)=1 \), we have
        \begin{equation}\label{eq: term 3}
            \prod_{p \in P}  \varphi_\beta \bigl( (d\sigma)_p \bigr)  
            =
            \prod_{p \in P}  \varphi_\beta \bigl( (d\sigma')_p \bigr)  
            =
            \prod_{p \in P_N}  \varphi_\beta \bigl( (d\sigma')_p \bigr)  .
        \end{equation}
        \label{item: iii}

        \item Since \( P = \support d\sigma \), if \( p \in P_N \smallsetminus P \), then \( (d\sigma')_p=0 \), and hence
        \begin{equation*}
            (d\sigma)_p = (d\sigma)_p - 0 = (d\sigma)_p-(d\sigma')_p =  (d\sigma-d\sigma')_p
             =  \bigl(d(\sigma-\sigma')\bigr)_p.
        \end{equation*}
        On the other hand, if \( p \in P \), then \( (d\sigma')_p = (d\sigma)_p \), and hence
        \begin{equation*}
            (d(\sigma-\sigma'))_p = (d\sigma-d\sigma')_p = (d\sigma)_p-(d\sigma')_p =   (d\sigma')_p -(d\sigma')_p =0.
        \end{equation*}
        Combining these observations and using that \( \varphi_\beta(0)=1 \), we obtain
        \begin{equation}\label{eq: term 4}
            \prod_{p \in P_N\smallsetminus P}  \varphi \bigl( (d\sigma)_p  \bigr)
            =
            \prod_{p \in P_N\smallsetminus P}  \varphi \bigl( (d(\sigma-\sigma'))_p  \bigr)
            =
            \prod_{p \in P_N}  \varphi \bigl( (d(\sigma-\sigma'))_p  \bigr).
        \end{equation}
    \end{enumerate}
    Using~\eqref{eq: term 1},~\eqref{eq: term 2},~\eqref{eq: term 3}, and~\eqref{eq: term 4} to substitute the corresponding factors in~\eqref{eq: eq with terms}, we obtain
    \begin{equation*} 
        \begin{split}
        & \varphi_{\beta,\kappa}(\sigma) 
        =
        \biggl[ \, \prod_{p \in P_N} \varphi_{\beta}  \bigl( (d\sigma')_p\bigr) \prod_{e \in E_N}\varphi_\kappa(\sigma'_e) \biggr] \cdot \biggl[ \,  \prod_{p \in P_N } \varphi_{\beta}  \bigl( (d(\sigma-\sigma'))_p\bigr) \prod_{e \in E_N }\varphi_\kappa((\sigma-\sigma')_e) \biggr]
        \\&\qquad=
        \varphi_{\beta,\kappa}(\sigma') \varphi_{\beta,\kappa}(\sigma-\sigma') 
        \end{split}
    \end{equation*}
    as desired.  
\end{proof}

\section{Distribution of vortices and gauge field configurations}\label{sec: distribution of configurations}

In this section, we state and prove two propositions. The first proposition estimates the probability that a gauge field configuration \(\sigma\) locally agrees with another given gauge field configuration \(\sigma'\). The second proposition estimates the probability that \(d\sigma\) locally agrees with a given irreducible plaquette configuration \(\nu\).

\begin{proposition}\label{proposition: edgecluster flipping ii}
    Let  \( \sigma' \in \Sigma_{E_N} \), let \( \beta \in [0,\infty] \), and let \( \kappa \geq 0 \). Then
    \begin{equation*}
        \mu_{N,\beta,\kappa}\bigl(\{ \sigma \in \Sigma_{E_N} \colon
        \sigma' \leq \sigma \} \bigr) 
        \leq 
        \varphi_{\beta,\kappa}(\sigma').
    \end{equation*}
\end{proposition}

\begin{proof}
    Define 
    \begin{equation*}
        \Sigma^{\sigma'} \coloneqq \{ \sigma \in \Sigma_{E_N} \colon \sigma' \leq \sigma \} \quad \text{and} \quad
        \bar \Sigma^{\sigma'} \coloneqq \{ \sigma \in \Sigma_{E_N} \colon \sigma + \sigma' \in \Sigma^{\sigma'} \}.
    \end{equation*} 
    Using~\eqref{eq: mubetakappaphi}, and that \( \bar \Sigma^{\sigma'} \subseteq \Sigma_{E_N} \), we obtain
    \begin{equation}\label{eq: first ineq eq}
        \mu_{N,\beta,\kappa}\bigl(\{ \sigma \in \Sigma_{E_N} \colon \sigma' \leq \sigma \} \bigr) 
        = 
        \mu_{N,\beta,\kappa}(\Sigma^{\sigma'}) 
        =\frac{\sum_{\sigma \in \Sigma^{\sigma'} } \varphi_{\beta,\kappa}(\sigma)}{\sum_{\sigma \in \Sigma_{E_N}} \varphi_{\beta,\kappa}(\sigma)} 
        \leq\frac{\sum_{\sigma \in \Sigma^{\sigma'} } \varphi_{\beta,\kappa}(\sigma)}{\sum_{\sigma \in \bar \Sigma^{\sigma'}} \varphi_{\beta,\kappa}(\sigma)} .
    \end{equation}
    If \( \sigma \in \Sigma^{\sigma'} \), then \( \sigma' \leq \sigma \), and hence, by Lemma~\ref{lemma: factorization property}, \( \varphi_{\beta,\kappa}(\sigma) = \varphi_{\beta,\kappa}(\sigma')\varphi_{\beta,\kappa}(\sigma-\sigma') \). On the other hand, the mapping \( \sigma \mapsto \sigma - \sigma' \) is a bijection from \( \Sigma^{\sigma'} \) to \( \bar \Sigma^{\sigma'} \). Combining these observations, we obtain
    \begin{equation*}
        \sum_{\sigma \in \Sigma^{\sigma'} } \varphi_{\beta,\kappa}(\sigma) = 
        \sum_{\sigma \in \Sigma^{\sigma'} } \varphi_{\beta,\kappa}(\sigma')\varphi_{\beta,\kappa}(\sigma-\sigma') = 
        \varphi_{\beta,\kappa}(\sigma')
        \sum_{\sigma \in \Sigma^{\sigma'} } \varphi_{\beta,\kappa}(\sigma-\sigma')= 
        \varphi_{\beta,\kappa}(\sigma')
        \sum_{\sigma \in \bar \Sigma^{\sigma'} } \varphi_{\beta,\kappa}(\sigma).
    \end{equation*}
    This shows that the right-hand side of~\eqref{eq: first ineq eq} is equal to \(\varphi_{\beta,\kappa}(\sigma') \), and hence the desired conclusion follows.
\end{proof}

The next proposition gives a version of Proposition~\ref{proposition: edgecluster flipping ii} for plaquette configurations.
\begin{proposition}\label{proposition: first step}
    Let \( \nu \in \Sigma_{P_N} \)  be irreducible and non-trivial, and assume that \( \beta \geq 0 \) and  \( \kappa \geq 0 \) are such that~\ref{assumption: 1} holds. Then
    \begin{equation*} 
        \mu_{N,\beta,\kappa}\bigl(\{ \sigma \in \Sigma_{E_N} \colon \nu \leq d\sigma \} \bigr) 
        \leq    
        \frac{ \varphi_\beta(\nu) \, \bigl( 4 \alpha_0(\kappa)\bigr)^{m_0}  }{1-
        \alpha_1(\beta) -  4 \alpha_0(\kappa) }, 
\end{equation*}
    where 
    \begin{equation}\label{eq: m0def}
        m_0 \coloneqq \frac{1}{2} \min \bigl\{ |\support \sigma| \colon \sigma \in \Sigma_{E_N} \text{ and } \nu \leq d\sigma \bigr\}.
    \end{equation}
\end{proposition}

Before giving a proof of Proposition~\ref{proposition: first step}, we state and prove three lemmas which we will need.
The first of these lemmas will be used to enumerate all gauge field configurations \( \sigma \in \Sigma_{E_N}\) with \( \nu \leq d\sigma\) that are minimal in the sense of Lemma~\ref{lemma: reduction IV}. In this lemma, we associate two different sequences to a vortex \( \nu \) in an irreducible edge configuration \( \hat{\sigma} \): the first sequence is a sequence of edge configurations \( \sigma^{(0)}, \sigma^{(1)}, \ldots, \sigma^{(m)} \in \Sigma_{E_N} \) with increasing supports such that $\sigma^{(0)} = 0$ and $\sigma^{(m)} = \hat{\sigma}$, and the second sequence is an increasing sequence of subsets $L_0 \subseteq L_1 \subseteq  \dots \subseteq  L_m$ of $P_N$ such that $L_0 = \emptyset$ and $L_m = \support d\hat \sigma \smallsetminus \support \nu$. Together these two sequences can be thought of as a step-wise joint exploration process of \( \hat \sigma \) and \( \support \hat \sigma \smallsetminus \support \nu \).

\begin{lemma}\label{lemma: a constructive version of algorithm B}
    Let \( \nu \in \Sigma_{P_N}\), and let \( \hat \sigma \in \Sigma_{E_N} \) be such that
    \begin{enumerate}[label=\textnormal{(\alph*)}]
        \item \( \nu\) is a vortex in \( \hat  \sigma \), and
        \item there is no  \( \sigma' \in \Sigma_{E_N} \) with \( \sigma' < \hat  \sigma \) such that \( \nu \) is a vortex in \( \sigma' \). \label{item: constructive *}
    \end{enumerate}
    Assume that an arbitrary total ordering of the plaquettes in  \( P_N \) is given, and define 
    \begin{equation*}
        m \coloneqq m(\hat \sigma,\nu) \coloneqq \bigl(|\support d\hat \sigma \smallsetminus \support \nu| + |\support \hat \sigma| \bigr)/2 .
    \end{equation*} 
    Then there are sequences \( \sigma^{(0)},\sigma^{(1)}, \ldots, \sigma^{(m)} \in \Sigma_{E_N} \) and \( L_0 , L_1 , L_2 , \ldots,  L_m \subseteq P_N \), such that the following properties hold.
    \begin{enumerate}[label=\textnormal{(\roman*)}]
        \item \( \sigma^{(0)} \equiv 0 \) and \( \sigma^{(m)} = \hat \sigma \).\label{property va 1}
        \item \( L_0 = \emptyset \) and \( L_m = \support d\hat \sigma \smallsetminus \support \nu \).\label{property va 2}
    \end{enumerate}
    Moreover, for all \( j \in \{ 1,2, \ldots, m \} \), 
    \begin{enumerate}[resume,label=\textnormal{(\roman*)}] 
        \item  if  \( p \in L_{j}\), then \( \sigma^{(j)}_e = \hat \sigma_e \) for all \( e \in \partial p \),  \label{property va 3}
        \item   \(\hat \sigma|_{\support \sigma^{(j)}}  =  \sigma^{(j)} \), and \label{property va 4}
        \item   \( L_j \subseteq \support d\hat \sigma \smallsetminus \support \nu \).\label{property va 5}
    \end{enumerate}
    Further, for \( j = 1,2, \ldots, m \), let  
    \begin{equation}\label{Pjdef}
        P^j \coloneqq \bigl\{ p \in P_N \smallsetminus L_{j-1} \colon (d\sigma^{(j-1)})_p \neq \nu_p \bigr\} ,
    \end{equation} 
    and let \( p_j  \) be the plaquette in \( P^j \) which appears first in the given  ordering of the plaquettes. Then \( p_j \) is well defined, and we can choose the sequences \( \sigma^{(0)},\sigma^{(1)},\ldots, \sigma^{(m)} \) and \( L_0,L_1,\ldots, L_m \) such that either
    \begin{enumerate}[label=\textnormal{(vi')}]
        \item \( \support \sigma^{(j)} = \support \sigma^{(j-1)} \cup \{e_j,-e_j \} \) for some \( e_j \in \partial p_j \), and \( L_{j} =  L_{j-1}    \), or\label{property va 6'} 
    \end{enumerate}
    \begin{enumerate}[label=\textnormal{(vi'')}] 
        \item \( L_{j} =  L_{j-1} \cup \{ p_j,-p_j \} \) and \(  \support \sigma^{(j)}  = \support \sigma^{(j-1)}, \)\label{property va 6''}
    \end{enumerate}
    for \(j = 1,2, \dots, m\).
    Moreover, if the plaquettes in \( \support \nu \) occur first in the ordering of the plaquettes, then
    \begin{enumerate}[label=\textnormal{(vii)}]
        \item \( L_0 = L_1 = \ldots = L_{m_0} = \emptyset, \)\label{property va 7}
    \end{enumerate}
    where \( m_0 \) is given by~\eqref{eq: m0def}.
\end{lemma}

For \( j = 1,2, \ldots, m \), we think of the plaquettes in the set \( L_{j-1} \) as being ``locked'' or ``frozen'' at step \( j \).

In Figure~\ref{fig: first example} and Figure~\ref{fig: second example}, we give examples of sequences \( \sigma^{(0)},\sigma^{(1)}, \ldots \) and \( L_0,L_1,\ldots \) which correspond to some given  \( \hat  \sigma \in \Sigma_{E_N}\) and  \( \nu \in \Sigma_{P_N}\) as in Lemma~\ref{lemma: a constructive version of algorithm B}. 
\begin{proof}[Proof of Lemma~\ref{lemma: a constructive version of algorithm B}]
    We will show that the conclusion of Lemma~\ref{lemma: a constructive version of algorithm B} holds by constructing two sequences \( \sigma^{(0)}, \sigma^{(1)}, \ldots, \sigma^{(m)} \in \Sigma_{E_N}\) and \( L_0 ,L_1,\ldots, L_m \subseteq P_N\) with the desired properties. Define \( \sigma^{(0)} \equiv 0 \) and \( L_0 \coloneqq \emptyset \).  
    Fix \( k \in \{ 1,2, \ldots, m\} \), and assume that \( \sigma^{(1)}, \ldots, \sigma^{(k-1)} \in \Sigma_{E_N}\) and \(  L_1, \ldots, L_{k-1} \subseteq P_N\) are such that for each \( j \in \{ 0,1,2, \ldots, k-1 \} \),~\ref{property va 3},~\ref{property va 4},~\ref{property va 5} and either~\ref{property va 6'} or~\ref{property va 6''} holds.
    Before we construct \( \sigma^{(k)} \) and \( L_k \), we show that these assumptions together imply that the following  claim holds.
 
    \begin{sublemma}
        \( p_k \) is well defined.
    \end{sublemma}

    \begin{proof}[Proof of claim]\renewcommand\qedsymbol{\Large$\square$}
        Assume for contradiction that \( P^k = \emptyset \). By the definition of \( P^k \) (see~\eqref{Pjdef}), it follows that  \( (d\sigma^{(k-1)})_p = \nu_p \) for all \( p \in P_N \smallsetminus L_{k-1} \). Since \( L_{k-1} \cap \support \nu = \emptyset \) by~\ref{property va 5}, we have \( (d\sigma^{(k-1)})_p = \nu_p \) for all \( p \in \support \nu \). Hence \( \nu \) is a vortex in \( \sigma^{(k-1)} \). 
        
        At the same time, if \( p \in \support d\sigma^{(k-1)}, \) then either \( p\in P_N \smallsetminus L_{k-1} \) and \( (d\sigma^{(k-1)})_p = \nu_p = (d\hat \sigma)_p \) (since \( \nu \) is a vortex in \( \hat \sigma \)), or \( p \in L_{k-1} \), in which case, by (iii), we have \( \sigma_e^{(k-1)} =  \hat \sigma_e \) for all \( e \in \partial p \), and hence \( (d\sigma^{(k-1)})_p =  (d\hat \sigma)_p. \) Consequently, for all \( p \in \support d\sigma^{(k-1)} \) we have \( (d\sigma^{(k-1)})_p = (d\hat \sigma)_p. \) Since \( \sigma^{(k-1)} = \hat \sigma|_{\support \sigma^{(k-1)}} \) by construction, it follows that \( \sigma^{(k-1)} \leq \hat \sigma. \)
        
        Since \( \sigma^{(k-1)} \leq \hat \sigma \) and \( \nu \) is a vortex in \(\sigma^{(k-1)}, \) it now follows from~\ref{item: constructive *} that \( \sigma^{(k-1)} = \hat  \sigma \).
        Since \( P^k \) is empty by assumption, it follows from the definition of \( P^k \) that \( (d\hat \sigma)_p = \nu_p \) for each \( p \in P_N \smallsetminus L_{k-1} \), and consequently, we have \( L_{k-1} = \support d\hat \sigma \smallsetminus \support \nu \). 
        Since either~\ref{property va 6'} or~\ref{property va 6''} holds for each \( j \in \{ 1,2, \ldots, k-1 \} \), we infer that
        \begin{equation}
            k-1 = \bigl(|L_{k-1}| + |\support \sigma^{(k-1)}|\bigr)/2
            = \bigl(|\support d \hat \sigma \smallsetminus \support \nu| + |\support \hat  \sigma |\bigr)/2,
        \end{equation}
        and hence by the definition of \( m \), we have \( k-1 = m \). Since this contradicts that \( k \leq m \), the set \( P^k \) must be nonempty, implying in particular that \( p_k \) is well defined.
    \end{proof}

    We now define \( \sigma^{(k)} \) and \( L_k \).

    If \( \sigma^{(k-1)}_e = \hat \sigma_e \) for all edges \( e \in \partial p_k \smallsetminus \partial L_{k-1} \) where $\partial L_{k-1} = \cup_{p \in L_{k-1}} \partial p$,  we let \( \sigma^{(k)} \coloneqq \sigma^{(k-1)} \) and \( L_k \coloneqq L_{k-1} \cup \{ p_k,-p_k \} \). Then~\ref{property va 3} holds for \( j = k \)  by assumption, and~\ref{property va 6''} holds for \( j = k \) by construction. Since \( \sigma^{(k)} = \sigma^{(k-1)} \) and~\ref{property va 4} holds for \( j = k-1 \) by assumption,~\ref{property va 4} holds with \( j = k \). 
    Finally, we show that~\ref{property va 5} holds. By~\ref{property va 3} and by the definition of $p_k$, $(d\hat{\sigma})_{p_k} = (d\sigma^{(k-1)})_{p_k} \neq \nu_{p_k}$. Thus, since \( \nu \) is a vortex in \( \hat{\sigma} \), we have \( \nu_{p_k} = 0 \), and hence \( p_k \in \support d\hat \sigma \smallsetminus \support \nu \). Since \( L_k = L_{k-1} \cup \{ p_k,-p_k \} \), and since~\ref{property va 5} holds for \( j = k-1 \) by assumption, it follows that \( L_k \subseteq \support d\hat \sigma \smallsetminus \support \nu \), and thus~\ref{property va 5} holds for \( j = k \).

    On the other hand, if there is an edge \( e_k \in \partial p_k\smallsetminus \partial L_{k-1} \) such that \( \sigma^{(k-1)}_{e_k} \neq \hat \sigma_{e_k} \), then let \( L_k \coloneqq L_{k-1} \) and, for \( e \in E_N \), define
    \begin{equation*}
        \sigma^{(k)}(e) \coloneqq \begin{cases} \hat \sigma_e            &\text{if } e = \pm e_k, \cr        \sigma^{(k-1)}_e &\text{else.} \end{cases}
    \end{equation*}
    Since \( L_{k} = L_{k-1} \), and, since $\sigma^{(k)}$ and \( \sigma^{(k-1)} \) differ only at \( \pm e_k \notin \partial L_{k-1} \),~\ref{property va 3} holds for \( j = k \) (since it holds for \( j = k-1 \) by assumption).
    Since~\ref{property va 4} holds for \( j = k-1 \) by assumption,~\ref{property va 4} holds also when \( j = k \). By construction,~\ref{property va 6'} holds. Finally, since~\ref{property va 5} holds for \( j = k-1 \) by assumption, we have \( L_k = L_{k-1} \subseteq \support d\hat \sigma \smallsetminus \support \nu \), and hence~\ref{property va 5} holds also for \( j = k \).
    
    To sum up, we have constructed sequences \( \sigma^{(0)}, \ldots, \sigma^{(k)} \in \Sigma_{E_N}\) and \(  L_0, \ldots, L_{k} \subseteq P_N\) such that, for each \( j \in \{ 1,2, \ldots, k \} \),~\ref{property va 3},~\ref{property va 4},~\ref{property va 5},  and either~\ref{property va 6'} or~\ref{property va 6''} holds. This completes the inductive step and shows that sequences \( \sigma^{(0)},\sigma^{(1)} ,\ldots, \sigma^{(m)} \) and \( L_0,L_1,\ldots, L_m \) can be defined by repeating the above steps recursively.
    
    We now prove that if \( \sigma^{(0)},\sigma^{(1)} ,\ldots, \sigma^{(m)} \) and \( L_0,L_1,\ldots, L_m \) are defined as above, then \( {\sigma^{(m)} = \hat \sigma }\).
    
    \begin{sublemma}
        \( \sigma^{(m)} = \hat  \sigma \) and  \( L_m = \support d\hat \sigma \smallsetminus \support \nu. \)
    \end{sublemma}

    \begin{proof}[Proof of claim]\renewcommand\qedsymbol{\Large$\square$}
        Since either~\ref{property va 6'} or~\ref{property va 6''} holds for each \( j \in \{ 1,2, \ldots, m \} \), we have 
        \begin{equation*}
            m = \bigl(|L_m| + |\support \sigma^{(m)}|\bigr)/2.
        \end{equation*}
        By definition, we also have
        \begin{equation*}
            m = \bigl(|\support d\hat \sigma \smallsetminus \support \nu| + |\support \hat \sigma|\bigr)/2.
        \end{equation*}
        Since \( L_m \subseteq \support d\hat \sigma \smallsetminus \support \nu \) by~\ref{property va 5}, and \( \support \sigma^{(m)} \subseteq \support \hat \sigma \) by~\ref{property va 4}, it follows that \( L_m = \support d\hat \sigma \smallsetminus \support \nu \) and  \( \support \sigma^{(m)} = \support \hat \sigma \). Again using~\ref{property va 4}, it follows that \( \sigma^{(m)} = \hat \sigma \).
    \end{proof}

    It only remains to prove that~\ref{property va 7} holds.

    \begin{sublemma}
        If the plaquettes in \( \support \nu \) appear first in the ordering of the plaquettes, then \( L_0 = L_1 = \cdots = L_{m_0} = \emptyset \).
    \end{sublemma}

    \begin{proof}[Proof of claim]\renewcommand\qedsymbol{\Large$\square$}
        Assume that the plaquettes in \( \support \nu \) appear first in the ordering of the plaquettes. 
        Furthermore, assume that \( k \leq m \) is the largest positive integer such that \( p_1,p_2,\ldots, p_k \in \support \nu \).
        Since \( L_k \cap \support \nu = \emptyset \) by~\ref{property va 5}, it follows that \( L_0 = L_1 = \dots = L_k = \emptyset \). Consequently, it suffices to show that \( k \geq m_0. \)

        If \( k<m \), then \( p_{k+1} \notin \support \nu \). Since the plaquettes in \( \support \nu \) appear first in the ordering of the plaquettes and \(P^{k+1} = \{p \in P_N \colon (d\sigma^{(k)})_p \neq \nu_p\}\), we must have \( (d\sigma^{(k)})_p = \nu_p \) for all \( p \in \support \nu \) by the choice of \( p_{k+1}\). On the other hand, if \( k = m \), then by~\ref{property va 1}, we have \( \sigma^{(k)} = \sigma^{(m)} = \hat \sigma \), and hence \( (d\sigma^{(k)})_p = (d\hat \sigma)_p = \nu_p \) for all \( p \in \support \nu \) also in this case.
        Since \( d\nu = 0 = d(d\sigma^{(k)})  \), if follows that \( \nu \leq d\sigma^{(k)} \). Consequently, by definition (see~\eqref{eq: m0def}), we have \( |\support \sigma^{(k)}| \geq 2m_0 \). 
        Since, by construction, \( |\support \sigma^{(k)}|\leq |\support \sigma^{(k-1)}| + 2 \) and \( |\support \sigma^{(0)}| = 0\), it follows that \( k \geq m_0 \). This concludes the proof of the claim.  
    \end{proof}
\end{proof}

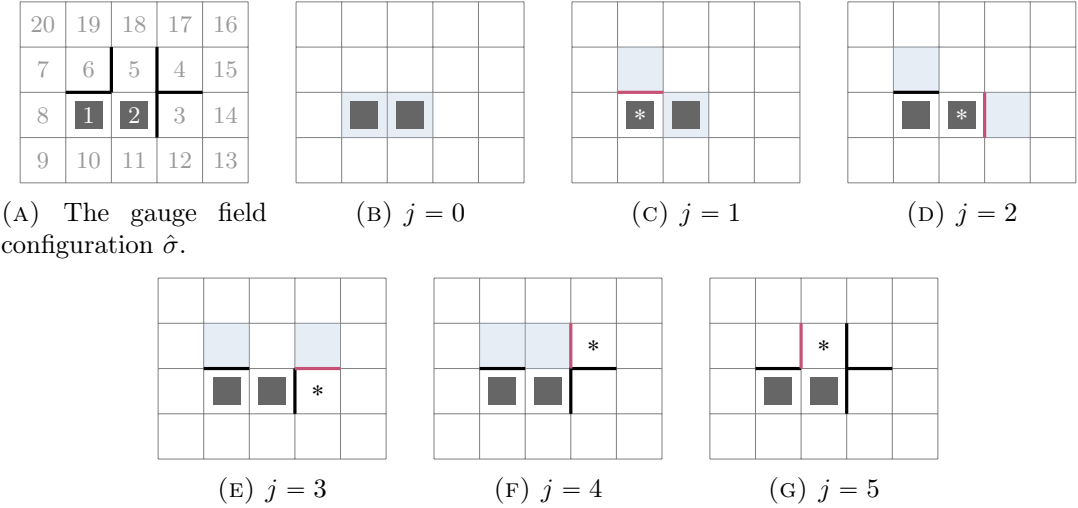
\begin{figure}[htp]
    \centering 
    \pgfmathsetmacro{\bd}{0.2}
     
     \begin{subfigure}[t]{0.22\textwidth}
     \centering\pgfmathsetmacro{\bd}{0.2}

    \begin{tikzpicture}[scale=0.6] 
    
        \draw[help lines] (3,2) grid (8,6);
        
        \draw[very thick, black] (4,4) -- (5,4);
        \draw[very thick, black] (6,3) -- (6,4) -- (7,4); 
        \draw[very thick, black] (6,4) -- (6,5);      
        \draw[very thick, black] (5,4) -- (5,5);      
          
         \fill[fill=black!60!white] (4+\bd,3+\bd) -- (4+\bd,4-\bd) -- (5-\bd,4-\bd) -- (5-\bd,3+\bd) -- (4+\bd,3+\bd);
         \fill[fill=black!60!white] (5+\bd,3+\bd) -- (5+\bd,4-\bd) -- (6-\bd,4-\bd) -- (6-\bd,3+\bd) -- (5+\bd,3+\bd); 
         
         \draw[white] (4.5,3.5) node {\footnotesize \( 1 \)};
         \draw[white] (5.5,3.5) node {\footnotesize \( 2 \)};
         \draw[black!40] (6.5,3.5) node {\footnotesize \( 3 \)};
         \draw[black!40] (6.5,4.5) node {\footnotesize \( 4 \)};
         \draw[black!40] (5.5,4.5) node {\footnotesize \( 5 \)};
         \draw[black!40] (4.5,4.5) node {\footnotesize \( 6 \)};
         \draw[black!40] (3.5,4.5) node {\footnotesize \( 7 \)};
         \draw[black!40] (3.5,3.5) node {\footnotesize \( 8 \)};
         \draw[black!40] (3.5,2.5) node {\footnotesize \( 9 \)};
         \draw[black!40] (4.5,2.5) node {\footnotesize \( 10 \)};
         \draw[black!40] (5.5,2.5) node {\footnotesize \( 11 \)};
         \draw[black!40] (6.5,2.5) node {\footnotesize \( 12 \)};
         \draw[black!40] (7.5,2.5) node {\footnotesize \( 13 \)};
         \draw[black!40] (7.5,3.5) node {\footnotesize \( 14 \)};
         \draw[black!40] (7.5,4.5) node {\footnotesize \( 15 \)};
         \draw[black!40] (7.5,5.5) node {\footnotesize \( 16 \)};
         \draw[black!40] (6.5,5.5) node {\footnotesize \( 17 \)};
         \draw[black!40] (5.5,5.5) node {\footnotesize \( 18 \)};
         \draw[black!40] (4.5,5.5) node {\footnotesize \( 19 \)};
         \draw[black!40] (3.5,5.5) node {\footnotesize \( 20 \)}; 
    \end{tikzpicture}
    \caption{The gauge field configuration \( \hat{\sigma} \).}
    \label{fig: first example a}
    \end{subfigure}
    \vspace{1ex} 
    \begin{subfigure}[t]{0.22\textwidth}
        \centering
        \begin{tikzpicture}[scale=0.6] 
            \draw[help lines] (3,2) grid (8,6);
 
            \fill[opacity = 0.3, detailcolor08] (4,3) -- (4,4) -- (6,4) -- (6,3) -- (4,3);
            \fill[fill=black!60!white] (4+\bd,3+\bd) -- (4+\bd,4-\bd) -- (5-\bd,4-\bd) -- (5-\bd,3+\bd) -- (4+\bd,3+\bd);
            \fill[fill=black!60!white] (5+\bd,3+\bd) -- (5+\bd,4-\bd) -- (6-\bd,4-\bd) -- (6-\bd,3+\bd) -- (5+\bd,3+\bd);   
        \end{tikzpicture}
    \caption{\( j = 0 \) }
    \label{fig: first example b}
    \end{subfigure}
     \begin{subfigure}[t]{0.22\textwidth}
        \centering
        \begin{tikzpicture}[scale=0.6] 
            \draw[help lines] (3,2) grid (8,6);
            
            \fill[opacity = 0.3, detailcolor08] (4,4) -- (4,5) -- (5,5) -- (5,4) -- (4,4);
            \fill[opacity = 0.3, detailcolor08] (5,3) -- (5,4) -- (6,4) -- (6,3) -- (5,3);
            \draw[very thick, detailcolor03] (4,4) -- (5,4); 
            
            \fill[fill=black!60!white] (4+\bd,3+\bd) -- (4+\bd,4-\bd) -- (5-\bd,4-\bd) -- (5-\bd,3+\bd) -- (4+\bd,3+\bd);
            \fill[fill=black!60!white] (5+\bd,3+\bd) -- (5+\bd,4-\bd) -- (6-\bd,4-\bd) -- (6-\bd,3+\bd) -- (5+\bd,3+\bd);
         
            \draw[white] (4.5,3.5) node {\small \( * \)}; 
        \end{tikzpicture}
    \caption{\( j = 1 \)}
    \label{fig: first example c}
    \end{subfigure}
    \begin{subfigure}[t]{0.22\textwidth}
    \centering
    \begin{tikzpicture}[scale=0.6] 
            \draw[help lines] (3,2) grid (8,6);
        
            \fill[opacity = 0.3, detailcolor08] (4,4) -- (4,5) -- (5,5) -- (5,4) -- (4,4);
            \fill[opacity = 0.3, detailcolor08] (6,3) -- (6,4) -- (7,4) -- (7,3) -- (6,3);

            \draw[very thick, black] (4,4) -- (5,4);
            \draw[very thick, detailcolor03] (6,3) -- (6,4);      
          
            \fill[fill=black!60!white] (4+\bd,3+\bd) -- (4+\bd,4-\bd) -- (5-\bd,4-\bd) -- (5-\bd,3+\bd) -- (4+\bd,3+\bd);
            \fill[fill=black!60!white] (5+\bd,3+\bd) -- (5+\bd,4-\bd) -- (6-\bd,4-\bd) -- (6-\bd,3+\bd) -- (5+\bd,3+\bd);
          
            \draw[white] (5.5,3.5) node {\small \( * \)};
        \end{tikzpicture}
    \caption{\( j = 2 \)}
    \label{fig: first example d}
    \end{subfigure}
      
    \begin{subfigure}[t]{0.22\textwidth}
    \centering
    \begin{tikzpicture}[scale=0.6] 
    
        \draw[help lines] (3,2) grid (8,6); 
        
        \fill[opacity = 0.3, detailcolor08] (4,4) -- (4,5) -- (5,5) -- (5,4) -- (4,4);
         \fill[opacity = 0.3, detailcolor08] (6,4) -- (6,5) -- (7,5) -- (7,4) -- (6,4);

        \draw[very thick, black] (4,4) -- (5,4);
        \draw[very thick, black] (6,3) -- (6,4) ;   
        \draw[very thick, detailcolor03] (6,4) -- (7,4);      
          
         \fill[fill=black!60!white] (4+\bd,3+\bd) -- (4+\bd,4-\bd) -- (5-\bd,4-\bd) -- (5-\bd,3+\bd) -- (4+\bd,3+\bd);
         \fill[fill=black!60!white] (5+\bd,3+\bd) -- (5+\bd,4-\bd) -- (6-\bd,4-\bd) -- (6-\bd,3+\bd) -- (5+\bd,3+\bd);
          
         \draw[black] (6.5,3.5) node {\small \( * \)};

    \end{tikzpicture}
    \caption{\( j = 3 \)}
    \label{fig: first example e}
    \end{subfigure}
    \begin{subfigure}[t]{0.22\textwidth}
    \centering 
        \begin{tikzpicture}[scale=0.6] 
    
        \draw[help lines] (3,2) grid (8,6);
        
        \fill[opacity = 0.3, detailcolor08] (4,4) -- (4,5) -- (5,5) -- (5,4) -- (4,4);
         \fill[opacity = 0.3, detailcolor08] (5,4) -- (5,5) -- (6,5) -- (6,4) -- (5,4);

        \draw[very thick, black] (4,4) -- (5,4);
        \draw[very thick, black] (6,3) -- (6,4) -- (7,4); 
        \draw[very thick, detailcolor03] (6,4) -- (6,5);      
          
         \fill[fill=black!60!white] (4+\bd,3+\bd) -- (4+\bd,4-\bd) -- (5-\bd,4-\bd) -- (5-\bd,3+\bd) -- (4+\bd,3+\bd);
         \fill[fill=black!60!white] (5+\bd,3+\bd) -- (5+\bd,4-\bd) -- (6-\bd,4-\bd) -- (6-\bd,3+\bd) -- (5+\bd,3+\bd);
         
         \draw[black] (6.5,4.5) node {\small \( * \)};

    \end{tikzpicture}
    \caption{\( j = 4 \)}
    \label{fig: first example f}
    \end{subfigure}
    \begin{subfigure}[t]{0.22\textwidth}
        \centering
        \begin{tikzpicture}[scale=0.6] 
    
        \draw[help lines] (3,2) grid (8,6);
        
        \draw[very thick, black] (4,4) -- (5,4);
        \draw[very thick, black] (6,3) -- (6,4) -- (7,4); 
        \draw[very thick, black] (6,4) -- (6,5);      
        \draw[very thick, detailcolor03] (5,4) -- (5,5);      
          
         \fill[fill=black!60!white] (4+\bd,3+\bd) -- (4+\bd,4-\bd) -- (5-\bd,4-\bd) -- (5-\bd,3+\bd) -- (4+\bd,3+\bd);
         \fill[fill=black!60!white] (5+\bd,3+\bd) -- (5+\bd,4-\bd) -- (6-\bd,4-\bd) -- (6-\bd,3+\bd) -- (5+\bd,3+\bd);
         
         \draw[black] (5.5,4.5) node {\small \( * \)};

        \end{tikzpicture}
        \caption{\( j = 5 \)}
        \label{fig: first example g}
    \end{subfigure}

    \caption{Assuming that \( G = \mathbb{Z}_2 \), consider the gauge field configuration \( \hat \sigma \) and plaquette ordering given in Figure~\subref{fig: first example a} above. Note that it contains a (minimal) vortex \( \nu \) with support on two plaquettes (dark gray). Figures~\subref{fig: first example b}, \subref{fig: first example c}, \subref{fig: first example d}, \subref{fig: first example e}, \subref{fig: first example f}, and \subref{fig: first example g}  above correspond to  \( j = 0,1,2,3,4,5 \) respectively, and in each picture we draw the vortex \( \nu \) (dark gray), the set \( P^{j+1} \) (blue), the plaquette \( p_j \) (\( * \)), the edge \( e_j \) (red), and the gauge field configuration \( \sigma^{(j)} \) (union of black and red edges) given in the proof of Lemma~\ref{lemma: a constructive version of algorithm B}, given the ordering of the plaquettes shown in~\subref{fig: first example a}.  Note that in this example, we have \( L_j = \emptyset \) for all \( j \). Note also that although in this case, there are no distinct \( i \) and \( j \) such that \( p_i = p_j \), this is in general not the case when \( G \neq \mathbb{Z}_2 \). Moreover, in  general, \( e_j \) is not always uniquely determined by the previous steps.}
    \label{fig: first example}
\end{figure}

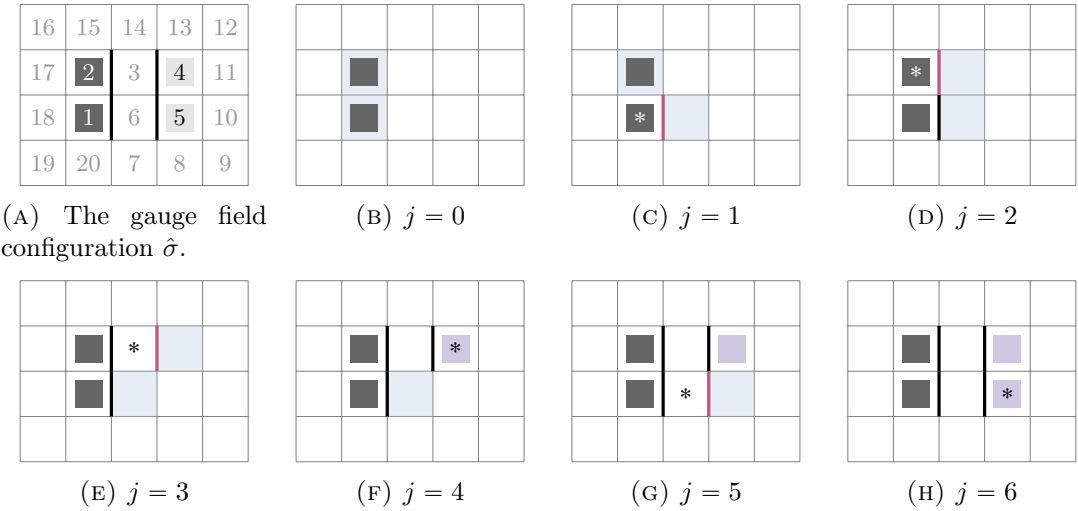
\begin{figure}[htp]
    \centering 
    \pgfmathsetmacro{\bd}{0.2}
     
     \begin{subfigure}[t]{0.22\textwidth}
     \centering
     \begin{tikzpicture}[scale=0.6] 
    
        \draw[help lines] (3,2) grid (8,6);
        
        \draw[very thick, black] (5,3) -- (5,5);  
        \draw[very thick, black] (6,3) -- (6,5); 
          
         \fill[fill=black!60!white] (4+\bd,3+\bd) -- (4+\bd,4-\bd) -- (5-\bd,4-\bd) -- (5-\bd,3+\bd) -- (4+\bd,3+\bd);
         \fill[fill=black!60!white] (4+\bd,4+\bd) -- (4+\bd,5-\bd) -- (5-\bd,5-\bd) -- (5-\bd,4+\bd) -- (4+\bd,4+\bd);
          
         \fill[fill=black!10!white] (6+\bd,3+\bd) -- (6+\bd,4-\bd) -- (7-\bd,4-\bd) -- (7-\bd,3+\bd) -- (6+\bd,3+\bd);
         \fill[fill=black!10!white] (6+\bd,4+\bd) -- (6+\bd,5-\bd) -- (7-\bd,5-\bd) -- (7-\bd,4+\bd) -- (6+\bd,4+\bd);
         
         \draw[white] (4.5,3.5) node {\footnotesize \( 1 \)}; 
         \draw[white] (4.5,4.5) node {\footnotesize \( 2 \)};
         \draw[black!40] (5.5,4.5) node {\footnotesize \( 3 \)};
         \draw[black!95] (6.5,4.5) node {\footnotesize \( 4 \)};
         \draw[black!95] (6.5,3.5) node {\footnotesize \( 5 \)}; 
         \draw[black!40] (5.5,3.5) node {\footnotesize \( 6 \)};
          \draw[black!40] (5.5,2.5) node {\footnotesize \( 7 \)};
         \draw[black!40] (6.5,2.5) node {\footnotesize \( 8 \)};
         \draw[black!40] (7.5,2.5) node {\footnotesize \( 9 \)};
         \draw[black!40] (7.5,3.5) node {\footnotesize \( 10 \)};
         \draw[black!40] (7.5,4.5) node {\footnotesize \( 11 \)};
         \draw[black!40] (7.5,5.5) node {\footnotesize \( 12 \)};
         \draw[black!40] (6.5,5.5) node {\footnotesize \( 13 \)};
         \draw[black!40] (5.5,5.5) node {\footnotesize \( 14 \)};
         \draw[black!40] (4.5,5.5) node {\footnotesize \( 15 \)};
         \draw[black!40] (3.5,5.5) node {\footnotesize \( 16 \)};
         \draw[black!40] (3.5,4.5) node {\footnotesize \( 17 \)};
         \draw[black!40] (3.5,3.5) node {\footnotesize \( 18 \)};
         \draw[black!40] (3.5,2.5) node {\footnotesize \( 19 \)};
         \draw[black!40] (4.5,2.5) node {\footnotesize \( 20 \)};
         
    \end{tikzpicture}
    \caption{The gauge field configuration \( \hat{\sigma} \).}
    \label{fig: second example a}
    \end{subfigure}
     \begin{subfigure}[t]{0.22\textwidth}
     \centering
     \begin{tikzpicture}[scale=0.6] 
    
        \draw[help lines] (3,2) grid (8,6);
 
        \fill[opacity = 0.3, detailcolor08] (4,3) -- (4,5) -- (5,5) -- (5,3) -- (4,3);  
        
             \fill[fill=black!60!white] (4+\bd,3+\bd) -- (4+\bd,4-\bd) -- (5-\bd,4-\bd) -- (5-\bd,3+\bd) -- (4+\bd,3+\bd);
         \fill[fill=black!60!white] (4+\bd,4+\bd) -- (4+\bd,5-\bd) -- (5-\bd,5-\bd) -- (5-\bd,4+\bd) -- (4+\bd,4+\bd);

    \end{tikzpicture}
    \caption{\( j = 0 \) }
    \label{fig: second example b}
    \end{subfigure}
     \begin{subfigure}[t]{0.22\textwidth}
     \centering
     \begin{tikzpicture}[scale=0.6] 
    
        \draw[help lines] (3,2) grid (8,6);
        
        \fill[opacity = 0.3, detailcolor08] (4,4) -- (4,5) -- (5,5) -- (5,4) -- (4,4); 
        \fill[opacity = 0.3, detailcolor08] (5,3) -- (5,4) -- (6,4) -- (6,3) -- (5,3);

        \draw[very thick, detailcolor03] (5,3) -- (5,4);   
        
        \fill[fill=black!60!white] (4+\bd,3+\bd) -- (4+\bd,4-\bd) -- (5-\bd,4-\bd) -- (5-\bd,3+\bd) -- (4+\bd,3+\bd);
         \fill[fill=black!60!white] (4+\bd,4+\bd) -- (4+\bd,5-\bd) -- (5-\bd,5-\bd) -- (5-\bd,4+\bd) -- (4+\bd,4+\bd);
         
         \draw[white] (4.5,3.5) node {\small \( * \)};

    \end{tikzpicture}
    \caption{\( j = 1 \)}
    \label{fig: second example c}
    \end{subfigure}
    \begin{subfigure}[t]{0.22\textwidth}
    \centering
    \begin{tikzpicture}[scale=0.6] 
    
        \draw[help lines] (3,2) grid (8,6);
        
        \fill[opacity = 0.3, detailcolor08] (5,4) -- (5,5) -- (6,5) -- (6,4) -- (5,4); 
        \fill[opacity = 0.3, detailcolor08] (5,3) -- (5,4) -- (6,4) -- (6,3) -- (5,3);  
        
        \draw[very thick, black] (5,3) -- (5,4); 
        \draw[very thick, detailcolor03] (5,4) -- (5,5);   
        
        \fill[fill=black!60!white] (4+\bd,3+\bd) -- (4+\bd,4-\bd) -- (5-\bd,4-\bd) -- (5-\bd,3+\bd) -- (4+\bd,3+\bd);
         \fill[fill=black!60!white] (4+\bd,4+\bd) -- (4+\bd,5-\bd) -- (5-\bd,5-\bd) -- (5-\bd,4+\bd) -- (4+\bd,4+\bd); 
         
         \draw[white] (4.5,4.5) node {\small \( * \)};

    \end{tikzpicture}
    \caption{\( j = 2 \)}
    \label{fig: second example d}
    \end{subfigure}
     
     \vspace{1ex} 
    \begin{subfigure}[t]{0.22\textwidth}
        \centering
        \begin{tikzpicture}[scale=0.6] 
    
            \draw[help lines] (3,2) grid (8,6); 
        
            \fill[opacity = 0.3, detailcolor08] (6,4) -- (6,5) -- (7,5) -- (7,4) -- (6,4); 
            \fill[opacity = 0.3, detailcolor08] (5,3) -- (5,4) -- (6,4) -- (6,3) -- (5,3);  
        
            \draw[very thick, black] (5,3) -- (5,4); 
            \draw[very thick, black] (5,4) -- (5,5);   
            \draw[very thick, detailcolor03] (6,4) -- (6,5);   
        
            \fill[fill=black!60!white] (4+\bd,3+\bd) -- (4+\bd,4-\bd) -- (5-\bd,4-\bd) -- (5-\bd,3+\bd) -- (4+\bd,3+\bd);
            \fill[fill=black!60!white] (4+\bd,4+\bd) -- (4+\bd,5-\bd) -- (5-\bd,5-\bd) -- (5-\bd,4+\bd) -- (4+\bd,4+\bd);
         
            \draw[black] (5.5,4.5) node {\small \( * \)};
        \end{tikzpicture}
        \caption{\( j = 3 \)}
        \label{fig: second example e}
    \end{subfigure}
    \begin{subfigure}[t]{0.22\textwidth}
    \centering 
    \begin{tikzpicture}[scale=0.6] 
    
        \draw[help lines] (3,2) grid (8,6); 
        
        \fill[opacity = 0.3, detailcolor08] (5,3) -- (5,4) -- (6,4) -- (6,3) -- (5,3);  
        
        \draw[very thick, black] (5,3) -- (5,4); 
        \draw[very thick, black] (5,4) -- (5,5);   
        \draw[very thick, black] (6,4) -- (6,5);   
        
        \fill[fill=black!60!white] (4+\bd,3+\bd) -- (4+\bd,4-\bd) -- (5-\bd,4-\bd) -- (5-\bd,3+\bd) -- (4+\bd,3+\bd);
         \fill[fill=black!60!white] (4+\bd,4+\bd) -- (4+\bd,5-\bd) -- (5-\bd,5-\bd) -- (5-\bd,4+\bd) -- (4+\bd,4+\bd);
           
         \fill[fill=detailcolor07,opacity=0.7] (6+\bd,4+\bd) -- (6+\bd,5-\bd) -- (7-\bd,5-\bd) -- (7-\bd,4+\bd) -- (6+\bd,4+\bd);
         
         \draw[black] (6.5,4.5) node {\small \( * \)};

    \end{tikzpicture}
    \caption{\( j = 4 \)}
    \label{fig: second example f}
    \end{subfigure}
    \begin{subfigure}[t]{0.22\textwidth}
    \centering
    \begin{tikzpicture}[scale=0.6] 
    
        \draw[help lines] (3,2) grid (8,6); 
        
        \fill[opacity = 0.3, detailcolor08] (6,3) -- (6,4) -- (7,4) -- (7,3) -- (6,3);  
        
        \draw[very thick, black] (5,3) -- (5,4); 
        \draw[very thick, black] (5,4) -- (5,5);   
        \draw[very thick, black] (6,4) -- (6,5);   
        \draw[very thick, detailcolor03] (6,3) -- (6,4);   
        
        \fill[fill=black!60!white] (4+\bd,3+\bd) -- (4+\bd,4-\bd) -- (5-\bd,4-\bd) -- (5-\bd,3+\bd) -- (4+\bd,3+\bd);
         \fill[fill=black!60!white] (4+\bd,4+\bd) -- (4+\bd,5-\bd) -- (5-\bd,5-\bd) -- (5-\bd,4+\bd) -- (4+\bd,4+\bd);
           
         \fill[fill=detailcolor07,opacity=0.7] (6+\bd,4+\bd) -- (6+\bd,5-\bd) -- (7-\bd,5-\bd) -- (7-\bd,4+\bd) -- (6+\bd,4+\bd);
         
         \draw[black] (5.5,3.5) node {\small \( * \)};

        \end{tikzpicture}
        \caption{\( j = 5 \)}
        \label{fig: second example g}
    \end{subfigure}
    \begin{subfigure}[t]{0.22\textwidth}
    \centering
    \begin{tikzpicture}[scale=0.6] 
    
        \draw[help lines] (3,2) grid (8,6);   
        
        \draw[very thick, black] (5,3) -- (5,4); 
        \draw[very thick, black] (5,4) -- (5,5);   
        \draw[very thick, black] (6,4) -- (6,5);   
        \draw[very thick, black] (6,3) -- (6,4);   
        
        \fill[fill=black!60!white] (4+\bd,3+\bd) -- (4+\bd,4-\bd) -- (5-\bd,4-\bd) -- (5-\bd,3+\bd) -- (4+\bd,3+\bd);
         \fill[fill=black!60!white] (4+\bd,4+\bd) -- (4+\bd,5-\bd) -- (5-\bd,5-\bd) -- (5-\bd,4+\bd) -- (4+\bd,4+\bd);
          
         \fill[fill=detailcolor07,opacity=0.7] (6+\bd,3+\bd) -- (6+\bd,4-\bd) -- (7-\bd,4-\bd) -- (7-\bd,3+\bd) -- (6+\bd,3+\bd);
         \fill[fill=detailcolor07,opacity=0.7] (6+\bd,4+\bd) -- (6+\bd,5-\bd) -- (7-\bd,5-\bd) -- (7-\bd,4+\bd) -- (6+\bd,4+\bd);
         
         \draw[black] (6.5,3.5) node {\small \( * \)};

    \end{tikzpicture}
    \caption{\( j = 6 \)}
    \label{fig: second example h}
    \end{subfigure}

    \caption{Assuming that \( G = \mathbb{Z}_2 \), consider the gauge field configuration \( \hat \sigma \) and plaquette ordering given by Figure~\subref{fig: second example a} above. Note that it contains two (minimal) vortices \( \nu_1 \) (dark gray) and \( \nu_2 \) (light gray), each with support on exactly  two plaquettes. Figures~\subref{fig: second example b}, \subref{fig: second example c}, \subref{fig: second example d}, \subref{fig: second example e}, \subref{fig: second example f}, \subref{fig: second example g}, and \subref{fig: second example h} correspond to  \( j =0,1,2,3,4,5,6 \) respectively. In each picture we draw the vortices \( \nu = \nu_1\) (dark gray), \( \nu_2 \) (light gray), the set \( P^{j+1} \) (blue), the plaquette \( p_j \) (\( * \)), the edge \( e_j \) (red), the gauge field configuration \( \sigma^{(j)} \) and the set \( L_j \) (purple) given in the proof of Lemma~\ref{lemma: a constructive version of algorithm B}, given the ordering of the plaquettes shown in Figure~\subref{fig: second example a}.}
    \label{fig: second example}
\end{figure}

\begin{lemma}\label{lemma: part 1 of upper bound}
    Let \( \beta \geq 0 \) and \( \kappa \geq 0 \). Further, let \( \nu \in \Sigma_{P_N} \)  be irreducible and non-trivial, and define \(m_0\) by~\eqref{eq: m0def}.
    For each \( m \geq 0 \), let   \( \Sigma_{E_N, \nu, m}  \) be the set of all \(  \sigma' \in \Sigma_{E_N} \) such that
    \begin{enumerate}[label=\textnormal{(\roman*)}]
        \item \( \nu \leq d \sigma' \),
        \item \( \nu \nleq d\sigma'' \) for any \( \sigma'' <  \sigma' \), and
        \item \( |\support d\sigma' \smallsetminus \support \nu| + |\support \sigma'| = 2m \).
    \end{enumerate} 
    Then
    \begin{equation}\label{eq: zeroeth}
        \begin{split}
            &\mu_{N,\beta,\kappa}\bigl(\{ \sigma \in \Sigma_{E_N} \colon \nu \leq d\sigma \} \bigr) 
            \leq    
            \sum_{m = m_0}^\infty
            \sum_{ \sigma' \in \Sigma_{E_N, \nu,m}}
            \varphi_{\beta,\kappa}(\sigma').
        \end{split}
    \end{equation}
\end{lemma}

\begin{proof} 
    Consider the event \( \{ \sigma \in \Sigma_{E_N} \colon \exists \sigma' \in \Sigma_{E_N,\nu,m}  \text{ such that } \sigma' \leq \sigma \} \).
    The set \( \Sigma_{E_N,\nu,m} \) is empty if \( m < m_0\). 
    By Lemma~\ref{lemma: vortex transfer} and Lemma~\ref{lemma: reduction IV}, we thus have
    \begin{equation*}
        \begin{split}
            &\mu_{N,\beta,\kappa}\bigl(\{ \sigma \in \Sigma_{E_N} \colon \nu \text{ is a vortex in } \sigma \} \bigr)  
            \\&\qquad =
            \mu_{N,\beta,\kappa}\biggl(\{ \sigma \in \Sigma_{E_N} \colon \exists \sigma' \in \bigcup_{m = m_0}^\infty \Sigma_{E_N,\nu,m} \text{ such that } \sigma' \leq \sigma \} \biggr) 
        \end{split}
     \end{equation*}
     and hence, by a union bound, 
     \begin{equation*}
         \begin{split}
             &\mu_{N,\beta,\kappa}\bigl(\{ \sigma \in \Sigma_{E_N} \colon \nu \text{ is a vortex in } \sigma \} \bigr)  
            \leq
            \sum_{m = m_0}^\infty \mu_{N,\beta,\kappa}\bigl(\{ \sigma \in \Sigma_{E_N} \colon \exists \sigma' \in  \Sigma_{E_N,\nu,m} \text{ such that }  \sigma' \leq \sigma \} \bigr).
        \end{split}
    \end{equation*}
    By a union bound, for any \( m \geq m_0 \), we have 
    \[
        \mu_{N,\beta,\kappa}\bigl(\{ \sigma \in \Sigma_{E_N} \colon \exists  \sigma' \in \Sigma_{E_N,\nu,m}  \text{ such that } \sigma' \leq \sigma \} \bigr) 
        \leq \sum_{\sigma' \in \Sigma_{E_N, \nu,m}} \mu_{N,\beta,\kappa} \bigl(\{ \sigma \in \Sigma_{E_N} \colon \sigma' \leq \sigma \} \bigr) .
    \] 
    Since, by Proposition~\ref{proposition: edgecluster flipping ii}, for any \(  \sigma' \in \Sigma_{E_N} \), we have
    \begin{equation*}
        \mu_{N,\beta,\kappa} \bigl(\{ \sigma \in \Sigma_{E_N} \colon \sigma' \leq \sigma \} \bigr) \leq \varphi_{\beta,\kappa}(\sigma')
    \end{equation*}
    we obtain~\eqref{eq: zeroeth} as desired. 
\end{proof}

\begin{lemma}\label{lemma: part 2 of upper bound}
   Let \( \beta \geq 0 \), and let \( \kappa \geq 0 \). Let \( \nu \in \Sigma_{P_N} \) be irreducible and non-trivial, and define \(m_0\) by~\eqref{eq: m0def}. Let \( \Sigma_{E_N, \nu, m}  \) be the sets defined in Lemma~\ref{lemma: part 1 of upper bound}.
    Then, for any \( m \geq m_0 \),
    \begin{equation*} 
        \begin{split}
            &\sum_{\hat \sigma \in \Sigma_{E_N, \nu,m}} \varphi_{\beta,\kappa} (\hat \sigma)
            \leq 
            \varphi_{\beta} (\nu) \, \bigl( 4\alpha_0(\kappa) \bigr)^{m_0}   
            \Bigl(  \alpha_1(\beta)+ 4\alpha_0(\kappa)\Bigr)^{m-m_0},
        \end{split}
    \end{equation*}
    where \(\alpha_0\) and \(\alpha_1\) are defined in~\eqref{eq: alpha01def}.
\end{lemma}

\begin{proof}
    Let \( V \coloneqq \support \nu \) and assume that \( \hat \sigma \in \Sigma_{E_N, \nu,m} \).
    By definition, 
    \begin{align*}
        &
        \varphi_{\beta,\kappa}(\hat \sigma)
        =
        \biggl[ \, \prod_{p' \in \support d\hat \sigma } \varphi_\beta\bigl( (d\hat \sigma)_p \bigr)  \biggr] \cdot \biggl[ \,  \prod_{e \in \support \hat \sigma} \varphi_\kappa (\hat \sigma_e)  \biggr].
    \end{align*}
    Using that  \( (d \hat \sigma) |_V = \nu\) and hence \( V \subseteq \support d \hat \sigma  \), we obtain
    \begin{equation}\label{phibetakappasigmahat}
        \varphi_{\beta,\kappa}(\hat \sigma)
        =
        \varphi_\beta(\nu) \biggl[ \, \prod_{p' \in (\support d\hat \sigma)\smallsetminus V } \varphi_\beta\bigl( (d\hat \sigma)_p \bigr)  \biggr] \cdot \biggl[ \,  \prod_{e \in \support \hat \sigma} \varphi_\kappa (\hat \sigma_e)  \biggr].
    \end{equation}
    Next note that the assumptions of Lemma~\ref{lemma: a constructive version of algorithm B} are satisfied.  
    Let the integer \( m \), the sets \( L_0,L_1, \ldots, L_{m} \) and the 1-forms \( \sigma^{(0)},\sigma^{(1)},\ldots, \sigma^{(m)} \) be as in Lemma~\ref{lemma: a constructive version of algorithm B} when applied with \( \hat \sigma \) and \( \nu \). Then the following hold for each \( j \in \{ 1,2, \ldots, m\} \).
    \begin{enumerate}
        \item By~\ref{property va 1} of Lemma~\ref{lemma: a constructive version of algorithm B}, we have \( \sigma^{(0)} = 0 \) and \( \sigma^{(m)} = \hat \sigma \), and by~\ref{property va 2} of Lemma~\ref{lemma: a constructive version of algorithm B}, we have \( L_0 = \emptyset \) and \( L_m = \support d\hat \sigma \smallsetminus \support \nu \).
        \item By~\ref{property va 6'} and~\ref{property va 6''} of Lemma~\ref{lemma: a constructive version of algorithm B},  either  \( L_j = L_{j-1} \) or \( \sigma^{(j)} = \sigma^{(j-1)} \) (but not both).
        \item By~\ref{property va 6'} and~\ref{property va 6''} of Lemma~\ref{lemma: a constructive version of algorithm B}, if \( L_j \neq L_{j-1} \), then there is a plaquette \( p_j \in \support d\hat \sigma \smallsetminus V \) such that \( L_j = L_{j-1} \sqcup \{ p_j,-p_j \} \).
        \item By~\ref{property va 6'} and~\ref{property va 6''} of Lemma~\ref{lemma: a constructive version of algorithm B}, if \( \sigma^{(j)} \neq \sigma^{(j-1)} \), then there is an edge \( e_j \in \support \hat \sigma \smallsetminus \support \sigma^{(j-1)} \) such that \( \sigma^{(j)}_{e'} = \sigma^{(j-1)}_{e'} + \hat \sigma_{e'} \mathbb{1}_{e' = \pm e_j}  \) for all \( e' \in E_N \).
    \end{enumerate} 
    Together, these observations imply that 
    \begin{equation}\label{phibetaphikappa}
        \begin{split} 
            &\biggl[ \, \prod_{p' \in (\support d\hat \sigma)\smallsetminus V } \varphi_\beta\bigl( (d\hat \sigma)_p \bigr)  \biggr] \cdot \biggl[ \,  \prod_{e \in \support \hat \sigma} \varphi_\kappa (\hat \sigma_e)  \biggr]
            \\ 
            &\qquad =    \prod_{j = 1}^m \Bigl[ \mathbb{1}\bigl(L_j \neq L_{j-1}\bigr) \varphi_\beta\bigl( (d\hat{\sigma})_{p_j}\bigr)^2+ \mathbb{1}\bigl({\sigma^{(j)} \neq \sigma^{(j-1)}}\bigr) \varphi_\kappa (\hat \sigma_{e_j})^2 \Bigr],
        \end{split}
    \end{equation}
    where \( \mathbb{1}(L_j \neq L_{j-1}) \) is equal to one if \( L_j \neq L_{j-1} \) and zero otherwise; analogously,  \( \mathbb{1}(\sigma^{(j)} \neq \sigma^{(j-1)})\) equals one if \( \sigma^{(j)}\neq \sigma^{(j-1)} \) and zero otherwise.
    If $L_j \neq L_{j-1}$, then \( p_j \in \support d\hat \sigma \), and thus \( \varphi_\beta\bigl( (d\hat{\sigma})_{p_j} \bigr)^2 \leq  \alpha_1(\beta) \) by the definition~\eqref{eq: alpha01def} of $\alpha_1$. 
    Consequently,~\eqref{phibetakappasigmahat} and~\eqref{phibetaphikappa} yield
    \begin{equation}\label{eq: the last product equation before the observations} 
        \varphi_{\beta,\kappa}( \hat \sigma ) 
        \leq 
        \varphi_\beta(\nu)  \prod_{j = 1}^m \Bigl[ \mathbb{1}\bigl(L_j \neq L_{j-1}\bigr)  \alpha_1(\beta) + \mathbb{1}\bigl({\sigma^{(j)} \neq \sigma^{(j-1)}}\bigr) \varphi_\kappa (\hat \sigma_{e_j})^2\Bigr].
    \end{equation}

    Given some \( j \in \{ 1,2, \ldots, m \} \), we now make a few additional observations regarding the above indicator functions.
    \begin{enumerate}
        \item If for each plaquette \( p \in P_N \), we let \( \partial p[1] \), \( \partial p[2] \), \( \partial p[3] \), and \( \partial p[4] \) denote the positively oriented edges in \( \partial p \), then by~\ref{property va 6'} of Lemma~\ref{lemma: a constructive version of algorithm B} we have
        \begin{equation} \label{eq: split indicators}
            \begin{split}
                &\mathbb{1}(\sigma^{(j)} \neq \sigma^{(j-1)})  
                \\&\qquad= \sum_{k \in \{ 1,2,3,4 \}}\sum_{g \in G\smallsetminus \{ 0 \}} \mathbb{1} \bigl( \support \sigma^{(j)} \smallsetminus \support \sigma^{(j-1)} = \{ \partial p_j[k], -\partial p_j[k] \} \text{ and }  \sigma^{(j)}_{\partial p_j[k]} = g  \bigr).
            \end{split}
        \end{equation} 
        \item If we for each \( j \in \{ 1,2, \ldots, m \} \) know which indicator (if any) in the right-hand side of~\eqref{eq: split indicators} is non-zero, then we know \( \sigma^{(m)} \). Since \( \hat \sigma = \sigma^{(m)} \), this fully determines \( \hat \sigma \).
        \item By~\ref{property va 7} of Lemma~\ref{lemma: a constructive version of algorithm B}, we have \(  L_0 = L_1 = \dots = L_{m_0} = \emptyset\).  
    \end{enumerate}
    By combining~\eqref{eq: the last product equation before the observations} with the above observations, we obtain
    \begin{equation*} 
        \begin{split}
            &\sum_{\hat \sigma \in \Sigma_{E_N, \nu,m}} \varphi_{\beta,\kappa} (\hat \sigma) \leq \varphi_{\beta} (\nu)  \Bigl[ 4\sum_{g \in G \smallsetminus \{ 0 \}} \varphi_\kappa(g)^2  \Bigr]^{m_0} 
            \prod_{j = m_0+1}^{m} \biggl[  \alpha_1(\beta)+ 4\sum_{g \in G \smallsetminus \{ 0 \}} \varphi_\kappa(g)^2\biggr]
            \\&\qquad = \varphi_{\beta} (\nu) \bigl( 4 \alpha_0(\kappa) \bigr)^{m_0}   
            \Bigl(  \alpha_1(\beta)+ 4\alpha_0(\kappa) \Bigr)^{m-m_0}
        \end{split}
    \end{equation*}
    as desired.
\end{proof}

\begin{proof}[Proof of Proposition~\ref{proposition: first step}]
    Let \(m_0\) be given by~\eqref{eq: m0def} and let   \( \Sigma_{E_N, \nu, m}  \) be the sets defined in Lemma~\ref{lemma: part 1 of upper bound}. 
    By Lemma~\ref{lemma: part 1 of upper bound}, we have
    \begin{equation*}
        \begin{split}
            &\mu_{N,\beta,\kappa}\bigl(\{ \sigma \in \Sigma_{E_N} \colon \nu \text{ is a vortex in } \sigma \} \bigr) 
            \leq    
            \sum_{m = m_0}^\infty
            \sum_{\hat \sigma \in \Sigma_{E_N, \nu,m}}
            \varphi_{\beta,\kappa}(\hat \sigma).
        \end{split}
    \end{equation*}
    On the other hand, by Lemma~\ref{lemma: part 2 of upper bound}, for any \( m \geq m_0 \), we have
    \begin{equation*} 
        \begin{split}
            &\sum_{\hat \sigma \in \Sigma_{E_N, \nu,m}} \varphi_{\beta,\kappa} (\hat \sigma)
            \leq 
            \varphi_{\beta} (\nu) \bigl( 4\alpha_0(\kappa) \bigr)^{m_0}
            \Bigl(  \alpha_1(\beta)+ 4\alpha_0(\kappa)\Bigr)^{m-m_0}.
        \end{split}
    \end{equation*}
    Combining these two inequalities, we get
    \begin{equation*}
        \begin{split}
            &\mu_{N,\beta,\kappa}\bigl(\{ \sigma \in \Sigma_{E_N} \colon \nu \text{ is a vortex in } \sigma \} \bigr) 
            \leq    
            \varphi_{\beta} (\nu) \bigl( 4\alpha_0(\kappa) \bigr)^{m_0} \sum_{m = m_0}^\infty
            \bigl(  \alpha_1(\beta)+ 4\alpha_0(\kappa)\bigr)^{m-m_0}.
        \end{split}
    \end{equation*}
    Since \(  \alpha_1(\beta)+ 4\alpha_0(\kappa)< 1 \) by assumption~\ref{assumption: 1}, we have
    \begin{equation*}
        \begin{split}
            &  \sum_{j = 0}^\infty 
            \bigl( \alpha_1(\beta)+ 4\alpha_0(\kappa)\bigr)^{j}
            =    \frac{1 }{1- \alpha_1(\beta)- 4\alpha_0(\kappa) },
        \end{split}
    \end{equation*}
    we hence obtain
    \begin{equation*}
        \begin{split}
            &\mu_{N,\beta,\kappa}\bigl(\{ \sigma \in \Sigma_{E_N} \colon \nu \text{ is a vortex in } \sigma \} \Bigr)  
            \leq 
            \frac{ \varphi_\beta(\nu)   \bigl( 4\alpha_0(\kappa) \bigr)^{m_0}  }{1-  \alpha_1(\beta) - 4 \alpha_0(\kappa) } ,
        \end{split}
    \end{equation*}
    which is the desired conclusion.
\end{proof}

\section{Distribution of frustrated plaquettes}
\label{section: distributions of frustrated plaquettes} 

In this section we establish an upper bound for the probability that a gauge field configuration contains a large vortex with support containing a given plaquette.

In order to simplify notation, given \( \sigma \in \Sigma_{E_N} \) and \( p \in P_N \), we let \( (d\sigma)^p  \) be the set of  vortices in \( \sigma \) with support containing \( p \).  

\begin{proposition}[Compare with Proposition~5.9 in~\cite{flv2020}]\label{proposition: the vortex flipping lemma}
Let \( \beta \geq0 \) and \( \kappa \geq 0 \)  be such that~\ref{assumption: 1}  and~\ref{assumption: 2} hold. Further, let \( p_0 \in P_N \) and \( M \geq 1 \). Then
    \begin{equation}\label{eq: the vortex flipping lemma}
        \begin{split}
            &\mu_{N,\beta,\kappa}\bigl(\{ \sigma \in \Sigma_{E_N} \colon \exists \nu \in (d\sigma)^{p_0} \text{ with } |\support \nu| \geq 2M \} \bigr)
            \leq 
            C_0^{(M)} \alpha_2(\beta,\kappa)^{\mathrlap{M}} ,
        \end{split}
    \end{equation} 
    where
    \begin{equation}\label{eq: C0M}
        C_0^{(M)} \coloneqq \frac{   5^M  4^{ M/6} }{\bigl(1-
            \alpha_1(\beta)- 4\alpha_0(\kappa)\bigr) \bigl(1-2^{1/3} 5 \alpha_2 (\beta,\kappa) \bigr) }.
    \end{equation}
\end{proposition}

Before we give a proof of Proposition~\ref{proposition: the vortex flipping lemma}, we state and prove the following lemma, which essentially corresponds to the first part of the proof of Proposition~2.11 in~\cite{flv2020}.

\begin{lemma}\label{lemma: counting vortex configurations iii}
    Let \(  p_0 \in P_N \),  \( m \geq 1 \), and \( \beta \geq0 \). Let \( \Lambda_m \) be the set of all irreducible plaquette configurations \( \nu \in \Sigma_{P_N}  \) with \( p_0 \in \support \nu \) and \( |\support \nu| = 2m \). Then 
    \begin{equation}\label{eq: counting vortex configurations ii}
        \sum_{\nu \in \Lambda_m} \varphi_\beta(\nu) \leq  5^{m-1} \alpha_0(\beta)^m .
    \end{equation}
\end{lemma}

\begin{proof}
    We will prove the lemma by giving a mapping from \( \Lambda_m \) to a set of  sequences \( \nu^1 ,\nu^2, \ldots, \nu^m \) of \( G \)-valued 2-forms on \( P_N \).
    To this end, assume that total orderings of the plaquettes and 3-cells in \( B_N \) are given.

    Fix some \( \nu \in \Lambda_m\).
    Let \( p_1 = p_1(\nu) \coloneqq p_0 \), and define 
    \begin{equation*}
        \nu^1(p) \coloneqq 
        \begin{cases}
            \nu_p &\text{if } p = p_1, \cr 
            -\nu_p &\text{if } p = -p_1, \cr 
            0 &\text{otherwise.}
        \end{cases}
    \end{equation*}
    Next,  assume that for some \( k\in \{ 1,2, \ldots, m\} \), we are given 2-forms \( \nu^1, \nu^2 , \ldots, \nu^k \)  such that for each \( j \in \{ 1,2, \ldots, k \} \) we have
    \begin{enumerate}[label=(\alph*)]
        \item \( \support \nu^j \smallsetminus \support \nu^{j-1} = \{ p_j,-p_j \} \) for some \( p_j \in P_N \), and\label{item: old i}
        \item \( \nu|_{\support \nu^j} = \nu^j \).\label{item: old ii}
    \end{enumerate}
    If \( d\nu^k = 0 \), then \( \nu^k \in \Sigma_{P_N} \) and, by~\ref{item: old ii}, we have \( \nu|_{\support \nu^k} = \nu^k \). Since \( \nu \) is irreducible by assumption, and \( \support \nu^k \neq \emptyset \) by~\ref{item: old i}, it follows that \( \nu^k = \nu \). Since \( |\support \nu^k| = 2k \) by~\ref{item: old i}, we conclude that  \( k = m \).
    Consequently, if \( k <m \), then \( d\nu^k \not \equiv 0 \), so there is at least one oriented 3-cell \( c \in B_N\) for which \( (d \nu^k)_c \neq 0 \). Let \( c_{k+1} \) be the first oriented 3-cell (with respect to the ordering of the 3-cells) for which \( (d\nu^k)_{c_{k+1}} \neq 0\). Since \( \nu \in \Sigma_{P_N} \), we have \( (d\nu)_{c_{k+1}} = 0 \), and consequently there exists at least one plaquette \( p\in  \support \nu \cap \partial c_{k+1} \smallsetminus  \support \nu^k. \)  Let \( p_{k+1} \coloneqq p_{k+1}(\nu) \) be the first such plaquette (with respect to the ordering of the plaquettes) and define
    \begin{equation*}
        \nu^{k+1}(p) \coloneqq 
        \begin{cases}
            \nu_p &\text{if } p = p_{k+1}, \cr 
            -\nu_p &\text{if } p = -p_{k+1}, \cr 
            \nu^k(p) &\text{otherwise.}
        \end{cases}
    \end{equation*}
    Note that if \( \nu^1,\nu^2, \ldots, \nu^k \)   satisfies~\ref{item: old i} and~\ref{item: old ii}, then so does \( \nu^{k+1} \).

    We now show that \( \nu^m = \nu \). By~\ref{item: old i}, \( |\support \nu^m | = 2m \) and by~\ref{item: old ii}, \( \nu|_{\support \nu^m} = \nu^m \). Since \( |\support \nu| = 2m \), it follows that \( \nu^m = \nu \).

    In order to simplify notation, let \( [m] \coloneqq \{ 1,2,\ldots, m\}\). To obtain~\eqref{eq: counting vortex configurations ii}, note first that the above mapping allows us to write
    \begin{align*}
        &\sum_{\nu \in \Lambda_m} \varphi_\beta(\nu) 
        =
        \sum_{\nu \in \Lambda_m}  
        \prod_{j=1}^m  \varphi_\beta\bigl(\nu_{p_j(\nu)}\bigr)^2 
        =
        \sum_{g_1,g_2,\ldots, g_m \in G\smallsetminus \{ 0 \}} \sum_{\substack{\nu \in \Lambda_m\colon \\ \nu_{p_j(\nu)} = g_j  \,\forall j\in [m]}}
        \, \prod_{j=1}^m  \varphi_\beta\bigl(g_j\bigr)^2 . 
    \end{align*}
    Given a sequence \( p_1,p_2,\ldots, p_m \in P_N \) and a sequence \( g_1,\dots, g_m \in G\smallsetminus \{ 0 \} \), there is at most one  \( \nu \in \Lambda_m \) with \( p_j = p_j(\nu) \) and \( \nu_{p_j} = g_j \) for all \( j \in [m] \).
    This implies in particular that the right-hand side of the previous equation is equal to
    \begin{align*}
        &
        \sum_{g_1,g_2,\ldots, g_m \in G\smallsetminus \{ 0 \} } \sum_{\substack{p_1,p_2,\ldots,p_m \in P_N \colon\\ \substack{\exists \nu \in \Lambda_m \colon \forall j \in [m] \colon\\ p_j = p_j(\nu) \text{ and } \nu_{p_j}=g_j }}}
        \, \prod_{j=1}^m  \varphi_\beta\bigl(g_j\bigr)^2 . 
    \end{align*}
    Given \( \nu^k \), the 3-cell \( c_{k+1} \) is uniquely defined and satisfies \( |\partial c_{k+1} \cap \support \nu^k| \geq1. \) Since we always have \( p_{k+1} \in \partial c_{k+1} \smallsetminus \support \nu^k, \) given \( \nu^k, \) there can be at most five possible choices for \( p_{k+1} \). Since \( p_1 = p_0 \) by construction, this implies in particular that
    \begin{align*}
        &
        \sum_{g_1,g_2,\ldots, g_m \in G\smallsetminus \{ 0 \} } \sum_{\substack{p_1,p_2,\ldots,p_m \in P_N \colon\\ \substack{\exists \nu \in \Lambda_m \colon \forall j \in [m] \colon\\ p_j = p_j(\nu) \text{ and } \nu_{p_j}=g_j }}}
        \, \prod_{j=1}^m  \varphi_\beta\bigl(g_j\bigr)^2
        \leq 
        \sum_{g_1,g_2,\ldots, g_m \in G\smallsetminus \{ 0 \} } 5^{m-1}
        \, \prod_{j=1}^m  \varphi_\beta\bigl(g_j\bigr)^2.
    \end{align*}
    Finally, note that
    \begin{align*}
        &
        \sum_{g_1,g_2,\ldots, g_m \in G\smallsetminus \{ 0 \} } 5^{m-1}
        \, \prod_{j=1}^m  \varphi_\beta\bigl(g_j\bigr)^2
        =
         5^{m-1}
        \, \prod_{j=1}^m  \Bigl( \sum_{{g_j} \in G\smallsetminus \{ 0 \} } \varphi_\beta\bigl(g_j\bigr)^2 \Bigr) 
        =
         5^{m-1} \alpha_0(\beta)^m.
    \end{align*} 
    Combining the above equations, we obtain~\eqref{eq: counting vortex configurations ii} as desired.%
\end{proof}

\begin{remark}
    The statements of Lemma~\ref{lemma: a constructive version of algorithm B} and Lemma~\ref{lemma: counting vortex configurations iii} hold for lattices \( \mathbb{Z}^d \) whenever \( d \geq 3 \).
\end{remark}

We are now ready to give a proof of Proposition~\ref{proposition: the vortex flipping lemma}.

\begin{proof}[Proof of Proposition~\ref{proposition: the vortex flipping lemma}] 
    For each \( m \geq 1 \), let \( \Lambda_m \) be the set of all irreducible plaquette configurations \( \nu \in \Sigma_{P_N}  \) such that \( p_0 \in \support \nu \) and \( |\support \nu| = 2m \). By Lemma~\ref{lemma: counting vortex configurations iii}, we have 
\begin{equation} \label{eq: bound 1}
    \sum_{\nu \in \Lambda_m} \varphi_\beta(\nu) \leq  5^{m-1} \alpha_0(\beta)^m .
\end{equation}
Next, by Lemma~\ref{lemma: 6 plaquettes per edge}, if \( \nu \in \Lambda_m \) and  \( \sigma \in \Sigma_{E_N} \) is such that \( \nu \leq d\sigma   \), then \( |\support \sigma| \geq 2 m/6  \). 
In particular, this shows that the constant $m_0$, defined in~\eqref{eq: m0def}, satisfies \(m_0 \geq  m/6 \). Consequently, since \( \alpha_1(\beta) + 4\alpha_0(\beta) < 1 \), we can apply Proposition~\ref{proposition: first step}  to obtain 
    \begin{equation}\label{eq: bound 2}
        \begin{split}
            &\mu_{N,\beta,\kappa}\bigl(\{ \sigma \in \Sigma_{E_N} \colon \nu \leq d\sigma \} \Bigr)  
            \leq     
            \frac{ \varphi_\beta(\nu) \bigl( 4 \alpha_0(\kappa) \bigr)^{ m/6} }{1-
            \alpha_1(\beta)- 4\alpha_0(\kappa) }
        \end{split}
    \end{equation}
for each  \( \nu \in \Lambda_m\). Combining these observations, we obtain
    \begin{align*}
        &\mu_{N,\beta,\kappa}\bigl(\{ \sigma \in \Sigma_{E_N} \colon \exists \nu \in (d\sigma)^{p_0} \text{ with }  |\support \nu| \geq 2M \} \bigr) 
        \leq 
        \sum_{m=M}^\infty \sum_{\nu \in \Lambda_m} \mu_{N,\beta,\kappa}\bigl(\{ \sigma \in \Sigma_{E_N} \colon  \nu \leq d\sigma\bigr) 
        \\&\qquad\overset{\eqref{eq: bound 2}}{\leq}
        \sum_{m=M}^\infty \sum_{\nu \in \Lambda_m} \frac{ \varphi_\beta(\nu) \bigl( 4 \alpha_0(\kappa) \bigr)^{ m/6 } }{1-
            \alpha_1(\beta)- 4\alpha_0(\kappa) }
            \overset{\eqref{eq: bound 1}}{\leq}
        \sum_{m=M}^\infty \frac{ 5^m  \alpha_0(\beta)^m \bigl( 4 \alpha_0(\kappa) \bigr)^{ m/6 } }{1-
            \alpha_1(\beta)- 4\alpha_0(\kappa) }.
     \end{align*}
    Computing the sum of the geometric series (which is finite by the assumptions on \( \beta \) and \( \kappa \)), we obtain
    \begin{equation*}
        \begin{split}
            &\mu_{N,\beta,\kappa}\bigl(\{ \sigma \in \Sigma_{E_N} \colon \exists \nu \in (d\sigma)^{p_0} \text{ with }  |\support \nu| \geq 2M \} \bigr) 
            \\&\qquad\leq  
            \frac{ 5^M  \alpha_0(\beta)^M \bigl( 4 \alpha_0(\kappa) \bigr)^{ M/6} }{\bigl(1-
            \alpha_1(\beta)- 4\alpha_0(\kappa)\bigr) \bigl(1-5  \alpha_0(\beta) ( 4 \alpha_0(\kappa) )^{ 1/6} \bigr) }.
        \end{split}
    \end{equation*}
    Recalling the definition~\eqref{eq: alpha234def} of \( \alpha_2 (\beta,\kappa) \), the desired conclusion follows.
\end{proof}

\section{Coupling the lattice Higgs model and the \texorpdfstring{\( \mathbb{Z}_n \)}{Zn} model}\label{sec: coupling}

Heuristically, when \( \beta \) is very large and \( \sigma \sim \mu_{\beta,\kappa} \), we expect very few of the plaquettes in \( P_N \) to be frustrated in \( d\sigma \), and so, far away from these plaquettes, we expect the distribution of \( \sigma \)  to behave similarly to the gradient field of the \( \mathbb{Z}_n \) model with parameter \( \kappa \), defined in~\eqref{eq: Zn model def}. 
In particular, if \( \sigma \sim \mu_{N,\beta,\kappa} \) and we condition on the spins \( \sigma_e\) of edges \( e \) in a connected component \( E \) of edges with non-zero spin which supports a vortex, then the distribution of \( \sigma \) outside \( E\) has the same distribution as \( \sigma' \sim \mu_{N,\infty,\kappa} \) conditioned to be equal to zero on \( E. \)  
In this section, we make this precise by constructing a coupling between \( \mu_{\beta,\kappa} \) and \( \mu_{\infty,\kappa} \).
To this end, we first define a graph which will be used in the definition of the coupling. 
The main purpose of this graph is to be used as a tool to define a set of edges where two configurations will not be coupled.
\begin{definition}
    Given \( \sigma,\sigma' \in \Sigma_{E_N}  \), let  \( \mathcal{G}(\sigma, \sigma') \) be the graph with vertex set \( E_N \), and with an edge between two distinct vertices  \( e,e' \in E_N  \) if either 
    \begin{enumerate}[label=\textnormal{(\roman*)}]
        \item \( e' = -e \), or 
        \item \( e,e' \in \support \sigma \cup   \support \sigma'  \), and either \( \hat \partial e \cap \hat \partial e' \neq \emptyset \) or \( \hat \partial e \cap \hat \partial (-e') \neq \emptyset \).
    \end{enumerate} 
    Given \( \sigma,\sigma' \in \Sigma_{E_N}  \), \( \mathcal{G} \coloneqq  \mathcal{G}(\sigma, \sigma') \), and \( e \in E_N \), we let \( \mathcal{C}_{\mathcal{G}}(e) \) be set of all edges \( e'\in E_N \) which belong to the same connected component as \( e \) in \( \mathcal{G} \). For \( E \subseteq E_N \), we let \( \mathcal{C}_{\mathcal{G}}(E) \coloneqq \cup_{e \in E} \mathcal{C}_{\mathcal{G}}(e).  \) 
\end{definition}

We next establish a few useful properties of the above definitions.

\begin{lemma}\label{lemma: cluster is subconfig}
    Let \( \sigma ,\sigma' \in \Sigma_{E_N}  \), \( \mathcal{G} \coloneqq  \mathcal{G}(\sigma, \sigma') \), \( E \subseteq E_N \), and \( E' \coloneqq \mathcal{C}_{\mathcal{G}}(E) \). Then 
    \begin{enumerate}[label=\textnormal{(\roman*)}]
        \item \( \sigma |_{E'} \leq \sigma \), \label{item: subconfig i}
        \item \( \sigma |_{E_N\smallsetminus E'} \leq \sigma \),\label{item: subconfig ii}
        \item \( \sigma' |_{E'} \leq \sigma' \), and\label{item: subconfig iii}
        \item \( \sigma' |_{E_N \smallsetminus E'} \leq \sigma' \).\label{item: subconfig iv}
    \end{enumerate} 
\end{lemma}

\begin{proof}
    Assume  that \( p \in \support d (\sigma |_{E'} ) \). Then the set \( \partial p \cap E' \) is non-empty, and hence, by the definition of \( E' \), we have $    \partial p \cap \support \sigma  \subseteq E'$. Consequently, $(d\sigma)_p = \bigl(d(\sigma|_{E'})\bigr)_p$. Since $\sigma|_{\support (\sigma|_{E'})} = \sigma|_{E'}$, this shows that~\ref{item: subconfig i} holds. Using Lemma~\ref{lemma: the blue lemma}\ref{property 4}, we obtain~\ref{item: subconfig ii}. Finally, by symmetry, we obtain~\ref{item: subconfig iii} and~\ref{item: subconfig iv}. This concludes the proof. 
\end{proof}

Given $\sigma ,\sigma' \in \Sigma_{E_N}$, we define the set of edges $E_{\sigma, \sigma'} \subseteq E_N$ by
\begin{equation}\label{eq: Esigmasigmadef}
  E_{\sigma, \sigma'} \coloneqq \mathcal{C}_{\mathcal{G}(\sigma, \sigma')}(\support  \sigma \cap \partial \support d \sigma).
\end{equation}
The set \( E_{\sigma, \sigma'}\) will turn out to always contain the set where the two coupled configurations differ. With this definition at hand, we now take the first step towards the construction of the coupling, by making the following observation.

\begin{lemma}\label{lemma: closed is closed}
    Let \( \hat \sigma \in \Sigma_{E_N} \) and \( \hat \sigma' \in \Sigma_{E_N}^0 \), and define \(\sigma' \coloneqq \hat \sigma'|_{E_{\hat \sigma,\hat \sigma'}} + \hat \sigma|_{E_N\smallsetminus E_{\hat \sigma,\hat \sigma'}} \). Then \( \sigma' \in \Sigma_{E_N}^0 \).
\end{lemma}

\begin{proof}
    Since \( \hat \sigma \in \Sigma_{E_N} \) and \( \hat \sigma' \in \Sigma_{E_N}^0 \), we clearly have  \( \sigma' \in \Sigma_{E_N} \). The desired conclusion will thus follow if we can show that \( d\sigma' = 0 \). 
    To see this, note first that by Lemma~\ref{lemma: cluster is subconfig}, applied with \( \hat \sigma \), \( \hat \sigma' \), and  \( E=\support \hat \sigma \cap \partial \support d\hat \sigma \), we have \( \hat \sigma|_{E_{\hat \sigma,\hat \sigma'}} \leq \hat \sigma \) and \( \hat \sigma|_{E_N\smallsetminus E_{\hat \sigma,\hat \sigma'}} \leq \hat \sigma \).
    
    Let us now show that \( d(\hat \sigma |_{E_N\smallsetminus E_{\hat \sigma,\hat \sigma'}}) = 0 . \)
    If $(d( \hat \sigma |_{E_N\smallsetminus E_{ \hat \sigma, \hat \sigma'}}))_p \neq 0$ for some plaquette \( p \in P_N\), then there exists an $e \in \partial p$ such that $( \hat \sigma |_{E_N\smallsetminus E_{ \hat \sigma, \hat \sigma'}})_e \neq 0$; in particular $e \in (E_N\smallsetminus E_{ \hat \sigma, \hat \sigma'}) \cap \support {\hat \sigma}$.
    Also, the relation $\hat \sigma|_{E_N\smallsetminus E_{\hat \sigma,\hat \sigma'}} \leq \hat \sigma$ implies $\support d(\hat \sigma |_{E_N\smallsetminus E_{\hat \sigma,\hat \sigma'}}) \subseteq \support d\hat \sigma$.
    This means that \( p \in \support d\hat \sigma\), and hence that \( e \in \support  \hat{\sigma} \cap \partial \support d \hat{\sigma}. \) 
	 Since $\support  \hat{\sigma} \cap \partial \support d \hat{\sigma} \subseteq E_{\hat{\sigma}, \hat{\sigma}'}$ by the definition~\eqref{eq: Esigmasigmadef} of $E_{\hat{\sigma}, \hat{\sigma}'}$, we deduce that $e \in E_{\hat{\sigma}, \hat{\sigma}'}$, which contradicts the fact that $e \in E_N\smallsetminus E_{\hat \sigma,\hat \sigma'}$. This proves that \(d(\hat \sigma |_{E_N\smallsetminus E_{\hat \sigma,\hat \sigma'}}) = 0. \)

    Next, since \( \hat \sigma' \in \Sigma_{E_N}^0 \), we have \( d\hat \sigma' = 0 \). Since, by Lemma~\ref{lemma: cluster is subconfig}, we have \( \hat \sigma'|_{E_{\hat \sigma, \hat \sigma'}} \leq \hat \sigma' \), it follows that \( d(\hat \sigma'|_{E_{\hat \sigma, \hat \sigma'}})=0 \).
    Finally, since  \( d( \hat \sigma|_{E_N\smallsetminus E_{\hat \sigma,\hat \sigma'}}) = 0\) and \( d(\hat \sigma'|_{E_{\hat \sigma, \hat \sigma'}})=0 \), we obtain
    \begin{equation*}
        d\sigma' = d(\hat \sigma|_{E_N\smallsetminus E_{\hat \sigma,\hat \sigma'}} + \hat \sigma'|_{E_{\hat \sigma, \hat \sigma'}})
        = d(\hat \sigma|_{E_N\smallsetminus E_{\hat \sigma,\hat \sigma'}}) + d(\hat \sigma'|_{E_{\hat \sigma, \hat \sigma'}}) = 0+0 = 0
    \end{equation*}
    as desired. 
\end{proof}

\begin{lemma}\label{lemma: E in coupling}
    Let \( \hat \sigma \in \Sigma_{E_N} \), \( \hat \sigma' \in \Sigma_{E_N}^0 \), and define 
    \begin{equation*}
        \begin{cases}
            \sigma  \coloneqq \hat \sigma|_{E_{\hat \sigma,\hat \sigma'}}
            +
            \hat \sigma'|_{E_N \smallsetminus E_{\hat \sigma,\hat \sigma'}},
            \cr 
            \sigma' \coloneqq  \hat \sigma'|_{E_{\hat \sigma,\hat \sigma'}}
            +
            \hat \sigma|_{E_N \smallsetminus E_{\hat \sigma,\hat \sigma'}}.
        \end{cases}
    \end{equation*}
    Then \(   E_{\hat \sigma, \hat \sigma'} = E_{\hat \sigma,\sigma'} = E_{\sigma,\hat \sigma'}  = E_{\sigma, \sigma'} \). In particular, \( \hat \sigma_e = \sigma'_e \) and \(  \sigma_e = \hat \sigma'_e \) for all \( e \in E_N \smallsetminus E_{\sigma,\sigma'} \).
\end{lemma}

\begin{proof}
    Before we prove the statement of the lemma, we state and prove a few claims which will simplify the rest of the proof.

    \begin{sublemma}\label{sublemma: 1}
        We have \( d\sigma = d\hat \sigma. \)
    \end{sublemma}

    \begin{subproof}
        By Lemma~\ref{lemma: cluster is subconfig}\ref{item: subconfig ii}, we have \( \hat \sigma'|_{E_N\smallsetminus E_{\hat \sigma, \hat \sigma'}} \leq \hat \sigma'. \) Since \( \hat \sigma' \in \Sigma_{E_N}^0, \) it follows from the definition of \( \leq \) that \( d(\hat \sigma'|_{E_N\smallsetminus E_{\hat \sigma, \hat \sigma'}}) = 0 . \) 
        Consequently, 
        \begin{equation*}
            d\sigma 
            = 
            d(\hat \sigma|_{E_{\hat \sigma,\hat \sigma'}} + \hat \sigma'|_{E_N\smallsetminus E_{\hat \sigma, \hat \sigma'}}) 
            = 
            d(\hat \sigma|_{E_{\hat \sigma,\hat \sigma'}}) + d(\hat \sigma'|_{E_N\smallsetminus E_{\hat \sigma, \hat \sigma'}}) 
            = d(\hat \sigma|_{E_{\hat \sigma,\hat \sigma'}}) + 0 = d(\hat \sigma|_{E_{\hat \sigma,\hat \sigma'}}).
        \end{equation*}  
        By Lemma~\ref{lemma: closed is closed}, we have \( \sigma' = \hat \sigma'|_{E_{\hat \sigma,\hat \sigma'}} + \hat \sigma|_{E_N\smallsetminus E_{\hat \sigma,\hat \sigma'}} \in \Sigma_{E_N}^0, \) and thus \( d\sigma' = 0. \) Since, by Lemma~\ref{lemma: cluster is subconfig}\ref{item: subconfig ii}, we have \( \hat \sigma'|_{E_{\hat \sigma,\hat \sigma'}} \leq \hat \sigma' \) and \( \hat \sigma' \in \Sigma_{E_N}^0, \) it follows from the definition of \( \leq \) that \( d(\hat \sigma'|_{E_{\hat \sigma,\hat \sigma'}}) = 0. \) Hence 
         \begin{equation*}
            \begin{split}
            &d\hat \sigma
            = 
            d(\hat \sigma|_{ E_{\hat \sigma,\hat \sigma'}}+\hat \sigma|_{E_N\smallsetminus E_{\hat \sigma,\hat \sigma'}})
            =  
            d(\hat \sigma|_{ E_{\hat \sigma,\hat \sigma'}}+\hat \sigma'|_{E_{\hat \sigma,\hat \sigma'}} + \hat \sigma|_{E_N\smallsetminus E_{\hat \sigma,\hat \sigma'}} - \hat \sigma'|_{E_{\hat \sigma,\hat \sigma'}})
            \\&\qquad
            =  
            d(\hat \sigma|_{ E_{\hat \sigma,\hat \sigma'}}+\sigma' - \hat \sigma'|_{E_{\hat \sigma,\hat \sigma'}}) 
            =
            d(\hat \sigma|_{ E_{\hat \sigma,\hat \sigma'}})+d\sigma' - d(\hat \sigma'|_{E_{\hat \sigma,\hat \sigma'}})
            =
            d(\hat \sigma|_{ E_{\hat \sigma,\hat \sigma'}})+0-0
            \\&\qquad 
            =
            d(\hat \sigma|_{ E_{\hat \sigma,\hat \sigma'}}).
            \end{split}
        \end{equation*}
        Combining the two previous equations, we obtain \( d\sigma = d\hat \sigma \) as desired.
    \end{subproof}

    \begin{sublemma}\label{sublemma: 2}
        If \( e \in \support \sigma \cap \partial \support d\sigma, \) then there is a plaquette \( p \in \hat \partial e \cap \support d\sigma \) and an edge \( e' \in \partial p \) such that \( \sigma_{e'} = \hat \sigma_{e'} \neq 0 . \)
    \end{sublemma}

    \begin{subproof}
        Since \( e \in \support \sigma \cap \partial \support d\sigma, \) there exists \( p \in \hat \partial e \) with \( (d\sigma)_p  \neq 0. \) Since, by Claim~\ref{sublemma: 1}, we have \( d\sigma = d\hat \sigma \), we must have \( (d\hat \sigma)_p = (d\sigma)_p \neq 0. \) Since \( (d\hat \sigma)_p \neq 0, \)  there exists an edge \( e' \in \partial p \) with \( \hat \sigma_{e'} \neq 0. \) Since \( \hat \sigma_{e'} \neq 0, \) \( e' \in \partial p, \) and \( (d\hat \sigma)_p \neq 0, \) we have \( e' \in E_{\hat \sigma,\hat \sigma'}, \) and hence, by the definition of \( \sigma, \) it follows that \( \sigma_{e'} = \hat \sigma_{e'} \neq 0 . \) This concludes the proof.
    \end{subproof}

    \begin{sublemma}\label{sublemma: 3}
        We have \( \support \sigma \cap \partial \support d\sigma \subseteq  \mathcal{C}_{\mathcal{G}( \sigma ,  0)}(\support \hat \sigma \cap \partial \support d\hat \sigma).\) 
    \end{sublemma}

    \begin{subproof}
        Let \( e \in \support \sigma \cap \partial \support d\sigma. \) Then, by Claim~\ref{sublemma: 2},  there is \( p \in \hat \partial e \cap \support d \sigma   \) and \( e' \in \partial p \) such that \( \sigma_{e'} = \hat \sigma_{e'} \neq 0 . \) Using the definition of \( \mathcal{C}_{\mathcal{G}( \sigma ,  0)} ,\)  it follows that \( e \in \mathcal{C}_{\mathcal{G}( \sigma ,  0)}(e').\)
        Since \( p \in \support d\sigma \) and \( d\sigma = d\hat \sigma \) (by Claim~\ref{sublemma: 1}), we have \( p \in \support d\hat \sigma. \) Since \( e'\in \partial p \) and \( \hat \sigma_{e'} \neq 0, \) it follows that  \( e' \in \support \hat \sigma \cap \partial \support  d\hat \sigma, \) and hence \( \mathcal{C}_{\mathcal{G}( \sigma ,  0)}(e') \subseteq  \mathcal{C}_{\mathcal{G}( \sigma ,  0)}(\support \hat \sigma \cap \partial \support d\hat \sigma).\) 
        Combining the above observations, it follows that \( e \in \mathcal{C}_{\mathcal{G}( \sigma ,  0)}(\support \hat \sigma \cap \partial \support d\hat \sigma).\) 
        Since this holds for any \( e \in \support \sigma \cap \partial \support d\sigma, \) the desired conclusion follows.
    \end{subproof}

    \begin{sublemma}\label{sublemma: 5}
        We have \( \support \hat \sigma \cap \partial \support d\hat \sigma
            \subseteq \support  \sigma \cap \partial \support d \sigma. \)
    \end{sublemma}

    \begin{proof}
        By definition, we have  \(  \sigma|_{E_{\hat \sigma, \hat \sigma'}} = \hat \sigma|_{E_{\hat \sigma, \hat \sigma'}}. \) Since \( \support \hat \sigma \cap \partial \support d\hat \sigma \subseteq E_{\hat \sigma,\hat \sigma'} ,\) and, by Claim~\ref{sublemma: 1}, we have \( d\hat \sigma = d\sigma, \) it follows that
        \begin{equation*}
            \support \hat \sigma \cap \partial \support d\hat \sigma
            =
            \support \hat \sigma|_{E_{\hat \sigma,\hat \sigma'}} \cap \partial \support d\hat \sigma
            =
            \support  \sigma|_{E_{\hat \sigma,\hat \sigma'}} \cap \partial \support d \sigma
            \subseteq \support  \sigma \cap \partial \support d \sigma
        \end{equation*}
        as desired.
    \end{proof}

    \begin{sublemma}\label{sublemma: 4}
        \( \mathcal{G}(\hat \sigma,\sigma') \) and \( \mathcal{G}( \sigma,\hat\sigma') \) are both subgraphs of \( \mathcal{G}(\hat \sigma,\hat \sigma') = \mathcal{G}( \sigma, \sigma'). \)
    \end{sublemma}
    
    \begin{subproof}
        By definition, we have \( \support \hat \sigma \cup \support \sigma'   \subseteq \support \hat \sigma \cup \support \hat \sigma', \) \(  \support  \sigma \cup \support \hat \sigma' \subseteq \support \hat \sigma \cup \support \hat \sigma'  , \)
        and 
        \(  \support \hat \sigma \cup \support \hat \sigma' = \support \sigma \cup \support \sigma'.\) From this the desired conclusion immediately follows.
    \end{subproof}
    
    We now use the above claims to prove that \( E_{\sigma,\sigma'} = E_{\sigma,\hat \sigma'} = E_{\hat \sigma, \sigma'} = E_{\hat \sigma,\hat \sigma}. \)
    
    By combining Claim~\ref{sublemma: 3} and Claim~\ref{sublemma: 4}, we see that 
    \begin{equation}\label{eq: new proof eq 6}
        \begin{split}
            &E_{\sigma,  \sigma'}  = \mathcal{C}_{ {\mathcal{G}( \sigma,   \sigma')}}(\support   \sigma \cap  \partial \support d \sigma)
            \subseteq 
            \mathcal{C}_{ {\mathcal{G}( \sigma,   \sigma')}}\bigl(\mathcal{C}_{\mathcal{G}( \sigma ,  0)}(\support \hat \sigma \cap \partial \support d\hat \sigma)\bigr)
            \\&\qquad = 
            \mathcal{C}_{ {\mathcal{G}( \sigma,   \sigma')}}(\support \hat \sigma \cap \partial \support d\hat \sigma) 
            =
            \mathcal{C}_{ {\mathcal{G}( \hat \sigma,   \hat \sigma')}}(\support \hat \sigma \cap \partial \support d\hat \sigma) = E_{\hat \sigma,\hat \sigma'}
        \end{split}
    \end{equation}
    At the same time, by combining Claim~\ref{sublemma: 5} and Claim~\ref{sublemma: 4}, we obtain
          \begin{equation}\label{eq: new proof eq 3}
          \begin{split}
              &E_{\hat\sigma,\hat\sigma'} = \mathcal{C}_{\mathcal{G}(\hat\sigma,\hat\sigma')}(\support \hat \sigma \cap \partial \support d\hat \sigma) 
              \subseteq 
              \mathcal{C}_{\mathcal{G}(\hat\sigma,\hat\sigma')}(\support  \sigma \cap \partial \support d \sigma)
              \\&\qquad=
              \mathcal{C}_{\mathcal{G}(\sigma,\sigma')}(\support  \sigma \cap \partial \support d \sigma) = E_{ \sigma, \sigma'}.
          \end{split}
    \end{equation}
    Combining~\eqref{eq: new proof eq 6}~and~\eqref{eq: new proof eq 3}, we obtain \( E_{\sigma,\sigma'} \subseteq E_{\hat \sigma,\hat \sigma'} \subseteq E_{\sigma,\sigma'}, \) and hence \( E_{\sigma,\sigma'} = E_{\hat \sigma,\hat \sigma'}. \) 
    
    Next, using Claim~\ref{sublemma: 4}, we immediately see that
    \begin{equation}\label{eq: new proof eq 1}
        E_{\hat \sigma, \sigma'} 
        =  
        \mathcal{C}_{ {\mathcal{G}( \hat \sigma,  \sigma')}}(\support  \hat \sigma \cap  \partial \support d \hat \sigma)
        \subseteq 
        \mathcal{C}_{ {\mathcal{G}( \hat \sigma,  \hat \sigma')}}(\support  \hat \sigma \cap  \partial \support d \hat \sigma) = E_{\hat \sigma,\hat \sigma'}
    \end{equation} 
    and
    \begin{equation}\label{eq: new proof eq 5}
        \begin{split}
            E_{\sigma,\hat \sigma'}
            = \mathcal{C}_{\mathcal{G}(\sigma,\hat \sigma')}(\support \sigma \cap \partial \support d\sigma)
            \subseteq \mathcal{C}_{\mathcal{G}(\sigma, \sigma')}(\support \sigma \cap \partial \support d\sigma) = E_{\sigma,\sigma'}.
        \end{split}
    \end{equation} 
    
    Since 
    \( \sigma'|_{E_{\hat \sigma,\hat \sigma'}} = \hat \sigma'|_{E_{\hat \sigma,\hat \sigma'}}, \) the graphs \( \mathcal{G}( \sigma, \sigma') \) and \( \mathcal{G}(\sigma,\hat \sigma') \)  are identical when restricted to the set \( E_{\hat \sigma,\hat \sigma'}.\) Since \( E_{\hat \sigma,\hat \sigma'} = E_{\sigma,\sigma'} \) and the set \( E_{\sigma,\sigma'} \) contains the set \( \support \sigma \cap \partial \support d\sigma,\) this implies in particular that
    \begin{equation}\label{eq: new proof eq 4}
        \begin{split}
            E_{\sigma,\sigma'} = \mathcal{C}_{\mathcal{G}(\sigma,\sigma')}(\support \sigma \cap \partial \support d\sigma) \subseteq \mathcal{C}_{\mathcal{G}(\sigma,\hat \sigma')}(\support \sigma \cap \partial \support d\sigma) = E_{\sigma,\hat \sigma'}.
        \end{split}
    \end{equation}  
    Combining~\eqref{eq: new proof eq 5}~and~\eqref{eq: new proof eq 4}, we obtain \( E_{\sigma,\sigma'} \subseteq E_{\sigma,\hat \sigma'} \subseteq E_{\sigma,\sigma'}\), and hence \( E_{\sigma,\hat \sigma'} = E_{\sigma,\sigma'}. \)
    Similarly, since \( \sigma|_{E_{\hat \sigma,\hat \sigma'}} = \hat \sigma|_{E_{\hat \sigma,\hat \sigma'}}, \) 
    the graphs \( \mathcal{G}( \sigma, \sigma') \) and  \( \mathcal{G}(\hat  \sigma,\sigma') \)  are identical when restricted to the set \( E_{\hat \sigma,\hat \sigma'},\) which contains the set \( \support \hat \sigma \cap \partial \support d\hat \sigma.\) Using also Claim~\ref{sublemma: 3}, we obtain
    \begin{equation}\label{eq: new proof eq 7}
        \begin{split}
            &E_{\sigma,\sigma'} 
            = 
            \mathcal{C}_{\mathcal{G}(\sigma,\sigma')}(\support \sigma \cap \partial \support d\sigma) 
            \subseteq 
            \mathcal{C}_{\mathcal{G}(\sigma,\sigma')}\bigl(\mathcal{C}_{\mathcal{G}(\sigma,0)}(\support \hat \sigma \cap \partial \support d\hat \sigma) \bigr)
            \\&\qquad =
            \mathcal{C}_{\mathcal{G}(\sigma,\sigma')}(\support \hat \sigma \cap \partial \support d\hat \sigma) 
            \subseteq
            \mathcal{C}_{\mathcal{G}(\hat \sigma,\sigma')}(\support \hat{\sigma} \cap \partial \support d\hat{\sigma}) 
            = 
            E_{\hat \sigma,\sigma'}.
        \end{split}
    \end{equation} 
    Combining~\eqref{eq: new proof eq 3},~\eqref{eq: new proof eq 1}~and~\eqref{eq: new proof eq 7}, we obtain  \( E_{\sigma,\sigma} \subseteq E_{\hat \sigma, \sigma'} \subseteq E_{\hat \sigma,\hat \sigma'} \subseteq E_{\sigma,\sigma'}, \) and hence \( E_{\hat \sigma, \sigma'} = E_{\sigma,\sigma'}. \)
\end{proof}

We are now ready to describe the coupling between \( \mu_{N,\beta,\kappa} \) and \( \mu_{N,\infty,\kappa} \).

\begin{definition}[The coupling to the \( \mathbb{Z}_n \) model]\label{def: the coupling}
    Let \( \beta,\kappa \geq 0 \). 
    For \( \sigma \in \Sigma_{E_N} \) and \( \sigma' \in \Sigma_{E_N}^0 \), we define
    \begin{align*}
        &\mu_{N, (\beta,\kappa),(\infty,\kappa)}(\sigma, \sigma') 
        \\&\qquad\coloneqq \mu_{N, \beta,\kappa} \times \mu_{N,\infty, \kappa}\bigl(\big\{(\hat{\sigma}, \hat{\sigma}')\in \Sigma_{E_N} \times \Sigma_{E_N}^0 \colon \sigma=\hat{\sigma} \text{ and } \sigma' = \hat \sigma'|_{E_{\hat \sigma,\hat \sigma'}} + \hat \sigma|_{E_N \smallsetminus E_{\hat \sigma,\hat \sigma'}}\bigr\} \bigr).
    \end{align*}
    We let \( \mathbb{E}_{N,(\beta,\kappa),(\infty,\kappa)} \) denote the corresponding expectation. 
\end{definition}

\begin{remark}
    We mention that, by definition, if \( \hat \sigma \sim \mu_{N,\beta,\kappa} \) and \( \hat \sigma' \sim \mu_{N,\infty,\kappa} \) are independent, and we let 
    \begin{equation}\label{couplingdef}
        \begin{cases}
            \sigma  \coloneqq \hat \sigma,  \cr 
            \sigma' \coloneqq  \hat \sigma'|_{E_{\hat \sigma,\hat \sigma'}}
            +
            \hat \sigma|_{E_N \smallsetminus E_{\hat \sigma,\hat \sigma'}},
        \end{cases}
    \end{equation}
    then \( (\sigma,\sigma') \sim \mu_{N,(\beta,\kappa),(\infty,\kappa)} \). In Figure~\ref{fig: the coupling}, we illustrate the construction of the pair \( (\sigma,\sigma') \) as described above.
\end{remark}

\begin{figure}[htp]
    \centering
    \begin{subfigure}[t]{0.32\textwidth}\centering
        \includegraphics[width=\textwidth]{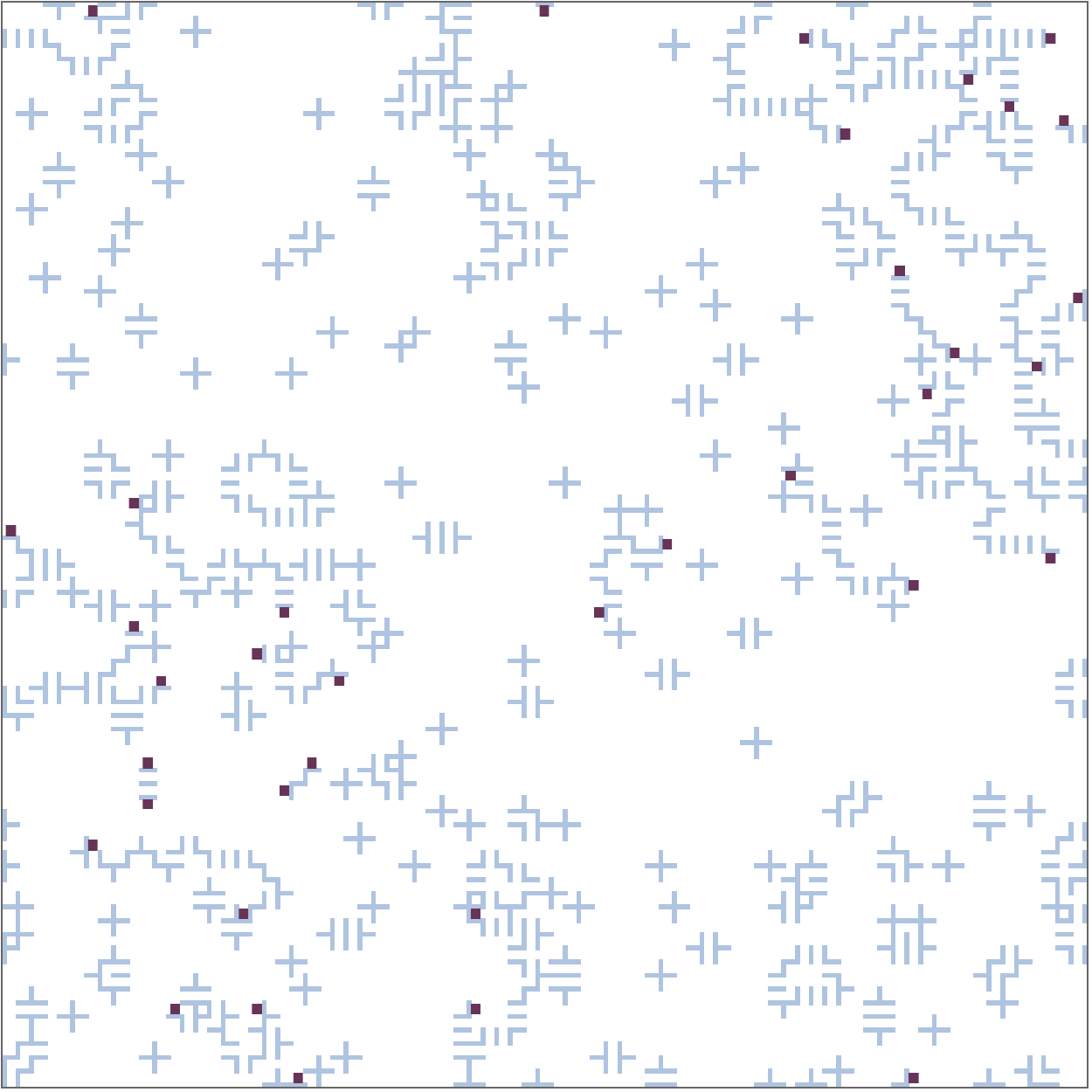}
        \caption{Blue edges correspond to the support of \( \hat \sigma \sim \mu_{N,\beta,\kappa}\), and black squares to the support of \( d\hat \sigma \).}
    \end{subfigure}
    \hfil
    \begin{subfigure}[t]{0.32\textwidth}\centering
        \includegraphics[width=\textwidth]{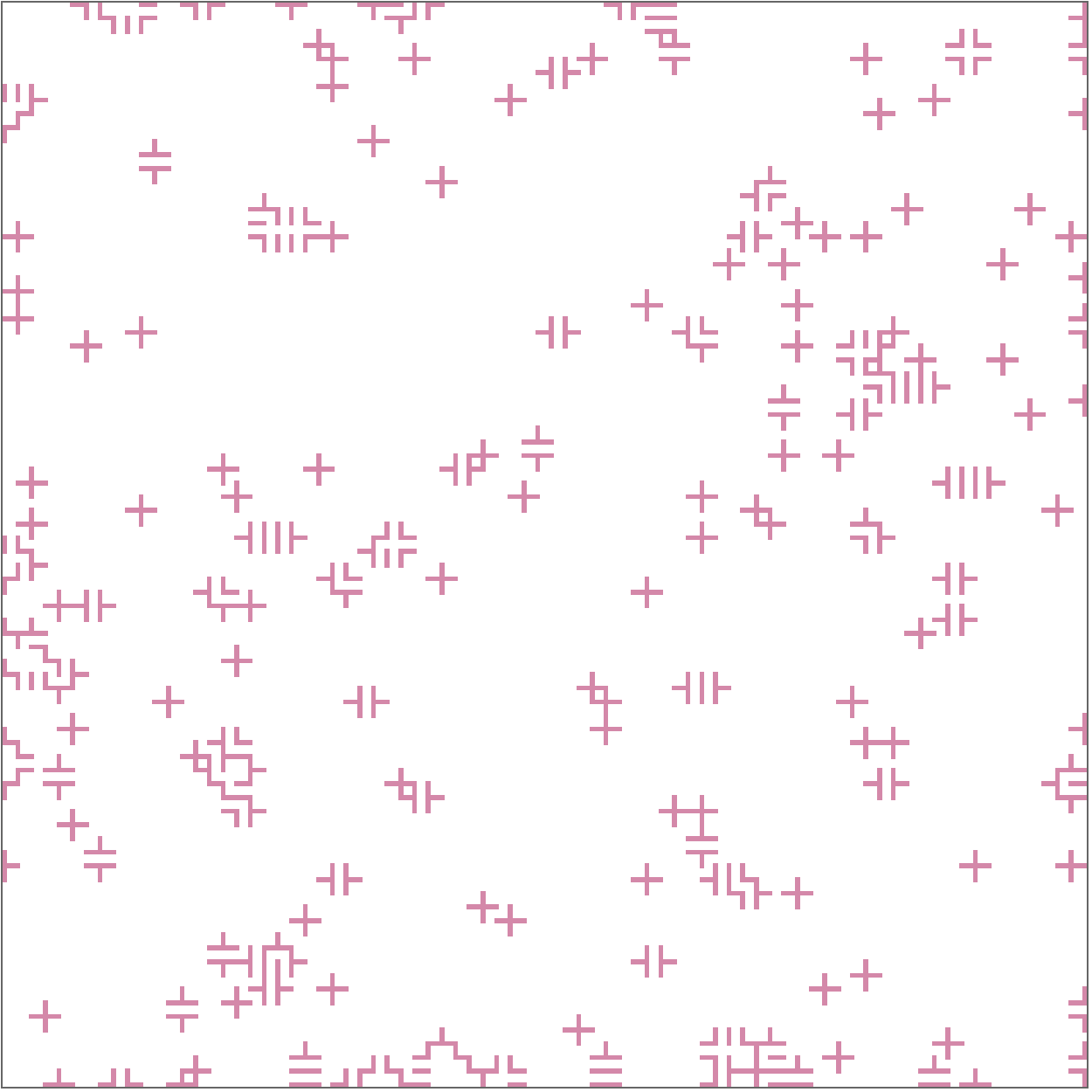}
        \caption{Red edges correspond to the support of \( \hat \sigma' \sim \mu_{N,\infty,\kappa}\). In this case we automatically have \( d\hat \sigma' = 0\).}
    \end{subfigure}
    \hfil
    \begin{subfigure}[t]{0.32\textwidth}\centering
        \includegraphics[width=\textwidth]{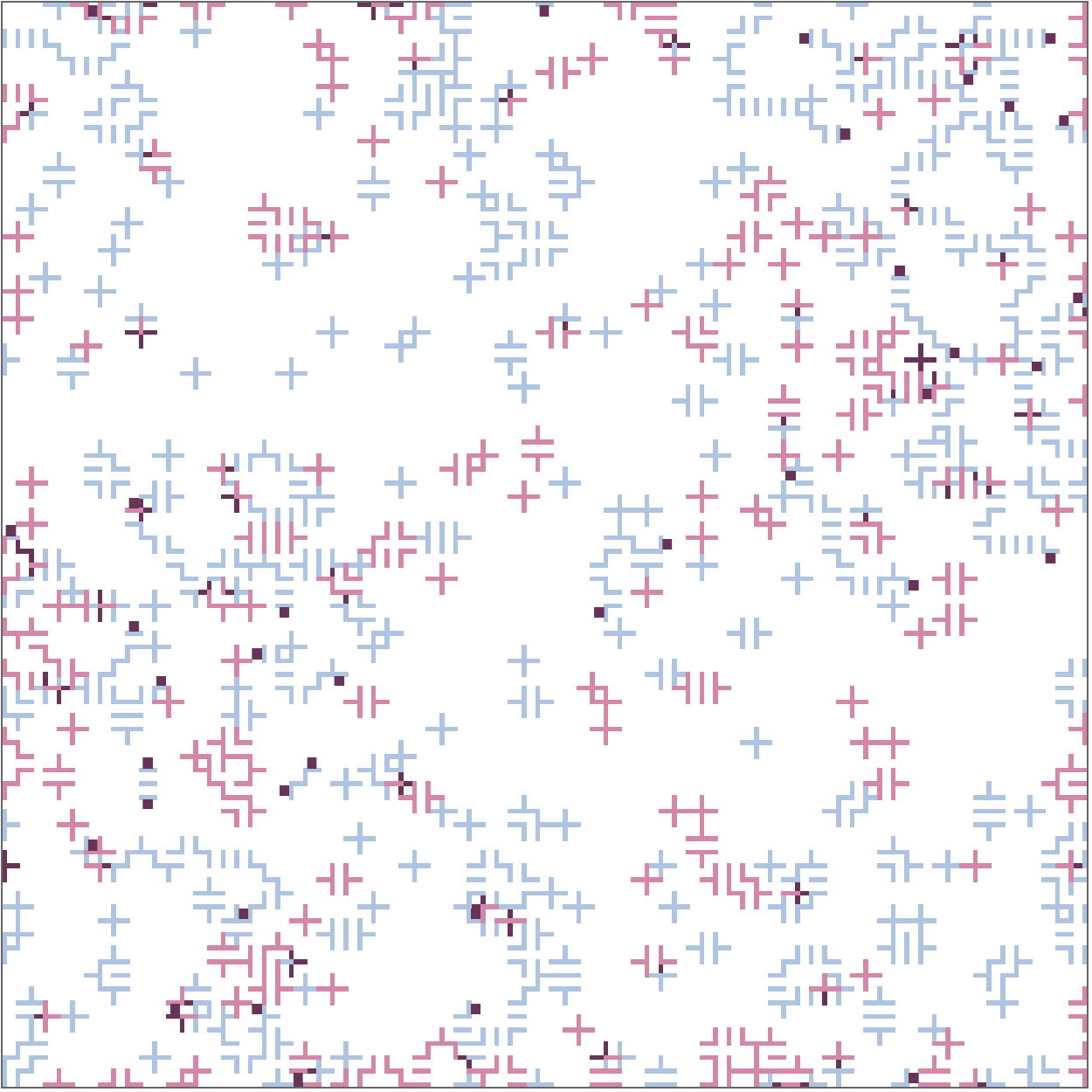}
        \caption{Blue edges correspond to the support of \( \hat \sigma \), red edges correspond to the support of \( \hat \sigma'\), and black squares to the support of \( d\hat \sigma \).}
    \end{subfigure}

    \vspace{1ex}
    \begin{subfigure}[t]{0.32\textwidth}\centering
        \includegraphics[width=\textwidth]{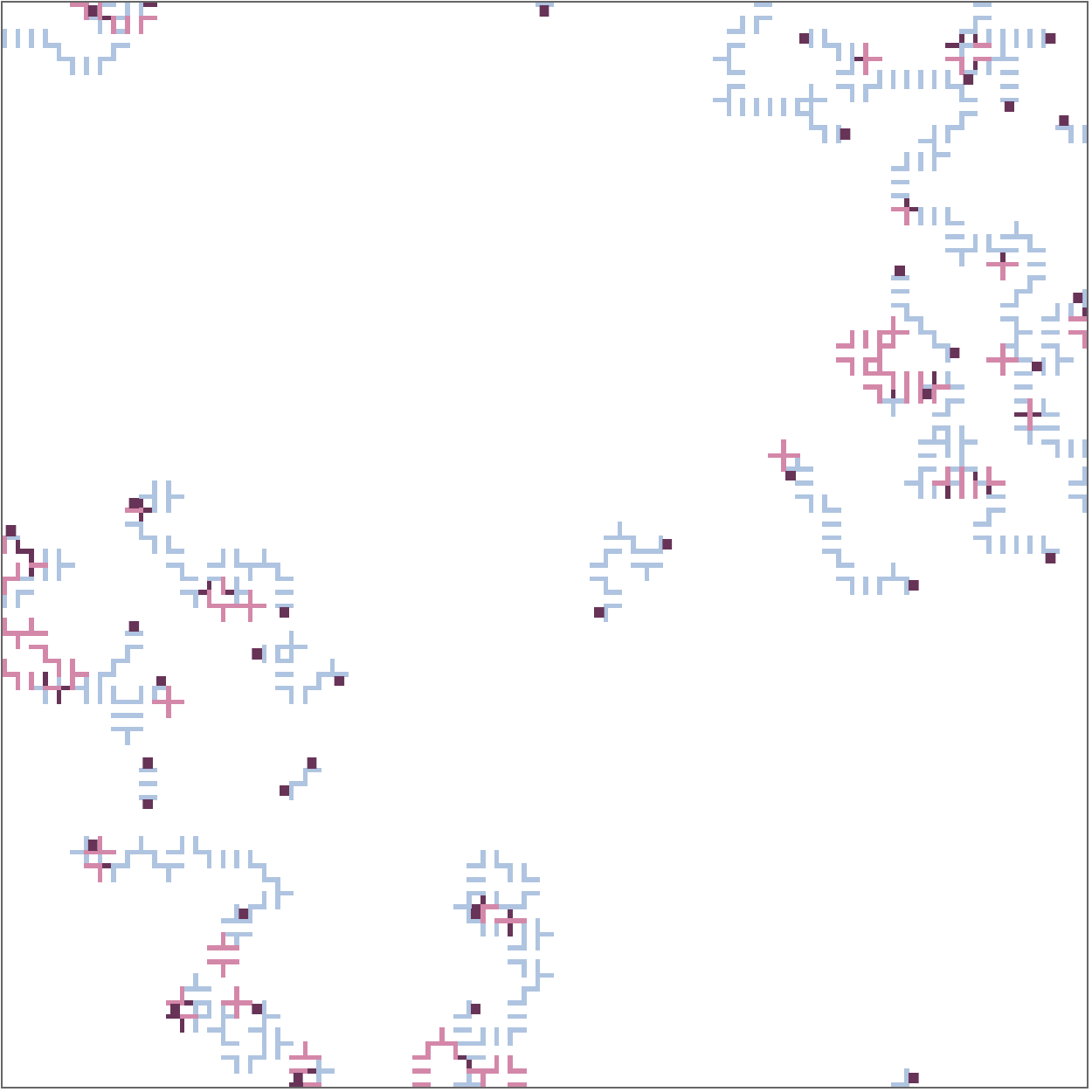}
        \caption{Blue edges correspond to the support of  \(  \hat \sigma|_{E_{\hat \sigma,\hat \sigma'}}  \),  red edges correspond to the support of \(  \hat \sigma'|_{E_{\hat \sigma,\hat \sigma'}} \), and  black squares correspond to the support of \( d\hat \sigma  \).}\label{subfig: the coupling d}
    \end{subfigure}
    \hfil
    \begin{subfigure}[t]{0.32\textwidth}\centering
        \includegraphics[width=\textwidth]{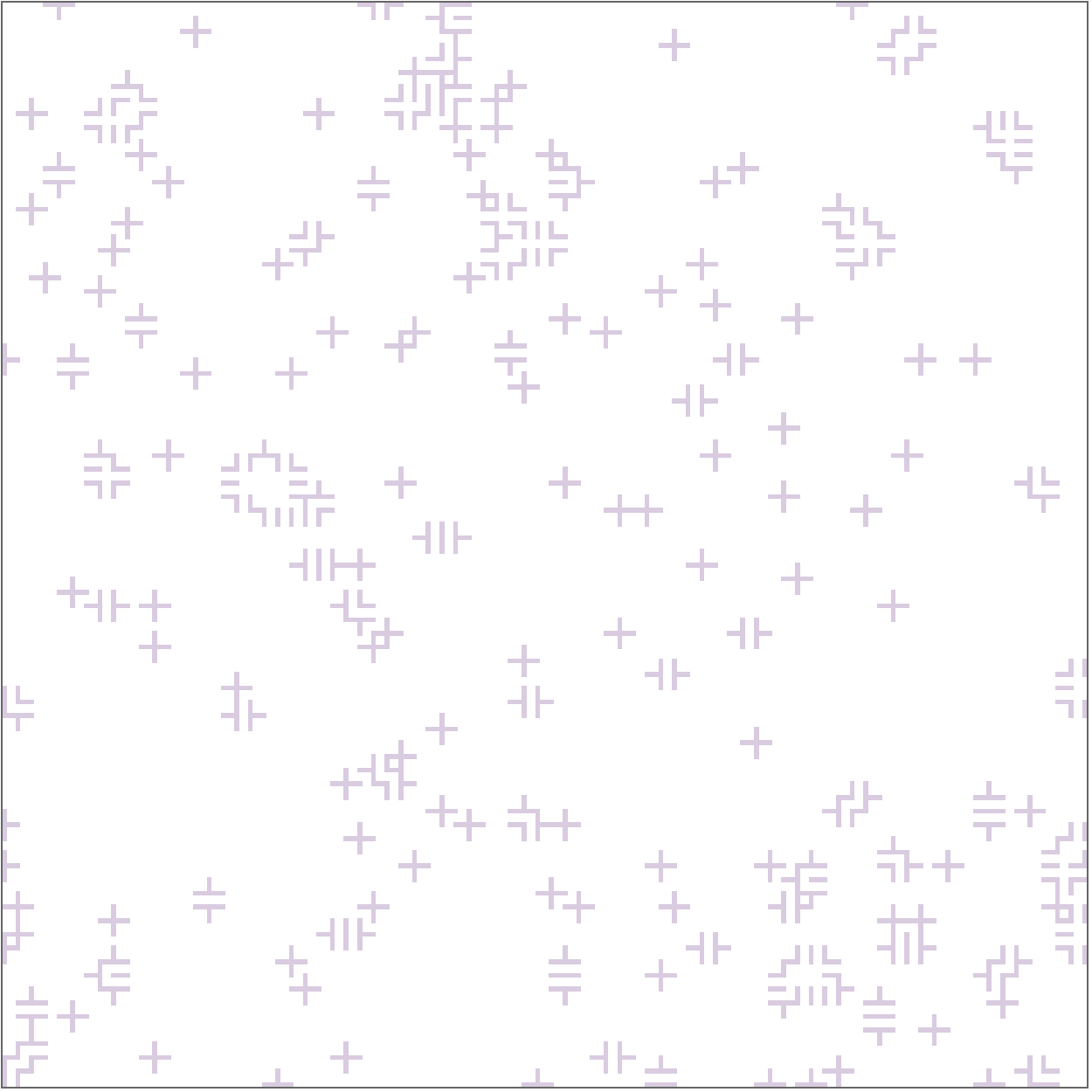}
        \caption{Purple edges correspond to the support of  \(  \hat \sigma|_{E_N\smallsetminus E_{\hat \sigma,\hat \sigma'}}  \).}\label{subfig: the coupling e}
    \end{subfigure}
    \hfil
    \begin{subfigure}[t]{0.32\textwidth}\centering
        \includegraphics[width=\textwidth]{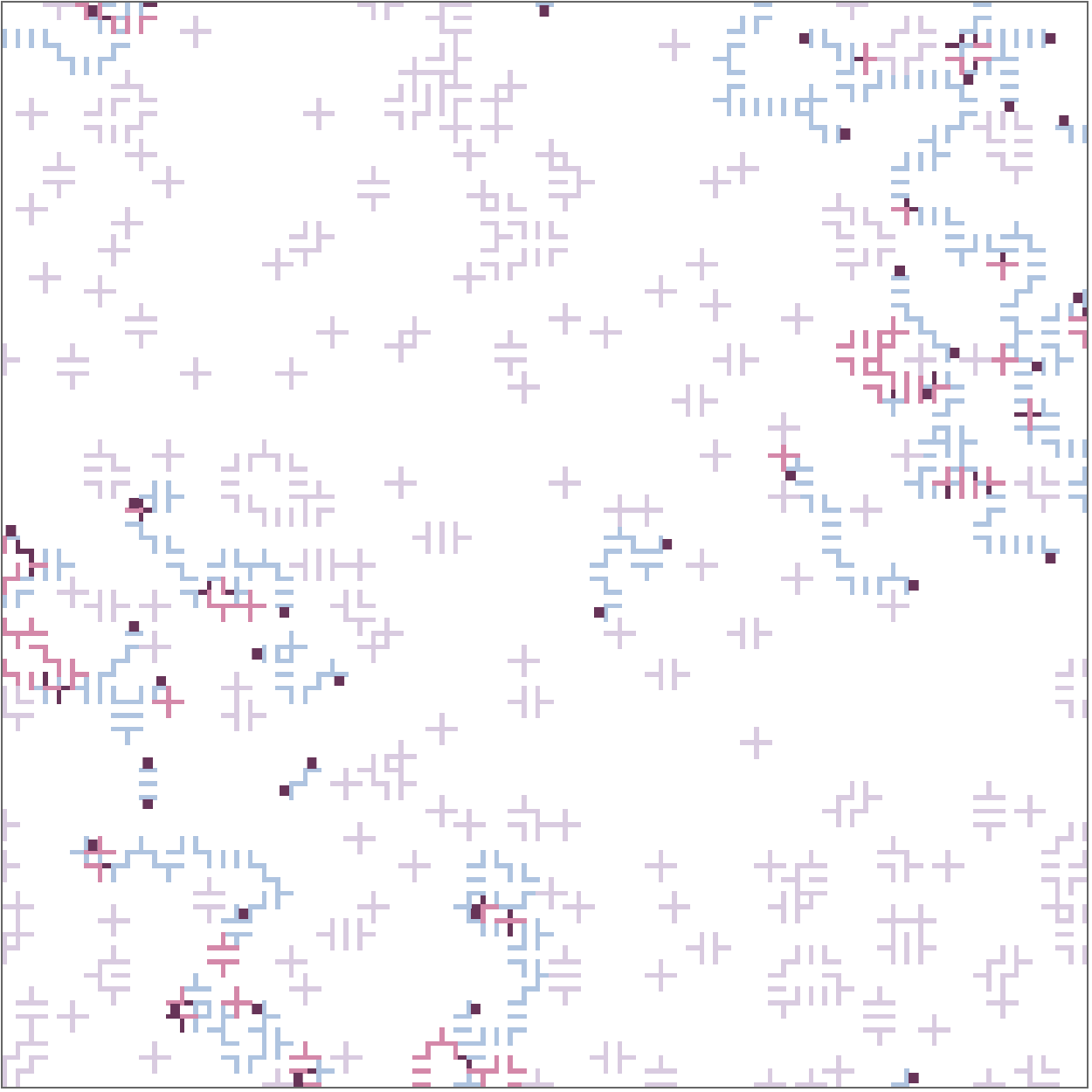}
        \caption{Blue edges correspond to the support of \( \hat \sigma|_{E_{\hat \sigma, \hat \sigma'}}\),  red edges correspond to the support of \( \hat \sigma'|_{E_{\hat \sigma, \hat \sigma'}}\),  purple edges correspond to the support of \( \hat \sigma|_{E_N \smallsetminus E_{\hat \sigma, \hat \sigma'}}\), and  black squares correspond to the support of \( d\hat \sigma \).}
    \end{subfigure}
    \caption{Illustration of the coupling \( (\sigma, \sigma') = (\hat \sigma, \hat \sigma'|_{E_{\hat \sigma, \hat \sigma'}} + \hat \sigma|_{E_N \smallsetminus E_{\hat \sigma, \hat \sigma'}}) \) defined in Definition~\ref{def: the coupling} (simulated on a 2-dimensional lattice, with \( G = \mathbb{Z}_2 \)). }
    \label{fig: the coupling}
\end{figure}

\begin{proposition}\label{proposition: coupling is a coupling}
    Let \( \beta,\kappa \geq 0 \). Then \( \mu_{N,(\beta,\kappa),{(\infty,\kappa})} \) is a coupling of \( \mu_{N,\beta,\kappa} \) and \( \mu_{N,\infty,\kappa} \). 
\end{proposition}

\begin{proof} 
    It is immediate from the definition that if \( (\sigma,\sigma') \sim \mu_{N,(\beta,\kappa),(\infty,\kappa)}\), then \( \sigma \sim\mu_{N,\beta,\kappa} \), and it is hence sufficient to show that \( \sigma' \sim \mu_{N,\infty,\kappa} \).
	To this end, fix some \(  \sigma' \in \Sigma_{E_N}^0 \). We need to show that
	\begin{equation*}	\mu_{N,\beta,\kappa}\times \mu_{N,\infty,\kappa} \bigl( \bigl\{ (\hat \sigma,\hat \sigma') \in \Sigma_{E_N} \times \Sigma_{E_N^0} \colon \hat \sigma|_{E_N\smallsetminus E_{\hat \sigma,\hat \sigma'}}  + \hat \sigma'|_{E_{\hat \sigma,\hat \sigma'}} =  \sigma' \bigr\} \bigr)
		=
		\mu_{N,\infty,\kappa}( \sigma'),
	\end{equation*}
	or, equivalently, using~\eqref{eq: mubetakappaphi} (see also~\eqref{eq: mubetakappaphiinfty}), that
	\begin{equation}\label{eq: coupling goal eq}
		\sum_{\hat \sigma \in \Sigma_{E_N}} \sum_{\hat \sigma' \in \Sigma_{E_N}^0}
		\varphi_{\beta,\kappa}(\hat \sigma) \varphi_{\beta,\kappa}(\hat \sigma') \cdot \mathbb{1} \bigl(  \hat \sigma|_{E_N\smallsetminus E_{\hat \sigma,\hat \sigma'}} + \hat \sigma'|_{E_{\hat \sigma,\hat \sigma'}} =  \sigma' \bigr)
		=
		\varphi_\kappa( \sigma') \sum_{ \sigma \in \Sigma_{E_N}} \varphi_{\beta,\kappa} ( \sigma),
	\end{equation}
	where we have used that \(\varphi_\kappa(\hat{\sigma}') = \varphi_{\beta, \kappa}(\hat{\sigma}')\) because \(d\hat{\sigma}' = 0\).
	We now rewrite the left-hand side of~\eqref{eq: coupling goal eq} in order to see that this equality indeed holds.
	
	Let \( \hat \sigma \in \Sigma_{E_N} \) and \( \hat \sigma' \in \Sigma_{E_N}^0 \).
	By Lemma~\ref{lemma: cluster is subconfig}, we have \( \hat \sigma |_{E_{\hat \sigma,\hat \sigma'}} \leq \hat \sigma \) and \( \hat \sigma' |_{E_{\hat \sigma,\hat \sigma'}} \leq \hat \sigma'\) 
	and hence, by Lemma~\ref{lemma: factorization property},
    \begin{equation}\label{eq: useful eq 1a}
     \varphi_{\beta,\kappa}(\hat \sigma) = \varphi_{\beta,\kappa}\bigl( \hat \sigma |_{E_{\hat \sigma,\hat \sigma'}} \bigr)\varphi_{\beta,\kappa}\bigl( \hat \sigma |_{E_N\smallsetminus E_{\hat \sigma,\hat \sigma'}} \bigr)
        \;\; \text{and} \;\;
        \varphi_{\beta,\kappa}(\hat \sigma') = \varphi_{\beta,\kappa}\bigl( \hat \sigma' |_{E_{\hat \sigma,\hat \sigma'}} \bigr)\varphi_{\beta,\kappa}\bigl( \hat \sigma' |_{E_N\smallsetminus E_{\hat \sigma,\hat \sigma'}} \bigr).
    \end{equation} 
	Since \( \hat{\sigma}' \in \Sigma_{E_N}^0 \), we have \( d \hat{\sigma}' = 0 \).  Since \( \hat{\sigma}'|_{E_{\hat{\sigma},\hat{\sigma}'}} \leq \hat{\sigma}' \), and hence \( d(\hat{\sigma}'|_{E_{\hat{\sigma},\hat{\sigma}'}} )= 0 \). As \( \hat \sigma'|_{E_{\hat{\sigma},\hat{\sigma}'}} \) and  \( \hat \sigma|_{E_N\smallsetminus E_{\hat{\sigma},\hat{\sigma}'}}\)  have disjoint supports, it follows that \( \hat \sigma'|_{E_{\hat \sigma,\hat \sigma'}} \leq \hat \sigma|_{E_N\smallsetminus E_{\hat \sigma,\hat \sigma'}} + \hat \sigma'|_{E_{\hat \sigma,\hat \sigma'}} \). Thus, by Lemma~\ref{lemma: factorization property},
    \begin{equation}\label{eq: useful eq 1b}
        \varphi_{\beta,\kappa}\bigl( \hat \sigma |_{E_N\smallsetminus E_{\hat \sigma,\hat \sigma'}} \bigr)
         \varphi_{\beta,\kappa}\bigl( \hat \sigma' |_{E_{\hat \sigma,\hat \sigma'}} \bigr)
         	=
         	\varphi_{\beta,\kappa}\bigl( \hat \sigma |_{E_N\smallsetminus E_{\hat \sigma,\hat \sigma'}} +\hat \sigma' |_{E_{\hat \sigma,\hat \sigma'}} \bigr).
    \end{equation} %
	Similarly, by Lemma~\ref{lemma: cluster is subconfig}, we have \( \hat{\sigma}'|_{E_N\smallsetminus E_{\hat{\sigma},\hat{\sigma}'}} \leq \hat{\sigma}' \), and hence \( d(\hat{\sigma}'|_{E_N\smallsetminus E_{\hat{\sigma},\hat{\sigma}'}})=0 \). Since \( \hat \sigma'|_{E_N\smallsetminus E_{\hat{\sigma},\hat{\sigma}'}} \) and \( \hat \sigma|_{ E_{\hat{\sigma},\hat{\sigma}'}}\) have disjoint supports, it follows that \( \hat \sigma'|_{E_N\smallsetminus E_{\hat \sigma,\hat \sigma'}} \leq \hat \sigma|_{ E_{\hat \sigma,\hat \sigma'}}  + \hat \sigma'|_{E_N\smallsetminus E_{\hat \sigma,\hat \sigma'}} \). Consequently, by Lemma~\ref{lemma: factorization property}, we have
    \begin{equation}\label{eq: useful eq 1c}
        \varphi_{\beta,\kappa}\bigl( \hat \sigma |_{E_{\hat \sigma,\hat \sigma'}} \bigr)
         \varphi_{\beta,\kappa}\bigl( \hat \sigma' |_{E_N\smallsetminus E_{\hat \sigma,\hat \sigma'}} \bigr)
         	=
         	\varphi_{\beta,\kappa}\bigl( \hat \sigma |_{E_{\hat \sigma,\hat \sigma'}} +\hat \sigma' |_{E_N\smallsetminus E_{\hat \sigma,\hat \sigma'}} \bigr).
    \end{equation} %
    Using~\eqref{eq: useful eq 1a},~\eqref{eq: useful eq 1b}, and~\eqref{eq: useful eq 1c}, we can rewrite the left-hand side of~\eqref{eq: coupling goal eq} as follows:
	\begin{equation}\label{eq: coupling goal eq ii} 
		\begin{split}
			&\sum_{\hat \sigma \in \Sigma_{E_N}} \sum_{\hat \sigma' \in \Sigma_{E_N}^0}
			\varphi_{\beta,\kappa}(\hat \sigma) \varphi_{\beta,\kappa}(\hat \sigma') \cdot \mathbb{1} \bigl(  \hat \sigma|_{E_N\smallsetminus E_{\hat \sigma,\hat \sigma'}} + \hat \sigma'|_{E_{\hat \sigma,\hat \sigma'}} =  \sigma' \bigr) 
			\\&\qquad =
			\sum_{\hat \sigma \in \Sigma_{E_N}} \sum_{\hat \sigma' \in \Sigma_{E_N}^0}
			\varphi_{\beta,\kappa}\bigl( \hat \sigma |_{E_N\smallsetminus E_{\hat \sigma,\hat \sigma'}} +\hat \sigma' |_{E_{\hat \sigma,\hat \sigma'}} \bigr)
			\varphi_{\beta,\kappa}\bigl( \hat \sigma |_{E_{\hat \sigma,\hat \sigma'}} +\hat \sigma' |_{E_N\smallsetminus E_{\hat \sigma,\hat \sigma'}} \bigr)
			\\&\qquad\qquad \cdot 
			\mathbb{1} \bigl(  \hat \sigma|_{E_N\smallsetminus E_{\hat \sigma,\hat \sigma'}} + \hat \sigma'|_{E_{\hat \sigma,\hat \sigma'}} =  \sigma' \bigr) 
			\\&\qquad = 
			\varphi_{\beta,\kappa}\bigl(  \sigma' \bigr)
			\sum_{\hat \sigma \in \Sigma_{E_N}} \sum_{\hat \sigma' \in \Sigma_{E_N}^0}
			\varphi_{\beta,\kappa}\bigl( \hat \sigma |_{E_{\hat \sigma,\hat \sigma'}} +\hat \sigma' |_{E_N\smallsetminus E_{\hat \sigma,\hat \sigma'}} \bigr)
			\cdot 
			\mathbb{1} \bigl(  \hat \sigma|_{E_N\smallsetminus E_{\hat \sigma,\hat \sigma'}} + \hat \sigma'|_{E_{\hat \sigma,\hat \sigma'}} =  \sigma' \bigr) 
			\\&\qquad =
  			\varphi_{\beta,\kappa}\bigl(  \sigma' \bigr)
			\sum_{ \sigma \in \Sigma_{E_N}} \varphi_{\beta,\kappa}( \sigma) \sum_{\hat \sigma \in \Sigma_{E_N}} \sum_{\hat \sigma' \in \Sigma_{E_N}^0}
			\mathbb{1}\bigl( \hat \sigma |_{E_{\hat \sigma,\hat \sigma'}} +\hat \sigma' |_{E_N\smallsetminus E_{\hat \sigma,\hat \sigma'}} =  \sigma \bigr)
			\\&\qquad\qquad \cdot 
			\mathbb{1} \bigl(  \hat \sigma|_{E_N\smallsetminus E_{\hat \sigma,\hat \sigma'}} + \hat \sigma'|_{E_{\hat \sigma,\hat \sigma'}} =  \sigma' \bigr).  
			\end{split}	
	\end{equation}  
	Now note that, by Lemma~\ref{lemma: E in coupling}, we have \( E_{\hat \sigma,\hat \sigma'} = E_{\sigma,\sigma'} \).
	Consequently, for any \(  \sigma \in \Sigma_{E_N} \), we have
	\begin{equation}\label{eq: doublesum11}
		\begin{split}
			&\sum_{\hat \sigma \in \Sigma_{E_N}} \sum_{\hat \sigma' \in \Sigma_{E_N}^0}
			\mathbb{1}\bigl( \hat \sigma |_{E_{\hat \sigma,\hat \sigma'}} +\hat \sigma' |_{E_N\smallsetminus E_{\hat \sigma,\hat \sigma'}} =  \sigma \bigr)
			\cdot 
			\mathbb{1} \bigl(  \hat \sigma|_{E_N\smallsetminus E_{\hat \sigma,\hat \sigma'}} + \hat \sigma'|_{E_{\hat \sigma,\hat \sigma'}} = \sigma' \bigr)
			\\&\qquad=
			\sum_{\hat \sigma \in \Sigma_{E_N}} \sum_{\hat \sigma' \in \Sigma_{E_N}^0}
			\mathbb{1}\bigl( \hat \sigma |_{E_{ \sigma, \sigma'}} +\hat \sigma' |_{E_N\smallsetminus E_{ \sigma, \sigma'}} =  \sigma \bigr)
			\cdot 
			\mathbb{1} \bigl(  \hat \sigma|_{E_N\smallsetminus E_{ \sigma, \sigma'}} + \hat \sigma'|_{E_{ \sigma, \sigma'}} =  \sigma' \bigr)
			\\&\qquad=
			\sum_{\hat \sigma \in \Sigma_{E_N}} \sum_{\hat \sigma' \in \Sigma_{E_N}^0}
			\mathbb{1}\bigl( \hat \sigma    =  \sigma|_{E_{ \sigma, \sigma'}}+ \sigma'|_{E_N\smallsetminus E_{ \sigma, \sigma'}} \bigr)
			\cdot 
			\mathbb{1}\bigl(  \hat \sigma' =  \sigma |_{E_N\smallsetminus E_{ \sigma, \sigma'}} +   \sigma' |_{E_{ \sigma, \sigma'}} \bigr).
		\end{split}
	\end{equation}
	Since \( \sigma \in \Sigma_{E_N} \) and \( \sigma' \in \Sigma_{E_N}^0 \), we can apply Lemma~\ref{lemma: closed is closed} to obtain \( \sigma'|_{E_{\sigma,\sigma'}} + \sigma_{E_N\smallsetminus E_{\sigma,\sigma'}} \in \Sigma_{E_N}^0 \).
	From this it follows that the double sum on the right-hand side of~\eqref{eq: doublesum11} is equal to $1$. Inserting this into the last line of~\eqref{eq: coupling goal eq ii} and using that $\varphi_{\beta,\kappa}(  \sigma') = \varphi_{\kappa}(  \sigma')$, we obtain~\eqref{eq: coupling goal eq} as desired. This concludes the proof.
\end{proof}

In what follows, we collect a few useful properties of the coupling introduced in Definition~\ref{def: the coupling} for easy reference in later sections.
 
\begin{lemma}\label{lemma: properties of coupling measure} 
    Let \( \beta,\kappa \geq 0 \) and assume that \( \mu_{N,(\beta,\kappa),(\infty,\kappa)}(\sigma , \sigma') \neq 0\). Then \( \sigma_e = \sigma'_e \) for all \( e \in E_N \smallsetminus E_{\sigma,\sigma'} \).  
\end{lemma}

\begin{proof}
    Since \(\mu_{N,(\beta,\kappa),(\infty,\kappa)}(\sigma , \sigma') \neq 0\), by definition, there is \( (\hat \sigma, \hat \sigma') \in E_N \times E_N^0 \) such that \( \sigma = \hat \sigma \) and \( \sigma' = \hat \sigma'|_{E_{\hat \sigma,\hat \sigma'}} + \hat{\sigma}|_{E_N \smallsetminus E_{\hat{\sigma}, \hat{\sigma}'}}\). Using Lemma~\ref{lemma: E in coupling}, we immediately obtain the desired conclusion.
\end{proof}

\begin{lemma}\label{lemma: Edsigmalemma}
    Let \((\sigma, \sigma') \in \Sigma_{E_N} \times \Sigma_{E_N}^0\). Then \( e\in E_{\sigma, \sigma'}\) if and only if  \(d(\sigma |_{\mathcal{C}_{\mathcal{G}(\sigma, \sigma')}(e)}) \neq 0\).
\end{lemma}

\begin{proof}
    Let  \( \mathcal{G} \coloneqq \mathcal{G}(\sigma, \sigma')\).
    Suppose first that \( e\in E_{\sigma, \sigma'}\). By the definition of \(\mathcal{C}_{\mathcal{G}}(e)\), there exists an edge \( e_1 \in \mathcal{C}_{\mathcal{G}}(e)\) such that \(e_1 \in \support \sigma \cap \partial \support d\sigma\). We infer that there exists a plaquette \(p_1 \in \hat{\partial} e_1\) such that \((d\sigma)_{p_1} \neq 0\). Since \( e_1\in \mathcal{C}_{\mathcal{G}}(e)\), all edges in \(\support \sigma \cap \partial p_1\) are contained in \(\mathcal{C}_{\mathcal{G}}(e)\), and so \((d(\sigma |_{\mathcal{C}_{\mathcal{G}}(e)}))_{p_1} = (d\sigma)_{p_1} \neq 0\).

    Conversely, suppose that there is a plaquette \(p_1\in P_N \) such that \((d(\sigma |_{\mathcal{C}_{\mathcal{G}}(e)}))_{p_1} \neq 0\). Then there exists an edge \(e_1 \in \partial p_1\) with $\sigma_{e_1} \neq 0$ and $e_1 \in \mathcal{C}_{\mathcal{G}}(e)$. It follows that $\support \sigma \cap \partial p_1 \subseteq \mathcal{C}_{\mathcal{G}}(e)$, and so $(d\sigma)_{p_1} = (d(\sigma |_{\mathcal{C}_{\mathcal{G}}(e)}))_{p_1} \neq 0$; in particular $e_1 \in \partial p_1 \subseteq \partial \support d\sigma$.
    Consequently,  \(e_1\) must belong to the set \(\support \sigma \cap \partial \support d\sigma\). In particular, \(e_1 \in E_{\sigma, \sigma'}\) and thus \(e \in \mathcal{C}_{\mathcal{G}}(e_1) \subseteq E_{\sigma, \sigma'}\).
\end{proof}

For \( e \in E_N \), we let \( \dist(e,B_N^c) \) denote the smallest Euclidean distance between the midpoint of \( e \) and the midpoints of the edges on the boundary of \( B_N \). For \( E \subseteq E_N \), we let \( \dist(E,B_N^c ) \coloneqq \min_{e \in E} \dist(e, B_N^c)\).

In the remainder of this section, we provide proofs of the following three propositions, which will all be used later in the proofs of our main results.
Heuristically, these propositions all give upper bounds for local events related to the coupling. For all of these local events, the main proof idea is to combine Proposition~\ref{proposition: edgecluster flipping ii} with an enumeration of all spin configurations which belong to the given local event.

\begin{proposition}\label{proposition: coupling and conditions 2}
    Let \( \beta,\kappa \geq 0 \) be such that~\ref{assumption: 3} holds, and let \( e \in E_N \) be such that \( \partial \hat \partial \partial \hat \partial e \subseteq E_N \). Then
    \begin{equation}\label{eq: coupling and conditions 2}
        \begin{split}  &\mu_{N,(\beta,\kappa),(\infty,\kappa)}\bigl( \bigl\{ (\sigma,\sigma') \in \Sigma_{E_N} \times \Sigma_{E_N}^0 \colon e \in E_{\sigma,\sigma'}  \text{ and } \sigma_e' \neq 0 \bigr\}\bigr) 
            \\&\qquad \leq
            C_{c,1}   \alpha_0(\kappa)^9\alpha_1(\beta)^6
            +
           C_{c,1}'  \Bigl( 18^{2} \bigl(2+ \alpha_0(\kappa)  \bigr) \alpha_0(\kappa)  \Bigr)^{\dist(e,B_N^c)} , 
        \end{split}
    \end{equation} 
    where
    \begin{equation}\label{eq: Cc1}
        C_{c,1} \coloneqq 18^{13} \cdot \frac{ \sum_{j=1}^8 \binom{8}{j} 2^{8-j}\alpha_0(\kappa)^{j-1}  + 2^8
        18^{2} \bigl(2  + \alpha_0(\kappa) \bigr) }{1-
        18^{2} \bigl(2\alpha_0(\kappa) + \alpha_0(\kappa)^2 \bigr) }
    \end{equation}
    and
    \begin{equation}\label{eq: Cc1'}
        C_{c,1}' \coloneqq \frac{1}{18^3\big(1-18^{2} (2 + \alpha_0(\kappa) ) \alpha_0(\kappa)\big)}.
    \end{equation}
\end{proposition}

\begin{remark}
    With very small modifications to the proof of Proposition~\ref{proposition: coupling and conditions 2}, one can prove a result which considers the same event as in~\eqref{eq: coupling and conditions 2} but without the condition \( \sigma_e' \neq 0. \) This shows that far from the boundary of \( B_N \), the probability that the coupled configurations \( \sigma \) and \( \sigma' \) disagree on a given edge \( e \in E_N \) is of order \( \alpha_2(\beta,\kappa)^6. \)
\end{remark}

\begin{proposition}\label{proposition: coupling and conditions 1}
    Let \( \beta,\kappa \geq 0 \) be such that~\ref{assumption: 3} holds. If \( e \in E_N \) satisfies \( \partial \hat \partial \partial \hat \partial e \subseteq E_N \), then
    \begin{equation}\label{eq: coupling and conditions 1}
        \begin{split}          &\mu_{N,(\beta,\kappa),(\infty,\kappa)}\bigl( \bigl\{ (\sigma,\sigma') \in \Sigma_{E_N} \times \Sigma_{E_N}^0 \colon 
            e \in \partial \hat \partial E_{\sigma,\sigma'} \text{ and } \sigma \in \mathcal{E}\}\bigr) 
            \\&\qquad \leq
            C_{c,2}  \alpha_1(\beta)^6 \alpha_0(\kappa)^6 \max\bigl(   \alpha_0(\kappa),  \alpha_1(\beta)^6\bigr)
            +
           C_{c,2}'  \Bigl( 18^{2} \bigl(2+ \alpha_0(\kappa)  \bigr) \alpha_0(\kappa)  \Bigr)^{\dist(e,B_N^c)} ,  
        \end{split}
    \end{equation}
    where the event $\mathcal{E}$ is defined by
    \begin{equation*}
        \mathcal{E} \coloneqq \bigl\{ \sigma \in \Sigma_{E_N} \colon \exists g \in G\smallsetminus \{ 0 \} \text{ such that }\sigma_e - (d\sigma)_p = g \text{ for all } p \in \hat \partial e \bigr\},
    \end{equation*}  
    \begin{equation}\label{eq: Cc2}
        C_{c,2} \coloneqq   
        \frac{2\cdot 18^{11} (18^2+1) \bigl( 1 + 2\alpha_0(\kappa) + \alpha_0(\kappa)^2 \bigr) \bigl(2 + \alpha_0(\kappa) \bigr)^{7}}{1-18^2(2\alpha_0(\kappa) + \alpha_0(\kappa)^2)},
    \end{equation}
    and
    \begin{equation}\label{eq: Cc2'}
        C_{c,2}' \coloneqq \frac{ 1 + 2\alpha_0(\kappa) + \alpha_0(\kappa)^2 }{18^3(1-18^{2} \bigl(2 + \alpha_0(\kappa) \bigr) \alpha_0(\kappa))} \cdot \biggl( 1 + \frac{1}{ \bigl(2 + \alpha_0(\kappa) \bigr) \alpha_0(\kappa)}\biggr).
    \end{equation}
\end{proposition}

\begin{remark}
    We note that if \( N \to \infty \) while \( e \), \( \beta \), and \( \kappa \) remain fixed, then the terms involving $\dist(e,B_N^c)$ on the right-hand sides of~\eqref{eq: coupling and conditions 2} and~\eqref{eq: coupling and conditions 1} tend to zero.
\end{remark}

\begin{proposition} \label{proposition: minimal Ising}
    Let \( \beta,\kappa \geq 0 \) be such that~\ref{assumption: 3} holds, and let \( e \in E_N^+ \) be such that \( \partial \hat \partial \partial \hat \partial e \in E_N \). Then
    \begin{equation}\label{eq: minimal Ising}
        \mu_{N,\infty,\kappa}   \bigl( \{ \sigma' \in  \Sigma_{E_N}^0 \colon  \sigma_e'\neq0 \}\bigr) 
        \leq 
        C_{I} \alpha_0(\kappa)^8,
    \end{equation}
    where
    \begin{equation}\label{eq: Ising constant}
        C_{I} \coloneqq \frac{18^{13}  }{1-18^{2}  \alpha_0(\kappa)},
    \end{equation}
\end{proposition}

Before we provide proofs of the three propositions above, we introduce some additional notation and prove two useful lemmas.

If \( \mathcal{G} \) is a finite graph, we say that a walk \( W = (v_1,\ldots, v_m) \) in \( \mathcal{G} \) is a \emph{spanning walk} of \( \mathcal{G} \) if it passes through each vertex of \( G \) at least once. We say that a walk \( W = (v_1,\ldots, v_m) \) has \emph{length} $m-1$.

\begin{lemma}\label{lemma: spanning paths}
    Let \( \mathcal{G} \) be a connected graph on \( m \geq 2 \) vertices. Then there is a walk \( W \) in \( \mathcal{G} \) of length \( 2m-3 \) which passes through each vertex at least once.
\end{lemma} 

\begin{proof} 
    Since \( \mathcal{G} \) is connected, \( \mathcal{G} \) has a spanning tree. Any such spanning tree must contain exactly \( m-1 \) edges. Using this tree, one can find a spanning walk which uses each edge in the tree exactly twice. This walk must have length \( 2(m-1) = 2m-2 \). Removing the last edge in this walk, we obtain a walk with the desired properties. This concludes the proof.
\end{proof}

Fix some \( g \in G\smallsetminus \{ 0 \} \) and define \( \sigma \in \Sigma_{E_N} \) by letting \( \sigma_e = g \) for all \( e \in E_N^+ \). Let \( \bar {\mathcal{G}} \coloneqq \mathcal{G}(\sigma,0) \) and note that \( \bar {\mathcal{G}} \) does not depend on the choice of \( g \). 
Note also that if \( \sigma^{(0)},\sigma^{(1)} \in \Sigma_{E_N} \), then \( \mathcal{G}(\sigma^{(0)}, \sigma^{(1)}) \) is a subgraph of \( \bar {\mathcal{G}} \).

\begin{lemma}\label{lemma: number of paths}
    Let \( e \in E_N \) and let \( m \geq 1 \). Then there are at most \( 18^{m}\) walks of length \( m\) in \( \bar {\mathcal{G}} \) containing only edges in \( E_N^+ \), and starting at \( e \). 
\end{lemma}

\begin{proof}
    For each \( e \in E_N \), we have \( |(\partial \hat \partial e)\smallsetminus \{ e \}|= 6 \cdot 3 = 18\). Moreover, 
    \begin{equation*}
        \bigl((\partial \hat \partial e)\smallsetminus \{ e \} \bigr)
        \cap \Bigl( -\bigl((\partial \hat \partial e)\smallsetminus \{ e \} \bigr) \Bigr) = \emptyset,
    \end{equation*}
   and hence
    \begin{equation*}
        \Bigl| \bigl((\partial \hat \partial e)\smallsetminus \{ e \} \bigr)^+
        \cup \Bigl( -\bigl((\partial \hat \partial e)\smallsetminus \{ e \} \bigr) \Bigr)^+ \Bigr| = |(\partial \hat \partial e)\smallsetminus \{ e \}|= 18,
    \end{equation*}
    from which the desired conclusion follows.
\end{proof}

\begin{proof}[Proof of Proposition~\ref{proposition: coupling and conditions 2}]
    By using first the definition of \( \mu_{N,(\beta,\kappa),(\infty,\kappa)} \) and then Lemma~\ref{lemma: E in coupling}, we see that
    \begin{equation*}  
        \begin{split}
            &\mu_{N,(\beta,\kappa),(\infty,\kappa)}   \bigl( \big\{ ( \sigma, \sigma') \in \Sigma_{E_N} \times \Sigma_{E_N}^0 \colon  e \in E_{ \sigma,  \sigma'} \text{ and } \sigma_e' \neq 0\big\}\bigr)
            \\&\qquad= 
            \mu_{N,\beta,\kappa} \times \mu_{N,\infty,\kappa}   \bigl( \big\{ (\hat \sigma,\hat \sigma') \in \Sigma_{E_N} \times \Sigma_{E_N}^0 \colon  e \in E_{\hat \sigma, \hat \sigma'|_{E_{\hat \sigma,\hat \sigma'}}+ \hat \sigma|_{E_N\smallsetminus E_{\hat \sigma,\hat \sigma'}}}  \text{ and } \hat \sigma_e' \neq 0\big\}\bigr)\nonumber
            \\&\qquad= 
            \mu_{N,\beta,\kappa} \times \mu_{N,\infty,\kappa}   \bigl( \big\{ (\hat \sigma,\hat \sigma') \in \Sigma_{E_N} \times \Sigma_{E_N}^0 \colon  e \in E_{\hat \sigma, \hat \sigma'} 
             \text{ and } \hat \sigma_e' \neq 0\big\}\bigr). 
        \end{split}
    \end{equation*}
    By Lemma~\ref{lemma: Edsigmalemma}, for any \( (\hat \sigma,\hat \sigma') \in \Sigma_{E_N} \times \Sigma_{E_N}^0 \), we have
    \begin{equation*}
        e\in E_{\hat \sigma,\hat \sigma'}
        \Leftrightarrow 
        d(\hat{\sigma} |_{\mathcal{C}_{\mathcal{G}(\hat \sigma , \hat \sigma')}(e)}) \neq 0.
    \end{equation*}
    Consequently, the desired conclusion will follow if we can show that
    \begin{equation}\label{eq: hat goal 2}
        \begin{split}
            &\mu_{N,\beta,\kappa} \times \mu_{N,\infty,\kappa} \bigl( \big\{ (\hat \sigma,\hat \sigma') \in \Sigma_{E_N} \times \Sigma_{E_N}^0 \colon d(\hat{\sigma} |_{\mathcal{C}_{\mathcal{G}(\hat \sigma , \hat \sigma')}(e)}) \neq 0 \text{ and } \hat \sigma_e' \neq 0\big\}\bigr) 
            \\&\qquad \leq 
            C_{c,1} \alpha_0(\kappa)^9\alpha_1(\beta)^6+
            C_{c,1}' 
            \Bigl( 18^{2} \bigl(2+ \alpha_0(\kappa)  \bigr) \alpha_0(\kappa)  \Bigr)^{\dist(e,B_N^c)}.
        \end{split}
    \end{equation}
    
    If \( (\hat \sigma,\hat \sigma') \in \Sigma_{E_N} \times \Sigma_{E_N}^0 \) and \( d(\hat \sigma|_{\mathcal{C}_{\mathcal{G}(\hat \sigma, \hat \sigma')}(e)}) \neq 0 \), the set \(   \mathcal{C}_{\mathcal{G}(\hat \sigma,\hat \sigma')}(e) \) must be non-empty. Since \( \mathcal{C}_{\mathcal{G}(\hat \sigma,\hat \sigma')}(e) \) is symmetric, the cardinality \( |\mathcal{C}_{\mathcal{G}(\hat \sigma,\hat \sigma')}(e)| \geq 1 \) must be even. From these observations, we obtain
    \begin{align*}
        \begin{split}
            &\mu_{N,\beta,\kappa} \times \mu_{N,\infty,\kappa} \bigl( \big\{ (\hat \sigma,\hat \sigma') \in \Sigma_{E_N} \times \Sigma_{E_N}^0 \colon  d(\hat \sigma|_{\mathcal{C}_{\mathcal{G}(\hat \sigma, \hat \sigma')}(e)}) \neq 0 \text{ and } \hat \sigma_e' \neq 0\big\} \bigr) 
            \\& = 
            \sum_{m=1}^\infty  
            \mu_{N,\beta,\kappa} \times \mu_{N,\infty,\kappa} \bigl( \big\{ (\hat \sigma,\hat \sigma') \in \Sigma_{E_N} \times \Sigma_{E_N}^0 \colon d(\hat \sigma|_{\mathcal{C}_{\mathcal{G}(\hat \sigma, \hat \sigma')}(e)})\neq 0,\, \hat \sigma_e' \neq 0 ,\,  |\mathcal{C}_{\mathcal{G}(\hat \sigma, \hat \sigma')}(e)| = 2m \big\} \bigr).  
        \end{split}
    \end{align*}

    Next, for \( k \geq 1 \), let \( T_{e,k} \) be the set of all walks in \( \bar {\mathcal{G}}  \) which start at \( e \), have length \( 2k-3 \), contain only edges in \( E_N^+ \), and visit exactly \( k \) vertices in \( \bar {\mathcal{G}}  \).
    Given \( (\hat \sigma,\hat \sigma')   \in \Sigma_{E_N} \times \Sigma_{E_N}^0  \), and \( \mathcal{G} = \mathcal{G}(\hat \sigma, \hat \sigma') \), 
    let \( \mathcal{G}|_{(\mathcal{C}_{\mathcal{G}}(e))^+ } \) denote the graph \( \mathcal{G} \) restricted to the set \( (\mathcal{C}_{\mathcal{G}}(e))^+ \), and let \( \mathcal{W}_{\hat \sigma,\hat \sigma',e} \) be the set of all spanning walks of \( \mathcal{G}|_{( \mathcal{C}_{\mathcal{G}}(e))^+ } \) which have length \( 2|(\mathcal{C}_{\mathcal{G}}(e))^+|-3 \) and start at \( e \).
    By Lemma~\ref{lemma: spanning paths}, the set \( \mathcal{W}_{\hat \sigma,\hat \sigma',e} \) is non-empty. Moreover, since any walk on \( \mathcal{G}\) is a walk on \( \bar {\mathcal{G}}  \), we have \( \mathcal{W}_{\hat \sigma,\hat \sigma',e} \subseteq T_{e,|(\mathcal{C}_{\mathcal{G}}(e))^+|} \).
    Hence
    \begin{align}
        &\sum_{m=1}^\infty  
        \mu_{N,\beta,\kappa} \times \mu_{N,\infty,\kappa} \bigl( \big\{ (\hat \sigma,\hat \sigma') \in \Sigma_{E_N} \times \Sigma_{E_N}^0 \colon d(\hat \sigma|_{\mathcal{C}_{\mathcal{G}(\hat \sigma, \hat \sigma')}(e)})\neq 0,\, \hat \sigma_e' \neq 0,\,  |\mathcal{C}_{\mathcal{G}(\hat \sigma, \hat \sigma')}(e)| = 2m \big\} \bigr) \nonumber
        \\&\qquad \leq 
        \sum_{m=1}^\infty \sum_{W \in T_{e,m}} 
        \mu_{N,\beta,\kappa} \times \mu_{N,\infty,\kappa} \bigl( \big\{ (\hat \sigma,\hat \sigma') \in \Sigma_{E_N} \times \Sigma_{E_N}^0 \colon d(\hat \sigma|_{\mathcal{C}_{\mathcal{G}(\hat \sigma, \hat \sigma')}(e)})\neq 0,\, \hat \sigma_e' \neq 0, \label{eq: last line 2p} \,\\[-1ex]&\hspace{19em} |\mathcal{C}_{\mathcal{G}(\hat \sigma, \hat \sigma')}(e)| = 2m \text{ and }  W \in \mathcal{W}_{\hat \sigma,\hat \sigma',e} \big\} \bigr). \nonumber
    \end{align}
    Now assume that \( W \in T_{e,m} \) for some \( m \geq 1 \). Let \( \support W \) be the set of edges in \( E_N^+ \) traversed by \( W \). Then \( \bigl( \mathcal{C}_{\mathcal{G}(\hat \sigma,\hat \sigma')}(e)\bigr)^+ = \support W\) whenever $|\mathcal{C}_{\mathcal{G}(\hat \sigma, \hat \sigma')}(e)| = 2m$ and \( W \in \mathcal{W}_{\hat \sigma,\hat\sigma',e} \). Consequently, the expression on the right-hand side of~\eqref{eq: last line 2p} is bounded from above by
    \begin{equation}\label{eq: muNbetainftysumsum} 
        \begin{split}
            &\sum_{m=1}^\infty \sum_{W \in T_{e,m}} 
            \mu_{N,\beta,\kappa} \times \mu_{N,\infty,\kappa} \bigl( \{ (\hat \sigma,\hat \sigma') \in \Sigma_{E_N} \times \Sigma_{E_N}^0 \colon d(\hat \sigma|_{\mathcal{C}_{\mathcal{G}(\hat \sigma, \hat \sigma')}(e)})\neq 0, \, \hat\sigma_e' \neq 0, \\[-1ex]&\hspace{24em}  \bigl( \mathcal{C}_{\mathcal{G}(\hat \sigma, \hat \sigma')}(e)\bigr)^+ = \support W \}\bigr).
        \end{split}
    \end{equation}
    Given \( (\hat \sigma,\hat \sigma') \in \Sigma_{E_N} \times \Sigma_{E_N}^0 \), if we let \( \sigma \coloneqq  \hat \sigma|_{\mathcal{C}_{\mathcal{G}(\hat \sigma,\hat \sigma')}(e)} \) and \( \sigma' \coloneqq  \hat \sigma'|_{\mathcal{C}_{\mathcal{G}(\hat\sigma,\hat \sigma')}(e)} \), then the following statements hold.
    \begin{enumerate}
        \item By Lemma~\ref{lemma: cluster is subconfig}, \( \sigma' \leq \hat \sigma' \). Since \( \hat \sigma' \in \Sigma_{E_N}^0 \), it follows that \( \sigma' \in \Sigma_{E_N}^0 \). 
        \item If \( d(\hat \sigma|_{\mathcal{C}_{\mathcal{G}(\hat \sigma, \hat \sigma')}(e)})\neq 0 \), then \( d\sigma \neq 0. \)
        \item If \( d(\hat \sigma|_{\mathcal{C}_{\mathcal{G}(\hat \sigma, \hat \sigma')}(e)})\neq 0 \), then \( e \in \support \hat \sigma \cup \support \hat \sigma'\) and thus
        \( (\mathcal{C}_{\mathcal{G}(\hat \sigma,\hat \sigma')}(e))^+ = (\support \sigma \cup \support \sigma')^+\). Hence, if $d(\hat \sigma|_{\mathcal{C}_{\mathcal{G}(\hat \sigma, \hat \sigma')}(e)})\neq 0$, then \( (\mathcal{C}_{\mathcal{G}(\hat \sigma,\hat \sigma')}(e))^+ = \support W \) if and only if \(   (\support \sigma \cup \support \sigma')^+ = \support W. \)
        \item 
        \( \hat \sigma_e' = \sigma_e'. \)
    \end{enumerate}
    As a consequence, the expression in~\eqref{eq: muNbetainftysumsum} is bounded from above by
    \begin{equation*} 
        \begin{split}
            &\sum_{m=1}^\infty \sum_{W \in T_{e,m}} 
            \sum_{\substack{\sigma \in \Sigma_{E_N},\,  \sigma' \in \Sigma_{E_N}^0 \colon \\
            \substack{(\support  \sigma \cup \support  \sigma')^+ = \support W,\\ d \sigma \neq 0,\, \sigma'_e \neq 0} }} 
            \mu_{N,\beta,\kappa} \times \mu_{N,\infty,\kappa} \bigl( \{ (\hat \sigma,\hat \sigma') \in \Sigma_{E_N} \times \Sigma_{E_N}^0 \colon
            \\[-7ex]&\hspace{22em}
            \hat \sigma|_{\mathcal{C}_{\mathcal{G}(\hat \sigma, \hat \sigma')}(e)} =  \sigma
            , \text{ and }
            \hat \sigma'|_{\mathcal{C}_{\mathcal{G}(\hat \sigma, \hat \sigma'}(e)} =  \sigma'\} \bigr) \\ &\mbox{}
        \end{split}
    \end{equation*}
    For any \( m\), \( W \), \(  \sigma \), and \(  \sigma' \) as in the sum above, by Lemma~\ref{lemma: cluster is subconfig}, we have 
    \begin{equation*}
        \begin{split}
            &\mu_{N,\beta,\kappa} \times \mu_{N,\infty,\kappa} \bigl( \{ (\hat \sigma,\hat \sigma') \in \Sigma_{E_N} \times \Sigma_{E_N}^0 \colon
            \hat \sigma|_{\mathcal{C}_{\mathcal{G}(\hat \sigma, \hat \sigma')}(e)} =  \sigma \text{ and } 
            \hat \sigma'|_{\mathcal{C}_{\mathcal{G}(\hat \sigma, \hat \sigma'}(e)} =  \sigma'\} \bigr) 
            \\&\qquad\leq \mu_{N,\beta,\kappa} \times \mu_{N,\infty,\kappa} \bigl( \{ (\hat \sigma,\hat \sigma') \in \Sigma_{E_N} \times \Sigma_{E_N}^0 \colon
            \sigma \leq \hat \sigma  \text{ and } 
            \sigma' \leq \hat \sigma' \} \bigr) 
            \\&\qquad= \mu_{N,\beta,\kappa} \bigl( \{ \hat \sigma  \in \Sigma_{E_N}  \colon
            \sigma \leq \hat \sigma   \} \bigr) 
            \mu_{N,\infty,\kappa} \bigl( \{ \hat \sigma' \in  \Sigma_{E_N}^0 \colon
            \sigma' \leq \hat \sigma' \} \bigr) \leq \varphi_{\beta,\kappa}( \sigma) \varphi_{\infty,\kappa} (\sigma'),
        \end{split}
    \end{equation*}
    where the last inequality follows by applying Proposition~\ref{proposition: edgecluster flipping ii} twice.
    Taken together, the above equations show that
    \begin{equation}\label{eq: first summary B}  
        \begin{split}
        &\mu_{N,\beta,\kappa} \times \mu_{N,\infty,\kappa} \bigl( \{ (\hat \sigma,\hat \sigma') \in \Sigma_{E_N} \times \Sigma_{E_N}^0 \colon  e \in E_{\hat \sigma, \hat \sigma'}  ,\, \hat{\sigma}'_e \neq 0 \} \bigr) \leq
        \sum_{m=1}^\infty \sum_{W \in T_{e,m}} 
          J_{\beta, \kappa}(W),
        \end{split}
    \end{equation}
    where
    \begin{equation*}
        J_{\beta, \kappa}(W) \coloneqq 
        \sum_{\substack{ \sigma \in \Sigma_{E_N},\,  \sigma' \in \Sigma_{E_N}^0 \colon \\
        \substack{(\support  \sigma \cup \support  \sigma')^+ = \support W,\\ d \sigma \neq 0,\, 
        \sigma'_e \neq 0 
        }}} \varphi_{\beta,\kappa}  ( \sigma )\,  \varphi_{\infty,\kappa} (  \sigma'  ).
    \end{equation*} 
    
    Fix some \( m \geq 1 \) and \( W \in T_{e,m} \). Then 
    \begin{equation*}
        \begin{split}
            &J_{\beta, \kappa}(W) = \sum_{\substack{ \sigma \in \Sigma_{E_N},\,  \sigma' \in \Sigma_{E_N}^0 \colon \\
            \substack{(\support  \sigma \cup \support  \sigma')^+ = \support W,\\ d \sigma \neq 0,\, 
            \sigma'_e \neq 0 
             }}}
            \prod_{e' \in E_N} \varphi_\kappa(\sigma_{e'}) \prod_{p \in P_N}\varphi_\beta \bigl((d\sigma)_p \bigr) \prod_{e'' \in E_N} \varphi_\kappa(\sigma_{e''}')
            \\&\qquad =
            \sum_{\substack{ \sigma \in \Sigma_{E_N},\,  \sigma' \in \Sigma_{E_N}^0 \colon \\
            \substack{(\support  \sigma \cup \support  \sigma')^+ = \support W,\\ d \sigma \neq 0,\, 
            \sigma'_e \neq 0 
            }}}
            \prod_{e' \in E_N^+} (\varphi_\kappa(\sigma_{e'})^2 \varphi_\kappa(\sigma_{e'}')^2) \prod_{p \in P_N^+}\varphi_\beta \bigl((d\sigma)_p \bigr)^2. 
        \end{split}
    \end{equation*} 
    If \( \sigma' \in \Sigma_{E_N}^0 \) satisfies \( \sigma'_e \neq 0 \), then, since \( \partial \hat \partial \partial \hat \partial e \subseteq E_N \) by assumption, we have \( |(\support \sigma')^+| \geq 8 \) by Lemma~\ref{lemma: small 1forms}. Hence \( J_{\beta,\kappa}(W) =0\) if \( m \leq 7 \) and \( W \in T_{e,m} \).
    
    Now recall that for any \( r \geq 0 \) and \( g \in G \), we have  \( \varphi_r(0) = 1\) and \( \varphi_r(g) = e^{r \Re(\rho(g) - 1)} \in (0,1]\). If \( g \neq 0\), then \( \varphi_r(g) < 1\) and hence \(\varphi_\beta(g)^2 \leq \alpha_1(\beta) < 1\).
    If \( \sigma \) is as above, then \( d\sigma \neq 0 \), and we must be in one of the following two cases.
    \begin{enumerate}
        \item If \( |(\support d\sigma)^+| \geq 6\), then 
        \begin{equation*}
            \prod_{p \in P_N^+}\varphi_\beta \bigl((d\sigma)_p \bigr)^2 \leq \alpha_1(\beta)^6.
        \end{equation*} 
        \item If \( |(\support d\sigma)^+| < 6\), then, by Lemma~\ref{lemma: minimal vortex I}, \( \sigma \) must support a vortex with support at the boundary of \( B_N \), and hence we must have \( |\support W| \geq \dist(e,B_N^c).\) At the same time, we trivially also have \( \prod_{p \in P_N^+}\varphi_\beta \bigl((d\sigma)_p \bigr)^2 \leq 1. \)
    \end{enumerate}
    Consequently, if \( m \geq 8 \) and \( W \in T_{e,m} \), then
    \begin{equation*}
        J_{\beta, \kappa}(W)  \leq \bigl(\alpha_1(\beta)^6+\mathbb{1}_{|\support W| \geq \dist(e,B_N^c)}\bigr) \sum_{\substack{ \sigma \in \Sigma_{E_N},\,  \sigma' \in \Sigma_{E_N}^0 \colon \\
        \substack{(\support  \sigma \cup \support  \sigma')^+ = \support W, \\ d\sigma \neq 0 ,\, \sigma_e' \neq 0 }}}
        \prod_{e' \in E_N^+} \bigl(\varphi_\kappa(\sigma_{e'})^2 \varphi_\kappa(\sigma_{e'}')^2 \bigr).
    \end{equation*}
    By replacing the condition \( d\sigma \neq 0\)  with the condition \( \sigma \neq 0 \)
    we make the sum larger. Hence
    \begin{align*}
        J_{\beta, \kappa}(W) \leq \bigl(\alpha_1(\beta)^6+\mathbb{1}_{|\support W| \geq \dist(e,B_N^c)}\bigr) \sum_{\substack{ \sigma \in \Sigma_{E_N},\,  \sigma' \in \Sigma_{E_N}^0 \colon \\
        \substack{(\support  \sigma \cup \support  \sigma')^+ = \support W, \\  \sigma,\sigma_e' \neq 0 }}}
        \prod_{e' \in E_N^+} \bigl(\varphi_\kappa(\sigma_{e'})^2 \varphi_\kappa(\sigma_{e'}')^2 \bigr).
    \end{align*}
    If \( \sigma \in \Sigma_{E_N} \) and \( e' \notin \support \sigma \), then \( \varphi_\kappa(\sigma_{e'}) = \varphi_\kappa(0)=1 \). Also, if \( \sigma, \sigma' \in \Sigma_{E_N} \) and \( e' \in (\support \sigma \cup \support \sigma')^+ = \support W \), then either \( \sigma_{e'} \neq 0 \) and \( \sigma'_{e'} = 0 \), \( \sigma_{e'}=0 \) and \(\sigma_{e'}' \neq 0 \), or \( \sigma_{e'},\sigma_{e'}' \neq 0 \). At the same time, if \( \sigma \in \Sigma_{E_N} \) is such that \( \sigma \neq 0 \), then \( |(\support \sigma)^+| \geq 1 \), and, by Lemma~\ref{lemma: small 1forms}, if \( \sigma' \in \Sigma_{E_N}^0 \) is such that \( \sigma_e' \neq 0 \), then \( |(\support \sigma')^+| \geq 8 \).   %
    Combining these observations, we obtain
    \begin{align}\nonumber
        &\sum_{\substack{ \sigma \in \Sigma_{E_N},\,  \sigma' \in \Sigma_{E_N}^0 \colon \\
        \substack{(\support  \sigma \cup \support  \sigma')^+ = \support W, \\  \sigma,\sigma_e' \neq 0 }}}
        \prod_{e' \in E_N^+} \bigl(\varphi_\kappa(\sigma_{e'})^2 \varphi_\kappa(\sigma_{e'}')^2 \bigr) 
        \\&\qquad\leq  
        \prod_{e' \in \support W}
        \biggl\{ 
        \varphi_\kappa(0)^2
        \bigg(\sum_{\sigma_{e'}' \in G \smallsetminus \{0\}} \varphi_\kappa(\sigma_{e'}')^2\bigg)
        + 
        \biggl(\sum_{\sigma_{e'} \in G \smallsetminus \{0\}}
        \varphi_\kappa(\sigma_{e'})^2 \biggr)
        \varphi_\kappa(0)^2
        \\\nonumber
        &\qquad\qquad\qquad +\bigg(\sum_{\sigma_{e'} \in G \smallsetminus \{0\}} 
        \varphi_\kappa(\sigma_{e'})^2 \bigg)
        \bigg(\sum_{\sigma_{e'}' \in G \smallsetminus \{0\}} \varphi_\kappa(\sigma_{e'}')^2\bigg)
        \biggr\}\bigg|_{\varphi_\kappa(\cdot)^{2p}, p \geq 9},
            \\ \nonumber
        & \qquad = \prod_{e' \in \support W}
        \big(\alpha_0(\kappa) + \alpha_0(\kappa)
        + \alpha_0(\kappa)^2\big)\big|_{\alpha_0(\kappa)^{p}, p \geq 9}
            \\ \nonumber
        & \qquad =   \bigl(2 \alpha_0(\kappa)
        + \alpha_0(\kappa)^2\bigr)^{|\support W|}\big|_{\alpha_0(\kappa)^{p}, p \geq 9},
    \end{align} 
    where the notation $|_{\alpha_0(\kappa)^{p}, p \geq 9}$ indicates that only terms involving  $\alpha_0(\kappa)$ raised to at least the $9$th power should be included, and $|_{\varphi_\kappa(\cdot)^{2p}, p \geq 9}$ indicates that only terms containing at least 9 factors of the type \( \varphi_\kappa(g)^2 \) for \( g \neq 0 \) should be included. 
    We thus have
    \begin{align}\label{eq: innermost sum B}
        &J_{\beta, \kappa}(W)
        \leq \bigl(\alpha_1(\beta)^6+\mathbb{1}_{|\support W| \geq \dist(e,B_N^c)}\bigr) \Bigl( \bigl(2 \alpha_0(\kappa)
        + \alpha_0(\kappa)^2\bigr)^{|\support W|} \Bigr)\big|_{\alpha_0(\kappa)^{p}, p \geq 9}.
    \end{align}
    Since \( W \in T_{e,m} \), we have $|\support W| =m$. Also, by Lemma~\ref{lemma: number of paths}, \( |T_{e,m}| \leq 18^{2m-3} \) for any \( m \geq 1 \). 
    Combining~\eqref{eq: first summary B} and~\eqref{eq: innermost sum B}, we thus find that  
    \begin{align*}
        &\mu_{N,(\beta,\kappa),(\infty,\kappa)}\bigl( \bigl\{ (\sigma,\sigma') \in \Sigma_{E_N} \times \Sigma_{E_N}^0 \colon e \in E_{\sigma,\sigma'}  \text{ and } \sigma_e' \neq 0 \bigr\}\bigr)   
        \\&\qquad\leq  
     \sum_{m=8}^\infty  
        18^{2m-3} \bigl(2\alpha_0(\kappa) + \alpha_0(\kappa)^2 \bigr)^m  \alpha_1(\beta)^6 |_{\alpha_0(\kappa)^{p}, p \geq 9}
        \\&\qquad\qquad +
        \sum_{m=\dist(e,B_N^c)}^\infty  
        18^{2m-3} \bigl(2\alpha_0(\kappa) + \alpha_0(\kappa)^2 \bigr)^m 
        \\&\qquad\leq  
        \sum_{m=8}^\infty  
        18^{2m-3} \bigl(2\alpha_0(\kappa) + \alpha_0(\kappa)^2 \bigr)^m  \alpha_1(\beta)^6
        -
        18^{13}  2^8\alpha_0(\kappa)^8   \alpha_1(\beta)^6
        \\&\qquad\qquad +
        \sum_{m=\dist(e,B_N^c)}^\infty  
        18^{2m-3} \bigl(2\alpha_0(\kappa) + \alpha_0(\kappa)^2 \bigr)^m .
    \end{align*}
     Computing the above geometric sums, we obtain~\eqref{eq: coupling and conditions 2}.
\end{proof}

\begin{proof}[Proof of Proposition~\ref{proposition: coupling and conditions 1}]
    Without loss of generality, we can assume that \( e \in E_N^+ \).
    Define
    \begin{equation*}
        {\tilde { \mathcal{C}}}_{\mathcal{G}(\sigma,\sigma')}( e) \coloneqq \bigcup_{e' \in \partial \hat \partial e \cap (\support \sigma \cup \support \sigma')} \mathcal{C}_{\mathcal{G}(\sigma,\sigma')}( e').
    \end{equation*}
    Using the same notation as in the proof of Proposition~\ref{proposition: coupling and conditions 2}, we find the following analog of \eqref{eq: last line 2p}--\eqref{eq: muNbetainftysumsum}:   
   \begin{equation}\label{eq: muNbetainftysumsum2} 
        \begin{split}
            &   \mu_{N,(\beta,\kappa),(\infty,\kappa)}\bigl( \bigl\{ (\sigma,\sigma') \in \Sigma_{E_N} \times \Sigma_{E_N}^0 \colon e \in \partial \hat \partial E_{\sigma,\sigma'}  \text{ and } \sigma \in \mathcal{E} \bigr\}\bigr) 
            \\&\qquad \leq
            \sum_{m=1}^\infty \sum_{W \in T_{e,m}} 
            \mu_{N,\beta,\kappa} \times \mu_{N,\infty,\kappa} \Bigl( \bigl\{ (\hat \sigma,\hat \sigma') \in \Sigma_{E_N} \times \Sigma_{E_N}^0 \colon d(\hat \sigma|_{{\tilde { \mathcal{C}}}_{\mathcal{G}(\hat \sigma, \hat \sigma')}( e)})\neq 0,  \, \\[-1ex]&\hspace{16em}  \bigl( {\tilde { \mathcal{C}}}_{\mathcal{G}(\hat \sigma, \hat \sigma')}(e)\bigr)^+ \cup \{ e \} = \support W,  \text{ and } \hat \sigma \in \mathcal{E}\bigr\} \Bigr).
        \end{split}
    \end{equation} 
    Given \( (\hat \sigma,\hat \sigma') \in \Sigma_{E_N} \times \Sigma_{E_N}^0 \), let \( \sigma \coloneqq  \hat \sigma|_{{\tilde { \mathcal{C}}}_{\mathcal{G}(\hat \sigma,\hat \sigma')}(e)} \) and \( \sigma' \coloneqq  \hat \sigma'|_{{\tilde { \mathcal{C}}}_{\mathcal{G}(\hat\sigma,\hat \sigma')}(e)} \). If \( d(\hat \sigma|_{{\tilde { \mathcal{C}}}_{\mathcal{G}(\hat \sigma, \hat \sigma')}( e)})\neq 0 \) and \( \hat \sigma \in \mathcal{E} \), then  \(  \sigma \in \mathcal{E} \).  This leads to the following analog of~\eqref{eq: first summary B}:
    \begin{equation*} 
        \begin{split}
        &\mu_{N,\beta,\kappa} \times \mu_{N,\infty,\kappa} \bigl( \{ (\hat \sigma,\hat \sigma') \in \Sigma_{E_N} \times \Sigma_{E_N}^0 \colon  e \in \partial \hat \partial E_{\hat \sigma, \hat \sigma'},\,  \sigma \in \mathcal{E} \} \bigr) \leq
        \sum_{m=1}^\infty \sum_{W \in T_{e,m}} 
          J_{\beta, \kappa}(W),
        \end{split}
    \end{equation*}
    where
    \begin{equation*}
        J_{\beta, \kappa}(W) \coloneqq 
        \sum_{\substack{ \sigma \in \Sigma_{E_N},\,  \sigma' \in \Sigma_{E_N} \colon \\
        \substack{(\support  \sigma \cup \support  \sigma')^+ \cup \{ e \} = \support W,\\ d \sigma \neq 0,\, \sigma \in \mathcal{E}}}} \varphi_{\beta,\kappa}  ( \sigma )\,  \varphi_{\infty,\kappa} (  \sigma'  ).
    \end{equation*} 
    
    If \( \sigma \in \mathcal{E} \), then \( |(\support \sigma)^+\smallsetminus \{ e \}| \geq 6 \), and hence \( J_{\beta,\kappa}(W) =0\) whenever \( m <7 \) and \( W \in T_{e,m} \). Furthermore, if \( m \geq 7 \) and \( W \in T_{e,m} \), then 
    \begin{align*}
            &J_{\beta, \kappa}(W) =
            \sum_{\substack{ \sigma \in \Sigma_{E_N},\,  \sigma' \in \Sigma_{E_N} \colon \\
            \substack{(\support  \sigma \cup \support  \sigma')^+ \cup \{ e \} = \support W, \\ d \sigma \neq 0,\, \sigma \in \mathcal{E}}}}
            \prod_{e' \in E_N^+} (\varphi_\kappa(\sigma_{e'})^2 \varphi_\kappa(\sigma_{e'}')^2) \prod_{p \in P_N^+}\varphi_\beta \bigl((d\sigma)_p \bigr)^2.
    \end{align*}
    
    If \( \sigma \in \mathcal{E} \) and \( |(\support \sigma)^+| = 6 \), then \( e \notin \support \sigma \), and, since \( \partial \hat \partial \partial \hat \partial e \subseteq E_N \), we have $|(\support d\sigma)^+| \geq 12$ by Lemma~\ref{lemma: other minimal configuration}. Consequently, in this case, we have
    \begin{equation*}
        \prod_{p \in P_N^+}\varphi_\beta \bigl((d\sigma)_p \bigr)^2 \leq \alpha_1(\beta)^{12},
    \end{equation*}
    On the other hand, if \( |(\support \sigma)^+| \geq 7\), then since \( d\sigma \neq 0 \), it follows from~Lemma~\ref{lemma: minimal vortex I} that we either have \( |(\support d\sigma)^+| \geq 6 \), or \( |\support W| \geq \dist(e,B_N^c).\)
    Hence
     \begin{equation*}
        \prod_{p \in P_N^+}\varphi_\beta \bigl((d\sigma)_p \bigr)^2 \leq \alpha_1(\beta)^6 + \mathbb{1}_{|\support W| \geq \dist(e,B_N^c)}.
    \end{equation*}
    Consequently,
    \begin{equation*}
        \begin{split}
            J_{\beta, \kappa}(W)  
            &\leq \alpha_1(\beta)^{12} \sum_{\substack{ \sigma \in \Sigma_{E_N},\,  \sigma' \in \Sigma_{E_N} \colon \\
            \substack{(\support  \sigma \cup \support  \sigma')^+ \cup \{ e \} = \support W, \\    |(\support \sigma)^+|=6   }}}
            \prod_{e' \in E_N^+} \bigl(\varphi_\kappa(\sigma_{e'})^2 \varphi_\kappa(\sigma_{e'}')^2 \bigr)
            \\&\qquad+
            \bigl( \alpha_1(\beta)^6 + \mathbb{1}_{|\support W| \geq \dist(e,B_N^c)}\bigr) \sum_{\substack{ \sigma \in \Sigma_{E_N},\,  \sigma' \in \Sigma_{E_N} \colon \\ \substack{(\support  \sigma \cup \support  \sigma')^+ \cup \{ e \}= \support W, \\    |(\support \sigma)^+| \geq 7  }}}
            \prod_{e' \in E_N^+} \bigl(\varphi_\kappa(\sigma_{e'})^2 \varphi_\kappa(\sigma_{e'}')^2 \bigr)
        \end{split}
    \end{equation*}
    Now note that we can have either \( e \in \support \sigma \cup \support \sigma' \) or \( e \notin \support \sigma \cup \support \sigma'. \) Using this observation,
    we obtain the following analog of~\eqref{eq: innermost sum B}:
    \begin{align*}
        &J_{\beta, \kappa}(W)
        \\&\qquad\leq \alpha_1(\beta)^{12} \bigl(1+2 \alpha_0(\kappa)
        + \alpha_0(\kappa)^2\bigr)\bigl(2 \alpha_0(\kappa)
        + \alpha_0(\kappa)^2\bigr)^{|\support W|-1} \cdot \mathbb{1}(|\support W| \geq 7)
        \\&\qquad\quad +
        \bigl( \alpha_1(\beta)^6 + \mathbb{1}_{|\support W| \geq \dist(e,B_N^c)}\bigr) \bigl(2 \alpha_0(\kappa)
        + \alpha_0(\kappa)^2\bigr)^{|\support W|}  \cdot \mathbb{1}(|\support W| \geq 7)
        \\&\qquad\quad +
        \bigl( \alpha_1(\beta)^6 + \mathbb{1}_{|\support W| \geq \dist(e,B_N^c)}\bigr) \bigl(1+2 \alpha_0(\kappa)
        + \alpha_0(\kappa)^2\bigr) \bigl(2 \alpha_0(\kappa)
        + \alpha_0(\kappa)^2\bigr)^{|\support W|-1} \cdot \mathbb{1}(|\support W| \geq 8).
    \end{align*}
    Since \( W \in T_{e,m} \), we have $|\support W| =m$. Also, by Lemma~\ref{lemma: number of paths}, $|T_{e,m}| \leq 18^{2m-3}$ for any $m \geq 1$.
    We conclude that  
    \begin{align*}
        &\mu_{N,\beta,\kappa} \times \mu_{N,\infty,\kappa} \bigl( \{ (\hat \sigma,\hat \sigma') \in \Sigma_{E_N} \times \Sigma_{E_N}^0 \colon  e \in \partial \hat \partial E_{\hat \sigma, \hat \sigma'},\,  \sigma \in \mathcal{E} \} \bigr)  
        \\&\qquad\leq  
     \sum_{m=7}^\infty  
        18^{2m-3} \bigl( 1 + 2\alpha_0(\kappa) + \alpha_0(\kappa)^2 \bigr)\bigl(2\alpha_0(\kappa) + \alpha_0(\kappa)^2 \bigr)^{m-1}  \alpha_1(\beta)^{12}
        \\&\qquad\qquad +
     \sum_{m=7}^\infty  
        18^{2m-3} \bigl(2\alpha_0(\kappa) + \alpha_0(\kappa)^2 \bigr)^m  \alpha_1(\beta)^{6}
        \\&\qquad\qquad +
     \sum_{m=8}^\infty  
        18^{2m-3} \bigl( 1 + 2\alpha_0(\kappa) + \alpha_0(\kappa)^2 \bigr) \bigl(2\alpha_0(\kappa) + \alpha_0(\kappa)^2 \bigr)^{m-1}  \alpha_1(\beta)^{6}
        \\&\qquad\qquad+
     \bigl( 1 + 2\alpha_0(\kappa) + \alpha_0(\kappa)^2 \bigr) \sum_{m=\dist(e,B_N^c)}^\infty  
        18^{2m-3} \bigl(2\alpha_0(\kappa) + \alpha_0(\kappa)^2 \bigr)^{m-1} 
    \end{align*}
    and thus~\eqref{eq: coupling and conditions 1} follows by evaluating the above geometric sums. 
\end{proof}

\begin{proof}[Proof of Proposition~\ref{proposition: minimal Ising}]
    Note first that if \( \hat \sigma' \in \Sigma_{E_N}^0, \) then \( d\hat \sigma' = 0 , \) and hence, using Lemma~\ref{lemma: cluster is subconfig}\ref{item: subconfig i} and the definition of \( \leq \), it follows that \( d\hat \sigma'|_{\mathcal{C}_{\mathcal{G}(0,\hat \sigma')}(e)} = 0 ,\) and hence \( \hat \sigma'|_{\mathcal{C}_{\mathcal{G}(0,\hat \sigma')}(e)} \in \Sigma_{E_N}^0.\)
    If, in addition, \( \hat \sigma'_e \neq 0,  \) then, by definition, we have \( (\hat \sigma'|_{\mathcal{C}_{\mathcal{G}(0,\hat \sigma')}(e)})_e \neq 0. \) Applying Lemma~\ref{lemma: small 1forms}, it thus follows that \( |(\support \hat \sigma'|_{\mathcal{C}_{\mathcal{G}(0,\hat \sigma')}(e)})^+| \geq 2 \cdot 8, \) and hence \( |\mathcal{C}_{\mathcal{G}(0 , \hat \sigma')}(e)| = |(\support \hat \sigma'|_{\mathcal{C}_{\mathcal{G}(0,\hat \sigma')}(e)})^+| \geq 2 \cdot 8. \)
    Consequently, the desired conclusion will follow if we can show that
    \begin{equation}\label{eq: hat goal}
        \mu_{N,\infty,\kappa} \bigl( \hat \sigma' \in  \Sigma_{E_N}^0 \colon  | \mathcal{C}_{\mathcal{G}(0 , \hat \sigma')}(e) | \geq 2 \cdot 8 \big\}\bigr) \leq C_I \alpha_0(\kappa)^8.
    \end{equation}
    
    For \( k \geq 1 \), let \( T_{e,k} \) be the set of all walks in \( \bar {\mathcal{G}}  \) which start at \( e \), have length \( 2k-3 \), contain only edges in \( E_N^+ \), and visit exactly \( k \) vertices in \( \bar {\mathcal{G}}  \).
    Given \( \hat \sigma'   \in\Sigma_{E_N}^0  \), and \( \mathcal{G} = \mathcal{G}(0, \hat \sigma') \), 
    let \( \mathcal{G}|_{(\mathcal{C}_{\mathcal{G}}(e))^+ } \) denote the graph \( \mathcal{G} \) restricted to the set \( (\mathcal{C}_{\mathcal{G}}(e))^+ \), and let \( \mathcal{W}_{0,\hat \sigma',e} \) be the set of all spanning walks of \( \mathcal{G}|_{( \mathcal{C}_{\mathcal{G}}(e))^+ } \) which have length \( 2|(\mathcal{C}_{\mathcal{G}}(e))^+|-3 \) and start at \( e \).
    By Lemma~\ref{lemma: spanning paths}, the set \( \mathcal{W}_{0,\hat \sigma',e} \) is non-empty. Moreover, since any walk on \( \mathcal{G}\) is a walk on \( \bar {\mathcal{G}}  \), we have \( \mathcal{W}_{0,\hat \sigma',e} \subseteq T_{e,|(\mathcal{C}_{\mathcal{G}}(e))^+|} \). 
    Hence
    \begin{align}
        & \mu_{N,\infty,\kappa} \bigl( \big\{ \hat \sigma' \in  \Sigma_{E_N}^0 \colon  | \mathcal{C}_{\mathcal{G}(0 , \hat \sigma')}(e) | \geq 2 \cdot 8 \big\}\bigr)
        =
        \sum_{m=8}^\infty  
        \mu_{N,\infty,\kappa} \bigl( \big\{ \hat \sigma' \in \Sigma_{E_N}^0 \colon  |\mathcal{C}_{\mathcal{G}(0, \hat \sigma')}(e)| = 2m \big\} \bigr) \nonumber
        \\&\qquad \leq 
        \sum_{m=8}^\infty \sum_{W \in T_{e,m}} 
        \mu_{N,\infty,\kappa} \bigl( \big\{ \hat \sigma' \in \Sigma_{E_N}^0 \colon  |\mathcal{C}_{\mathcal{G}(0, \hat \sigma')}(e)| = 2m
        \text{ and } W \in \mathcal{W}_{0,\hat \sigma',e} \big\} \bigr). \label{eq: last line}
    \end{align}

    Now assume that \( W \in T_{e,m} \) for some \( m \geq 1 \). Let \( \support W \) be the set of edges in \( E_N^+ \) traversed by \( W \). Then \( \bigl( \mathcal{C}_{\mathcal{G}(0,\hat \sigma')}(e)\bigr)^+ = \support W\) whenever \(|\mathcal{C}_{\mathcal{G}(0, \hat \sigma')}(e)| = 2m\) and \( W \in \mathcal{W}_{0,\hat\sigma',e} \). Consequently, the expression on the right-hand side of~\eqref{eq: last line} is bounded from above by
    \begin{equation*}
        \begin{split}
            &\sum_{m=8}^\infty \sum_{W \in T_{e,m}} 
             \mu_{N,\infty,\kappa} \bigl( \big\{ \hat \sigma' \in \Sigma_{E_N}^0 \colon  \bigl( \mathcal{C}_{\mathcal{G}(0, \hat \sigma')}(e)\bigr)^+ = \support W \big\}\bigr)
             \\&\qquad\leq
            \sum_{m=8}^\infty \sum_{W \in T_{e,m}} 
            \sum_{\substack{ \sigma' \in \Sigma_{E_N} \colon \\
            \substack{(\support  \sigma')^+ = \support W}}} 
             \mu_{N,\infty,\kappa} \bigl( \big\{ \hat \sigma' \in \Sigma_{E_N}^0 \colon
            \hat \sigma'|_{\mathcal{C}_{\mathcal{G}(0, \hat \sigma')}(e)} =  \sigma'\big\} \bigr) .
        \end{split}
    \end{equation*}
    For any \( m\), \( W \),  and \(  \sigma' \) as in the sum above, by Lemma~\ref{lemma: cluster is subconfig}, we have 
    \begin{equation*}
        \begin{split}
            &\mu_{N,\infty,\kappa} \bigl( \big\{ \hat \sigma' \in \Sigma_{E_N}^0 \colon
            \hat \sigma'|_{\mathcal{C}_{\mathcal{G}(0, \hat \sigma')}(e)} =  \sigma'\big\} \bigr) 
            \leq  
            \mu_{N,\infty,\kappa} \bigl( \big\{ \hat \sigma' \in  \Sigma_{E_N}^0 \colon
            \sigma' \leq \hat \sigma' \big\} \bigr) \leq \varphi_{\infty,\kappa} (\sigma'),
        \end{split}
    \end{equation*}
    where the last inequality follows by applying Proposition~\ref{proposition: edgecluster flipping ii}.
    
    To sum up, we have shown that
    \begin{equation*}
        \begin{split}
        &\mu_{N,\infty,\kappa} \bigl( \big\{ \hat \sigma' \in \Sigma_{E_N}^0 \colon  |\mathcal{C}_{\mathcal{G}(0 , \hat \sigma')}(e)| \geq 2 \cdot 8  \big\} \bigr) \leq
        \sum_{m=8}^\infty \sum_{W \in T_{e,m}} 
          J_{\beta, \kappa}(W),
        \end{split}
    \end{equation*}
    where
    \begin{equation*}
        J_{\beta, \kappa}(W) \coloneqq 
        \!\!\!\!\sum_{\substack{ \sigma' \in \Sigma_{E_N} \colon \\
        \substack{(\support  \sigma')^+ = \support W}}} \!\!\!\! \varphi_{\infty,\kappa} (  \sigma'  )
        =
        \!\!\!\!\sum_{\substack{ \sigma' \in \Sigma_{E_N} \colon \\
        \substack{(\support  \sigma')^+ = \support W}}}\!\!\!\! \prod_{e' \in \support W}\varphi_{\infty,\kappa} (  \sigma'_{e'}  )^2
    \end{equation*} 
    Now fix some \( W \in T_{e,m} \). By changing the order of summation and multiplication, we have
    \begin{equation*}
        \sum_{\substack{ \sigma' \in \Sigma_{E_N} \colon \\
        \substack{(\support  \sigma')^+ = \support W}}}\!\!\!\! \prod_{e' \in \support W}\varphi_{\infty,\kappa} (  \sigma'_{e'}  )^2
        =
        \prod_{e' \in \support W} \sum_{\sigma_{e'}' \in G\smallsetminus \{ 0 \}} \varphi_\kappa(\sigma_{e'}')^2
        =
        \prod_{e' \in \support W} \alpha_0(\kappa)
        =
        \alpha_0(\kappa)^{|\support W|}.
    \end{equation*}
    Since \( W \in T_{e,m} \), we have \( |\support W| =m\).  Consequently, $J_{\beta, \kappa}(W) \leq \alpha_0(\kappa)^m$.   
Next, by Lemma~\ref{lemma: number of paths}, for any \( m \geq 1 \) we have \( |T_{e,m}| \leq 18^{2m-3} \). Combining the previous equations, we thus find that  
    \begin{align*}
        &\mu_{N,\infty,\kappa} \bigl( \{ \hat \sigma' \in \Sigma_{E_N}^0 \colon  |\mathcal{C}_{\mathcal{G}(\hat \sigma , \hat \sigma')}(e)| \geq 2 \cdot 8  \} \bigr)  
        \leq  
        \sum_{m=8}^\infty  
        18^{2m-3}  \alpha_0(\kappa)^m.
    \end{align*}
    This is a geometric sum which converges if~\ref{assumption: 3} holds. Using the formula for geometric sums, we obtain~\eqref{eq: hat goal}. This concludes the proof.
\end{proof}

\section{Properties of \texorpdfstring{\( \theta_{\beta,\kappa}\)}{theta}} \label{section: technical inequalities}

In this section, we establish a few useful properties of the function \( \theta_{\beta,\kappa} \), as well as of the closely related function \( S_{\beta,\kappa} \) defined in~\eqref{Sbetakappadef} below.

Recall that we work under the standing assumptions listed in Section~\ref{sec:standing assumptions}. In particular, we assume that \( \rho \) is a unitary and faithful 1-dimensional representation of \( G = \mathbb{Z}_n \). 

We now define the function \( S_{\beta,\kappa} \). To this end, assume that \( \beta,\kappa \geq 0 \), that \( K \) is a finite and non-empty index set, and that an element \( g_k \in G = \mathbb{Z}_n\) is given for each \( k \in K \). Define
\begin{equation}\label{Sbetakappadef}
    S_{\beta, \kappa} \bigl(\{ g_k\}_{k \in K}\bigr) \coloneqq \frac{\sum_{g \in G} \rho(g) \bigl(\prod_{k \in K}\varphi_\beta(g+g_k)\bigr) \varphi_\kappa(g)}{\sum_{g \in G} \bigl(\prod_{k \in K}\varphi_\beta(g+g_k)\bigr) \varphi_\kappa(g)}.
\end{equation}
For \( \beta,\kappa \geq0 \) and \( \hat g \in G \), the function $\theta_{\beta,\kappa}(\hat g)$ defined in~\eqref{eq: def of thetag} can be expressed as
\begin{equation*}
    \theta_{\beta,\kappa}(\hat g) 
    =
    \frac{\sum_{h \in G} \rho(h)  \varphi_{2\beta}(h)^{6}  \varphi_{2\kappa}(h+\hat{g})}{\sum_{h \in G} \varphi_{2\beta}(h)^{6} \varphi_{2\kappa}(h+\hat{g})}. 
\end{equation*} 
Hence, letting \( K = \{ 1,2,\ldots, 6 \} \) and performing the change of variables $h = g + \hat{g}$, we find
\begin{align}
    \nonumber
    &  \theta_{\beta,\kappa}(-\hat g) 
    =
    \frac{\sum_{h \in G} \rho(h)  \varphi_{2\beta}(h)^{6}  \varphi_{2\kappa}(h-\hat{g})}{\sum_{h \in G} \varphi_{2\beta}(h)^{6} \varphi_{2\kappa}(h-\hat{g})}
    =
    \frac{\sum_{g \in G} \rho(g + \hat{g})  \varphi_{2\beta}(g + \hat{g})^{6}  \varphi_{2\kappa}(g)}{\sum_{g \in G} \varphi_{2\beta}(g + \hat{g})^{6} \varphi_{2\kappa}(g)}
    \\ \label{eq: theta vs S}
    &\qquad    = \rho (\hat g)S_{2\beta,2\kappa}(\{ \hat g\}_{k\in K}).
\end{align}

The next lemma provides simple upper bounds on \( \theta_{\beta,\kappa} \) and \( S_{\beta,\kappa} \).

\begin{lemma}\label{lemma: theta bound}
     For any \( \beta,\kappa \geq 0 \), any finite and non-empty set \( K \), and any choice of \( g_k \in G \) for \( k \in K \), we have 
     \begin{equation*}
         |S_{\beta,\kappa}(\{g_k\}_{k \in K})| \leq 1 .
     \end{equation*} 
     Consequently, for any \( g \in G \) we have
     \begin{equation*}
         |\theta_{\beta,\kappa}(g)| \leq 1.
     \end{equation*}
 \end{lemma}
 
\begin{proof}
    Since \( \rho \) is unitary and one-dimensional, we have \( |\rho(g)|=1 \) for all \( g \in G \). Thus,
    \begin{align*}
        &\bigl| S_{\beta, \kappa} \bigl(\{ g_k\}_{k \in K}\bigr) \bigr| 
        \leq 
        \frac{\sum_{g \in G} |\rho(g)| \bigl(\prod_{k \in K}\varphi_\beta(g+g_k)\bigr) \varphi_\kappa(g)}{\sum_{g \in G} \bigl(\prod_{k \in K}\varphi_\beta(g+g_k)\bigr) \varphi_\kappa(g)} 
        =
        1.
    \end{align*}
    Using~\eqref{eq: theta vs S}, the estimate $|\theta_{\beta,\kappa}(g)| \leq 1$ immediately follows.
\end{proof}

\begin{lemma}\label{lemma: theta inequalities Dplusii}
    Let \( \beta,\kappa \geq 0 \), and for each \( g \in G \), let \( j_g >0 \) be given. Further, let \( j \coloneqq \sum_{g \in G} j_g \). Then
    \begin{equation*}
        \Bigl| \prod_{g \in G} \theta_{\beta,\kappa}(g)^{j_g} \Bigr| \leq e^{-j \alpha_5(\beta,\kappa)}.
    \end{equation*}
\end{lemma}

\begin{proof}
    Using~\eqref{eq: theta vs S}, for each \( g \in G \) we have
    \begin{equation*}
        |\theta_{\beta,\kappa}(g)| \leq 1-\alpha_5(\beta,\kappa) \leq e^{-\alpha_5(\beta,\kappa)},
    \end{equation*}
where $\alpha_5$ is defined in~\eqref{eq: alpha5}. From this the desired conclusion immediately follows.
\end{proof}

\section{From \texorpdfstring{\( W_\gamma \)}{Wgamma}  to \texorpdfstring{\( W_\gamma' \)}{Wgamma'}}\label{sec:WtoWprime}
In this section, we prove Proposition~\ref{prop: first part of proposition proof}, which shows that the expected value $\mathbb{E}_{N,\beta,\kappa} [W_\gamma]$ of the Wilson loop observable $W_\gamma$ can be well approximated by the expected value of a simpler observable $W_\gamma'$. 
The proof is based on the same idea as the proof of Theorem 1.1 in~\cite{c2019}, namely, to show that the main contribution to $\mathbb{E}_{N,\beta,\kappa} [W_\gamma]$ stems from the minimal vortices centered on edges of the loop $\gamma$.

Given a simple loop \( \gamma \),  we let \( \gamma_c \) denote the set of corner edges in \( \gamma ,\) and let \( \gamma_1 \coloneqq \gamma\smallsetminus \gamma_c .\) Further, given a configuration \( \sigma \in \Sigma_{E_N}, \) we let  \( \gamma' \) denote the set of edges \(e \in  \gamma_1\) such that there are \( p,p' \in \hat \partial e \) with \( (d\sigma)_p \neq (d\sigma)_{p'}\). 
Recall that we are always working under the standing assumptions in Section~\ref{sec:standing assumptions}.

\begin{proposition}\label{prop: first part of proposition proof}
    Let \( \beta,\kappa \geq 0 \) be such that~\ref{assumption: 1} and~\ref{assumption: 2} hold,
    and let \( \gamma \) be a simple loop in \( E_N \) such that \( \hat \partial \gamma \subseteq P_N \). 
    For each \( e \in \gamma_1 \smallsetminus \gamma' \), fix \( p_e \in \hat \partial e \) and define
    \begin{equation*}
          W_\gamma' \coloneqq    \rho \Bigl( \sum_{e \in \gamma_1 \smallsetminus \gamma'} (d\sigma)_{p_e}   \Bigr).
    \end{equation*}
    Then
    \begin{equation*} 
        \Bigl|\mathbb{E}_{N,\beta,\kappa} \bigl[W_\gamma\bigr]- \mathbb{E}_{N,\beta,\kappa}\bigl[W_\gamma' \bigr]\Bigr|
        \leq 
        4C_1C_0^{(7)} |\gamma| \alpha_2(\beta,\kappa)^{7}
        + 2C_3|\gamma_c| \alpha_2(\beta,\kappa)^{6}
        + 2C_0^{(25)} |\gamma|^4 \alpha_2(\beta,\kappa)^{25}.
    \end{equation*}
\end{proposition}

In the proof of Proposition~\ref{prop: first part of proposition proof},  we use the following additional notation.   
\begin{description}
    %\item[\( B \)] A cube of  width \( |\gamma| \) which contains \( \gamma \). Since \( \gamma \) has length \( |\gamma| \), such a cube exists.
    
    \item[\( q \)] An oriented surface such that \( \gamma \) is the boundary of \( q \), and such that the support of \( q \) is contained in a cube of side length \( |\gamma|/2 \) which also contains \( \gamma \). Since \( \gamma \) has length \( |\gamma| \), such a cube exists, and the existence of such a surface is then guaranteed by Lemma~\ref{lemma: oriented loops}.
    
     \item[\( Q \)] The support of the oriented surface \( q \).
    \item[\( b \)] The smallest number such that any irreducible \( \nu \in \Sigma_{P_N} \) with \( |\support \nu| \leq 2 \cdot 24 \) is contained in a box of width \( b \). By the definition of irreducible plaquette configurations, \( b \) is a finite universal constant.
    \item[\( Q' \)] The set of plaquettes   \( p \in Q \) that are so far away from \( \gamma \) that any cube of width \( b+2 \) containing \( p \) does not intersect \( \gamma \).
\end{description}
 
We will consider the following ``good'' events:
\begin{description}
    \item[\( \mathcal{A}_1 \)] There is no vortex \( \nu \) in \( \sigma \) with \( |\support \nu | \geq 2 \cdot 25 \) which intersects \( Q \).
    \item[\( \mathcal{A}_2 \)] There is no vortex \( \nu \) in \( \sigma \) with \( |\support \nu| \geq 2\cdot 7 \) which  intersects \( Q \smallsetminus Q' \).
    \item[\( \mathcal{A}_3\)] There is no minimal vortex \( \nu \) in \( \sigma \) with \( \support \nu = \hat\partial e \cup \hat\partial(-e) \) for some \( e \in \gamma_c \).
\end{description}

Before we give a proof of Proposition~\ref{prop: first part of proposition proof}, we  state and prove a few short lemmas which bound the probabilities of the events \( \neg  \mathcal{A}_1 \), \(\neg \mathcal{A}_2 \), and \( \neg \mathcal{A}_3 \).
  
\begin{lemma}\label{lemma: A1}
    Let \( \beta,\kappa \geq 0 \) be such that~\ref{assumption: 1} and~\ref{assumption: 2} hold, and let \( \gamma \) be a simple loop in \( E_N \). Then
    \begin{equation*}
        \mu_{N,\beta,\kappa}(\neg \mathcal{A}_1) \leq  C_0^{(25)} |\gamma|^4 \alpha_2(\beta,\kappa)^{25}.
    \end{equation*}
\end{lemma}
 
\begin{proof}
    By Proposition~\ref{proposition: the vortex flipping lemma} and a union bound, we have
    \begin{equation*}
        \mu_{N,\beta,\kappa}(\neg \mathcal{A}_1) \leq  C_0^{(25)} |Q| \alpha_2(\beta,\kappa)^{25} .
    \end{equation*}
    Since $Q$ is contained in a cube of side $|\gamma|/2$, we have \( |Q| \leq |\gamma|^4\) and the desired conclusion follows.
\end{proof}

\begin{lemma}\label{lemma: A2}
    Let \( \beta,\kappa \geq 0 \) be such that~\ref{assumption: 1} and~\ref{assumption: 2} hold, and let \( \gamma \) be a simple loop in \( E_N \). Then
    \begin{equation}\label{eq: 7.4}
        \mu_{N,\beta,\kappa}(\neg \mathcal{A}_2) \leq C_1 C_0^{(7)} |\gamma| \alpha_2(\beta,\kappa)^{7}.
    \end{equation}
\end{lemma}
 
\begin{proof}
    Since each plaquette of \( Q \smallsetminus Q' \) is contained in a cube of width \( b+2 \) which intersects \( \gamma \), it follows that 
    \begin{equation}\label{eq: C1}
        | Q \smallsetminus Q' | \leq C_1 |\gamma|,
    \end{equation}
    where \( C_1 \) is a universal constant which depends on \( b \).
    Therefore, using Proposition~\ref{proposition: the vortex flipping lemma} and a union bound, we obtain~\eqref{eq: 7.4}.
\end{proof}

\begin{lemma}\label{lemma: A3}
    Let \( \beta,\kappa \geq 0 \) be such that~\ref{assumption: 1} holds, and let \( \gamma \) be a simple loop in \( E_N \) which is such that \( \hat \partial \gamma \subseteq P_N \). Then
    \begin{equation*}
        \mu_{N,\beta,\kappa}(\neg \mathcal{A}_3) \leq  C_3|\gamma_c|  \alpha_2(\beta,\kappa)^{6},
    \end{equation*}
    where
    \begin{equation}\label{eq: C3 def}
        C_3 \coloneqq \frac{4}{1-\alpha_1(\beta)-4\alpha_0(\kappa)}.
    \end{equation}
\end{lemma}

\begin{proof}
    For \( e\in \gamma \) and \( g \in G\smallsetminus \{ 0 \} \), let \( \nu^{e,g} \in \Sigma_{P_N}\) denote the minimal vortex with \( \nu_p = g \) for all \( p \in \hat \partial e \). By Lemma~\ref{lemma: minimal vortex II}, any minimal vortex is of this form for some choices of $e$ and $g$. Moreover, by Proposition~\ref{proposition: first step}, 
    \begin{equation*}
        \mu_{N,\beta,\kappa} \bigl( \{ \sigma \in \Sigma_{E_N} \colon \nu^{e,g} \leq d\sigma \} \bigr) \leq 
        \frac{\varphi_\beta(\nu^{e,g}) \cdot 4\alpha_0(\kappa)}{1-\alpha_1(\beta) - 4\alpha_0(\kappa)}.
    \end{equation*}
    Consequently, by a union bound, 
    \begin{align*}
        &\mu_{N,\beta,\kappa}(\neg \mathcal{A}_3) 
        \leq 
        \sum_{e \in \gamma_c} \sum_{g \in G\smallsetminus \{ 0 \}} \mu_{N,\beta,\kappa} \bigl( \{ \sigma \in \Sigma_{E_N} \colon \nu^{e,g} \leq d\sigma \} \bigr)
        \leq 
        \sum_{e \in \gamma_c} \sum_{g \in G\smallsetminus \{ 0 \}}
        \frac{\varphi_\beta(\nu^{e,g})\cdot 4\alpha_0(\kappa)}{1-\alpha_1(\beta) - 4\alpha_0(\kappa)}.
    \end{align*}
    For each \( e \in \gamma_c \) and \( g \in G\smallsetminus \{ 0 \} \), we have \( \varphi_\beta(\nu^{e,g}) = \varphi_\beta(g)^{12} \). Thus, for any \( e \in \gamma_c \), it holds that
    \begin{align*}
        &
        \sum_{g \in G\smallsetminus \{ 0 \}}
        \varphi_\beta(\nu^{e,g}) 
        =
        \sum_{g \in G\smallsetminus \{ 0 \}}
        \varphi_\beta(g)^{12} 
        \leq 
        \Bigl( \sum_{g \in G\smallsetminus \{ 0 \}}
        \varphi_\beta(g)^{2} \Bigr)^6
        = \alpha_0(\beta)^6.
    \end{align*}
    Combining the previous equations, we finally obtain
    \begin{equation*} \mu_{N,\beta,\kappa}(\neg \mathcal{A}_3) \leq |\gamma_c| \alpha_0(\kappa)\alpha_0(\beta)^6 \cdot  \frac{4}{1-\alpha_1(\beta)-4\alpha_0(\kappa)}
        =
        |\gamma_c|  \alpha_2(\beta,\kappa)^6  \cdot  \frac{4}{1-\alpha_1(\beta)-4\alpha_0(\kappa)},
    \end{equation*} 
    which is the desired conclusion.
\end{proof}

\begin{proof}[Proof of Proposition~\ref{prop: first part of proposition proof}] 
    Let \( \sigma \in \Sigma_{E_N} \). 
    By Lemma~\ref{lemma: lemma sum of irreducible vortices}, there is a set \( \Omega \) of  non-trivial and irreducible plaquette configurations \( \omega_1,\dots,\omega_k \leq d\sigma \) with disjoint supports, such that \( d\sigma = \omega_1 + \dots + \omega_k \). Fix such a set \( \Omega \), and note that each \( \nu \in \Omega \) is a vortex in \( \sigma \).
    Let \( V \) be the set of plaquette configurations in \( \Omega \) whose support intersects \( Q \coloneqq \support q \).
    Then, by Lemma~\ref{lemma: stokes}, we have 
    \begin{equation*}
        W_\gamma =    \rho  \Bigl(\, \sum_{p \in P_N} q^+_p (d\sigma)_p \Bigr)  =   \rho  \Bigl( \, \sum_{\nu \in V} \sum_{p \in \support \nu  } q_p^+ (d\sigma)_p  \Bigr).
    \end{equation*}
    
    Let
    \begin{equation*}
        V_0 \coloneqq \bigl\{ \nu \in V \colon |\support \nu| \leq 2 \cdot 24 \bigr\}
    \end{equation*}
    and define
    \begin{equation*}
        W_\gamma^0 \coloneqq   \rho  \Bigl( \sum_{\nu \in V_0} \sum_{p \in \support \nu  } q^+_p (d\sigma)_p \Bigr) .
    \end{equation*}
    If the event \( \mathcal{A}_1 \) occurs, then \( W_\gamma = W_\gamma^0 \), and hence, by Lemma~\ref{lemma: A1},
    \begin{equation*}
        \mathbb{E}_{N,\beta,\kappa} \Bigl[ \bigl| W_\gamma - W_\gamma^0 \bigr| \Bigr] \leq 2   \mu_{N,\beta,\kappa}(\neg \mathcal{A}_1)  \leq 2 C_0^{(25)} |\gamma|^4 \alpha_2(\beta,\kappa)^{25}.
    \end{equation*}

    Next define
    \begin{equation*}
        V_1 \coloneqq \bigl\{ \nu \in V_0 \colon \support \nu \cap Q' \neq \emptyset \bigr\}.
    \end{equation*}
    Take any vortex \( \nu \in V_1 \). By the definition of \( Q' \) and \( V_1 \), it follows that any cube \( B \) of width~\( b \) that contains \( \support \nu \) has the property that \(   
    (*{*B}) \cap Q\) only contains internal plaquettes of \( q \). Therefore, by Lemma~\ref{lemma: 3.2},
    \begin{equation*}
	    \sum_{p \in \support \nu } q_p^+ (d\sigma)_p = 0.
    \end{equation*}
    In particular, if we let \( V_2 \coloneqq V_0 \smallsetminus V_1 \), then
    \begin{equation}\label{eq: 7.3}
        W_\gamma^0 =  \rho  \Bigl( \sum_{\nu \in V_2} \sum_{p \in \support \nu  } q_p^+  (d\sigma)_p \Bigr).
    \end{equation}

    Let 
    \begin{equation*}
        V_3 \coloneqq \bigl\{ \nu \in V_2 \colon |\support \nu| = 2 \cdot 6 \bigr\},
    \end{equation*}
    i.e., let \( V_3 \) be the set of all minimal vortices in \( \Omega  \) whose support intersects \( Q \) but not \( Q' \).
    Define 
    \begin{equation*}
        W_\gamma^3 \coloneqq   \rho \Bigl( \sum_{\nu \in V_3} \sum_{p \in \support \nu  } q_p^+  (d\sigma)_p \Bigr) .
    \end{equation*}
    If the event \( \mathcal{A}_2 \) occurs, then \( V_3 = V_2 \), and hence by~\eqref{eq: 7.3}, \( W_\gamma^0 = W_\gamma^3 \). Consequently, by Lemma~\ref{lemma: A2}, we have
    \begin{equation*}
        \mathbb{E}_{N,\beta,\kappa}\Bigl[\bigl|W_\gamma^0-W_\gamma^3\bigr|\Bigr] \leq 2  \mu_{N,\beta,\kappa}(\neg \mathcal{A}_2) \leq 2C_1 C_0^{(7)} |\gamma| \alpha_2(\beta,\kappa)^{7}.
    \end{equation*}
    
    Next, let
    \begin{equation*}
        V_4 \coloneqq \bigl\{ \nu \in V_3 \colon \exists e \in \gamma \text{ such that\ } \support \nu = \hat\partial e \cup \hat\partial (-e) \bigr\}
    \end{equation*}
    and define 
    \begin{equation*}
        W_\gamma^4 \coloneqq \rho  \Bigl( \sum_{\nu \in V_4} \sum_{p \in \support \nu  } q_p^+  (d\sigma)_p \Bigr).
    \end{equation*}
    Recall that if \( \nu \) is a minimal vortex centered at \( e \in E_N \), and \( e \) is an internal edge of the oriented surface \( q \), then by Lemma~\ref{lemma: removing internal min vort}, we have 
    \begin{equation*}
        \sum_{p\in \support \nu } q_p^+  (d\sigma)_p  
        =
        \sum_{p\in \support \nu } q_p^+  \nu_p   = 0.
    \end{equation*}
    Hence 
    \begin{equation*}
        W_\gamma^3  = \rho  \Bigl(\sum_{\nu \in V_3} \sum_{p \in \support \nu   } q_p^+ (d\sigma)_p \Bigr) 
        =  \rho  \Bigl( \sum_{\nu \in V_4} \sum_{p \in \support \nu  } q_p^+ (d\sigma)_p  \Bigr)= W_\gamma^4 .
    \end{equation*} 
    Finally, let
    \begin{equation*}
        V_5 \coloneqq \bigl\{ \nu \in V_3 \colon \exists e \in \gamma\smallsetminus  \gamma_c  \text{ such that\ } \support \nu = \hat\partial e \cup \hat\partial (-e) \bigr\}
    \end{equation*}
    and define
    \begin{equation*}
        W_\gamma^5 \coloneqq   \rho \Bigl( \sum_{\nu \in V_5} \sum_{p \in \support \nu  } q_p^+ (d\sigma)_p \Bigr)  .
    \end{equation*} 
    If the event \( \mathcal{A}_3 \) occurs, then \( W_\gamma^4 = W_\gamma^5 \), and hence by Lemma~\ref{lemma: A3} we have
    \begin{equation*}
        \mathbb{E}_{N,\beta,\kappa} \Bigl[\bigl|W_\gamma^4 - W_\gamma^5\bigr|\Bigr] \leq 2 \mu_{N,\beta,\kappa}(\neg \mathcal{A}_3)
        \leq 2C_3 |\gamma_c| \alpha_2(\beta,\kappa)^{6}.
    \end{equation*}
 
    If \( \nu \in V_5 \), then \( \support \nu = \hat \partial e  \cup \hat \partial (-e) \) for some edge \( e \in \gamma\smallsetminus \gamma_c \). Let
    \begin{equation*}
        E_5 \coloneqq \bigl\{ e \in \gamma\smallsetminus \gamma_c \colon \exists \nu \in V_5 \text{ such that\ } \support \nu = \hat \partial e \cup \hat \partial (-e)\bigr\}.
    \end{equation*}
     Since \( q \) is an oriented surface, by definition, we have 
    \begin{equation}\label{eq: 7.8}
         W_\gamma^5 \coloneqq  \rho \Bigl(  \sum_{\nu \in V_5} \sum_{p \in \support \nu} q_p^+ (d\sigma)_p  \Bigr)  
         =  \rho \Bigl(  \sum_{e \in E_5} \sum_{p \in \pm \hat \partial e } q_p^+ (d\sigma)_p   \Bigr)
         =  \rho \Bigl(  \sum_{e \in E_5} \sum_{p \in  \hat \partial e } q_p (d\sigma)_p   \Bigr).  
    \end{equation} 

    For any two distinct edges \( e,e' \in \gamma\smallsetminus \gamma_c \), the sets \( \hat \partial e \cup \hat \partial (-e) \) and \( \hat \partial e' \cup \hat \partial (-e') \) are disjoint. 
    Moreover, if \( \nu \) is a minimal vortex in \( \sigma \) centered at \( e \in \gamma \), then \( (d\sigma)_p=(d\sigma)_{p'} \) for all \( p,p' \in \hat \partial e \) by Lemma~\ref{lemma: minimal vortex II}. 
    With this in mind, define
    \begin{equation*}
        E_6  \coloneqq \bigl\{ e \in \gamma\smallsetminus \gamma_c \colon (d\sigma)_p = (d\sigma)_{p'} \text{ for all } p,p' \in \hat \partial e   \bigr\}.
    \end{equation*}
    Since \( q \) is an oriented surface, for any choice of a plaquette \( p_e \in \hat \partial e  \) for each \( e \in E_6 \), we have
    \begin{equation*}
        \sum_{p \in \hat \partial e } q_p (d\sigma)_p = (d\sigma)_{p_e} 
    \end{equation*}
  for all \( e \in E_6 \).  Define 
    \begin{equation*}
        W_\gamma^6 \coloneqq \rho \Bigl( \sum_{e \in E_6} (d\sigma)_{p_e}   \Bigr) 
        =\rho \Bigl( \sum_{e \in E_6} \sum_{p \in \partial e } q_p (d\sigma)_p   \Bigr) .
    \end{equation*} 
    We clearly have \( E_5 \subseteq E_6 \). Moreover, on the event \( \mathcal{A}_2 \), \( (d\sigma)_p = 0\) whenever \( p \in \hat{\partial} e \) for some \( e \in E_6 \smallsetminus E_5 \). Hence, in this case,  
    \begin{equation*}
         W_\gamma^5   
         =  \rho \Bigl(  \sum_{e \in E_5} \sum_{p \in \partial e } q_p (d\sigma)_p   \Bigr) 
         = W_\gamma^6.
    \end{equation*} 
    Again using Lemma~\ref{lemma: A2}, it follows that
    \begin{equation*}%\label{eq: 7.9}
        \mathbb{E}_{N,\beta,\kappa} \Bigl[ \bigl|W_\gamma^5-W_\gamma^6\bigr|\Bigr] \leq 2\mu_{N,\beta,\kappa}(\neg \mathcal{A}_2) \leq 2C_1C_0^{(7)} |\gamma| \alpha_2(\beta,\kappa)^{7}.
    \end{equation*}
    Combining the equations above, we arrive at
    \begin{equation*} 
        \begin{split}
          &\bigl|\mathbb{E}_{N,\beta,\kappa} [W_\gamma]- \mathbb{E}_{N,\beta,\kappa}[W_\gamma^6]\bigr| \leq   
          \mathbb{E}_{N,\beta,\kappa}\Bigl[\bigl|W_\gamma^6-W_\gamma^5\bigr|\Bigr]+ \mathbb{E}_{N,\beta,\kappa}\Bigl[\bigl|W_\gamma^5-W_\gamma^4\bigr|\Bigr] 
            + \mathbb{E}_{N,\beta,\kappa}\Bigl[\bigl|W_\gamma^4-W_\gamma^3\bigr|\Bigr]
            \\&\hspace{4.2cm}
            + \mathbb{E}_{N,\beta,\kappa}\Bigl[\bigl|W_\gamma^3-W_\gamma^0\bigr|\Bigr]
            + \mathbb{E}_{N,\beta,\kappa}\Bigl[\bigl|W_\gamma^0-W_\gamma\bigr|\Bigr]
            \\& \quad \leq 
             2C_1C_0^{(7)} |\gamma| \alpha_2(\beta,\kappa)^{7} 
            +  2C_3|\gamma_c| \alpha_2(\beta,\kappa)^{6} 
            +  0 
            +  2C_1 C_0^{(7)} |\gamma| \alpha_2(\beta,\kappa)^{7} 
            + 2C_0^{(25)} |\gamma|^4 \alpha_2(\beta,\kappa)^{25}  .
        \end{split}
    \end{equation*}
    Noting that \( E_6 = \gamma_1 \smallsetminus \gamma' \), the desired conclusion follows.
\end{proof}

\section{A resampling trick}\label{sec: resampling}
Recall from Section~\ref{sec:WtoWprime}, that given a simple loop \( \gamma \), we let \( \gamma_c \) denote the set of corner edges of \( \gamma \), we let \( \gamma_1 = \gamma\smallsetminus \gamma_c \), and we let \( \gamma' \) denote the set of edges \(e \in  \gamma_1\) such that there are \( p,p' \in \hat \partial e \) with \( (d\sigma)_p \neq (d\sigma)_{p'}\). Furthermore, for each \( e \in \gamma_1 \smallsetminus \gamma' \) we have fixed some (arbitrary) \( p_e \in \hat \partial e \), and defined
\begin{equation*}
    W_\gamma' =    \rho \Bigl( \sum_{e \in \gamma_1 \smallsetminus \gamma'} (d\sigma)_{p_e}   \Bigr). 
\end{equation*} 
In this section, we use a resampling trick, first introduced (in a different setting) in~\cite{c2019}, to rewrite \( \mathbb{E}_{N,\beta,\kappa}[W_\gamma'] \).

\begin{proposition}\label{prop: resampling in main proof}
    Let \( \beta,\kappa \geq 0 \), and let \( \gamma \) be a simple loop in \( E_N \) such that \( \partial \hat \partial \gamma \subseteq E_N\). For each \( e \in \gamma_1 \smallsetminus \gamma'  \), fix one plaquette \( p_e \in \hat \partial e \). Then 
    \begin{equation*}
        \mathbb{E}_{N,\beta, \kappa}\bigl[W_{\gamma  }'  \bigr] 
        =
        \mathbb{E}_{N,\beta, \kappa} \Bigl[ \, \prod_{ e \in \gamma_1 \smallsetminus \gamma'}  \theta_{\beta,\kappa} \bigl( \sigma_e - (d\sigma)_{p_e} \bigr) \Bigr]. 
    \end{equation*}   
\end{proposition}

The proof of Proposition~\ref{prop: resampling in main proof} is based on the following lemma.

\begin{lemma}\label{lemma: independent spins}
    Let \( \beta,\kappa \geq 0 \), and let \( E \subseteq E_N \) be such that 
    \begin{enumerate}[label=\textnormal{(\roman*)}]
        \item \( E \cap (-E)= \emptyset \),
        \item \( \partial \hat \partial E \subseteq E_N \), and
        \item for any distinct edges  \( e,e' \in E \)   we have \( \bigl( \hat \partial e \cup \hat \partial (-e) \bigr)
        \cap \bigl( \hat \partial e' \cup \hat \partial (-e') \bigr) = \emptyset. \)
    \end{enumerate} 
    Further,  let \( \sigma \in \Sigma_{E_N} \), and let \( \mu_{N,\beta,\kappa}' \) denote conditional probability of \( \mu_{N,\beta,\kappa} \) given \( \sigma|_{(E \cup (-E))^c}\). Then the spins \( ( \sigma_e)_{e \in E} \) are independent under \( \mu_{N,\beta,\kappa}' \). Equivalently, for any \( \sigma',\sigma'' \in \Sigma_{E_N} \), if \( \sigma \sim \mu_{N,\beta,\kappa} \) then
     \begin{equation}\label{eq: muprodmu}
     \begin{split}
 &\mu_{N,\beta,\kappa} \Bigl(   \sigma|_{E} = \sigma'|_{E} \; \bigm| \; \sigma|_{(E \cup (-E))^c} = \sigma''|_{(E \cup (-E))^c} \Bigr)
 \\&\qquad = \prod_{e \in E} \mu_{N,\beta,\kappa} \Bigl(   \sigma_e = \sigma'_e \; \bigm| \; \sigma|_{(E \cup (-E))^c} = \sigma''|_{(E \cup (-E))^c} \Bigr). 
     \end{split}
 \end{equation}
\end{lemma}

\begin{proof}
    To simplify notation, define \( \bar E \coloneqq E \cup \{ e \in E_N \colon -e \in E \} \) and \( \hat \partial \bar E \coloneqq \bigcup_{e \in \bar E} \hat \partial e \). 
If \( \sigma \in \Sigma_{E_N} \), then
    \begin{align*}
        &\varphi_{\beta,\kappa}(\sigma) 
        =
        \prod_{e \in E_N} \varphi_\kappa(\sigma_e)  \prod_{p \in P_N} \varphi_\beta \bigl( (d\sigma)_p \bigr)
        \\&\qquad = 
        \prod_{e \in E} \varphi_\kappa(\sigma_e)^2  
        \prod_{p \in \hat \partial E} \varphi_\beta \bigl( (d\sigma)_p \bigr)^2
        \cdot 
        \prod_{e \in E_N \smallsetminus \bar E} \varphi_\kappa(\sigma_e)
        \prod_{p \in P_N \smallsetminus \hat \partial \bar E} \varphi_\beta \bigl( (d\sigma)_p \bigr)
        \\&\qquad = 
        \prod_{e \in E} \varphi_\kappa(\sigma_e)^2  
        \prod_{p \in \hat \partial E} \varphi_\beta \biggl( \sigma_e + \sum_{e' \in \partial p \smallsetminus \{ e \}} \sigma_{e'} \biggr)^2
        \cdot 
        \prod_{e \in E_N \smallsetminus \bar E} \varphi_\kappa(\sigma_e)
        \prod_{p \in P_N \smallsetminus \hat \partial \bar E} \varphi_\beta \bigl( (d\sigma)_p \bigr).
    \end{align*}
   Now note that the term
    \begin{equation*}
        \prod_{e \in E_N \smallsetminus \bar E} \varphi_\kappa(\sigma_e)
        \prod_{p \in P_N \smallsetminus \hat \partial \bar E} \varphi_\beta \bigl( (d\sigma)_p \bigr)
    \end{equation*}
    does not depend on \( \sigma|_{\bar E} \).
    Consequently, if \( \sigma',\sigma'' \in \Sigma_{E_N} \), then, using the expression~\eqref{eq: mubetakappaphi} for $\mu_{N,\beta, \kappa}$,
    \begin{equation} \label{eq: first part of mup}
        \begin{split}
            &
            \mu_{N,\beta,\kappa} \Bigl(   \sigma|_{\bar E} = \sigma'|_{\bar E} \; \bigm| \; \sigma|_{\bar E^c} = \sigma''|_{\bar E^c} \Bigr)
            =
            \frac{\mu_{N,\beta,\kappa} \bigl( \{ \sigma \in \Sigma_{E_N} \colon \sigma|_{\bar E} = \sigma'|_{\bar E},\, \sigma|_{\bar E^c} = \sigma''|_{\bar E^c} \} \bigr)}{\mu_{N,\beta,\kappa} \bigl( \{ \sigma \in \Sigma_{E_N} \colon \sigma|_{\bar E^c} = \sigma''|_{\bar E^c} \} \bigr)}
            \\&\qquad =
            \frac{
            \prod_{e \in   E} \varphi_\kappa(\sigma_e')^2 \prod_{p \in \hat \partial e } \varphi_\beta \bigl( \sigma'_e + \sum_{e' \in \partial p \backslash \{ e \}} \sigma''_{e'} \bigr)^2  
        }{
        \sum_{\sigma  \in \Sigma_{E}} 
        \prod_{e \in E} \varphi_\kappa(\sigma_e)^2 \prod_{p \in \hat \partial e} \varphi_\beta \bigl( \sigma_e + \sum_{e' \in \partial p \backslash \{ e \}} \sigma''_{e'} \bigr)^2  
        }
        \\&\qquad =\prod_{e \in E}  \frac{ \varphi_\kappa(\sigma'_e)^2 \prod_{p \in \hat \partial e } \varphi_\beta  \bigl( \sigma'_e + \sum_{e' \in \partial p \backslash \{ e \}} \sigma''_{e'} \bigr)^2
        }{
        \sum_{\sigma_e  \in \Sigma_{\{ e,-e \}}} \varphi_\kappa(\sigma_e)^2
        \prod_{p \in \hat \partial e } \varphi_\beta \bigl( \sigma_e + \sum_{e' \in \partial p \backslash \{ e \}} \sigma''_{e'} \bigr)^2
        }. 
        \end{split}
    \end{equation}
    A similar argument shows that, for a fixed edge \( e \in E \),
    \begin{align}
        &\mu_{N,\beta,\kappa} \Bigl( \sigma_e = \sigma_e' \; \bigm| \; \sigma|_{E_N \backslash \{e,-e \}} = \sigma''|_{E_N \backslash \{e,-e \}} \Bigr) \nonumber
        \\&\qquad=
        \frac{
             \varphi_\kappa(\sigma'_e)^2 \prod_{p \in \hat \partial e } \varphi_\beta\bigl( \sigma'_e + \sum_{e' \in \partial p \backslash \{ e \}} \sigma''_{e'} \bigr)^2
        }{
        \sum_{\sigma_e  \in \Sigma_{\{ e,-e \}}} \varphi_\kappa(\sigma_e)^2
        \prod_{p \in \hat \partial e } \varphi_\beta \bigl( \sigma_e + \sum_{e' \in \partial p \backslash \{ e \}} \sigma''_{e'} \bigr)^2
        }.\label{eq: previous eq}
    \end{align}
    Since~\eqref{eq: previous eq} only depends on \( \sigma'' \) through \( \sigma'' |_{{\partial\hat \partial e} \backslash \{ e \}} \), and, by assumption, we have \( \support \sigma'' |_{{\partial\hat \partial e \backslash \{ e \}}} \subseteq \bar E^c \subseteq E_N\backslash \{ e,-e \}\), it follows that 
    \begin{align*}
        &\mu_{N,\beta,\kappa} \Bigl( \sigma_e = \sigma_e' \; \bigm| \; \sigma|_{E_N \backslash \{e,-e \}} = \sigma''|_{E_N \backslash \{e,-e \}} \Bigr) 
        =
        \mu_{N,\beta,\kappa} \Bigl( \sigma_e = \sigma_e' \; \bigm| \; \sigma|_{\bar E^c} = \sigma''|_{\bar E^c} \Bigr) .
    \end{align*}
    Combining this with~\eqref{eq: first part of mup}, equation~\eqref{eq: muprodmu} follows.
\end{proof}

\begin{proof}[Proof of Proposition~\ref{prop: resampling in main proof}]
    By the definition of  \( \gamma' \), if we condition on the spins of all edges which are not in \(  \pm \gamma_1 \) (the non-corner edges of \( \gamma \)), then we can recognize the set \( \gamma' \).
    Since no two non-corner edges belong to the same plaquette, from Lemma~\ref{lemma: independent spins} applied with \( E= \gamma_1 \) it follows that under this conditioning,  \( ( \sigma_e)_{e \in \gamma_1} \) are independent spins. 
    Letting \( \mu_{N,\beta,\kappa}' \) and \( \mathbb{E}_{N,\beta,\kappa}' \) denote conditional probability and conditional expectation given \( ( \sigma_e )_{e \notin \pm \gamma_1} \), we find
    \begin{equation*} 
            \mathbb{E}_{N,\beta, \kappa}'\bigl[W_{\gamma  }'  \bigr]
            =\mathbb{E}_{N,\beta, \kappa}' \biggl[    \rho  \Bigl( \, \sum_{e \in \gamma_1 \smallsetminus \gamma'  } (d\sigma)_{p_e}   \Bigr)\biggr]
            =      \mathbb{E}_{N,\beta, \kappa}' \biggl[\,  \prod_{e \in \gamma_1 \smallsetminus \gamma'  } \rho  (  (d\sigma)_{p_e}  ) \biggr]
            =       \prod_{e \in \gamma_1 \smallsetminus \gamma'  } \mathbb{E}_{N,\beta, \kappa}' \bigl[ \rho ( (d\sigma)_{p_e}  ) \bigr].
   \end{equation*} 
    For each edge \(  e \in \gamma_1 \smallsetminus \gamma' \), let \( \sigma^{ e} \coloneqq \sum_{\hat e \in \partial p_{ e} \smallsetminus \{  e \} } \sigma_{\hat e} \). Then  $(d\sigma)_p = \sigma_e + \sigma^e$ for each $p \in \hat{\partial} e$, and we obtain
   \begin{equation*} 
        \begin{split} 
            & \mathbb{E}_{N,\beta, \kappa}'\bigl[W_{\gamma  }'  \bigr] =   \prod_{e \in \gamma_1 \smallsetminus \gamma'  } \Bigl[ \,   \frac{ \sum_{g \in  G} \rho(g + \sigma^e)\, \varphi_\beta(g + \sigma^e)^{12} \varphi_\kappa(g)^2}{\sum_{g\in  G} \varphi_\beta(g + \sigma^e)^{12} \varphi_\kappa(g)^2} \, \Bigr]
            \\&\qquad
            =   \prod_{e \in \gamma_1 \smallsetminus \gamma'  } \Bigl[ \,   \frac{ \sum_{g \in  G} \rho(g) \varphi_\beta(g)^{12} \varphi_\kappa (g-\sigma^e)^{2}}{\sum_{g\in  G} \varphi_\beta(g)^{12} \varphi_\kappa (g-\sigma^e)^2} \, \Bigr]
            = 
            \prod_{ e \in \gamma_1 \smallsetminus \gamma'} \theta_{\beta,\kappa}(-\sigma^e)
            = 
            \prod_{ e \in \gamma_1 \smallsetminus \gamma'} \theta_{\beta,\kappa}(\sigma_e-(d\sigma)_{p_e}),
        \end{split}
    \end{equation*}   
    and hence
    \begin{equation*} 
        \begin{split} 
            &\mathbb{E}_{N,\beta, \kappa}\bigl[W_{\gamma  }'  \bigr] 
            =\mathbb{E}_{N,\beta, \kappa} \Bigl[ \mathbb{E}_{N,\beta,\kappa}'\bigl[W_{\gamma  }'  \bigr] \Bigr]
            = 
            \mathbb{E}_{N,\beta, \kappa} \Bigl[ \, \prod_{ e \in \gamma_1 \smallsetminus \gamma'}  \theta_{\beta,\kappa} \bigl( \sigma_e - (d\sigma)_{p_e} \bigr) \Bigr]
        \end{split}
    \end{equation*}   
    as desired.
\end{proof}

\section{Applying the coupling}\label{sec:applyingcoupling}
Recall from Section~\ref{sec:WtoWprime} that, given a loop \( \gamma \), we let \( \gamma_c \) denote the set of corner edges of \( \gamma \),  \( \gamma_1 = \gamma\smallsetminus \gamma_c \), and let \( \gamma' \) denotes the set of edges \(e \in  \gamma_1\) such that there are \( p,p' \in \hat \partial e \) with \( (d\sigma)_p \neq (d\sigma)_{p'}\). Further, for each \( e \in \gamma_1 \smallsetminus \gamma' \) we fix some (arbitrary) \( p_e \in \hat \partial e \).
In this section, we use the coupling introduced in Section~\ref{sec: coupling} to prove the following proposition.
    
\begin{proposition}\label{proposition: E}
    Let \( \beta,\kappa \geq 0 \) be such that~\ref{assumption: 3} holds, and let \( \gamma \) be a simple loop in \( E_N \) such that \( \partial \hat \partial \partial \hat \partial \gamma \subseteq E_N \).
    Then
    \begin{equation} \label{eq: E}
        \begin{split}
            &\biggl| 
            \mathbb{E}_{N,\beta, \kappa} \Bigl[ \, \prod_{ e \in \gamma_1 \smallsetminus \gamma'}  \theta_{\beta,\kappa} \bigl( \sigma_e - (d\sigma)_{p_e} \bigr) \Bigr]
            - \mathbb{E}_{N,\infty,\kappa} \Bigl[ \, \prod_{e \in \gamma } \theta_{\beta,\kappa}(\sigma_e) \Bigr] \biggr| 
             \\&\qquad \leq 
            2\sqrt{2 C_{c,2}  |\gamma| \alpha_4(\beta,\kappa) \alpha_2(\beta,\kappa)^6 \alpha_0(\kappa)^5 \max(\alpha_0(\kappa),\alpha_0(\beta)^6)  } 
            \\&\qquad\qquad + 
            4\sqrt{2 C_{c,1} |\gamma| \alpha_4(\beta,\kappa)     \alpha_2(\beta,\kappa)^6     \alpha_0(\kappa)^8  } 
            + 2 \sqrt{2 C_I |\gamma_c| \alpha_0(\kappa)^8\alpha_4(\beta,\kappa)}
            + 2 \sqrt{2 |\gamma_c| \alpha_3(\beta,\kappa)}
            %+ 2 \sqrt{2\cdot 19 C_c |\gamma| \alpha_2(\beta,\kappa)^6 \alpha_3(\beta,\kappa)}.
            \\&\qquad\qquad 
            + 2\sqrt{12C_0^{(6)}|\gamma|\alpha_2(\beta,\kappa)^6 \alpha_3(\beta,\kappa)}
            +  (2C_{c,1}'+C_{c,2}') |\gamma| \Bigl( 18^{2} \bigl(2+ \alpha_0(\kappa)  \bigr) \alpha_0(\kappa)  \Bigr)^{\dist(\gamma,B_N^c)}.
        \end{split}
    \end{equation}
\end{proposition}

For the proof of Proposition~\ref{proposition: E} we need two lemmas. 

\begin{lemma}\label{lemma: Chatterjees inequality ii}
    Assume that \( z_1,z_2,z_1',z_2' \in \mathbb{C} \) are such that \( |z_1|,|z_2|,|z_1'|,|z_2'| \leq 1 \). Then 
    \begin{equation*}
        |z_1z_2-z_1'z_2'| \leq |z_1-z_1'| + |z_2-z_2'|.
    \end{equation*} 
\end{lemma}

\begin{proof}
    Note first that
    \begin{equation*} 
        z_1z_2-z_1'z_2' = z_1(z_2-z_2') + z_2'(z_1-z_1').
    \end{equation*}
    Applying the triangle inequality, we thus obtain
    \begin{equation*} 
        \bigl|z_1z_2-z_1'z_2'\bigr| \leq \bigl|z_1(z_2-z_2') \bigr| + \bigl| z_2'(z_1-z_1')\bigr| =|z_1|\bigl|z_2-z_2' \bigr|+ |z_2'| \bigl| z_1-z_1'\bigr| .
    \end{equation*}
    Since \( |z_1| \leq 1 \) and \( |z_2'| \leq 1 \) by assumption, the desired conclusion follows.
\end{proof}

\begin{lemma}  \label{lemma: Chatterjees trick abs} 
    Let $a,b>0$.  Assume that \( A \subseteq E_N \) is a random set with \( \mathbb{E}[|A|] \leq a \),  and that
    \begin{enumerate}[label=\textnormal{(\roman*)}]
        \item \( X_e \in \mathbb{C} \) and \( |X_e| \leq 1 \) for all \( e \in E_N \), and \label{item: Chatterjees trick abs assump 1}
        \item there exists a \( c \in [-1,1] \) such that \( |X_e-c| \leq b \) for all \( e \in E_N \).\label{item: Chatterjees trick abs assump 2}
    \end{enumerate} 
    Then
    \begin{equation*} 
        \mathbb{E}\biggl[ \Bigl| \prod_{e \in A} c -  \prod_{e \in A} X_e \Bigr|  \biggr] 
        \leq
        2\sqrt{2ab} .
    \end{equation*}
\end{lemma}
     
\begin{proof} 
    For any \( j > 0 \),  we have
    \begin{equation*}
        \begin{split}
            &  \mathbb{E}  \biggl[\Bigl| \prod_{e \in A} X_e  
            -
            \mathbb{1}_{\{|A |\leq j\}}  \prod_{e \in A} X_e  \Bigr| \biggr] 
            = 
            \mathbb{E} \biggl[  \mathbb{1}_{\{|A | > j\}} \Bigl| \prod_{e \in A} X_e  \Bigr| \biggr] 
            \overset{\ref{item: Chatterjees trick abs assump 1}}{\leq} 
            1 \cdot  \mathbb{E} \bigl[ \mathbb{1}_{\{|A | > j\}}  \bigr]
            \leq 
            \frac{ \mathbb{E}  [|A |   ]}{j}
            \leq 
            \frac{ a}{j}.
        \end{split}
    \end{equation*} 
    Next, applying Lemma~\ref{lemma: Chatterjees inequality ii} several times, we deduce that
    \begin{equation*} 
        \mathbb{E} \biggl[ \Bigl|  \mathbb{1}_{\{|A |\leq j\}} \prod_{e \in A} X_e   -\mathbb{1}_{\{|A |\leq j\}} \prod_{e \in A} c \, \Bigr| \biggr]
        = 
        \mathbb{E}\biggl[  \mathbb{1}_{\{| A |\leq j\}} \sum_{e \in A} \Bigl|  X_e -c   \Bigr| \biggr]
        \overset{\ref{item: Chatterjees trick abs assump 2}}{\leq}
        b\cdot \mathbb{E}\big[\mathbb{1}_{|A| \leq j} |A|\big] 
        \leq b j . 
    \end{equation*} 
    Finally, since $|c|\leq 1$,
    \begin{equation*}
        \mathbb{E} \biggl[ \Bigl|   \mathbb{1}_{\{|A |\leq j\}} \prod_{e \in A} c - \prod_{e \in A} c \, \Bigr| \biggr]
        =
        \mathbb{E} \biggl[ \Bigl|   \mathbb{1}_{\{|A | > j\}} \prod_{e \in A} c   \, \Bigr| \biggr]
        \leq 
        1 \cdot  \mathbb{E} \bigl[ \mathbb{1}_{\{|A | > j\}}  \bigr]
        \leq 
        \frac{ \mathbb{E}  [|A |   ]}{j}
        \leq 
        \frac{ a}{j}.
    \end{equation*}
    Combining the previous estimates, using the triangle inequality, and choosing 
    \begin{equation*}
        j = \sqrt{2a/b},
    \end{equation*}
    we obtain the desired conclusion.
\end{proof}

\begin{proof}[Proof of Proposition~\ref{proposition: E}]
    Recall the coupling \( (\sigma,\sigma') \sim \mu_{N,(\beta,\kappa),(\infty,\kappa)} \) between \( \sigma \sim \mu_{N,\beta,\kappa} \) and \( \sigma' \sim \mu_{N,\infty,\kappa} \) described in Section~\ref{sec: coupling}, and the set \( E_{\sigma,\sigma'} \)  defined in~\eqref{eq: Esigmasigmadef}. Since \( \mu_{N,(\beta,\kappa),(\infty,\kappa)} \) is a coupling of \( \mu_{N,\beta,\kappa} \) and \( \mu_{N,\infty,\kappa} \), we have
    \begin{equation*}
        \begin{split}
        &\mathbb{E}_{N,\beta, \kappa} \Bigl[\, \prod_{ e \in \gamma_1 \smallsetminus \gamma'}  \theta_{\beta,\kappa} \bigl(  \sigma_e - (d\sigma)_{p_e} \bigr) \bigr]
        -
        \mathbb{E}_{N,\infty,\kappa} \Bigl[ \, \prod_{e \in \gamma } \theta_{\beta,\kappa}(\sigma_e) \Bigr]
        \\&\qquad =  
        \mathbb{E}_{N,(\beta,\kappa),(\infty,\kappa)} \Bigl[\, \prod_{ e \in \gamma_1 \smallsetminus \gamma'}  \theta_{\beta,\kappa}\bigl( \sigma_e - (d\sigma)_{p_e} \bigr)
        -
        \prod_{e \in \gamma } \theta_{\beta,\kappa}(\sigma'_e) \Bigr].
        \end{split}
    \end{equation*}
    
    Given \( (\sigma,\sigma') \in E_N \times E_N^0 \), define 
    \begin{equation*}
        \gamma'' \coloneqq \gamma_1 \cap \partial \hat \partial E_{\sigma,\sigma'}.
    \end{equation*}
    By Lemma~\ref{lemma: properties of coupling measure}, if \( e \in \gamma_1 \smallsetminus \gamma'' \), then \( \sigma_{e'} = \sigma_{e'}' \) for all \( e' \in \partial \hat \partial e \), and  hence  \( (d\sigma)_{p_e} = (d\sigma')_{p_e} = 0 \). In particular, this implies that if \( e \in \gamma_1 \smallsetminus \gamma'' \), then 
    \begin{equation}\label{eq: uncoupled edges eq}
        \sigma_e- (d\sigma)_{p_e}  = \sigma'_e- (d\sigma')_{p_e}   = \sigma'_e- 0 =  \sigma'_e.
    \end{equation}
    By the definition of \(\gamma'\), if \(e' \in \gamma'\) then there exists an edge \(e'' \in \partial \hat{\partial} e'\) such that \(e'' \in \support{\sigma} \cap \partial \support{d\sigma} \subseteq E_{\sigma, \sigma'}\), and thus \(e' \in \partial \hat{\partial}E_{\sigma, \sigma'}\).
    Hence \( \gamma' \subseteq \gamma'' \) and it follows that
    \begin{equation*}
        \begin{split}
            &\prod_{ e \in \gamma_1 \smallsetminus \gamma'}  \theta_{\beta,\kappa} \bigl( \sigma_e- (d\sigma)_{p_e} \bigr)
            =
            \prod_{ e \in \gamma_1 \smallsetminus \gamma''}  \theta_{\beta,\kappa} \bigl(\sigma_e- (d\sigma)_{p_e} \bigr)
            \prod_{ e \in \gamma'' \smallsetminus \gamma'}  \theta_{\beta,\kappa} \bigl( \sigma_e- (d\sigma)_{p_e} \bigr)
            \\&\qquad \overset{\eqref{eq: uncoupled edges eq}}{=}
            \prod_{ e \in \gamma_1 \smallsetminus \gamma''}  \theta_{\beta,\kappa}(\sigma_e')
            \prod_{ e \in \gamma'' \smallsetminus \gamma'}  \theta_{\beta,\kappa} \bigl(\sigma_e- (d\sigma)_{p_e} \bigr).
        \end{split}
    \end{equation*} 
    Consequently, using Lemma~\ref{lemma: properties of coupling measure}, we have
    \begin{equation*}
        \begin{split}
            &
            \prod_{ e \in \gamma_1 \smallsetminus \gamma'}  \theta_{\beta,\kappa} \bigl( \sigma_e- (d\sigma)_{p_e} \bigr)
            -  
            \prod_{ e \in \gamma  }  \theta_{\beta,\kappa} ( \sigma_e' )
            \\&\qquad=
            \prod_{ e \in \gamma_1 \smallsetminus \gamma''}  \theta_{\beta,\kappa} (\sigma_e)
            \prod_{ e \in \gamma'' \smallsetminus \gamma'}  \theta_{\beta,\kappa} \bigl( \sigma_e- (d\sigma)_{p_e} \bigr)  
            -
            \prod_{ e \in \gamma_1 \smallsetminus \gamma''  }  \theta_{\beta,\kappa} (\sigma'_e)
            \prod_{ e \in \gamma'' \smallsetminus \gamma'  }  \theta_{\beta,\kappa} (\sigma_e')
            \prod_{ e \in \gamma_c \sqcup \gamma' }  \theta_{\beta,\kappa} (\sigma_e') 
            \\&\qquad=
            \prod_{ e \in \gamma_1 \smallsetminus \gamma''}  \theta_{\beta,\kappa} (\sigma_e')
            \biggl( \, \prod_{ e \in \gamma'' \smallsetminus \gamma'}  \theta_{\beta,\kappa} \bigl( \sigma_e- (d\sigma)_{p_e} \bigr)  
            -
            \prod_{ e \in \gamma'' \smallsetminus \gamma'  }  \theta_{\beta,\kappa} (\sigma_e') \biggr)
            \\&\qquad\qquad
            +
            \prod_{ e \in \gamma_1 \smallsetminus \gamma'  }  \theta_{\beta,\kappa} (\sigma'_e) 
            \Bigl( 1-\prod_{ e \in \gamma_c  }  \theta_{\beta,\kappa}(\sigma_e')\Bigr)
            +
            \prod_{ e \in \gamma_1 \smallsetminus \gamma'  }  \theta_{\beta,\kappa} (\sigma'_e) \prod_{ e \in \gamma_c  }  \theta_{\beta,\kappa}(\sigma_e')
            \Bigl( 1-\prod_{ e \in  \gamma' }  \theta_{\beta,\kappa}(\sigma_e')\Bigr)
            .
        \end{split}
    \end{equation*}
    Combining the above equations, we find that
    \begin{equation*}
        \begin{split}
            &
            \mathbb{E}_{N,\beta, \kappa} \Bigl[\, \prod_{ e \in \gamma_1 \smallsetminus \gamma'}  \theta_{\beta,\kappa} \bigl(  \sigma_e - (d\sigma)_{p_e} \bigr) \bigr]   - \mathbb{E}_{N,\infty,\kappa} \Bigl[ \,  \prod_{ e \in \gamma  }  \theta_{\beta,\kappa} (\sigma_e) \Bigr] 
            \\&\qquad=  \mathbb{E}_{N,(\beta,\kappa),(\infty,\kappa)}  \biggl[\, \prod_{ e \in \gamma_1 \smallsetminus \gamma''}  \theta_{\beta,\kappa} (\sigma_e')
            \biggl( \, \prod_{ e \in \gamma'' \smallsetminus \gamma'}  \theta_{\beta,\kappa} \bigl( \sigma_e- (d\sigma)_{p_e} \bigr)  
            -
            \prod_{ e \in \gamma'' \smallsetminus \gamma'  }  \theta_{\beta,\kappa} (\sigma_e') \biggr)\biggr]
            \\&\qquad\qquad+
            \mathbb{E}_{N,(\beta,\kappa),(\infty,\kappa)}  \biggl[\, \prod_{ e \in \gamma_1 \smallsetminus \gamma'  }  \theta_{\beta,\kappa} (\sigma_e') 
            \Bigl( 1-\prod_{ e \in \gamma_c  }  \theta_{\beta,\kappa}(\sigma_e')\Bigr) \biggr]
           \\&\qquad\qquad +
            \mathbb{E}_{N,(\beta,\kappa),(\infty,\kappa)}  \biggl[\,\prod_{ e \in \gamma_1 \smallsetminus \gamma'  }  \theta_{\beta,\kappa} (\sigma'_e) \prod_{ e \in \gamma_c  }  \theta_{\beta,\kappa}(\sigma_e')
            \Bigl( 1-\prod_{ e \in  \gamma' }  \theta_{\beta,\kappa}(\sigma_e')\Bigr)\bigg].
        \end{split}
    \end{equation*}
    Now note that
    \begin{align*}
        & \prod_{ e \in \gamma'' \smallsetminus \gamma'}  \theta_{\beta,\kappa} \bigl( \sigma_e- (d\sigma)_{p_e} \bigr)  
            -
            \prod_{ e \in \gamma'' \smallsetminus \gamma'  }  \theta_{\beta,\kappa} (\sigma_e')
            \\&\qquad=
            \prod_{\substack{ e \in \gamma'' \smallsetminus \gamma' \mathrlap{\colon}\\ \sigma_e - (d\sigma)_{p_e}=0}}  \theta_{\beta,\kappa} ( 0 ) 
            \biggl( \, \prod_{\substack{ e \in \gamma'' \smallsetminus \gamma' \mathrlap{\colon}\\ \sigma_e - (d\sigma)_{p_e}\neq 0}}   \theta_{\beta,\kappa} \bigl( \sigma_e- (d\sigma)_{p_e} \bigr)  -\prod_{\substack{ e \in \gamma'' \smallsetminus \gamma' \mathrlap{\colon}\\ \sigma_e - (d\sigma)_{p_e}\neq 0}} \theta_{\beta,\kappa}(0)\biggr) 
            \\&\qquad\qquad+
            \prod_{\substack{ e \in \gamma'' \smallsetminus \gamma' \mathrlap{\colon}\\ \sigma_e'=0  }}  \theta_{\beta,\kappa} (0) \biggl(\,\prod_{\substack{ e \in \gamma'' \smallsetminus \gamma' \mathrlap{\colon}\\ \sigma_e'\neq 0  }}  \theta_{\beta,\kappa} (0)
            - \prod_{\substack{ e \in \gamma'' \smallsetminus \gamma' \mathrlap{\colon}\\ \sigma_e'\neq 0  }}  \theta_{\beta,\kappa} (\sigma_e')\biggr),
    \end{align*}
    and similarly, that
    \begin{equation*}
        1 - \prod_{e \in \gamma_c} \theta_{\beta,\kappa}(\sigma_e')
        =
        \bigl( 1 - \theta_{\beta,\kappa}(0)^{|\gamma_c|}\bigr)
        + \prod_{\substack{e \in \gamma_c \mathrlap{\colon}\\ \sigma_e'=0}} \theta_{\beta,\kappa}(0) \biggl( \prod_{\substack{e \in \gamma_c \mathrlap{\colon}\\ \sigma_e'\neq0}} \theta_{\beta,\kappa}(0)-\prod_{\substack{e \in \gamma_c \mathrlap{\colon}\\ \sigma_e'\neq0}} \theta_{\beta,\kappa}(\sigma_e') \biggr)
    \end{equation*}
    and
    \begin{equation*}
        1 - \prod_{e \in \gamma'} \theta_{\beta,\kappa}(\sigma_e')
        =
        \bigl( 1 - \theta_{\beta,\kappa}(0)^{|\gamma'|}\bigr)
        + 
        \prod_{\substack{e \in \gamma' \mathrlap{\colon}\\ \sigma_e'=0}} \theta_{\beta,\kappa}(0) \biggl( \prod_{\substack{e \in \gamma' \mathrlap{\colon}\\ \sigma_e'\neq0}} \theta_{\beta,\kappa}(0)-\prod_{\substack{e \in \gamma' \mathrlap{\colon}\\ \sigma_e'\neq0}} \theta_{\beta,\kappa}(\sigma_e') \biggr).
    \end{equation*}
    
    Consequently, by applying the triangle inequality and recalling from Lemma~\ref{lemma: theta bound} that \( |\theta_{\beta,\kappa}(g)| \leq 1 \) for all \( g \in G \), we obtain
    \begin{equation}\label{eq: eq before claims ii}
        \begin{split} 
            &\biggl| \mathbb{E}_{N,\beta, \kappa} \Bigl[\, \prod_{ e \in \gamma_1 \smallsetminus \gamma'}  \theta_{\beta,\kappa} \bigl(  \sigma_e - (d\sigma)_{p_e} \bigr) \bigr]    - \mathbb{E}_{N,\infty,\kappa} \Bigl[ \,  \prod_{ e \in \gamma  }  \theta_{\beta,\kappa} (\sigma_e) \Bigr]  \biggr|
            \\&\qquad\leq 
            \mathbb{E}_{N,(\beta,\kappa),(\infty,\kappa)}  \biggl[ 
            \Bigl|  \prod_{\substack{ e \in \gamma'' \smallsetminus \gamma' \mathrlap{\colon} \\ \sigma_e- (d\sigma)_{p_e} \neq 0}}   \theta_{\beta,\kappa} \bigl(\sigma_e- (d\sigma)_{p_e} \bigr)   - \prod_{\substack{ e \in \gamma'' \smallsetminus \gamma' \mathrlap{\colon} \\ \sigma_e- (d\sigma)_{p_e} \neq 0}} \theta_{\beta,\kappa}(0) 
            \Bigr|\biggr]
            \\&\qquad\quad
            +
          \mathbb{E}_{N,(\beta,\kappa),(\infty,\kappa)}  \biggl[  
            \Bigl|  \prod_{\substack{ e \in \gamma''\smallsetminus \gamma' \mathrlap{\colon} \\ \sigma_e' \neq 0}}   \theta_{\beta,\kappa}(0)  - \prod_{\substack{ e \in \gamma''\smallsetminus \gamma' \mathrlap{\colon} \\ \sigma_e' \neq 0}}  \theta_{\beta,\kappa} (\sigma_e')   \Bigr|\biggr]
             \\&\qquad\quad+   
           \mathbb{E}_{N,(\beta,\kappa),(\infty,\kappa)}  \biggl[ 
            \Bigl|  \prod_{\substack{e \in \gamma_c \mathrlap{\colon}\\ \sigma_e'\neq0}} \theta_{\beta,\kappa}(0)-\prod_{\substack{e \in \gamma_c \mathrlap{\colon}\\ \sigma_e'\neq0}} \theta_{\beta,\kappa}(\sigma_e') \Bigr|\biggr]
             + 
             \bigl( 1 - \theta_{\beta,\kappa}(0)^{|\gamma_c|}\bigr) 
             \\&\qquad\quad+   
           \mathbb{E}_{N,(\beta,\kappa),(\infty,\kappa)}  \biggl[ 
            \Bigl|  \prod_{\substack{e \in \gamma' \mathrlap{\colon}\\ \sigma_e'\neq0}} \theta_{\beta,\kappa}(0)-\prod_{\substack{e \in \gamma' \mathrlap{\colon}\\ \sigma_e'\neq0}} \theta_{\beta,\kappa}(\sigma_e') \Bigr|\biggr]
             + 
             \bigl( 1 - \theta_{\beta,\kappa}(0)^{|\gamma'|}\bigr) .
        \end{split}
    \end{equation}
    We now use Lemma~\ref{lemma: Chatterjees trick abs} to obtain upper bounds for each of the terms on the right-hand side of the previous equation. 
    To this end, note first that $\alpha_1(\beta) \leq \alpha_0(\beta)$. Moreover, for any \( e \in \gamma'' \), by definition, we have \( e \in \partial \hat \partial E_{\sigma,\sigma'} \). Consequently, for each \( e \in \gamma'' \) there is at least one edge \( e' \in \partial \hat \partial e \) such that \( e' \in E_{\sigma,\sigma'} \), and hence (by the definition of \( E_{\sigma,\sigma'} \)) if \( \sigma_e' \neq 0 \), we also have \( e \in E_{\sigma,\sigma'} \).   We can hence apply Proposition~\ref{proposition: coupling and conditions 2}, to obtain
    \begin{equation}\label{eq: gamma'' bound}
        \begin{split}
        &\mathbb{E}_{N,(\beta,\kappa),(\infty,\kappa)} \Bigl[ \bigl| \{ e \in \gamma'' \colon \sigma_e' \neq 0 \} \bigr| \Bigr] 
        \\&\qquad\leq 
        C_{c,1}  |\gamma| \alpha_2(\beta,\kappa)^6 \alpha_0(\kappa)^8 
        + 
        C_{c,1}' |\gamma| \Bigl( 18^{2} \bigl(2+ \alpha_0(\kappa)  \bigr) \alpha_0(\kappa)  \Bigr)^{\dist(\gamma,B_N^c)}.
    \end{split}
    \end{equation} 
    Next, since \( \gamma'' \subseteq \partial \hat \partial E_{\sigma,\sigma'} \), we can apply Proposition~\ref{proposition: coupling and conditions 1}, to get
    \begin{align*}
        &\mathbb{E}_{N,(\beta,\kappa),(\infty,\kappa)}\Bigl[ \bigl| \{ e \in \gamma'' \smallsetminus \gamma' \colon \sigma_e - (d\sigma)_{p_e} \neq 0 \} \bigr|\Bigr] 
        \\&\qquad \leq C_{c,2}  |\gamma| \alpha_2(\beta,\kappa)^6 \alpha_0(\kappa)^5 \max\bigl(   \alpha_0(\kappa),  \alpha_0(\beta)^6\bigr)
        + 
        C_{c,2}' |\gamma| \Bigl( 18^{2} \bigl(2+ \alpha_0(\kappa)  \bigr) \alpha_0(\kappa)  \Bigr)^{\dist(\gamma,B_N^c)}.
    \end{align*} 
    Further, by Proposition~\ref{proposition: minimal Ising}, we have
    \begin{equation*}
        \mathbb{E}_{N,(\beta,\kappa),(\infty,\kappa)} \Bigl[ \bigl| \{ e \in \gamma_c \colon \sigma_e' \neq 0 \} \bigr| \Bigr] \leq  C_I |\gamma_c| \alpha_0(\kappa)^8,
    \end{equation*}
    by Proposition~\ref{proposition: the vortex flipping lemma} we have
    \begin{equation*}
        \mathbb{E}_{N,\beta,\kappa}\bigl[|\gamma'|\bigr] \leq 6 C_0^{(6)}|\gamma|\alpha_2(\beta,\kappa)^6  
    \end{equation*}
    and finally, since \( \gamma' \subseteq \gamma'' \), we have
    \begin{equation*}
        \mathbb{E}_{N,(\beta,\kappa),(\infty,\kappa)}\Bigl[\bigl|\{e \in \gamma' \colon \sigma_e' \neq 0\}\bigr|\Bigr] \leq \mathbb{E}_{N,(\beta,\kappa),(\infty,\kappa)}\Bigl[\bigl|\{e \in \gamma'' \colon \sigma_e' \neq 0\}\bigr|\Bigr],
    \end{equation*}
    the right-hand side of which we have given an upper bound for in~\eqref{eq: gamma'' bound}.
    Applying Lemma~\ref{lemma: Chatterjees trick abs}, the desired conclusion~\eqref{eq: E} now follows from~\eqref{eq: eq before claims ii}. 
\end{proof}

\section{Proofs of the main results}\label{sec: main result}

The purpose of this section is to prove our two main theorems. We first prove Theorem~\ref{theorem: the main result} which provides the leading order term of $\mathbb{E}_{N,\beta,\kappa} [W_\gamma]$ under the assumption that $G = \mathbb{Z}_n$ for any integer $n \geq 2$. We then set $n = 2$ and prove Theorem~\ref{theorem:  main result Z2}.

\subsection{Proof of Theorem~\ref{theorem: main result}} 
The proof of Theorem~\ref{theorem: main result} uses two lemmas. The first lemma is obtained by combining the results of Sections~\ref{sec:WtoWprime},~\ref{sec: resampling}, and~\ref{sec:applyingcoupling}. 

\begin{lemma}\label{lemma: C}
    Let \( \beta,\kappa \geq 0 \) be such that~\ref{assumption: 1},~\ref{assumption: 2}, and~\ref{assumption: 3} hold, and let \( \gamma \) be a simple loop in \( E_N \) such that \( \partial \hat \partial \partial \hat \partial \gamma \subseteq E_N \).
    Then
    \begin{equation}\label{eq: first part of main result}
        \begin{split}
            &\Bigl|\mathbb{E}_{N,\beta,\kappa} [W_\gamma]- \mathbb{E}_{N,\infty,\kappa} \Bigl[ \, \prod_{e \in \gamma } \theta_{\beta,\kappa}(\sigma_e) \Bigr] \Bigr| 
            \\&\qquad \leq  
            4C_1C_0^{(7)} |\gamma| \alpha_2(\beta,\kappa)^{7} + 2C_3|\gamma_c| \alpha_2(\beta,\kappa)^{6} + 2C_0^{(25)} |\gamma|^4 \alpha_2(\beta,\kappa)^{25}
            \\&\qquad\qquad+
            2\sqrt{2 C_{c,2}  |\gamma| \alpha_4(\beta,\kappa) \alpha_2(\beta,\kappa)^6 \alpha_0(\kappa)^5 \max(\alpha_0(\kappa),\alpha_0(\beta)^6)  } 
            \\&\qquad\qquad + 
            4\sqrt{2 C_{c,1} |\gamma| \alpha_4(\beta,\kappa)     \alpha_2(\beta,\kappa)^6     \alpha_0(\kappa)^8  } 
             + 
            2\sqrt{2C_I|\gamma_c| \alpha_0(\kappa)^8\alpha_4(\beta,\kappa)}
            \\&\qquad\qquad+
            2\sqrt{2|\gamma_c| \alpha_3(\beta,\kappa)}
            + 2\sqrt{12C_0^{(6)}|\gamma|\alpha_2(\beta,\kappa)^6 \alpha_3(\beta,\kappa)}
            \\&\qquad\qquad+
            (2C_{c,1}'+C_{c,2}') |\gamma| \Bigl( 18^{2} \bigl(2+ \alpha_0(\kappa)  \bigr) \alpha_0(\kappa)  \Bigr)^{\dist(\gamma,B_N^c)}.
        \end{split}
    \end{equation}
\end{lemma}
\begin{proof}
Using the triangle inequality to write
    \begin{equation*}
        \begin{split}
         &\Bigl|\mathbb{E}_{N,\beta,\kappa} [W_\gamma]- \mathbb{E}_{N,\infty,\kappa} \Bigl[ \, \prod_{e \in \gamma } \theta_{\beta,\kappa}(\sigma_e) \Bigr] \Bigr| 
        \\&\qquad\leq 
         \bigl|\mathbb{E}_{N,\beta,\kappa} [W_\gamma]- \mathbb{E}_{N,\beta,\kappa} [W_{\gamma}'] \bigr|  +\Bigl|\mathbb{E}_{N,\beta,\kappa} \bigl[W_{\gamma}' \bigr]- \mathbb{E}_{N,\infty,\kappa} \Bigl[ \, \prod_{e \in \gamma } \theta_{\beta,\kappa}(\sigma_e) \Bigr] \Bigr|,
        \end{split}
    \end{equation*}
     the desired conclusion follows immediately from Proposition~\ref{prop: first part of proposition proof}, Proposition~\ref{prop: resampling in main proof} and Proposition~\ref{proposition: E}.
\end{proof}

\begin{remark}
    Each constant in Lemma~\ref{lemma: C} is defined in  previous sections. More precisely, $C_0^{(6)}$, \( C_0^{(7)} \) and \( C_0^{(25)} \) are defined in~\eqref{eq: C0M}, \( C_3 \) is defined in~\eqref{eq: C3 def}, \( C_I \) is defined in~\eqref{eq: minimal Ising}, \( C_{c,1} \) is defined in~\eqref{eq: Cc1}, \( C_{c,2} \) is defined in~\eqref{eq: Cc1}, \( C_1 \) is defined in~\eqref{eq: C1}, \( C_{c,1}' \) is defined in~\eqref{eq: Cc1'}, and \( C_{c,2}'\) is defined in~\eqref{eq: Cc2'}.
\end{remark}

The following lemma generalizes \cite[Lemma~7.12]{c2019} and \cite[Lemma~3.3]{flv2020}.  
  
\begin{lemma}\label{lemma: upper bound} 
    Let \( \beta,\kappa \geq 0 \), and let \( \gamma \) be a simple loop in \( E_N \) such that \( \partial \hat \partial \gamma \subseteq E_N \). Then
    \begin{equation*}
        |\mathbb{E}_{N,\beta,\kappa} [ W_\gamma ]| \leq   e^{-   |\gamma \smallsetminus \gamma_c| \alpha_5(\beta,\kappa)} .
    \end{equation*} 
\end{lemma}
 
\begin{proof}
    As in Proposition~\ref{prop: resampling in main proof}, let \( \gamma_1 \) be the set of all non-corner edges of \( \gamma \), and let \( \mu_{N,\beta,\kappa}' \) denote conditional probability and conditional expectation given \( (\sigma_e)_{e \notin \pm \gamma_1 } \). As observed in Lemma~\ref{lemma: independent spins}, \( (\sigma_e)_{e \in \gamma_1} \) are independent spins under this conditioning.
 
    Take any \( e \in \gamma_1 \).
    Then, for any \( g' \in  G \), we have
    \begin{equation*}
        \mu_{N,\beta,\kappa}'(\sigma_e = g') =  \frac{ \varphi_{2\kappa} (g')  \prod_{p \in \hat \partial  e} \varphi_{2\beta} (g' + \sum_{e' \in \partial p \smallsetminus \{ e \}} \sigma_{e'}) }{\sum_{g \in G}\varphi_{2\kappa} (g)  \prod_{p \in \hat \partial  e} \varphi_{2\beta} (g + \sum_{e' \in \partial p \smallsetminus \{ e \}} \sigma_{e'}) } .
    \end{equation*} 
    This implies in particular that the expected value of \( \rho(\sigma_e) \) under \( \mu_{N,\beta,\kappa}' \) is given by
   \begin{equation*}
        \begin{split}
            &\mathbb{E}_{N,\beta,\kappa}'\bigl[\rho (\sigma_e ) \bigr] =    
            \frac{ \sum_{g \in G} \rho(g)\varphi_{2\kappa} (g)  \prod_{p \in \hat \partial  e} \varphi_{2\beta} (g + \sum_{e' \in \partial p \smallsetminus \{ e \}} \sigma_{e'}) }{\sum_{g \in G}\varphi_{2\kappa} (g)  \prod_{p \in \hat \partial  e} \varphi_{2\beta} (g + \sum_{e' \in \partial p \smallsetminus \{ e \}} \sigma_{e'}) } .
        \end{split}
    \end{equation*} 
    Recalling~\eqref{eq: alpha5}, this yields
    \begin{equation*} 
       \bigl|\mathbb{E}_{N,\beta,\kappa}'[\rho(\sigma_e) ]\bigr| 
        \leq  1 -   \alpha_5(\beta,\kappa) \leq e^{- \alpha_5(\beta,\kappa)}.
    \end{equation*}
    Since the spins \( (\sigma_e)_{ e \in \gamma_1} \) are independent given this conditioning, we obtain,  
    \begin{equation*}
        \bigl|\mathbb{E}_{N,\beta,\kappa}'\bigl[\rho (\sum_{e \in \gamma_1} \sigma_e ) \bigr]\bigr| 
        =  
        \Bigl| \mathbb{E}_{N,\beta,\kappa}'\bigl[ \prod_{e \in \gamma_1} \rho (\sigma_e )\bigr]\Bigr| 
        = 
        \prod_{e \in \gamma_1} \bigl|\mathbb{E}_{N,\beta,\kappa}'\bigl[ \rho (\sigma_e )\bigr]\bigr| 
        \leq
      e^{-   |\gamma_1| \alpha_5(\beta,\kappa)} .
    \end{equation*}
    Consequently,
    \begin{equation*}
        \begin{split}
            &\bigl|\mathbb{E}_{N,\beta,\kappa} [W_\gamma]\bigr| 
            = \Bigl|\mathbb{E}_{N,\beta,\kappa} \Bigl[\rho\bigl(\, \sum_{e \in \gamma} \sigma_e\bigr)\Bigr]\Bigr|
            = \Bigl|  \mathbb{E}_{N,\beta,\kappa} \Bigl[\, \prod_{e \in \gamma\smallsetminus \gamma_1} \rho(\sigma_e) \cdot  \mathbb{E}_{N,\beta,\kappa}'\bigl[ \, \prod_{e \in  \gamma_1} \rho(\sigma_e) \bigr]\Bigr]\Bigr|
            \\&\qquad\leq
            \mathbb{E}_{N,\beta,\kappa} \Bigl[\bigl|\prod_{e \in \gamma\smallsetminus \gamma_1} \rho(\sigma_e) \bigr| \, \bigl| \mathbb{E}_{N,\beta,\kappa}'\bigl[ \, \prod_{e \in  \gamma_1} \rho(\sigma_e) \bigr] \bigr|\Bigr] 
            \leq e^{-   |\gamma_1| \alpha_5(\beta,\kappa)} . 
        \end{split}
    \end{equation*}
    Recalling that \( \gamma_1 = \gamma \smallsetminus \gamma_c \), the desired conclusion follows.
\end{proof}

\begin{proof}[Proof of Theorem~\ref{theorem: main result}]
    Let \( N \) be sufficiently large, so that  \( \partial\hat \partial \partial \hat \partial \gamma \subseteq E_N \). 
    Then the assumptions of Lemma~\ref{lemma: C} hold, and rearranging the terms on the right-hand side of~\eqref{eq: first part of main result}, it follows that 
    \begin{equation*}
           \Bigl|\mathbb{E}_{N,\beta,\kappa} [W_\gamma]- \mathbb{E}_{N,\infty,\kappa} \Bigl[ \, \prod_{e \in \gamma } \theta_{\beta,\kappa}(\sigma_e) \Bigr] \Bigr| 
          \leq  A \cdot 2|\gamma| \alpha_5(\beta,\kappa)
          +
          A' \cdot 2\sqrt{2|\gamma| \alpha_5(\beta,\kappa)}
           + B \cdot \bigl( |\gamma| \alpha_5(\beta,\kappa) \bigr)^4,
    \end{equation*}  
    where \( A \), \( A' \), and \( B \) are short-hand notations for
    \begin{align*}
        A \coloneqq &\;
        \frac{2C_1C_0^{(7)}  \alpha_2(\beta,\kappa)^{7}}{\alpha_5(\beta,\kappa)} + \frac{ C_3|\gamma_c| \alpha_2(\beta,\kappa)^{6}}{|\gamma| \alpha_5(\beta,\kappa)}  
        + 
        2\sqrt{\frac{2 C_{c,1}  \alpha_4(\beta,\kappa)  \alpha_2(\beta,\kappa)^6     \alpha_0(\kappa)^8  }{|\gamma| \alpha_5(\beta,\kappa)^2}} 
        \\
        & +
        \sqrt{\frac{2 C_{c,2}   \alpha_4(\beta,\kappa) \alpha_2(\beta,\kappa)^6 \alpha_0(\kappa)^5 \max(\alpha_0(\kappa),\alpha_0(\beta)^6)  }{|\gamma| \alpha_5(\beta,\kappa)^2}}
        + \sqrt{\frac{12C_0^{(6)}\alpha_2(\beta,\kappa)^6 \alpha_3(\beta,\kappa)}{|\gamma|\alpha_5(\beta,\kappa)^2}}
        \\&+
        \frac{(2C_{c,1}'+C_{c,2}') \Bigl( 18^{2} \bigl(2+ \alpha_0(\kappa)  \bigr) \alpha_0(\kappa)  \Bigr)^{\dist(\gamma,B_N^c)}}{2 \alpha_5(\beta,\kappa)},
        \\
        A' \coloneqq &\; \sqrt{\frac{C_I|\gamma_c|  \alpha_0(\kappa)^8 \alpha_4(\beta,\kappa)}{|\gamma| \alpha_5(\beta,\kappa)}}
        +
        \sqrt{\frac{|\gamma_c| \alpha_3(\beta,\kappa) }{|\gamma| \alpha_5(\beta,\kappa)}},
        \\
        B \coloneqq &\; \frac{2C_0^{(25)}   \alpha_2(\beta,\kappa)^{25}}{\alpha_5(\beta,\kappa)^4}.
    \end{align*}       
    Using that for \( x>0 \), we have  \( x \leq e^x \), \( 2\sqrt{x}\leq e^x\), and  \( x^4 \leq e^{2x} \), it follows that
    \begin{equation} \label{eq: combined bounds 002}    
        \Bigl|\mathbb{E}_{N,\beta,\kappa} [W_\gamma]- \mathbb{E}_{N,\infty,\kappa} \Bigl[ \, \prod_{e \in \gamma } \theta_{\beta,\kappa}(\sigma_e) \Bigr] \Bigr| \leq (A+A'+B) e^{2 |\gamma| \alpha_5(\beta,\kappa) }.
    \end{equation}  
    Next, recall that \(|\mathbb{E}_{N,\beta,\kappa} ( W_\gamma )| \leq   e^{-   |\gamma \smallsetminus \gamma_c| \alpha_5(\beta,\kappa)}\) by Lemma~\ref{lemma: upper bound}.
    Using the triangle inequality and applying Lemma~\ref{lemma: theta inequalities Dplusii} it follows that
    \begin{equation}\label{eq: second part of main result}
        \begin{split}
            &\Bigl|\mathbb{E}_{N,\beta,\kappa}[ W_\gamma] -  \mathbb{E}_{N,\infty,\kappa} \Bigl[ \, \prod_{e \in \gamma } \theta_{\beta,\kappa}(\sigma_e) \Bigr] \Bigr| 
            \leq  
            \Bigl|\mathbb{E}_{N,\beta,\kappa}[ W_\gamma] \Bigr| 
            +
            \Bigl| \mathbb{E}_{N,\infty,\kappa} \Bigl[ \, \prod_{e \in \gamma } \theta_{\beta,\kappa}(\sigma_e) \Bigr] \Bigr|
            \\&\qquad \leq   
            e^{-  (|\gamma|-|\gamma_c|) \alpha_5(\beta,\kappa)}+ e^{ -  |\gamma| \alpha_5(\beta,\kappa)}
            \leq
            2  e^{ -  (|\gamma|-|\gamma_c|) \alpha_5(\beta,\kappa)}.
        \end{split}
    \end{equation}
    Combining~\eqref{eq: combined bounds 002} and~\eqref{eq: second part of main result}, we obtain
    \begin{equation}\label{eq: almost last equation}    \Bigl|\mathbb{E}_{N,\beta,\kappa}[ W_\gamma] -  \mathbb{E}_{N,\infty,\kappa} \Bigl[ \, \prod_{e \in \gamma } \theta_{\beta,\kappa}(\sigma_e) \Bigr] \Bigr|^{1 + 2|\gamma|/(|\gamma|-|\gamma_c|)}
             \leq
             2^{2|\gamma|/(|\gamma|-|\gamma_c|)}(A+A'+B).
    \end{equation}
    Here 
    \begin{align*}
        &A +A'+B 
        \\&\qquad \leq 
        2C_1C_0^{(7)} \varepsilon_1(\beta,\kappa)\cdot \alpha_2(\beta,\kappa)
        + C_3\varepsilon_1(\beta,\kappa) \cdot \frac{|\gamma_c|}{|\gamma|}
        +
       2\sqrt{C_{c,1}      }\, \varepsilon_1(\beta,\kappa) \varepsilon_3(\beta,\kappa)  (\sqrt{2}\alpha_0(\kappa)^3) \cdot \frac{1}{\sqrt{|\gamma|}}
        \\&\qquad\qquad+
        \sqrt{ C_{c,2}}\varepsilon_1(\beta,\kappa)\varepsilon_3(\beta,\kappa)\sqrt{  2\alpha_0(\kappa)^2 \max(\alpha_0(\kappa)^2,\alpha_2(\beta,\kappa)^6)  } \cdot \frac{1}{\sqrt{|\gamma|}}
        \\&\qquad\qquad +
        \sqrt{12C_0^{(6)}} \sqrt{\varepsilon_1(\beta,\kappa)\varepsilon_2(\beta,\kappa)^2} \cdot \frac{1}{\sqrt{|\gamma|}}
        +
        \sqrt{C_I}\varepsilon_2(\beta,\kappa) \cdot \sqrt\frac{|\gamma_c|}{|\gamma|}
        +
        2C_0^{(25)}   \varepsilon_1(\beta,\kappa)^4  \cdot \alpha_2(\beta,\kappa)
        \\&\qquad\qquad+
        \frac{(2C_{c,1}'+C_{c,2}') \Bigl( 18^{2} \bigl(2+ \alpha_0(\kappa)  \bigr) \alpha_0(\kappa)  \Bigr)^{\dist(\gamma,B_N^c)}}{2 \alpha_5(\beta,\kappa)}
        \\&\qquad\leq
        \biggl( \bigl( 2C_1C_0^{(7)} +C_3 \bigr) \varepsilon_1(\beta,\kappa) 
        +
       \Bigl( \sqrt{C_{c,1}}+\sqrt{C_{c,2}} \max\bigl( 1,\alpha_2(\beta,\kappa)^6 \bigr) \Bigr)\, \varepsilon_1(\beta,\kappa) \varepsilon_3(\beta,\kappa)   
        \\&\qquad\qquad+ 
        \sqrt{12C_0^{(6)}} \sqrt{\varepsilon_1(\beta,\kappa)\varepsilon_2(\beta,\kappa)^2}
        +
        \sqrt{C_I}\varepsilon_2(\beta,\kappa)  
        +
        2C_0^{(25)}   \varepsilon_1(\beta,\kappa)^4   \biggr) \cdot \bigl( \alpha_2(\beta,\kappa) + \sqrt{|\gamma_c|/|\gamma|}\bigr)
        \\&\qquad\qquad+
        \frac{(2C_{c,1}'+C_{c,2}') \Bigl( 18^{2} \bigl(2+ \alpha_0(\kappa)  \bigr) \alpha_0(\kappa)  \Bigr)^{\dist(\gamma,B_N^c)}}{2 \alpha_5(\beta,\kappa)},
    \end{align*}
    where we have used that,  by~\ref{assumption: 3}, we have \( 18^2 \alpha_0(\kappa)<1 \), and that \( C_I \geq 1 \) by~\ref{assumption: 3}.
    Recalling Proposition~\ref{proposition: unitary gauge one dim} and Proposition~\ref{proposition: limit exists}, and letting \( N \to \infty \), the desired conclusion thus follows from~\eqref{eq: almost last equation} after simplification.
\end{proof}

\subsection{Proof of Theorem~\ref{theorem: main result Z2}}\label{subsec: Z2} 

Before we give a proof of Theorem~\ref{theorem: main result Z2}, we note that when \( G = \mathbb{Z}_2 \) and \( \rho(G) = \{1,-1\} \), then
\begin{equation}\label{eq: theta for Z2}
    \theta_{\beta,\kappa}(0) = \frac{1-e^{-24\beta-4\kappa}}{1+e^{-24\beta-4\kappa}}
    \quad \text{and} \quad\theta_{\beta,\kappa}(1) = \frac{1-e^{-24\beta+4\kappa}}{1+e^{-24\beta+4\kappa}}.
\end{equation}
We also note that in this setting, the function $\alpha_5$ defined in~\eqref{eq: alpha5} can be written as
\begin{align*}
& \alpha_5(\beta,\kappa) 
= 1 - \max_{g_1,g_2, \ldots, g_6 \in G} \biggl|  \frac{\bigl(\prod_{k =1}^6 \varphi_\beta(g_k)^2\bigr) \varphi_\kappa(0)^2- \bigl(\prod_{k =1}^6 \varphi_\beta(1+g_k)^2\bigr) \varphi_\kappa(1)^2}{\bigl(\prod_{k =1}^6 \varphi_\beta(g_k)^2\bigr) \varphi_\kappa(0)^2 +  \bigl(\prod_{k =1}^6 \varphi_\beta(1+g_k)^2\bigr) \varphi_\kappa(1)^2} \biggr|
    \\
&\qquad
= 1 - \max_{g_1,g_2, \ldots, g_6 \in G} \biggl|\frac{2}{1 +  \bigl(\prod_{k =1}^6 \frac{\varphi_\beta(1+g_k)^2}{\varphi_\beta(g_k)^2}\bigr) \cdot \frac{\varphi_\kappa(1)^2}{\varphi_\kappa(0)^2}}-1 \biggr|.
\end{align*}
Since the expression within the absolute value signs on the right-hand side is a decreasing function of $\prod_{k =1}^6 \frac{\varphi_\beta(1+g_k)^2}{\varphi_\beta(g_k)^2}$, we conclude that the maximum is obtained when all the $g_k$ are equal to $1$ or when all the $g_k$ are equal to $0$. A comparison of these two cases, shows that the maximum is obtained when $g_1 = \cdots = g_6 = 0$. It follows that, for $n = 2$, the definitions  of \( \alpha_0 \), \( \alpha_1 \), \( \alpha_2 \), \( \alpha_3 \), \( \alpha_4 \), and \( \alpha_5 \) (see~\eqref{eq: alpha01def}--\eqref{eq: alpha5}) reduce to
\begin{equation}\label{eq: alphas for Z2}
    \begin{split}
        &  \alpha_0(r) = \alpha_1(r) = \varphi_r(1)^2 = e^{-4r},
        \qquad    \alpha_2(\beta,\kappa) = e^{-4(\beta+\kappa/6)},
        \\
        &
        \alpha_3(\beta,\kappa) = 
        1 - \theta_{\beta,\kappa}(0)  =
        \frac{ 2e^{-24\beta-4\kappa}}{1 + e^{-24\beta-4\kappa}},
        \qquad
        \alpha_5(\beta,\kappa) = 
        1 - \theta_{\beta,\kappa}(0)  =
        \frac{ 2e^{-24\beta-4\kappa}}{1 + e^{-24\beta-4\kappa}},
        \\
        &    \alpha_4(\beta,\kappa) = 
        \theta_{\beta,\kappa}(0) -  \theta_{\beta,\kappa}(1) 
        = 
        \frac{2 e^{-24\beta} (e^{4\kappa}-e^{-4\kappa})}{(1 + e^{-24\beta-4\kappa})(1 + e^{-24\beta+4\kappa})}.
    \end{split}
\end{equation}

\begin{proof}[Proof of Theorem~\ref{theorem: main result Z2}]  
    Let \( N \) be sufficiently large, so that  \( \partial\hat \partial \partial \hat \partial \gamma \subseteq E_N \).
    
    Using the expression~\eqref{eq: theta for Z2} for \( \theta_{\beta,\kappa}(0) \), we have
    \begin{align*}
        &\bigl|\theta_{\beta,\kappa}(0) - e^{-2e^{-24\beta-4\kappa}}\bigr|
        \leq
        \bigl|\theta_{\beta,\kappa}(0) - (1-2e^{-24\beta-4\kappa})\bigr|
        +
        \bigl|(1-2e^{-24\beta-4\kappa})-e^{-2e^{-24\beta-4\kappa}}\bigr| 
        \\&\qquad \leq 
        2e^{-48\beta-8\kappa}
        +
        2e^{-48\beta-8\kappa}
        =
        4e^{-48\beta-8\kappa}.
    \end{align*}
    Completely analogously, we also have
    \begin{align*}
        &\bigl|\theta_{\beta,\kappa}(1) - e^{-2e^{-24\beta+4\kappa}}\bigr|
        \leq
        \bigl|\theta_{\beta,\kappa}(1) - (1-2e^{-24\beta+4\kappa})\bigr|
        +
        \bigl|(1-2e^{-24\beta+4\kappa})-e^{-2e^{-24\beta+4\kappa}}\bigr| 
        \\&\qquad \leq 
        2e^{-48\beta+8\kappa}
        +
        2e^{-48\beta+8\kappa}
        =
        4e^{-48\beta+8\kappa}.
    \end{align*}
    Combining these estimates and using Lemma~\ref{lemma: Chatterjees inequality ii} several times, it follows  that for any integer \( m \in [0,|\gamma|] \), we have
    \begin{align*}     
        &\bigl|\theta_{\beta,\kappa}(0)^{|\gamma|-m}\theta_{\beta,\kappa}(1)^{m}  - e^{-2(|\gamma|-m)e^{-24\beta-4\kappa}}e^{-2me^{-24\beta+4\kappa}}  \bigr| \leq 
        4\bigl(|\gamma|-m\bigr) e^{-48\beta-8\kappa} + 4m e^{-48\beta+8\kappa}
        \\&\qquad 
        \leq
        4|\gamma| e^{-48\beta+8\kappa}
        \leq 
        2e^{-24\beta+12\kappa} \cdot 2|\gamma|\alpha_5(\beta,\kappa)
        \leq 
        2e^{-24\beta+12\kappa} \cdot e^{2|\gamma|\alpha_5(\beta,\kappa)},
    \end{align*}
    and hence, choosing $m = |\{ e \in \gamma \colon \sigma_e = 1\}|$,
    \begin{equation}\label{eq: rewriting theta term 2}
        \biggl| \mathbb{E}_{N,\infty,\kappa}\Bigl[\, \prod_{e \in \gamma} \theta_{\beta,\kappa}(\sigma_e) -  e^{-2|\gamma|e^{-24\beta-4\kappa}  (1 + \frac{|\{ e \in \gamma \colon \sigma_e = 1\}|}{|\gamma|}(e^{ 8\kappa}-1))}\Bigr] \biggr|
        \leq 
        2e^{-24\beta+12\kappa} \cdot e^{2|\gamma|\alpha_5(\beta,\kappa)}.
    \end{equation}
    Using~\eqref{eq: theta for Z2}, it follows that if $\beta$ and $\kappa$ satisfy the assumptions of Theorem~\ref{theorem: main result Z2}, then~\ref{assumption: 1},~\ref{assumption: 2}, and~\ref{assumption: 3} all hold. For this reason, we can proceed as in the proof of Theorem~\ref{theorem: main result}. Combining~\eqref{eq: combined bounds 002} with~\eqref{eq: rewriting theta term 2},
    we obtain
    \begin{equation} \label{eq: combined bounds 002 ii}       
        \Bigl|\mathbb{E}_{N,\beta,\kappa} [W_\gamma]-   \mathbb{E}_{N,\infty,\kappa}\bigl[ e^{-2|\gamma|e^{-24\beta-4\kappa}  (1 + \frac{|\{ e \in \gamma \colon \sigma_e = 1\}|}{|\gamma|}(e^{ 8\kappa}-1))}\bigr]   \Bigr| \leq (A+A'+B + 2e^{-24\beta+12\kappa}  ) e^{2 |\gamma| \alpha_5(\beta,\kappa) }.
    \end{equation}

    Next, recall that $|\mathbb{E}_{N,\beta,\kappa} [ W_\gamma ]| \leq   e^{-   |\gamma \smallsetminus \gamma_c| \alpha_5(\beta,\kappa)}$ by Lemma~\ref{lemma: upper bound}. At the same time, using~\eqref{eq: alphas for Z2}, we have
    \begin{equation*}
        \mathbb{E}_{N,\infty,\kappa}\bigl[ e^{-2|\gamma|e^{-24\beta-4\kappa}  (1 + \frac{|\{ e \in \gamma \colon \sigma_e = 1\}|}{|\gamma|}(e^{ 8\kappa}-1))}\bigr] 
        \leq 
        e^{-2|\gamma|e^{-24\beta-4\kappa}} 
        = 
        e^{-|\gamma|\alpha_5(\beta,\kappa)(1+e^{-24\beta-4\kappa})}  
        \leq
        e^{-|\gamma|\alpha_5(\beta,\kappa)}.
    \end{equation*}
    Using the triangle inequality, it follows that
    \begin{equation}\label{eq: second part of main result ii}
        \begin{split}
            &\Bigl|\mathbb{E}_{N,\beta,\kappa} [W_\gamma]-   \mathbb{E}_{N,\infty,\kappa}\bigl[ e^{-2|\gamma|e^{-24\beta-4\kappa}  (1 + \frac{|\{ e \in \gamma \colon \sigma_e = 1\}|}{|\gamma|}(e^{ 8\kappa}-1))}\bigr]   \Bigr|
            \\&\qquad\leq 
            \Bigl|\mathbb{E}_{N,\beta,\kappa} [W_\gamma] \Bigr|
            +
            \Bigl| \mathbb{E}_{N,\infty,\kappa}\bigl[ e^{-2|\gamma|e^{-24\beta-4\kappa}  (1 + \frac{|\{ e \in \gamma \colon \sigma_e = 1\}|}{|\gamma|}(e^{ 8\kappa}-1))}\bigr]   \Bigr|
            \\&\qquad\leq e^{-(|\gamma|- |\gamma_c|) \alpha_5(\beta,\kappa)} + e^{-|\gamma|\alpha_5(\beta,\kappa)} \leq 2e^{-(|\gamma|- |\gamma_c|) \alpha_5(\beta,\kappa)}.
        \end{split}
    \end{equation}
    Combining~\eqref{eq: combined bounds 002 ii} and~\eqref{eq: second part of main result ii}, we obtain
    \begin{equation*}    
    \begin{split}
        &\Bigl|\mathbb{E}_{N,\beta,\kappa}[ W_\gamma] - \mathbb{E}_{N,\infty,\kappa}\bigl[ e^{-2|\gamma|e^{-24\beta-4\kappa}  (1 + \frac{|\{ e \in \gamma \colon \sigma_e = 1\}|}{|\gamma|}(e^{ 8\kappa}-1))}\bigr] \Bigr|^{1 + 2|\gamma|/(|\gamma|-|\gamma_c|)}
        \\&\qquad\leq
        2^{2|\gamma|/(|\gamma|-|\gamma_c|)}\bigl(A+A'+B+2e^{-24\beta+12\kappa} \bigr). 
    \end{split}
    \end{equation*}
    Now note that, by~\eqref{eq: alphas for Z2}, we have
    \begin{equation*}
        \varepsilon_1(\beta,\kappa) = \frac{\alpha_2(\beta,\kappa)^6}{\alpha_5(\beta,\kappa)} = \frac{{1 + e^{-24\beta-4\kappa}}}{2} \leq 1,
    \qquad   
        \varepsilon_3(\beta,\kappa) = \frac{  \alpha_4(\beta,\kappa) \alpha_0(\kappa)^2}{\alpha_5(\beta,\kappa)} =   
        \frac{1-e^{-8\kappa}}{1 + e^{-24\beta+4\kappa}} \leq 1,
    \end{equation*}
    and
    \begin{equation*} 
     \varepsilon_2(\beta,\kappa) =  \sqrt{\frac{\alpha_0(\kappa)^8 \alpha_4(\beta,\kappa)}{ \alpha_5(\beta,\kappa)}}
         +
         \sqrt{\frac{\alpha_3(\beta,\kappa) }{\alpha_5(\beta,\kappa)}}  = \sqrt{ \frac{  e^{-24\kappa}(1-e^{-8\kappa})}{1 + e^{-24\beta+4\kappa}}} + 1 \leq 2.
\end{equation*}

    Consequently, we have
    \begin{equation*}
        \begin{split}
            &A +A'+B 
            \leq
        2C_1C_0^{(7)} \cdot \alpha_2(\beta,\kappa)
        + C_3\cdot \frac{|\gamma_c|}{|\gamma|}
        +
       2\sqrt{C_{c,1} }\,  (\sqrt{2}\alpha_0(\kappa)^3) \cdot \frac{1}{\sqrt{|\gamma|}}
        \\&\qquad\qquad+
        \sqrt{ C_{c,2}} \sqrt{  2\alpha_0(\kappa)^2 \max(\alpha_0(\kappa)^2,\alpha_2(\beta,\kappa)^6)  } \cdot \frac{1}{\sqrt{|\gamma|}}
        +
        2\sqrt{12C_0^{(6)}} \cdot \frac{1}{\sqrt{|\gamma|}}
      + 2\sqrt{C_I} \cdot \sqrt\frac{|\gamma_c|}{|\gamma|}
        \\&\qquad\qquad+
        2C_0^{(25)} \cdot \alpha_2(\beta,\kappa)
        +
        \frac{(2C_{c,1}'+C_{c,2}') \Bigl( 18^{2} \bigl(2+ \alpha_0(\kappa)  \bigr) \alpha_0(\kappa)  \Bigr)^{\dist(\gamma,B_N^c)}}{2 \alpha_2(\beta,\kappa)}.
        \end{split}
    \end{equation*}
    Note now that by definition, \( |\gamma_c|\geq 4. \) Also, note that by assumption, we have \( 2\kappa \leq 3\beta  \), and hence \( e^{-24\beta+12\kappa} \leq e^{-4(\beta+\kappa/6)} = \alpha_2(\beta,\kappa)\leq 1\). Finally, note that by~\ref{assumption: 3}, we have \( 18^2\alpha_0(\kappa) \leq 1 \). 
    Recalling Proposition~\ref{proposition: unitary gauge one dim} and Proposition~\ref{proposition: limit exists}, letting \( N \to \infty \), and simplifying, we thus obtain
    \begin{equation*}
        \begin{split}
            &\bigl| \langle W_\gamma\rangle_{\beta,\kappa} - \Upsilon_{\gamma,\beta,\kappa}  \bigr|^{1 + 2|\gamma|/(|\gamma|-|\gamma_c|)}
            \\&\qquad \leq
            2^{2|\gamma|/(|\gamma|-|\gamma_c|)}
            \cdot 
            \bigl( 2C_1C_0^{(7)} + C_3  + \sqrt{ C_{c,1} }  + \sqrt{ C_{c,2}  } +
            2\sqrt{12C_0^{(6)}}
            + 2\sqrt{C_I}+2C_0^{(25)}+2
            \bigr)
            \\&\qquad \qquad
             \cdot \Bigl( e^{-4(\beta+\kappa/6)} + \sqrt{|\gamma_c|/|\gamma|}\Bigr).
        \end{split}
    \end{equation*} 
    
    If \( |\gamma_c|/|\gamma| \leq a \in (0,1) \), then 
    \begin{equation*}
        \frac{1-a}{3-a} \leq \frac{1}{1+2|\gamma|/(|\gamma|-|\gamma_c|)} \leq \frac{1}{3}.
    \end{equation*} 
    Consequently, in this case, if \( e^{-4(\beta+\kappa/6)} + \sqrt{|\gamma_c|/|\gamma|}\leq 1 \), then it follows that
    \begin{align}\nonumber
            &\bigl| \langle W_\gamma\rangle_{\beta,\kappa} - \Upsilon_{\gamma,\beta,\kappa}  \bigr|
            \\\nonumber 
            &\qquad \leq
            2^{1-\frac{1-a}{3-a}}
            \cdot 
            \bigl\{ 2C_1C_0^{(7)} + C_3  + \sqrt{ C_{c,1} }  + \sqrt{ C_{c,2}  } +2\sqrt{12C_0^{(6)}}+2\sqrt{C_I}+2C_0^{(25)}+2
            \bigr\}^{1/3} 
            \\\label{eq: Z2 main result before simplification}
            & \qquad\qquad
            \cdot \Bigl( e^{-4(\beta+\kappa/6)} + \sqrt{|\gamma_c| /|\gamma|}\Bigr)^{\frac{1-a}{3-a}}.
    \end{align}  
   Since \( |\rho(g)| = 1 \) for all $g \in G$, we always have
    \begin{equation*}
        \begin{split}
            &\bigl| \langle W_\gamma\rangle_{\beta,\kappa} - \Upsilon_{\gamma,\beta,\kappa}  \bigr| \leq 2.
        \end{split}
    \end{equation*}  
    On the other hand, \( C_0^{(25)} \geq 5^{24}\) by~\eqref{eq: C0M}, and hence \( (C_0^{(25)})^{1/3} \geq 5^8 \). This implies in particular that if \( e^{-4(\beta+\kappa/6)} + \sqrt{|\gamma_c|/|\gamma|}\geq 1 \), then~\eqref{eq: Z2 main result before simplification} automatically holds. Also it follows that if \( |\gamma_c|/|\gamma| > a \), then~\eqref{eq: Z2 main result before simplification} automatically holds if \( a \) is not too small. One verifies that it is sufficient to choose \( a \geq 5^{-42} \). Choosing \( a = 1/5 \) and noting that the assumption that \( 3\beta \geq 2\kappa \) implies that \( e^{-24\beta+12\kappa} \leq 1 \), we obtain
    the desired estimate~\eqref{eq: main result Z2}. This completes the proof of Theorem~\ref{theorem: main result Z2}.
\end{proof}

\noindent
{\bf Acknowledgements} {\it MPF and JL acknowledge support from the Ruth and Nils-Erik Stenb\"ack Foundation, the Swedish Research Council, Grant No. 2015-05430, and the European Research Council, Grant Agreement No. 682537.
FV acknowledges support from the Ruth and Nils-Erik Stenb\"ack Foundation, the Swedish Research Council, and the Knut and Alice Wallenberg Foundation.}

\end{document}